\documentclass[11pt]{amsart}
\usepackage{geometry}                
\usepackage[applemac]{inputenc}
\usepackage{graphicx, color, xypic}
\usepackage[all]{xy}
\usepackage{amsmath, amssymb,amsfonts, amsthm,amscd,latexsym,amstext,mathrsfs}
\usepackage[ german, frenchb, english]{babel}
\newtheorem{thm}{Theorem}[section]
\newtheorem{prop}[thm]{Proposition}
\newtheorem{cor}[thm]{Corollary}
\newtheorem{lem}[thm]{Lemma}

\newtheorem{rem}[thm]{Remark}

\usepackage{float}
\usepackage{epsfig}
\restylefloat{table}
\usepackage{rotating}
\usepackage{chngcntr}
\counterwithout{table}{section}
\newcommand{\OP}[1]{\OOP\left(#1\right)} 
\newcommand{\OX}[1]{\OOX\left(#1\right)} 
\newcommand{\OMP}[1]{\OOMP\left(#1\right)} 
\newcommand{\OMX}[1]{\OOMX\left(#1\right)} 
\usepackage{colortbl}
\newcommand{\SL}{\mathrm{SL}_2} 
\newcommand{\GL}{\mathrm{GL}_2} 
\newcommand{\PGL}{\mathrm{PGL}_2} 
\newcommand{\Kum}{\mathrm{Kum}} 
\newcommand{\Wed}{\mathrm{Wed}}

\DeclareMathOperator{\Hom}{Hom}
\DeclareMathOperator{\elm}{elm}

\DeclareMathOperator{\bun}{bun}
\DeclareMathOperator{\Con}{{Con}}
\DeclareMathOperator{\Bun}{Bun}
\DeclareMathOperator{\ind}{ind}
\DeclareMathOperator{\tr}{tr}

\usepackage{tikz}

\newcommand{\pgftextcircled}[1]{
    \setbox0=\hbox{#1}%
    \dimen0\wd0%
    \divide\dimen0 by 2%
    \begin{tikzpicture}[baseline=(a.base)]%
        \useasboundingbox (-\the\dimen0,0pt) rectangle (\the\dimen0,1pt);
        \node[circle,draw,outer sep=0pt,inner sep=0.1ex] (a) {#1};
    \end{tikzpicture}
}

\def\Z{\mathbb{Z}}

\def\R{\mathbb{R}}
\def\C{\mathbb{C}}
\def\P{\mathbb{P}}

\def\OO{\mathcal O}
\def\OOP{\OO_{\P^1}}
\def\OOX{\OO_X}
\def\OOM{\Omega^1}
\def\OOMP{\OOM_{\P^1}}
\def\OOMX{\OOM_X}
\def\KX{ \mathrm{K}_X}

\def\p{\boldsymbol{p}}
\def\s{\boldsymbol{s}}
\def\q{\boldsymbol{q}}
\def\b{\boldsymbol{b}}
\def\Mu{\boldsymbol{\mu}}

\def\CON{\mathfrak{Con}}
\def\BUN{\mathfrak{Bun}}
\def\HIGGS{\mathfrak{Higgs}}

\begin{document}

\let\textcircled=\pgftextcircled

\title[Flat rank 2 vector bundles over genus 2 curves]{Flat rank 2 vector bundles on genus 2 curves}
  
\author[V.Heu]{Viktoria Heu}
\address{IRMA, 7 rue Ren\'e-Descartes, 67084 Strasbourg Cedex, France}
\email{heu@math.unistra.fr}

\author[F.Loray]{Frank LORAY}
\address{IRMAR, Campus de Beaulieu, 35042 Rennes Cedex, France}
\email{frank.loray@univ-rennes1.fr}

\date{\today}


\thanks{\noindent The first author is supported by the ANR grants ANR-13-BS01-0001-01 and ANR-13-JS01-0002-01. \\ The second author is supported by CNRS}

\begin{abstract}
We study the moduli space of trace-free irreducible rank 2 connections over a curve of genus 2 and the forgetful map towards 
the moduli space of underlying vector bundles (including unstable bundles), for which we compute a natural Lagrangian rational section. As a particularity of the genus $2$ case, 
 connections as above are invariant under the hyperelliptic involution : they descend as rank $2$ logarithmic connections over the Riemann sphere. 
We establish explicit links between the well-known moduli space of the underlying parabolic bundles with the classical approaches 
by Narasimhan-Ramanan, Tyurin and Bertram.
This allow us to explain a certain number of geometric phenomena in the considered moduli spaces such as the classical $(16,6)$-configuration of the Kummer surface. We also recover a Poincar\'e family due to Bolognesi
on a degree 2 cover of the Narasimhan-Ramanan moduli space. We explicitly compute the Hitchin integrable system
on the moduli space of Higgs bundles and compare the Hitchin Hamiltonians with those found by vanGeemen-Previato. We explicitly describe the isomonodromic foliation in the moduli space of vector bundles with   $\mathfrak{sl}_2{\C}$-connection over curves of genus 2 and prove the transversality of the induced flow with the locus of unstable bundles. \end{abstract}

\maketitle
\tableofcontents
\section*{Introduction}

Let $X$ be a smooth projective curve of genus $2$ over $\C$. 
A rank $2$ holomorphic connection on $X$
is the data $\left(E,\nabla\right)$ of a rank $2$ vector bundle $E\to X$ together with a $\C$-linear map
$\nabla: E\to E\otimes\OOMX$ satisfying the Leibniz rule. The trace $\tr\left(\nabla\right)$ defines a holomorphic
connection on $\det\left(E\right)$; we say that $\left(E,\nabla\right)$ is  trace-free (or a $\mathfrak{sl}_2$-connection)
when $\left(\det\left(E\right),\tr\left(\nabla\right)\right)$ is the trivial connection $\left(\OOX,\mathrm{d}z\right)$.
From the analytic point of view, $\left(E,\nabla\right)$ is determined (up to bundle isomorphism) by its monodromy representation, 
\emph{i.e.} an element of $\Hom\left(\pi_1\left(X\right),\SL\right)/_{\PGL}$ (up to conjugacy).
The goal of this paper is the explicit construction and study of the moduli stack $\CON\left(X\right)$ of these
connections and in particular the forgetful map $\left(E,\nabla\right)\mapsto E$ towards the moduli stack $\BUN\left(X\right)$ of  vector bundles that can be endowed with connections. Over an open set of the base, the map $\bun:\CON\left(X\right)\to\BUN\left(X\right)$ is known to be
an affine $\mathbb A^3$-bundle.
The former moduli space may be constructed by Geometric Invariant Theory  (see \cite{Nitsure,InabaIwasakiSaito1,InabaIwasakiSaito2})
and we get a quasi-projective variety $\Con^{ss}\left(X\right)$ whose stable locus $\Con^{s}\left(X\right)$ is open, smooth and
parametrizes equivalence classes of irreducible connections. In the strictly semi-stable locus however, several equivalence classes of reducible connections 
may be identified to the same point. 

The moduli space of bundles, even after restriction to the moduli space $\BUN^{\mathrm{irr}}(X)$ of those bundles admitting an irreducible connection,
is non Hausdorff as a topological space, due to the fact that some unstable bundles arise in this way. We can start with
the classical moduli space $\Bun^{ss}\left(X\right)$ of semi-stable bundles constructed by Narasimhan-Ramanan (see \cite{NR}), 
but we have to investigate how to complete this picture with missing flat unstable bundles.

{\bf Hyperelliptic descent.}
The main tool of our study, elaborated in Section \ref{SecMainConstruction}, 
directly follows from the hyperellipticity property of such objects. Denote by 
$\iota:X\to X$ the hyperelliptic involution, by $\pi:X\to\P^1$ the quotient map and by $\underline{W}$
the critical divisor on $\P^1$ (projection of the $6$ Weierstrass points).
We can think of $\P^1=X/\iota$ as an orbifold quotient  (see \cite{Uniformisation}) and any representation
$\rho\in\Hom\left(\pi_1^{\mathrm{orb}}\left(X/\iota\right),\GL\right)$ of the 
orbifold fundamental group, \emph{i.e.} with $2$-torsion around points of $\underline{W}$, can be lifted on $X$ 
to define an element $\pi^*\rho$ in $\Hom\left(\pi_1\left(X\right),\SL\right)$. As a particularity of the genus 2 case, both 
 moduli spaces of representations have the same dimension $6$ and one can check that the map $\Hom\left(\pi_1^{\mathrm{orb}}\left(X/\iota\right),\GL\right)\to\Hom\left(\pi_1\left(X\right),\SL\right)$ is dominant:
any irreducible $\SL$-representation of the fundamental group of $X$
is in the image, is invariant under 
the hyperelliptic involution $\iota$ and can be pushed down to $X/\iota$.

From the point of view of connections, this means that every irreducible connection $\left({E},{\nabla}\right)$ on $X$
is invariant by the hyperelliptic involution $\iota:X\to X$. By pushing forward $\left({E},{\nabla}\right)$ to the quotient
$X/\iota\simeq\P^1$, we get a rank $4$ logarithmic connection that splits into the direct sum
$\pi_*\left({E},{\nabla}\right)=\left(\underline{E}_1,\underline{\nabla}_1\right)\oplus\left(\underline{E}_2,\underline{\nabla}_2\right)$ of two rank $2$ connections. 
Precisely, each $\underline{E}_i$ has degree $-3$ and $\underline{\nabla}_i:\underline{E}_i\to \underline{E}_i\otimes\Omega^1_{\P^1}\left(\underline{W}\right)$ 
is logarithmic with residual eigenvalues $0$ and $\frac{1}{2}$ at each pole. Conversely, $\pi^*\left(\underline{E}_i,\underline{\nabla}_i\right)$
is a logarithmic connection on $X$ with only apparent singular points: residual eigenvalues are now $0$ and $1$ 
at each pole, \emph{i.e.} at each Weierstrass point of the curve. After performing a birational bundle modification
(an elementary transformation over each of the $6$ Weierstrass points) one can turn it into a holomorphic and  trace-free connection on $X$: we recover the initial connection $\left({E},{\nabla}\right)$.
In restriction to the irreducible locus, we deduce a $\left(2:1\right)$ map 
$\Phi:\CON\left(X/\iota\right)\longrightarrow \CON\left(X\right)$
where $\CON\left(X/\iota\right)$ denotes the moduli space of logarithmic connections like above.
Moduli spaces of logarithmic connections on $\P^1$ have been widely studied by many authors. Note that the idea of descent to $\mathbb{P}^1$ for studying sheaves on hyperelliptic curves already appears in work of S. Ramanan and his student U. Bhosle (see for example \cite{Ramanan} and \cite {Bhosle}).

One can associate to a connection $\left(\underline{E},\underline{\nabla}\right)\in\CON\left(X/\iota\right)$ a parabolic
structure $\underline{\p}$ on $\underline{E}$ consisting of the data of the residual eigenspace $p_j\subset \underline{E}\vert_{w_j}$
associated to the $\frac{1}{2}$-eigenvalue for each pole $w_j$ in the support of $ \underline{W}$. 
Denote by $\mathfrak{Bun}\left(X/\iota\right)$ the moduli space of such parabolic bundles $\left(\underline{E},\underline{\p}\right)$,
\emph{i.e.} defined by a logarithmic connection $\left(\underline{E},\underline{\nabla}\right)\in\CON\left(X/\iota\right)$. In fact, the descending procedure
described above can already be constructed at the level of bundles  (see \cite{BiswasOrbifold}) and we can construct
a $\left(2:1\right)$ map $\phi:\mathfrak{Bun}\left(X/\iota\right)\to\mathfrak{Bun}\left(X\right)$ making the following diagram commutative:
\begin{equation}\label{diagram} \xymatrix{
    \CON\left(X/\iota\right) \ar[r]^{2:1}_{\Phi} \ar[d]_{\bun} & \CON\left(X\right) \ar[d]_{\bun} \\
    \BUN\left(X/\iota\right) \ar[r]^{2:1}_{\phi} & \BUN\left(X\right)}
\end{equation}
Vertical arrows are locally trivial affine $\mathbb A^3$-bundles in restriction to a large open set of the bases.

{\bf Narasimhan-Ramanan moduli space.}
Having this picture at hand, we study in Section \ref{SecFlatOnX} the structure of $\BUN\left(X\right)$, 
partly surveying   Narasimhan-Ramanan's classical work \cite{NR}. They construct a quotient map 
$$\mathrm{NR}:\Bun^{ss}\left(X\right)\to\P^3_{\mathrm{NR}}:=\vert 2\Theta\vert$$
defined on the open set $\Bun^{ss}\left(X\right)\subset \BUN\left(X\right)$ of semi-stable bundles onto the $3$-dimensional linear system
generated by twice the $\Theta$-divisor on $\mathrm{Pic}^1(X)$. This map is one-to-one in restriction to the 
open set $\Bun^{s}\left(X\right)$ of stable bundles; it however identifies some strictly semi-stable bundles,  as usually does GIT theory to get a Hausdorff quotient. Precisely, the Kummer surface $\mathrm{Kum}(X)=\mathrm{Jac}(X)/_{\pm1}$
naturally parametrizes the set of decomposable semi-stable bundles, and the classifying map $\mathrm{NR}$ 
provides an embedding $\mathrm{Kum}(X)\hookrightarrow\P^3_{\mathrm{NR}}$ as a quartic surface with $16$ nodes. 
The open set of stable bundles is therefore parametrized by the complement $\P^3_{\mathrm{NR}}\setminus\mathrm{Kum}(X)$.
Over a smooth point of $\mathrm{Kum}(X)$, the fiber of $\mathrm{NR}$ consists in $3$ isomorphism classes of semi-stable bundles,
namely a decomposable one $L_0\oplus L_0^{-1}$ and the two non trivial extensions between $L_0$ and $L_0^{-1}$. 
The latter ones, which we call affine bundles, are precisely the bundles occurring in $\BUN (X)\setminus \BUN^{\mathrm{irr}}(X)$, where $\BUN^{\mathrm{irr}}(X)$ denotes the moduli space of rank 2 bundles over $X$ that can be endowed with an irreducible trace-free connection. 
Over each singular point of $\mathrm{Kum}(X)$, the fiber of $\mathrm{NR}$ consists in
a decomposable bundle $E_\tau$ (a twist of the trivial bundle by a $2$-torsion point $\tau$ of $\mathrm{Jac}(X)$)
and the (rational) one-parameter family of non trivial extensions of $\tau$ by itself. The latter ones we call (twists of) unipotent bundles;
each of them is arbitrarily close to $E_\tau$ in $\BUN\left(X\right)$.
To complete this classical picture, we have to add flat unstable bundles: by Weil's criterion,
these are exactly the unique non-trivial extensions $\vartheta\to E_\vartheta\to \vartheta^{-1}$
where $\vartheta\in \mathrm{Pic}^1(X)$ runs over the $16$ theta-characteristics $\vartheta^2=\KX$.
We call them Gunning bundles in reference to \cite{GunningCoord}: 
 connections defining a projective $\mathrm{PGL}_2$-structure on $X$
(an oper in the sense of \cite{BeilinsonDrinfeld}, see also \cite{Zvi}) are defined on these very special bundles $E_\vartheta$,
including the uniformization equation for $X$. These bundles occur as non Hausdorff points of $\BUN\left(X\right)$:
the bundles arbitrarily close to $E_\vartheta$ are precisely semi-stable extensions of the form $\vartheta^{-1}\to E\to \vartheta$. They are sent onto a plane $\Pi_\vartheta\subset\P^3_{\mathrm{NR}}$ by the  Narasimhan-Ramanan classifying map. We call them Gunning planes:
they are precisely the $16$ planes involved in the classical $(16,6)$-configuration of Kummer surfaces (see \cite{Hudson,GonzalezDorrego}).
As far as we know, these planes have had no modular interpretation so far.
We supplement this geometric study with explicit computations of Narasimhan-Ramanan coordinates,
together with the equation of $\mathrm{Kum}(X)$, as well as the $16$-order symmetry group.
These computations are done for the genus $2$ curve defined by an affine equation
$y^2=x(x-1)(x-r)(x-s)(x-t)$ as functions of the free parameters $(r,s,t)$.

{\bf The branching cover $\phi:\BUN\left(X/\iota\right)\stackrel{2:1}{\longrightarrow}\BUN\left(X\right)$.}
In Section  \ref{SecFlatPar}, we provide a full description of this map which is a double cover of $\BUN^{\mathrm{irr}}\left(X\right)$
branching over the locus of decomposable bundles, including the trivial bundle and its $15$ twists.
The $16$ latter bundles lift as $16$ decomposable parabolic bundles.
If we restrict ourselves to the complement of these very special bundles, we can follow the previous work 
of \cite{ArinkinLysenko,LoraySaito}: the moduli space $\BUN^{ind}\left(X/\iota\right)$ of indecomposable parabolic bundes 
can be constructed by patching together GIT quotients $\Bun^{ss}_{\Mu}(X/\iota)$ of $\Mu$-semi-stable parabolic bundles
for a finite number of weights $\Mu\in[0,1]^6$. These moduli spaces are smooth projective manifolds
and they are patched together along Zariski open subsets, giving $\BUN^{ind}\left(X/\iota\right)$ the structure of a smooth non separated  scheme.
In the present work, we mainly study a one-parameter family of weights, namely the diagonal family $\Mu=(\mu,\mu,\mu,\mu,\mu,\mu)$.
For $\mu=\frac{1}{2}$, the restriction map $\phi:\Bun^{ss}_{\frac{1}{2}}(X/\iota)\to\P^3_{\mathrm{NR}}$ is exactly the $2$-fold cover of $\P^3_{\mathrm{NR}}$
ramifying over the Kummer surface $\mathrm{Kum}(X)$. The space $\Bun^{ss}_{\Mu}(X/\iota)$ it is singular for this special value $\mu=\frac{1}{2}$. 
We thoroughly study the chart given by any $\frac{1}{6}<\mu<\frac{1}{4}$ which is a $3$-dimensional projective space,
that we will denote $\P^3_{\boldsymbol{b}}$: it is naturally isomorphic to  the space of extensions studied by Bertram and Bolognesi 
\cite{Bertram,Bolognesi,Bolognesi2}.
The Narasimhan-Ramanan classifying map $\phi:\P^3_{\boldsymbol{b}}\dashrightarrow\P^3_{\mathrm{NR}}$ is rational
and also related to the classical geometry of Kummer surfaces.  There is no universal bundle for the Narasimhan-Ramanan moduli space $\P^3_{\mathrm{NR}}$, but there is one for the $2$-fold cover $\P^3_{\boldsymbol{b}}$. This universal bundle, due to Bolognesi \cite{Bolognesi2} is explicitly constructed in Section \ref{SecTyurinPar} from the Tyurin point of view. 

We establish a complete dictionary 
between special (in the sense of non stable) bundles $E$ in $\BUN\left(X\right)$ (listed in Section \ref{SecFlatOnX}) and special parabolic bundles $(\underline E,\underline{\p})$ in $\BUN\left(X/\iota\right)$ allowing us to describe the geometry of the non separated singular schemes $\BUN\left(X/\iota\right)$ and $\BUN^{\mathrm{irr}}\left(X\right)$.

{\bf Anticanonical subbundles and Tyurin parameters.}
In order to establish this dictionary, we study in Section \ref{SecTyurin} the space of sheaf  inclusions of the form $\OX{-\KX}\hookrightarrow E$ for each type of bundle $E$. This is a $2$-dimensional vector space for a generic vector bundle $E$
and defines a $1$-parameter family of line subbundles. Only two of these anti-canonical subbundles are invariant under the hyperelliptic involution. 
In the generic case, the fibres over the Weierstrass points of these two subbundles define precisely the two possible parabolic structures $\p$ and $\p'$ on $E$  that arise in the context of hyperelliptic descent. 
This allows us to relate our moduli space $\BUN\left(X/\iota\right)$ to the space of $\iota$-invariant extensions
$-\KX\to E\to \KX$ studied by Bertram and Bolognesi: their moduli space coincides with our chart $\P^3_{\boldsymbol{b}}$.

On the other hand,
anticanonical morphisms provide, for a generic bundle $E$, a birational morphism $\OX{-\KX}\oplus\OX{-\KX}\to E$,
or after tensoring by $\OX{\KX}$, a birational and minimal trivialisation $E_0\to E$.
Precisely, this birational bundle map consists in $4$ elementary tranformations for a parabolic structure on the trivial bundle $E_0$
supported by a divisor belonging to the linear system $\vert 2\KX\vert$.
The moduli space of such parabolic structures is a birational model for $\BUN(X)$ (from which we easily deduce 
the rationality of this moduli stack). 

We provide the explicit change of coordinates between the Tyurin parameters
and the other previous parameters. 

{\bf Higgs bundles and the Hitchin fibration.}
Section \ref{SecHiggsCon} contains some  applications of our previous study of diagram (\ref{diagram}) to the space of Higgs bundles $\HIGGS$ over $X$, respectively $X/\iota$, which can be interpreted as the homogeneous part of the affine bundle $\CON \to \BUN$. We provide an explicit universal Higgs bundle for $\HIGGS\left(X/\iota\right)$ and we compute the 
Hitchin Hamiltonians for the Hitchin system on $\HIGGS\left(X/\iota\right)$. Using the natural identification
with the cotangent bundle $\mathrm{T}^*\BUN(X/\iota)$ together with the double cover $\phi:\BUN(X/\iota)\to\BUN(X)$,
we derive the explicit Hitchin map $\HIGGS(X)\to \mathrm{H}^0(X,2 \KX)\ ;\ (E,\Theta)\mapsto\det(\Theta)$ 
in a very direct way in Section \ref{SecHitchin}. This allows us to relate the six Hamiltonians described by G. van Geemen and E. Previato in \cite{Emma} to the three Hamiltonian coefficients of the Hitchin map.

{\bf The geometry of $\CON(X)$.} 
The computations of the Tyurin parameters in Section \ref{SecTyurinPar} and their relation to the so-called apparent map on $\CON$ defined in Section \ref{SecApparentRST} allow us to construct an explicit rational section
$\BUN\left(X\right) \dashrightarrow \CON\left(X\right)$ which is regular over 
the stable open subset of $\BUN\left(X\right)$,  and is, moreover,  Lagrangian (see Section \ref{lagr}). 
In other words, over the stable open set, the Lagrangian fiber-bundle $\CON\left(X\right)\to\BUN\left(X\right)$ is isomorphic
to the cotangent bundle $\mathrm{T}^*\BUN\left(X\right)$ (i.e. $\mathrm{T}^*P^3_{\mathrm{NR}}$) as a symplectic manifold.
Together with a natural basis of the space of Higgs bundles over $X$ we thereby obtain a universal connection parametrizing an affine chart of $\CON\left(X\right)$. 

{\bf Isomonodromic deformations.}
On the moduli stack
$\mathcal{M}$ of triples $(X,E,\nabla)$, where $X$ is a genus two curve, and $(E,\nabla)\in\CON^*(X)$
a $\iota$-invariant but non trivial  $\mathfrak{sl}_2$-connection on $X$, isomonodromic deformations form the leaves of a $3$-dimensional
holomorphic foliation, the isomonodromy foliation. It is locally defined by the fibers of the analytic Riemann-Hilbert map,
which to a connection associate its monodromy representation. 
Our double-cover construction $\Phi:\CON(X/\iota)\to\CON(X)$ is compatible with isomonodromic 
deformations when we let the complex structure of $X$ vary. Therefore, isomonodromic deformation 
equations for holomorphic $\SL$-connections on $X$ reduce to a Garnier system. 

Hence in the moduli stack $\mathcal{M}$, we can explicitly describe the isomonodromy foliation $\mathcal{F}_{\mathrm{iso}}$ as well as the locus of special bundles, for example the locus $\Sigma\subset \mathcal{M}$ of connections on Gunning bundles. We show that the isomonodromy foliation is transverse to the locus of Gunning bundles by direct computation  in Theorem \ref{Thm:TransvGunningBundle}. As a corollary, we obtain 
 a new proof of a result of Hejhal \cite{Hejhal}, stating that the monodromy map from the space 
of projective structures on the genus two curves to the space of $\SL$-representations of the fundamental group is a local diffeomorphism.

\section{Preliminaries on connections}

In this section, we introduce the objects and methods related to the notion of connection relevant for this paper, such as parabolic logarithmic connections and their elementary transformations. 
More detailed introductions can be found for example in \cite{Sabbah}, \cite{Griffiths} and \cite{GomezMont}.

\subsection{Logarithmic connections}

Let $X$ be a smooth projective curve over $\C$ and $E\to X$ be a rank $r$ vector bundle.
Let $D$ be a reduced effective divisor on $X$. Note that in general, we make no difference in notation between a reduced effective divisor and its support, as well as between the total space of a vector bundle and its locally free sheaf of holomorphic sections. 
A \emph{logarithmic connection} on $E$ with polar divisor $D$ is a $\C$-linear map
$$\nabla:E\to E\otimes \OOMX\left(D\right)$$
satifying the Leibniz rule
$$\nabla\left(f\cdot s\right)=\mathrm{d}f\otimes s+f\cdot\nabla\left(s\right)$$
for any local section $s$ of $E$ and fonction $f$ on $X$. 
Locally, for a trivialization of $E$, the connection writes $\nabla=\mathrm{d}_X+A$ 
where $\mathrm{d}_X:\OOX\to\OOMX$ is the differential
operator on $X$ and  $A$ is a $r\times r$ matrix
with coefficients in $\OOMX\left(D\right)$, thus $1$-forms having at most simple poles located along $D$. 
The \emph{true polar divisor}, \emph{i.e.} the singular set of such a logarithmic connection $\nabla$ is a subset of $D$. Depending on the context, we may assume them to be equal.  
At each pole $x_0\in D$, the residual matrix intrinsically defines an endomorphism 
of the fiber $E_{x_0}$ that we denote $\mathrm{Res}_{x_0}\nabla$. \emph{Residual  eigenvalues} and \emph{residual eigenspaces} in $E_{x_0}$ hence are well-defined. 

\subsection{Twists and trace}

As before, let $E$ be a rank $r$ vector bundle endowed with a logarithmic connection $\nabla$ on a curve $X$. The connection $\nabla$ induces a logarithmic connection $\tr\left(\nabla\right)$ on the determinant line bundle $\mathrm{det}\left(E\right)$ over $X$ with $$\mathrm{Res}_{x_0} \tr\left(\nabla\right) =\tr\left(\mathrm{Res}_{x_0}\nabla\right)$$ for each $x_0 \in D$. By the residue theorem, the sum of residues of a global meromorphic $1$-form on $X$ is zero. We thereby obtain \emph{Fuchs' relation}:
\begin{equation}\label{fuchs}\mathrm{deg}\left(E\right)+\sum_{x_0 \in D}\tr\left(\mathrm{Res}_{x_0}\nabla\right)=0.\end{equation}

We can define the \emph{twist} of the connection $\left(E,\nabla\right)$ by a rank $1$ meromorphic connection $\left(L,\zeta\right)$ as the rank $r$ connection  $\left(E',\nabla'\right)$ with

$$\left(E',\nabla'\right)=\left(E,\nabla\right)\otimes \left(L, \zeta\right) := \left(E \otimes L,\nabla\otimes \mathrm{id}_L + \mathrm{id}_E \otimes \zeta\right).$$
We have
$$\det\left(E'\right)=\det\left(E\right)\otimes L^{\otimes r}\ \ \ \text{and}\ \ \ \tr\left(\nabla'\right)=\tr\left(\nabla\right)\otimes\zeta^{\otimes r}.$$
If $L\to X$ is a line bundle such that $L^{\otimes r} \simeq \mathcal{O}_X$, then there is a unique (holomorphic) connection $\nabla_L$ on $L$ such that the connection $\nabla_L^{\otimes r}$ is the trivial connection on $L^{\otimes r}  \simeq \mathcal{O}_X$. The twist by such a $r$-torsion connection has no effect on the trace: modulo isomorphism, we have $\det\left(E'\right)=\det\left(E\right)$ and $\tr\left(\nabla'\right)=\tr\left(\nabla\right)$.

\subsection{Projective connections and Riccati foliations}

From now on, let us assume the rank to be $r=2$. After projectivizing the bundle $E$, we get a 
$\P^1$-bundle $\P E$ over $X$ whose total space is a ruled surface $S$.
Since $\nabla$ is $\C$-linear, it defines a projective connection $\P\nabla$
on $\P E$ and the graphs of horizontal sections define a foliation by curves $\mathcal F$
on the ruled surface $S$. The foliation $\mathcal F$ is transversal to a generic member of 
the ruling $S\to X$ and is thus a Riccati foliation  (see \cite{Brunella}, chapter 4).
If the connection locally writes
$$\nabla:\begin{pmatrix}z_1\\ z_2\end{pmatrix}\ \mapsto  \mathrm{d}\begin{pmatrix}z_1\\ z_2\end{pmatrix}+\begin{pmatrix}\alpha&\beta\\ \gamma&\delta\end{pmatrix}\begin{pmatrix}z_1\\ z_2\end{pmatrix},$$
then in the corresponding trivialization $\left(z_1:z_2\right)=\left(1:z\right)$ of the ruling, the foliation is defined
by the (pfaffian) Riccati  equation
$$\mathrm{d}z-\beta z^2+\left(\delta-\alpha\right)z+\gamma=0.$$
Tangencies between $\mathcal F$ and the ruling are  concentrated on fibers 
over the (true) polar divisor $D$ of $\nabla$. These singular fibers are totally $\mathcal F$-invariant. According to the number of residual eigendirections of $\nabla$, the restriction of $\mathcal F$ to such a fibre is the union of a leaf and $1$ or $2$ points.

Any two connections $\left(E,\nabla\right)$ and $\left(E',\nabla'\right)$ on $X$ define the same Riccati foliation 
 if, and only if, $\left(E',\nabla'\right)=\left(E,\nabla\right)\otimes\left(L,\zeta\right)$ for a rank $1$
connection $\left(L,\zeta\right)$. Conversely, a Riccati foliation $\left(S,\mathcal F\right)$ is always the projectivization
of a connection $\left(E,\nabla\right)$: once we have chosen a lift $E$ of $S$ and a 
rank $1$ connection $\zeta$ on $\det\left(E\right)$, there is a unique connection $\nabla$ on $E$
such that $\mathrm{trace}\left(E\right)=\zeta$ and $\P \nabla=\mathcal F$.

\subsection{Parabolic structures}

A \emph{parabolic structure} on $E$ supported by a reduced divisor $D=x_1 + \ldots + x_n$ on $X$ is the data 
$\p=\left(p_1,\ldots,p_n\right)$ of a $1$-dimensional subspace $p_i\in E_{x_i}$
for each $x_i\in D$. A \emph{parabolic connection} is the data $\left(E,\nabla,\p\right)$
of a logarithmic connection $\left(E,\nabla\right)$ with polar divisor $D$ and a parabolic 
structure $\p$ supported by $D$ such that, at each pole $x_i\in D$,
the parabolic direction $p_i$ is an eigendirection of the residual endomorphism $\mathrm{Res}_{x_i}\nabla$. For the corresponding Riccati foliation, 
 $\p$ is the data, on the ruled surface $S$, of a singular point of the foliation $\mathcal F$
for each fiber over $D$.

\begin{rem} Note that our definition is non standard here: in the literature, a parabolic structure on $E$ is usually defined as the data $\p$ (a quasi-parabolic structure) together with a collection of weights $\Mu=(\mu_1,\ldots, \mu_n)\in \mathbb{R}^n$. 
\end{rem}

\subsection{Elementary transformations}

Let  $\left(E,p\right)$ be a parabolic bundle on $X$ supported 
by a single point $x_0\in X$. Consider the vector bundle $E^-$ 
defined by the subsheaf of those sections $s$ of $E$ 
such that $s(x_0)\in p$. A natural parabolic direction on $E^-$
is defined by those sections of $E$ which are vanishing at $x_0$ (and thus belong to $E^-$).
If $x$ is a local coordinate at $x_0$ and $E$ is generated near $x_0$ by $\langle e_1,e_2\rangle$ 
with $e_1(x_0) \in p$, then $E^-$ 
is locally generated by $\langle e_1,e_2'\rangle$ with $e'_2:=x e_2$ and we define $p^-\subset E^-|_{x_0}$ to be $\mathbb{C}e_2'(x_0)$. 
By identifying the sections of $E$ and $E^-$ outside $x_0$, we obtain a natural birational morphism (see also \cite{Machu})
$$\elm_{x_0}^-:E\dashrightarrow E^-.$$

In a similar way, we define the parabolic bundle $\left(E^+,p^+\right)$ by the sheaf of those
meromorphic sections of $E$ having (at most) a single pole at $x_0$, whose residual part is an element of $p$. The parabolic $p^+$ then is defined by $$p^+:=\{s(x_0)~|~s \textrm{ is a holomorphic section of } E \textrm{ near } x_0 \}.$$ In other words, $E^+$ is generated by $\langle e_1', e_2\rangle$ with $e_1':=\frac{1}{x}e_1$and $p^+\subset E^+|_{x_0}$
defined by $\mathbb{C}e_2$. The natural morphism
$$\elm_{x_0}^+:E\dashrightarrow E^+$$
is now regular, but fails to be an isomorphism at $x_0$.

These \emph{elementary transformations} satisfy the following properties:
\begin{itemize}
\item[$\bullet$] $\det\left(E^{\pm}\right)=\det\left(E\right)\otimes\OOX\left(\pm [x_0]\right)$,\vspace{.2cm}
\item[$\bullet$] $\elm_{x_0}^+\circ\elm_{x_0}^-=\text{id}_{\left(E,p\right)}$ and $\elm_{x_0}^-\circ\elm_{x_0}^+=\text{id}_{\left(E,p\right)}$,\vspace{.2cm}
\item[$\bullet$] $\elm_{x_0}^+=\OOX\left([x_0]\right)\otimes\elm_{x_0}^-$.
\end{itemize}
In particular, positive and negative elementary transformations coincide for a projective parabolic bundle 
$\left(\P E,p\right)$.
They consist, for the ruled surface $S$, in composing the blowing-up of $p$ with the contraction
of the strict transform of the fiber \cite{GomezMont}. This latter contraction gives the new parabolic $p^{\pm}$.
Elementary transformations on projective parabolic bundles are clearly involutive.

More generally, given a parabolic bundle $\left(E,\p\right)$ with support $D$, 
we define the elementary transformations
$\elm_{D}^{\pm}$ as the composition of the (commuting) single
elementary transformations over all points of $D$.
We define $\elm_{D_0}^{\pm}$ for any subdivisor $D_0\subset D$
in the obvious way. 

Given a parabolic connection $\left(E,\nabla, \p\right)$ with support $D$,
the elementary transformations $\elm_{D}^{\pm}$ yield new parabolic connections
$\left(E^{\pm},\nabla^{\pm},\p^{\pm}\right)$. In fact, the compatibility condition between
$\p$ and the residual eigenspaces of $\nabla$ insures that $\nabla^{\pm}$
is still logarithmic. The monodromy is obviously left unchanged, but the residual eigenvalues
are shifted as follows: if $\lambda_1$ and $\lambda_2$ denote the residual eigenvalues of $\nabla$ at $x_0$, 
with $p$ contained in the $\lambda_1$-eigenspace, then 
\begin{itemize}
\item[$\bullet$] $\nabla^+$ has eigenvalues $(\lambda_1^+,\lambda_2^+):=(\lambda_1-1,\lambda_2)$,\vspace{.2cm}
\item[$\bullet$] $\nabla^-$ has eigenvalues $(\lambda_1^-,\lambda_2^-):=(\lambda_1,\lambda_2+1 )$,
\end{itemize}
and $p^\pm$ is now defined by the $\lambda_2^\pm$-eigenspace.

Finally, if the parabolic connections $\left(E,\nabla, \p\right)$ and $\left(\widetilde{E},\widetilde{\nabla}, \widetilde{\p}\right)$ are isomorphic, then one can easily check that $\left(E^\pm,\nabla^\pm, \p^\pm\right)$ and $\left(\widetilde{E}^\pm,\widetilde{\nabla}^\pm, \widetilde{\p}^\pm\right)$ are also isomorphic. This will allow us to define elementary transformations $\elm_{D}^{\pm}$
on moduli spaces of parabolic connections.

\subsection{Stability and moduli spaces}

Given  a collection $\Mu=\left(\mu_1,\cdots,\mu_n\right)$ of weights $\mu_i\in[0,1]$
attached to $p_i$, we define the \emph{parabolic degree} with respect to $\Mu$ of a line subbundle $L\hookrightarrow E$ as
$$\deg^{\mathrm{par}}_{\Mu}\left(L\right):=\deg\left(L\right)+\sum_{p_i\subset L}\mu_i$$
(where the summation is taken over those parabolics $p_i$ contained in the total space of $L\subset E$). Setting
$$\deg^{\mathrm{par}}_{\Mu}\left(E\right):=\deg\left(E\right)+\sum_{i=1}^n\mu_i$$
(where the summation is taken over all parabolics), we define the \emph{stability index} of $L$ by
$$\ind_{\Mu}\left(L\right):=\deg^{\mathrm{par}}_{\Mu}\left(E\right)-2\deg^{\mathrm{par}}_{\Mu}\left(L\right).$$
The parabolic bundle $\left(E,\p\right)$ is called \emph{semi-stable}  (resp. \emph{stable}) with respect to $\Mu$ if 
$$\ind_{\Mu}\left(L\right)\ge0 \ \text{(resp. } >0 \text{)  for each line subbundle }L\subset E.$$
For vanishing weights $\mu_1=\ldots =\mu_n= 0$, we get the usual definition of  (semi-)stability of vector bundles. We say a bundle is \emph{strictly semi-stable} if it is semi-stable but not stable. A bundle is called \emph{unstable} if it is not semi-stable. 

Semi-stable parabolic bundles admit a coarse moduli space $\Bun^{ss}_{\Mu}$ 
which is a normal projective variety; the stable locus $\Bun^{s}_{\Mu}$ is smooth (see \cite{MehtaSeshadri}).
Note that tensoring by a line bundle does not affect the stability index. In fact,
if $S$ denotes again the ruled surface defined by $\P E$, line bundles $L\hookrightarrow E$ 
are in one to one correspondence with sections $\sigma:X\to S$, and for vanishing weights,
$\ind_{\Mu}\left(L\right)$ is precisely the self-intersection number of the curve $C:=\sigma\left(X\right)\subset S$ (see also \cite{Maruyama}).
For general weights, we have 
$$\ind_{\Mu}\left(L\right)=\# (C\cdot C)+\sum_{p_i\not\in C}\mu_i-\sum_{p_i\in C}\mu_i.$$

For weighted parabolic bundles $\left(E,\p,\Mu\right)$, it is natural to extend the definition of elementary
transformations as follows. Given a subdivisor $D_0\subset D$, define
$$\elm^{\pm}_{D_0}:\left(E,\p,\Mu\right)\dashrightarrow \left(E',\p',\Mu'\right)$$
by setting 
$$
\mu_i'=\left\{\begin{matrix}1-\mu_i& \text{if}& p_i\in D_0,\\
\mu_i& \text{if}& p_i\not\in D_0.
\end{matrix}\right.$$
When $L'\hookrightarrow E'$ denotes the strict transform of $L$, we can easily check that $$\ind_{\Mu'}\left(L'\right)=\ind_{\Mu}\left(L\right).$$
Therefore, $\elm^{\pm}_{D_0}$ acts as an isomorphism between the moduli spaces $\Bun^{ss}_{\Mu}$ and $\Bun^{ss}_{\Mu'}$ (resp. $\Bun^{s}_{\Mu}$ and $\Bun^{s}_{\Mu'}$).
A parabolic connection $\left(E,\nabla,\p\right)$ is said to be \emph{semi-stable} (resp. \emph{stable}) with respect to ${\Mu}$ if 
$$\ind_{\Mu}\left(L\right)\ge0 \ \text{(resp. } >0\text{)  for all }\nabla \text{-invariant line subbundles }L\subset E.$$
In particular, irreducible connections are stable for any weight $\Mu \in [0,1]^n$.
Semi-stable parabolic connections admit a coarse moduli space $\Con^{ss}_{\Mu}$ 
which is a normal quasi-projective variety; the stable locus $\Con^{s}_{\Mu}$ is smooth (see \cite{Nitsure}).

\section{Hyperelliptic correspondence}\label{SecMainConstruction}

Let $X$ be the smooth complex projective curve given in an affine chart  of $\mathbb{P}^1\times\mathbb{P}^1$ by
$$y^2=x\left(x-1\right)\left(x-r\right)\left(x-s\right)\left(x-t\right).$$
Denote its hyperelliptic involution, defined in the above chart by $\left(x,y\right) \mapsto \left(x,-y\right)$, by $\iota : X \to X$ and denote its hyperelliptic cover, defined in the above chart by $\left(x,y\right)\mapsto x$, by
 $\pi:X\to\P^1$.
Denote by $\underline{W}=\{0,1,r,s,t,\infty\}$ the critical divisor on $\P^1$ and by $W=\{w_0,w_1,w_r,w_s,w_t,w_\infty\}$
the  \emph{Weierstrass divisor} on $X$, \emph{i.e.} the branching divisor with respect to $\pi$. 

Consider a  rank 2 vector bundle $\underline{E}\to\P^1$ of degree $-3$, endowed with a logarithmic connection 
$\underline{\nabla}:\underline{E}\to \underline{E}\otimes\OMP{\underline{W}}$ having residual eigenvalues $0$ and $\frac{1}{2}$ at each pole. 
We fix the parabolic structure $\underline{\p}$ attached 
to the $\frac{1}{2}$-eigenspaces over $\underline{W}$. 
After lifting the parabolic connection $\left(\underline{E},\underline{\nabla},\underline{{\p}}\right)$ via $\pi:X\to\P^1$, 
we get a parabolic connection on $X$ 
$$\left(\widetilde{E}\to X,\widetilde{\nabla},\widetilde{\p}\right)=\pi^*\left(\underline{E} \to \mathbb{P}^1,\underline{\nabla},\underline{\p}\right).$$
We have
$\det\left(\widetilde{E}\right)\simeq\OOX\left(-3 \KX\right)$
and the residual eigenvalues of the connection $\widetilde{\nabla}: \widetilde{E}\to \widetilde{E}\otimes\OMX{W}$ are  $0$ and $1$ at each pole, with parabolic structure $\widetilde{\p}$ 
defined by the $1$-eigenspaces. After applying elementary transformations directed by $\widetilde{\p}$, 
we get a new parabolic connection:
$$\elm_W^+:\left(\widetilde{E},\widetilde{\nabla},\widetilde{\p}\right)\dashrightarrow\left(E,\nabla,\p\right)$$
which is now holomorphic and  trace-free. 

Recall from the introduction that we denote by $ \CON\left(X/\iota\right)$ the moduli space of logarithmic rank 2 connections on $\mathbb{P}^1$ with residual eigenvalues $0$ and $1 \over 2$ at each pole in $\underline W$, and we denote by $\CON\left(X\right)$ the moduli space of   trace-free holomorphic rank 2 connections on $X$. Since to every element $\left(\underline E,\underline \nabla\right)$ of $ \CON\left(X/\iota\right)$, the parabolic structure $\underline \p$ is intrinsically defined as above, 
we have just defined a map
$$\Phi\ : \left\{ \begin{matrix} \CON\left(X/\iota\right)&\to&\CON\left(X\right)\ \\ \left(\underline E,\underline \nabla, \underline \p\right)&\mapsto&\left(E,\nabla\right).\end{matrix}\right.$$
Roughly counting dimensions, we see that both spaces of connections have same dimension $6$ up to bundle isomorphims.
We may expect to obtain most of all holomorphic and  trace-free rank 2 connections on $X$ by this construction. This turns out to be true
and will be proved along Section  \ref{topcons}. In particular, any \emph{irreducible} holomorphic and  trace-free rank 2 connection $\left(E,\nabla\right)$ on $X$  can be obtained like above. Note that the stability of $E$ is a sufficient condition for the irreducibility of $\nabla$.

\subsection{Topological considerations}\label{topcons}

By the Riemann-Hilbert correspondence, the two moduli spaces of connections considered above are
in one-to-one correspondence with moduli spaces of representations. Let us start with $\CON\left(X\right)$ which is easier.
The monodromy of a  trace-free holomorphic rank 2 connection $\left(E,\nabla\right)$ on $X$ gives rise to a monodromy representation,
namely a homomorphism $\rho:\pi_1\left(X,w\right)\to\SL$. In fact, this depends on the choice of a basis on the fiber $E_w$. Another choice will give the conjugate representation $M\rho M^{-1}$ for some
$M\in\GL$. The class $[\rho]\in \Hom\left(\pi_1\left(X,w\right),\SL\right)/_{\PGL}$ however is well-defined by $\left(E,\nabla\right)$.
Conversely, the monodromy $[\rho]$ characterizes the connection $\left(E,\nabla\right)$ on $X$ modulo isomorphism, which yields a bijective correspondence 
$$\mathrm{RH}:\CON\left(X\right)\stackrel{\sim}\longrightarrow\Hom\left(\pi_1\left(X,w\right),\SL\right)/_{\PGL}$$
which turns out to be complex analytic where it makes sense, \emph{i.e.} on the smooth part. Yet this map is highly transcendental, since
we have to integrate a differential equation to compute the monodromy. Note that the space of representations
only depends on the topology of $X$, not on the complex and algebraic structure. 

In a similar way, parabolic connections in $\CON\left(X/\iota\right)$
are in one-to-one correspondence with faithful representations $\rho:\pi_1^{\mathrm{orb}}\left(X/\iota\right)\to\GL$
of the orbifold fundamental group (killing squares of simple loops around punctures, see the proof of theorem \ref{thm:Goldman} below). Thinking of $\P^1=X/\iota$ as the orbifold quotient of $X$ by the hyperelliptic involution, these representations can also be seen as representations
$$\rho:\pi_1\left(\P^1\setminus \underline W,x\right)\to\mathrm{GL}_2$$
with $2$-torsion monodromy around the punctures, having eigenvalues $1$ and $-1$.

If $x=\pi\left(w\right)$, the branching cover $\pi: X\to X/\iota$ induces a monomorphism
$$\pi_*:\pi_1\left(X,w\right)\hookrightarrow\pi_1^{\mathrm{orb}}\left(X/\iota,x\right),$$
whose image consists of words of even length in the alphabet of a system of simple generators of $\pi_1^{\mathrm{orb}}\left(X/\iota,x\right).$ This allows to associate, to any representation $\rho:\pi_1^{\mathrm{orb}}\left(X/\iota,x\right)\to\GL$ as above,
a representation $\rho\circ\pi_*:\pi_1\left(X,w\right)\to\SL$. 
We have thereby defined a map $\Phi^{\mathrm{top}}$ between corresponding representation spaces, which makes the following diagram commutative 
\begin{equation}\label{RHcommutativeDiagram}
 \xymatrix{
\CON\left(X/\iota\right)  \ar[rrr]^{\hskip-40pt \mathrm{RH}}_{\hskip-40pt \sim} \ar[d]^{\Phi} &&& \Hom\left(\pi_1^{\mathrm{orb}}\left(X/\iota, x\right),\GL\right)/_{\PGL} \ar[d]^{\Phi^{\mathrm{top}}} \\
\CON\left(X\right) \ar[rrr]^{\hskip-40pt \mathrm{RH}}_{\hskip-40pt \sim} &&& \Hom\left(\pi_1\left(X,w\right),\SL\right)/_{\PGL}.
 }
 \end{equation}
 We now want to describe the map $\Phi^{\mathrm{top}}$. The quotient $\pi_1^{\mathrm{orb}}\left(X/\iota,x\right)\ /\ \pi_*\left(\pi_1\left(X,w\right)\right)\simeq\Z_2$ acts (by conjugacy) as outer automorphisms of $\pi_1\left(X,w\right)$. It coincides with the outer action of the 
hyperelliptic involution $\iota$. 

\begin{thm}\label{thm:Goldman}Given a representation $[\rho]\in\Hom\left(\pi_1\left(X\right),\SL\right)/_{\PGL}$, the following properties are equivalent:
\begin{enumerate}
\item[\emph{(a)}] $[\rho]$ is either irreducible or abelian;
\item[\emph{(b)}] $[\rho]$ is $\iota$-invariant;
\item[\emph{(c)}] $[\rho]$ is in the image of $\Phi^{\mathrm{top}}$.
\end{enumerate}
If these properties are satisfied, then $[\rho]$ has $1$ or $2$ preimages under $\Phi^{\mathrm{top}}$, depending on whether it is diagonal or not.
\end{thm}

\begin{proof}We start making explicit the monomorphism $\pi_*$ and the involution $\iota$. Let $x\in\P^1\setminus \underline W$
and $w\in X$ one of the two preimages. Choose simple loops around the punctures to generate
the orbifold fundamental group of $\P^1\setminus \underline W$ with the standard representation
$$\pi_1^{\mathrm{orb}}\left(X/\iota,x\right)=\left\langle \gamma_0,\gamma_1,\gamma_r,\gamma_s,\gamma_t,\gamma_\infty\ \left\vert\ \begin{matrix}\gamma_0^2=\gamma_1^2=\gamma_r^2=\gamma_s^2=\gamma_t^2=\gamma_\infty^2=1\\ 
\text{and}\ \ \ \gamma_0\gamma_1\gamma_r\gamma_s\gamma_t\gamma_\infty=1\end{matrix}\right.
\right\rangle.$$
Even words in these generators can be lifted as loops based in $w$ on $X$, generating 
the ordinary fundamental group of $X$. Using the relations, we see that 
$\pi_1\left(X,w\right)$ is actually generated by the following pairs 
$$\left\{\begin{matrix}\alpha_1:=\gamma_0\gamma_1\\ \beta_1:=\gamma_r\gamma_1\end{matrix}\right.\ \ \ 
\left\{\begin{matrix}\alpha_2:=\gamma_s\gamma_t\\ \beta_2:=\gamma_\infty\gamma_t\end{matrix}\right.$$
and they provide the standard presentation
\begin{equation}\label{eq:standardfondgroupg2}
\pi_1\left(X,w\right)=\left\langle \alpha_1,\beta_1,\alpha_2,\beta_2\ \vert\ [\alpha_1,\beta_1][\alpha_2,\beta_2]=1
\right\rangle,
\end{equation}
where $[\alpha_i,\beta_i]=\alpha_i\beta_i\alpha_i^{-1}\beta_i^{-1}$ denotes the commutator. 
In other words, the monomorphism $\pi_*$ is defined  by $\alpha_1\mapsto \gamma_0\gamma_1$ {\it et cetera} (see Figure 1).

 \begin{figure}[!h]\label{gengen2}
 \centering
 \resizebox{120mm}{!}{\input{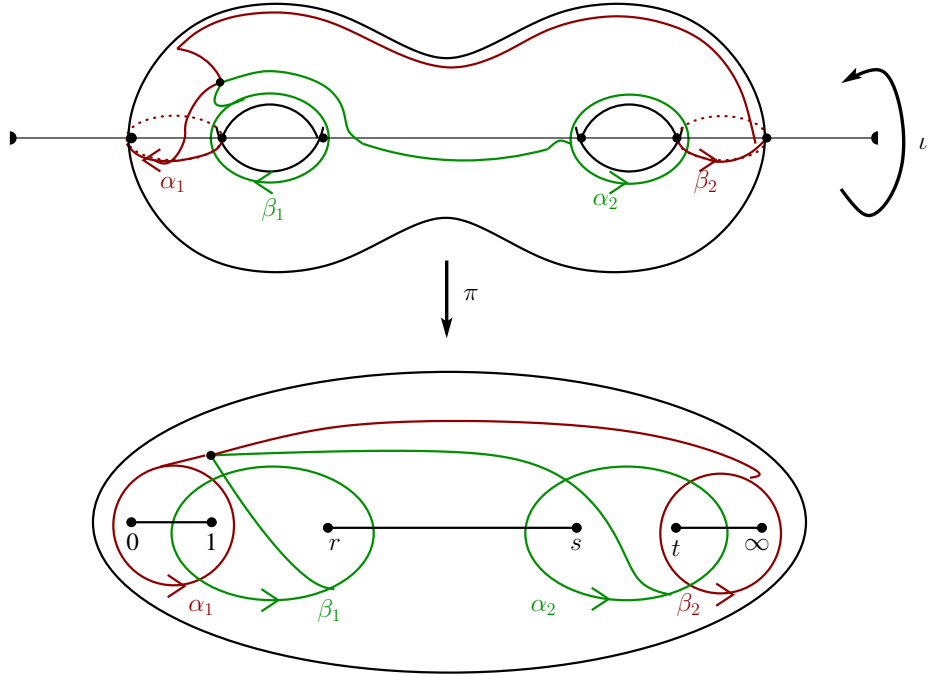}}      
 \caption{Elements of $\pi_1\left(\P^1\setminus \underline W, x\right)$ that lift as the generators of  $\pi_1\left(X,w\right)$.}
\end{figure}

After moving the base point  to a Weierstrass point, $w=w_i$ say, the involution $\iota$ acts as an involutive 
automorphism of $\pi_1\left(X,w_i\right)$: it coincides with the outer automorphism
given by $\gamma_i$-conjugacy. For instance, for $i=1$, we get
$$\left\{\begin{matrix}\alpha_1\mapsto \alpha_1^{-1}\\ \beta_1\mapsto \beta_1^{-1}\end{matrix}\right.\ \ \ 
\left\{\begin{matrix}\alpha_2\mapsto \gamma\alpha_2^{-1}\gamma^{-1}\\ \beta_2\mapsto \gamma\beta_2^{-1}\gamma^{-1}\end{matrix}\right.\ \ \ \text{with}\ \ \ \gamma=\beta_1^{-1}\alpha_1^{-1}\beta_2\alpha_2.$$
Let us now prove (a)$\Leftrightarrow$(b). That irreducible representations are $\iota$-invariant already
appears in the last section of \cite{Goldman}. Let us recall the argument given there. There is a natural surjective map
$$\Psi:\Hom\left(\pi_1\left(X\right),\SL\right)/_{\PGL}\longrightarrow \Hom\left(\pi_1\left(X\right),\SL\right)//_{\PGL}=:\chi$$
to the GIT quotient $\chi$, usually called character variety,  which is an affine variety. The singular locus 
is the image of reducible representations. There can be many different classes $[\rho]$ over each singular point.
The smooth locus of $\chi$ however is the geometric quotient of irreducible representations, which are called stable points in this context. The above map $\Psi$ is injective over this open subset. The involution $\iota$ acts on $\chi$ as a polynomial
automorphism and we want to prove that the action is trivial. First note that the canonical fuchsian representation 
given by the uniformisation $\mathbb H\to X$ must be invariant by the hyperelliptic involution $\iota:X\to X$. The corresponding point in $\chi$ therefore is  fixed by $\iota$. On the other hand,
the definition of $\chi$ only depends on the topology of $X$ and, considering all possible complex structures on $X$, 
we now get a large set of fixed points 
$\chi_{\mathrm{fuchsian}}\subset\chi$.
Those fuchsian representations actually form an open subset of $\Hom\left(\pi_1\left(X\right),\SL{\R}\right)//_{\SL{\R}}$,
and thus a Zariski dense subset of $\chi$. It follows that the action of $\iota$
is trivial on the whole space $\chi$. By injectivity of $\Psi$, any irreducible representation is $\iota$-invariant.

In other words, if an irreducible representation $\rho$ is defined by matrices $A_i,B_i\in\SL$, $i=1,2$ with $[A_1,B_1]\cdot [A_2,B_2]=I_2$,
then there exists $M\in\GL$ satisfying:
\begin{equation}\label{Eqinv} \left\{\begin{matrix}M^{-1}A_1M=A_1^{-1}\\ M^{-1}B_1M= B_1^{-1}\end{matrix}\right.\ \ \ 
\left\{\begin{matrix}M^{-1}A_2M=CA_2^{-1}C^{-1}\\ M^{-1}B_2M= CB_2^{-1}C^{-1}\end{matrix}\right.\ \ \ \text{with}\ \ \ C=B_1^{-1}A_1^{-1}B_2A_2.\end{equation}
Since the action of $\iota$ is involutive, $M^2$ commutes with $\rho$ and is thus a scalar matrix. The matrix $M$ has two opposite eigenvalues which can be normalized to $\pm1$ after replacing $M$ by a scalar multiple. There are exactly two such normalizations, namely $M$ and $-M$.

It remains to check what happens for reducible representations.
In the strict reducible case (\emph{i.e.} reducible but not diagonal), there is a unique common eigenvector for all matrices $A_1,B_1,A_2,B_2$;
the representation $\rho$ restricts to it as a representation $\pi_1\left(X\right)\to\C^*$ which must be $\iota$-invariant.
This (abelian) representation must therefore degenerate into $\{\pm1\}$. It follows that any reducible $\iota$-invariant representation is abelian. 
For abelian representations though,
the action of $\iota$ is simply given by 
$$A_i\mapsto A_i^{-1} \ \ \ \text{and}\ \ \ B_i\mapsto B_i^{-1} \ \ \ \text{for}\ \ \ i=1,2.$$
Hence all reducible $\iota$-invariant representations are abelian and, up to conjugacy, we have:
\begin{itemize}
\item[$\bullet$] either $A_1,B_1,A_2,B_2$ are diagonal and one can choose $M=\begin{pmatrix}0&1\\1&0\end{pmatrix}$,
\item[$\bullet$]  or $A_1,B_1,A_2,B_2$ are upper triangular with eigenvalues $\pm 1$ (projectively unipotent) and $M=\begin{pmatrix}1&0\\0&-1\end{pmatrix}$
works.
\end{itemize}

Let us now prove (b)$\Leftrightarrow$(c). 
Given a representation $[\rho]\in\Hom\left(\pi_1^{\mathrm{orb}}\left(X/\iota\right),\GL\right)/_{\PGL}$, its image under 
$\Phi^{\mathrm{top}}$ is $\iota$-invariant, \emph{i.e.} the action of $\iota$ coincides in this case with the conjugacy 
by $\rho\left(\gamma_1\right)\in\GL$. Conversely, let $[\rho]\in\Hom\left(\pi_1\left(X\right),\SL\right)/_{\PGL}$ be $\iota$-invariant, \emph{i.e.} 
$\iota^*\rho=M^{-1}\cdot\rho\cdot M$ for some $M\in\GL$ as in (\ref{Eqinv}). From the cases discussed above, we know that
$M$ can be chosen with eigenvalues $\pm1$. Then setting 
$$\left\{\begin{matrix}M_0:=&A_1M\\ M_1:=&M\\ M_r:=& B_1M\end{matrix}\right.\ \ \ 
\left\{\begin{matrix}M_s:=&B_2^{-1}A_1B_1M\\ M_t:=&A_1B_1MA_2B_2\\M_\infty:=&A_1B_1MA_2\end{matrix}\right.$$
we get a preimage of $[\rho]$. The preimage depends only of the choice of $M$. Any other choice writes $M':=CM$
with $C$ commuting with $\rho$. In the general case, \emph{i.e.} when $\rho$ is irreducible,  we get two preimages given by $M$ and $-M$. However, when $\rho$
is diagonal, we get only one preimage, because the anti-diagonal matrices $M$ and $-M$ are conjugated by a diagonal matrix (commuting with $\rho$).
\end{proof}

\begin{cor}\label{cor:GaloisTop}
The Galois involution of the double cover $\Phi^{\mathrm{top}}$ is given by 
$$\left\{\begin{array}{ccc}\Hom\left(\pi_1^{\mathrm{orb}}\left(X/\iota\right),\GL\right)/_{\PGL}&\longrightarrow&\Hom\left(\pi_1^{\mathrm{orb}}\left(X/\iota\right),\GL\right)/_{\PGL}
\\ \textrm{} [ \rho ]  &\mapsto & [-\rho]
 \end{array}\right\}.$$
\end{cor}

So far, Theorem \ref{thm:Goldman} provides an analytic description of the map $\Phi$:
although $\Phi^{\mathrm{top}}$ is a polynomial branching cover, the Riemann-Hilbert correspondence
is only analytic. In the next section, we will follow a more direct approach providing algebraic
informations about $\Phi$. However, note that we can already deduce the following:

\begin{cor}\label{lift}
An irreducible  trace-free holomorphic connection $\left(E,\nabla\right)$ on $X$ is invariant 
under the hyperelliptic involution:
there exists a bundle isomorphism $h:E\to \iota^*E$ conjugating $\nabla$ with $\iota^*\nabla$.
We can moreover assume $h\circ \iota^*h=\mathrm{id}_E$ and $h$ is unique up to a sign.
\end{cor}

\begin{rem} Note that $h$ acts as $-\mathrm{id}$ on the determinant line bundle
$\det\left(E\right)=\det\left(\iota^*E\right)\simeq \mathcal{O}$. \end{rem}

Each Weierstrass point $w\in X$ is fixed by $\iota$ and the restriction of $h$ to the fibre $E_{w}=\iota^*E_{w}$
is an automorphism with simple eigenvalues $\pm 1$. 

\subsubsection{Symmetry group}

The $16$-order group of $2$-torsion characters $\Hom(\pi_1(X),\{\pm1\})$ acts on the space of representations
$\Hom\left(\pi_1\left(X\right),\SL\right)/_{\PGL}$
by multiplication (changing signs of matrices $A_i,B_i$'s). This corresponds to the action of $2$-torsion rank one connections
on the moduli space $\CON\left(X\right)$: the unique unitary connection on a $2$-torsion line bundle is itself $2$-torsion;
if we twist a $\SL$-connection by this $2$-torsion one, we get a new $\SL$-connection. Together with the involution
of Corollary \ref{cor:GaloisTop}, we get a $32$-order group acting on 
$\Hom\left(\pi_1^{\mathrm{orb}}\left(X/\iota\right),\GL\right)/_{\PGL}$ also by changing signs of matrices $M_i$'s
(only even change signs). The generators are described as follows
$$
\begin{matrix}&(A_1,B_1,A_2,B_2)&(M_0,M_1,M_r,M_s,M_t,M_\infty)\\
\OOX([w_0]-[w_1])&(-,+,+,+)&(-,-,+,+,+,+)\\
\OOX([w_1]-[w_r])&(+,-,+,+)&(+,-,-,+,+,+)\\
\OOX([w_s]-[w_t])&(+,+,-,+)&(+,+,+,-,-,+)\\
\OOX([w_t]-[w_\infty])&(+,+,+,-)&(+,+,+,+,-,-)\\
\OOX&(+,+,+,+)&(-,-,-,-,-,-)
\end{matrix}
$$
The quotient for this action identifies with one of the two connected components of $\Hom\left(\pi_1\left(X\right),\PGL\right)/_{\PGL}$, namely the component of those representations that lift to $\SL$. We have seen in Theorem \ref{thm:Goldman}
that the fixed point set of the  Galois involution of $\Phi^{\mathrm{top}}$ is given by diagonal representations. 
We can also compute the fixed point locus of $\OOX([w_0]-[w_1])$ for instance.

\begin{prop}
The fixed points of the action of $\OOX([w_0]-[w_1])$ (with its unitary connection) 
on the space of representations $\Hom\left(\pi_1\left(X\right),\SL\right)/_{\PGL}$ is
parametrized by:
$$A_1=\pm I,\ \ \ B_1=\begin{pmatrix}0&1\\-1&0\end{pmatrix},\ \ \ A_2=\begin{pmatrix}a&0\\0&a^{-1}\end{pmatrix}
\ \ \ \text{and}\ \ \ B_2=\begin{pmatrix}b&0\\0&b^{-1}\end{pmatrix}$$
with $(a,b)\in\C^*\times\C^*$. 
\end{prop}


\subsection{A direct algebraic approach}\label{diralg}

Let $\left( E,  \nabla\right)$ be a holomorphic  trace-free rank 2 connection on $X$. As in Corollary \ref{lift}, let $h$ be a $ \nabla$-invariant lift to the vector bundle $ E$ of the action of $\iota$ on $X$. Following \cite{BiswasOrbifold} and \cite{BiswasHeu}, we can associate a parabolic logarithmic connection $\left(\underline E, \underline \nabla, \underline \p\right)$ on $\mathbb{P}^1$ with polar divisor $\underline W$ and a natural choice of parabolic weights $ \underline \Mu$. Let us briefly recall this construction.  
The isomorphism $h$ induces a non-trivial involutive automorphism on the rank 4 bundle $\pi_*  E$ on $\mathbb{P}^1$. The spectrum of such an automorphism is $\{-1, +1\}$
with respective multiplicities $2$, which yields a splitting $\pi_* E=\underline E \oplus \underline E '$ with $\underline E $ denoting the $h$-invariant subbundle.

In local coordinates, the automorphism $h$ acts on  $\pi_*  E$ in the following way. If $U\subset X$ is a sufficiently small open set outside of the critical points, we have $\Gamma\left(\pi\left(U\right),\pi_*   E\right)=\Gamma\left(U,  E\right)\oplus\Gamma\left(\iota\left(U\right),  E\right)$ and $h$ permutes both direct summands. Locally at a Weierstrass point 
with local coordinate $y$, one can choose sections $e_1$ and $e_2$ generating $ E$
such that $h\left(e_1\right)=e_1$ and $h\left(e_2\right)=-e_2$ (recall that $h$ has 
eigenvalues $\pm1$ in restriction to the Weierstrass fiber). On the corresponding open set of $\P^1$, the bundle $\pi_* E$ is generated by 
$\langle e_1,e_2,ye_1,ye_2\rangle$, and we see that 
$\langle e_1,ye_2\rangle$ spans the $h$-invariant subspace. 
Since the connection $ \nabla$ on $ E$ is $h$-invariant, we can choose the sections $e_1$ and $e_2$ above to be horizontal for  $ \nabla$. Then considering the basis $\underline{e}_1=e_1$ and $\underline{e}_2=ye_2$ of $\underline E$, we get
$$\underline \nabla \underline{e_1}=  \nabla e_1=0\ \ \ \text{and}\ \ \ \underline \nabla \underline{e_2}= \nabla ye_2=\mathrm{d}y\otimes e_2=\frac{\mathrm{d}y}{y}\otimes \underline{e_2}=\frac{1}{2}\frac{\mathrm{d}x}{x}\otimes \underline{e_2}$$
so that $\underline \nabla$ is logarithmic with eigenvalues $0$ and $\frac{1}{2}$. To each pole in $\underline W$, we associate the parabolic $\underline {p_i}$ defined by the eigenspace with eigenvalue ${1 \over 2}$, with the natural (in the sense of   \cite{BiswasHeu}) parabolic weight $\underline{\mu_i} = { 1 \over 2}$.

\begin{table}[h!]
$$ \xymatrix{
\left(\widetilde{E}, \widetilde{\nabla} ,\widetilde{\p}\right)\ar@{<--}[r]^{\mathrm{elm}_W^-}    & \left( E,\nabla,  \p\right)\ar@{<~>}[r]^h & \left( E,\nabla\right)\ar[d]^{\pi_*}   & \left( E, \nabla,  \p'\right)\ar@{<~>}[l]_{-h}
&  \left(\widetilde{E}', \widetilde{\nabla}',\widetilde{\p}'\right)\ar@{<--}[l]_{\mathrm{elm}_W^-} \\
\left(\underline E, \underline \nabla, \underline \p\right)\ar@{<~>}[r] ^{\mathrm{eig}_{1 \over 2}}\ar[u]^{\pi^*}    \ar@/_2pc/@{<-->}[rrrr]_{\sqrt{\mathrm{d}\log\left(\underline W\right)}\otimes\mathrm{elm}_{\underline W}^+}
  &  \left(\underline E,\underline \nabla\right)\ar@{^{(}->}[r]     & \left(\underline E,\underline\nabla\right)\oplus\left(\underline E',\underline \nabla'\right)      &    \left(\underline E',\underline \nabla'\right) \ar@{_{(}->}[l]     &    \left(\underline E', \underline \nabla',\underline \p'\right)\ar[u]^{\pi^*} \ar@{<~>}[l]_{\mathrm{eig}_{1 \over 2}}
  }$$
\caption{ Hyperelliptic descent, lift and involution.}\label{su}
\end{table}

However, since we consider the rank 2 case, this general  construction can also be viewed in the following way (summarized in Table \ref{su}): Denote by $ \p$ the parabolic structure on $ E$ defined by the $h$-invariant directions over $W=\{w_0,w_1,w_r,w_s,w_t,w_\infty\}$
 and associate the natural homogenous weight $ \Mu=0$. In the coordinates above, the basis $\left(\underline{e}_1, \underline{e}_2\right)$ generates the vector bundle $ E$ after one negative elementary transformation in that direction. Now the hyperelliptic involution acts trivially on the parabolic logarithmic connection on $X$ defined by
$$\left(\widetilde{E}, \widetilde{\nabla}, \widetilde{\p}, \widetilde{ \Mu}\right) := \mathrm{elm}_W^-\left( E,  \nabla,  \p,  \Mu\right)$$
 and we have
$$\left(\widetilde{E}, \widetilde{\nabla}, \widetilde{\p}, \widetilde{\Mu}\right) =  \pi^*\left(\underline E, \underline \nabla, \underline \p, \underline \Mu\right).$$

\subsubsection{Galois involution and symmetry group}\label{SecSymGal}

With the notations above, let $\left(\underline E', \underline \nabla'\right)$ be the connection on $\mathbb{P}^1$ we obtain for the other possibility of a lift of the hyperelliptic involution on $\left( E\to X,  \nabla\right)$, namely for $h'=-h$. 
It is straightforward to check that the map from $\left(\underline E', \underline \nabla',\p'\right)$ to $\left(\underline E, \underline \nabla,\p\right)$ and vice-versa is obtained by the elementary transformations $\elm_{\underline W}^+$ over $\mathbb{P}^1$, followed by the tensor product with a certain logarithmic rank 1 connection $\sqrt{\mathrm{d}\log\left(\underline{W}\right)}$ over $\mathbb{P}^1$ we now define:

There is a unique rank $1$ logarithmic connection $\left(L,\zeta\right)$ on $\P^1$ having polar divisor $\underline{W}$
and eigenvalues $1$; note that $L=\OOP\left(-6\right)$. We denote by $\mathrm{d}\log\left(\underline{W}\right)$ this connection and by $\sqrt{\mathrm{d}\log\left(\underline{W}\right)}$ its unique 
square root. In a similar way, define $\sqrt{\mathrm{d}\log\left(D\right)}$ for any even order subdivisor $D\subset \underline{W}$.

The Galois involution of our map $\Phi\ :\ \CON\left(X/\iota\right)\to\CON\left(X\right)$ is therefore given by 
$$\sqrt{\mathrm{d}\log\left(\underline{W}\right)}\otimes\elm_{\underline W}^+:\CON\left(X/\iota\right)\to \CON\left(X/\iota\right).$$

There is a $16$-order group of symmetries on $\BUN\left(X\right)$ (resp. $\CON\left(X\right)$) 
consisting of twists with $2$-torsion line bundles (resp. rank $1$ connections).
It can be lifted as a $32$-order group of symmetries on $\BUN\left(X/\iota\right)$ (resp. $\CON\left(X/\iota\right)$),
namely those transformations $\sqrt{\mathrm{d}\log\left(D\right)}\otimes\elm_{D}^+$ with $D\subset \underline{W}$ even.
For instance, if $D=w_i+w_j$, then its action on $\CON\left(X/\iota\right)$ corresponds via $\Phi$
to the twist by the $2$-torsion connection on $\OOX\left(w_i+w_j- \KX\right)$.
In particular, it permutes the two parabolics (of $\p$ and $\p'$) over $w_i$ and $w_j$.

\section{Flat vector bundles over $X$}\label{SecFlatOnX}

In this section, we provide a description of the space of   trace-free holomorphic connections on a given flat  rank $2$ vector bundle $E$ over the genus 2 curve $X$.
We first review the classical construction of
the moduli space of semi-stable such bundles due to Narasimhan and Ramanan.
We then present the \emph{special} (in the sense of flat but not stable) bundles and explain how they arise in the Narasimhan-Ramanan moduli space. The $16$-order group of $2$-torsion points of $\mathrm{Jac}(X)$ is naturally acting on $\BUN\left(X\right)$ by tensor product,
preserving each type of bundle. We compute this action for explicit coordinates in Section \ref{SecComputeNR} along with explicit equations of the Kummer surface of strictly semi-stable bundles. 
Moreover, we describe the set of connections on each of these bundles and the quotient of the irreducible ones by the automorphism group. 
This is summarized in Table \ref{TableSpecialBundles}; columns list for each type of bundle the projective part $\P\mathrm{Aut}(E)=\mathrm{Aut}(E)/_{\mathbb{G}_m}$of the bundle automorphism group, 
the affine space of connections and lastly the moduli space of irreducible connections up to bundle automorphism. Here $\mathbb{G}_m$ denotes the multiplicative group $(\mathbb{C}^\times ,\cdot )$ and $\mathbb{G}_a$ denotes the additive group $(\mathbb{C} ,+ )$.

Crucial for the understanding of the rest of the present paper is the case of Gunning bundles, where we explain the notion of two vector bundles being arbitrarily close in $\BUN (X)$, which, as we will see, is responsable for the classical geometry of the Kummer surface of strictly semi-stable bundles in $\mathcal{M}_{\mathrm{NR}}$.

\begin{table}[h]
\begin{center}
\begin{tabular}{|c|c|c|c|c|}
\hline
bundle type & $E$ & $\P\mathrm{Aut}(E)$ & connections & moduli \\
\hline
stable & $E$ & $1$ & $\mathbb A^3$ & $\mathbb A^3$ \\
\hline
decomposable & $L_0\oplus L_0^{-1}$ & $\mathbb G_m$ & $\mathbb A^4$ & $\C^2\times\C^*$ \\
\hline
affine & $L_0\to E\to L_0^{-1}$ & $1$ & $\mathbb{A}^3$ & $\emptyset$ \\
\hline
trivial+twists & $E_0$, $E_\tau$ & $\mathrm{PGL}_2(\C)$ & $\mathbb A^6$ & $\C^3_{(\nu_0, \nu_1, \nu_2)}\setminus\{\nu_1^2=4\nu_0\nu2\}$ \\
\hline
unipotent+twists & $\tau\to E\to \tau$ & $\mathbb G_a$ & $\mathbb A^4$ & $\C^2\times\C^*$ \\
\hline
Gunning & $\vartheta\to E_\vartheta\to \vartheta^{-1}$ & $H^0(X,\OOMX)$ &  $\mathbb A^5$ &  $\mathbb A^3$\\
\hline
\end{tabular}
\end{center}
\caption{Bundle automorphisms and moduli spaces of irreducible connections.}
\label{TableSpecialBundles}
\end{table}

\subsection{Flatness criterion}

Recall the well-known flatness criterion for vector bundles over curves \cite{Weil,Atiyah}.

\begin{thm}[Weil] A holomorphic vector bundle  on a compact Riemann surface is \emph{flat}, \emph{i.e.} it admits a holomorphic connection,
if and only if it is the direct sum of indecomposable bundles of degree $0$.
\end{thm}

In our case of rank 2 vector bundles $E$ over a genus 2 curve $X$ with trivial determinant bundle  $\det\left(E\right)=\OOX$, Weil's criterion demands that either $E$ is indecomposable, or it is the direct sum of degree $0$ line bundles. We get the following list of flat bundles:
\begin{itemize}
\item[$\bullet$] stable bundles (forming a Zariski-open subset of the moduli space),
\item[$\bullet$] decomposable bundles of the form $E=L\oplus L^{-1}$ where $L\in \mathrm{Jac}\left(X\right)$
is a degree $0$ line bundle,
\item[$\bullet$] strictly semi-stable indecomposable bundles,
\item[$\bullet$] Gunning bundles.
\end{itemize}

We recall that a \emph{Gunning bundle} over $X$ is an unstable indecomposable rank 2 vector bundle with trivial determinant bundle. There are precisely 16 such bundles: for each of the 16 line bundles $L \in \mathrm{Pic}^1\left(X\right)$ such that $L^{\otimes 2}=\mathcal{O}\left( \KX\right)$ there is a unique indecomposable extension  $0\to L\to E\to L^{-1}\to 0$ of $L^{-1}$ by $L$. 

Given a flat bundle $E$, and a $\mathfrak{sl}_2$-connection $\nabla$ on $E$, any other $\mathfrak{sl}_2$-connection writes
$$\nabla'=\nabla+\theta$$
where $\theta\in\mathrm{H}^0\left(\Hom_{\OOX} \left(\mathfrak{sl}\left(E\right)\otimes\OOMX\right)\right)$ is a Higgs field. Here, $\mathfrak{sl}\left(E\right)$ denotes the vector bundle whose sections are  trace-free endomorphisms of $E$. On the other hand, by the Riemann-Roch Theorem and Serre Duality we have 
\begin{equation}\label{dimconnex}\mathrm{h}^0\left(\mathfrak{sl}\left(E\right)\otimes\OOMX\right)=3\cdot\mathrm{genus}\left(X\right)-3+\mathrm{h}^0\left(\mathfrak{sl}\left(E\right)\right)\end{equation}
($\mathfrak{sl}\left(E\right)$ is self-dual). Since there is no natural choice 
for the initial connection $\nabla$, the set of connections on $E$ is an affine space.  We will see in the following that for generic bundles we have $h^0\left(\mathfrak{sl}\left(E\right)\right)=0$
and the {\it moduli space} of $\mathfrak{sl}_2$-connections on $E$ is $\mathbb A^3$ in this case. There are, however,
flat bundles with non-trivial automorphisms for which the {\it moduli space} of $\mathfrak{sl}_2$-connections
will be a quotient of some $\mathbb A^n$ by the automorphism group, yet the dimension of this quotient is always $3$, as suggested by (\ref{dimconnex}).

\subsection{Semi-stable bundles and the Narasimhan-Ramanan theorem}\label{NarRam}

Two semi-stable vector bundles of same rank and degree over a curve are called \emph{$S$-equivalent}, if the graded bundles associated to Jordan-H\"older filtrations of these bundles are isomorphic. In our case, \emph{i.e.} rank $2$ bundles with trivial determinant bundle, we get that
\begin{itemize}
\item[$\bullet$] two stable bundles are $S$-equivalent if and only if they are isomorphic;
\item[$\bullet$] two strictly semi-stable bundles are $S$-equivalent if and only if there is a  line bundle $L\in\mathrm{Jac}\left(X\right)$ such that each of the two bundles is an extension either of $L^{-1}$ by $L$ or of $L$ by $L^{-1}$.  
\end{itemize}
To a semi-stable bundle $E$, we associate (following \cite{NR}) the set 
$$C_E = \{L \in \mathrm{Pic}^1\left(X\right)~|~\mathrm{h}^0\left(X,E\otimes L\right)>0\}.$$  
Equivalently, $L\in C_E$ if and only if there is a non-trivial (and thus injective) homomorphism $L^{-1}\to E$ of locally free sheaves. For stable bundles, the quotient $E/_{L^{-1}}$ then is necessarily locally free and hence defines an embedding of the total space of $L^{-1}$ into the total space of $E$. The set $C_E$ then parametrizes
line subbundles of degree $-1$.

Narasimhan and Ramanan proved that this set $C_E$ is the support of a uniquely defined effective divisor $D_E$
on $\mathrm{Pic}^1\left(X\right)$ linearly equivalent to $2\Theta$, where 
$$\Theta=\{[p]~|~p\in X\}\subset \mathrm{Pic}^1\left(X\right)$$
is the locus of effective divisors of degree $1$, naturally parametrized by the curve $X$ itself.
Moreover, for strictly semi-stable bundles, the divisor $D_E$ only depends on the 
Jordan-H\"older filtration, \emph{i.e.} on the $S$-equivalence class of $E$.
We thus get a map 
$$\mathrm{NR}\ :\ \mathcal{M}_{\mathrm{NR}}\to\mathbb{P}\left(\mathrm{H}^0\left(\mathrm{Pic}^1\left(X\right),\mathcal{O}\left(2\Theta\right)\right)\right)$$
from the moduli space of $S$-equivalence classes to the linear system $\vert 2\Theta\vert$ on $\mathrm{Pic}^1\left(X\right)$.

\begin{thm}[Narasimhan-Ramanan]\label{NR}
The map $\mathrm{NR}$ defined above is an isomorphism. 
Let $\pi: \mathcal{E}\to T$ be a smooth family of semi-stable rank $2$ vector bundles with trivial determinant over $X$. Then the map $\phi : T \to  \mathcal{M}_{\mathrm{NR}}$ associating to $t\in T$ the $S$-equivalence class of $E_t=\pi^{-1}\left(t\right)$ is a morphism. 
\end{thm}

In particular, the moduli space of stable bundles naturally identifies with a Zariski open proper subset of $\mathcal{M}_{\mathrm{NR}}\simeq\P^3$.
A stable bundle has no non-trivial automorphism: we have $\mathrm{Aut}\left(E\right) =  \C^*$ acting by scalar multiplication in the fibres (see \cite{gunning}, thm 29).  Therefore, \emph{the moduli space of holomorphic connections $\nabla:E\to E\otimes\OOMX$ on a given  stable bundle $E$ is an $\mathbb A^3$-affine space}. Note that all holomorphic connections  on a stable bundle 
are irreducible.

\subsection{Semi-stable decomposable bundles}\label{SecDecFlatX}

Let $E=L_0\oplus L_0^{-1}$ with $L_0\in \mathrm{Jac}\left(X\right)$. Given $L\in\mathrm{Pic}^1\left(X\right)$, non-trivial sections
of $E\otimes L$ come from non-trivial sections of $L_0\otimes L$ or $L_0^{-1}\otimes L$. We promtly deduce 
that 
$$D_E=L_0\cdot\Theta+L_0^{-1}\cdot\Theta$$
where $L_0\cdot\Theta$ denotes the translation of $\Theta$ by $L_0$ for the group law on $\mathrm{Pic}\left(X\right)$.
A special case occurs for the $16$ torsion points $L_0^2=\OOX$ for which $L_0=L_0^{-1}$ and hence $D_E=2\left(L_0\cdot\Theta\right)$ is not reduced. 

The moduli space of semi-stable decomposable bundles naturally identifies with the \emph{Kummer variety} 
$$\mathrm{Kum}\left(X\right):=\mathrm{Jac}\left(X\right)/\iota,$$
the quotient of the Jacobian $\mathrm{Jac}\left(X\right)$ by the involution $\iota:L_0\mapsto \iota^*L_0=L_0^{-1}$.
The Narasimhan-Ramanan classifying map provides a canonical embedding
$$\mathrm{NR}:\mathrm{Kum}\left(X\right):=\mathrm{Jac}\left(X\right)/\iota\ \hookrightarrow \mathcal{M}_{\mathrm{NR}}$$
and the image is a quartic surface in $\mathcal{M}_{\mathrm{NR}}\simeq\P^3$.
The moduli space of stable bundles identifies with the complement of this surface.
The $16$ torsion points $L_0^2=\OOX$ of the Jacobian are precisely the fixed points of the involution $\iota$
and yield $16$ conic singularities on $\mathrm{Kum}\left(X\right)$. 

\subsubsection{The $2$-dimensional family of decomposable bundles}

When $L_0^2\not=\OOX$, the corresponding rank $2$ bundle $E=L_0\oplus L_0^{-1}$ 
lies on the smooth part of $\mathrm{Kum}\left(X\right)$. Non-scalar automorphisms come from the 
independent action of $\mathbb G_m$ on the two direct summands: we get a $\mathbb G_m$-action on $\P\left(E\right)$.

Given a connection on $L_0$, we easily deduce a totally reducible connection $\nabla_0$ on $E$
(preserving both direct summands). 
Any other connection will differ from $\nabla_0$ by a Higgs bundle: $\nabla=\nabla_0 + \theta$
where $\theta:E\to E\otimes\OOMX$ is $\OOX$-linear and may be represented 
in the matrix way 
$$\theta=\begin{pmatrix}\alpha&\beta\\ \gamma&-\alpha\end{pmatrix}\ \ \ \text{with}\ \ \ 
\left\{\begin{array}{l}\alpha:L_0\to L_0\otimes\OOMX,\\
\beta:L_0^{-1}\to L_0\otimes\OOMX,\\
\gamma:L_0\to L_0^{-1}\otimes\OOMX.\end{array}\right.$$
Under our assumption that  $L_0^2\not=\OOX$, our space of connections is parametrized
by $\C^2_\alpha\times\C^1_\beta\times\C^1_\gamma$.
Since $E$ has no degree $0$ subbundle other than $L_0$ and $L_0^{-1}$,  reducible connections on E are precisely those for which one of the two direct summands is invariant, \emph{i.e.} $\beta=0$ or $\gamma=0$.
The $\mathbb G_m$-action is trivial on $\alpha$ but not on the two other coefficients:
the quotient $\C^1_\beta\times\C^1_\gamma/\mathbb G_m$ is $\C^*$
after deleting reducible connections (for which $\beta=0$ or $\gamma=0$).
{\it The moduli space of irreducible connections on $E$ is thus given by $\C^2\times\C^*$.}

The involution $\iota$ preserves those connections that are irreducible or totally reducible.
{\it The moduli space of $\iota$-invariant connections is $\C^2\times\C$.}

\subsubsection{The trivial bundle and its $15$ twists}\label{SecTrivialFlatX}

All $16$ special decomposable bundles are equivalent to the trivial one after twisting
by a convenient line bundle. Let us study the case $E=\OOX\oplus\OOX$
which admits the trivial connection $\nabla_0=\mathrm{d}$.
Any other connection is obtained by adding a Higgs bundle of the matrix form
$$\theta=\begin{pmatrix}\alpha&\beta\\ \gamma&-\alpha\end{pmatrix}\ \ \ \text{with}\ \ \ 
\alpha,\beta,\gamma\in \mathrm{H}^0\left(X,\OOMX\right)$$
(here, a trivialization of $E$ is chosen).
Our space of connections is parametrized by 
$\C^2_\alpha\times\C^2_\beta\times\C^2_\gamma$
but now the group acting is $\mathrm{Aut}\left(\P E\right)=\mathrm{PGL}_2$.

Since the trivial connection is $\mathrm{PGL}_2$-invariant, the data of a connection $\nabla=\mathrm{d}+\theta$
is equivalent to the data of the Higgs field $\theta$ itself. Moreover, the determinant map
$$\det\ :\ \mathrm{H}^0\left(\mathfrak sl \left(E\right)\otimes\OOMX\right)\to \mathrm{H}^0\left(\OOMX\otimes\OOMX\right)\ ;\ 
\theta\mapsto -\left(\alpha\otimes\alpha+\beta\otimes\gamma\right)$$
is invariant under the $\mathrm{PGL}_2$-action. Through this map, we claim the following.

\begin{prop}\label{Prop:ModuliTrivialIrred} The moduli space 
of irreducible  trace-free connections on the trivial bundle of rank $2$ over $X$ coincides with the open set in $\mathrm{H}^0\left(X,\Omega_X^1\otimes \Omega_X^1 \right)$ of those quadratic differentials 
that are not the square $\omega\otimes\omega$ of a holomorphic $1$-form $\omega$.\end{prop}

\begin{proof}Note that in our usual coordinates on $X$, we have 
$$\mathrm{H}^0\left(X,\Omega_X^1\right) = \mathrm{Vect}_\mathbb{C}\left(\frac{\mathrm{d}x}{y}, x\frac{\mathrm{d}x}{y}\right).$$

The eigendirections of $\theta$
define a curve $C$ on the total space $X\times\P^1$ of the projectivized trivial bundle (for eigendirections to make sense, we have to compose $\theta$ by  local isomorphisms $E \otimes \Omega_X^1 \to E$; the resulting curve $C$ does not depend on this choice). In a concrete way, for each vector $v \in E$, we compute the 
determinant $v\wedge\theta(v)$. Under trivializing coordinates $\left(1:z\right)\in\P^1$ we find that $C$ is defined by
$$C\ :\ -\gamma +2z \alpha +z^2 \beta=0.$$
It follows that $C$ has bidegree $\left(2,2\right)$ in $X \times \P^1$ (\emph{i.e.} with respect to the variables $\left(y,z\right)$) and is invariant by the hyperelliptic involution $\iota$. Hence it defines a bidegree $\left(1,2\right)$ curve $\underline C\subset\P^1\times \P^1$ (\emph{i.e.} with respect to the variables $\left(x,z\right)$). It is  easy to check that $C$ is reducible 
if and only if $\nabla$ is reducible.
In the irreducible case, the curve $C$ defines a $\left(2:1\right)$-map $\P^1_z\to\P^1_x$
whose Galois involution may be normalized to $z\mapsto -z$ under the $\mathrm{PGL}_2$-action.
After this normalization, we get that $\alpha=0$ and the involution $\iota$ lifts as $\left(x,y,z\right)\mapsto\left(x,-y,-z\right)$.
in particular, $z=0$ and $z=\infty$ are the two $\iota$-invariant subbundles. This normalization is unique
up to action of the dihedral group $\mathbb D_\infty$ (preserving $z\in\{0,\infty\}$). Clearly, the determinant $-\beta\otimes\gamma$
is invariant and determines $\nabla$ up to this action since, in genus $2$, any quadratic form splits as a product
$\det\left(\theta\right)=-\beta\otimes\gamma$. Finally, one can easily check that the following properties are equivalent:
\begin{itemize}
\item[$\bullet$]  $\nabla$ (or $\theta$) is reducible,
\item[$\bullet$] the $\left(1,2\right)$ curve $\underline C$ splits,
\item[$\bullet$] the determinant $\det\left(\theta\right)$ viewed on $\P^1_x$ has a double zero,
\item[$\bullet$]  the determinant $\det\left(\theta\right)$ (viewed on $X$) is a square.
\end{itemize}
\end{proof}

It may be of interest to pursue the discussion of the proof above
in the reducible case. In this case, $C$ is reducible and has a bidegree $\left(0,1\right)$-factor which is $\nabla$-invariant.  We can normalize 
$$\theta=\begin{pmatrix}\alpha&\beta\\ 0&-\alpha\end{pmatrix}.$$
The gauge freedom is given by the group of upper-triangular matrices
and we are led to the following cases
\begin{enumerate}
\item[(1)] $\alpha\not=0$ and $\beta$ is not proportional to $\alpha$ (in particular $\not=0$); 
the monodromy is \emph{affine} but \emph{non-abelian} and the curve $\underline C$ splits as a union of 
irreducible bidegree $\left(0,1\right)$ and $\left(1,1\right)$ curves.
\item[(2)] $\alpha\not=0$ and $\beta$ is proportional to $\alpha$: we can normalize $\beta=0$; 
the monodromy is \emph{diagonal} and the curve $\underline C$ splits as a union of 
two bidegree $\left(0,1\right)$ curves and  one $\left(1,0\right)$ curve located at the vanishing point of $\alpha$.
\item[(3)] $\alpha=0$ and $\beta\not=0$; 
the monodromy is \emph{unipotent} but non-trivial and the curve $\underline C$ splits as a union of 
a bidegree $\left(0,1\right)$ curve with multiplicity $2$ and a bidegree $\left(1,0\right)$ curve located at the vanishing point of $\beta$.
\item[(4)] $\alpha=0$ and $\beta=0$ and we get the \emph{trivial} connection (the curve $\underline C$ has
vanishing equation and is not defined).
\end{enumerate}
The determinant map $\det$ defined above takes values in the set of quadratic differentials over $X$.  Those are of the form $$\nu = \frac{\nu_0+\nu_1x+\nu_2x^2}{x\left(x-1\right)\left(x-r\right)\left(x-s\right)\left(x-t\right)}\mathrm{d}x\otimes \mathrm{d}x.$$
It is a square, say $\det\left(\theta\right)=-\alpha\otimes\alpha$, if and only if $\nu_1^2=4\nu_0\nu_2$.
In this case, $\alpha$ is uniquely defined up to a sign. It follows that a fiber $\det^{-1}\left(\nu\right)$ of the determinant map above is
\begin{itemize}
\item [$\bullet$] a unique irreducible connection (up to $\PGL$-isomorphism) if $\nu_1^2\neq 4\nu_0\nu_2$; 
\item [$\bullet$] the union of two reducible connections of type (1) (upper and lower triangular once $\alpha$ is fixed) and a reducible connection of type~(2) over a smooth point of the cone  $\nu_1^2=4\nu_0\nu_2$; 
\item [$\bullet$] the union of the trivial connection (4) and a $1$-parameter family of reducible connections of type (3) over the singular point $\nu_0=\nu_1=\nu_2=0$.
\end{itemize}
The moduli space of $\iota$-invariant connections on the trivial bundle thus is not separated. Note that we obtain a double-cover of the  moduli space of $\iota$-invariant but non trivial connections on the trivial bundle by considering the family of connections of the form $$\mathrm{d}+\begin{pmatrix}0&\beta\\\gamma&0\end{pmatrix},$$
where one of the coefficients $\beta_0, \beta_1, \gamma_0, \gamma_1 \in \C$ is normalized to $1$ 
 (with the obvious transition maps). Here  $\beta = \beta_0\frac{\mathrm{d}x}{y}+\beta_1 x\frac{\mathrm{d}x}{y}$  and $\gamma = \gamma_0\frac{\mathrm{d}x}{y}+\gamma_1 x\frac{\mathrm{d}x}{y}$. 

\subsection{Semi-stable indecomposable bundles}

In this case, the bundle is a non-trivial extension $0\to L_0\to E\to L_0^{-1}\to 0$ for some $L_0\in\mathrm{Jac}\left(X\right)$.
It is $S$-equivalent in the sense of Narasimhan-Ramanan to  the corresponding trivial extension.
For fixed $L_0$, the moduli space of such extensions is isomorphic to $\P \mathrm{H}^1\left(X,L_0^2\right)$ which,
by Serre duality, identifies with $\P \mathrm{H}^0\left(X,L_0^{-2}\otimes\OOMX\right)$. Again, the discussion
splits into two cases.

\subsubsection{The $1$-dimensional family of unipotent bundles and its $15$ twists}

When $L_0^2=\OOX$, the moduli space of non-trivial extensions $0\to \OOX\to E\to \OOX\to 0$ 
is parametrized by $\P \mathrm{H}^1\left(X,\OOX\right)\simeq\P \mathrm{H}^0\left(X,\OOMX\right)\simeq\P^1$; we call \emph{unipotent bundles} such  bundles $E$.
Following \cite{MaruyamaAut}, the automorphism group of $E$ is $\mathrm{Aut}\left(E\right)=\mathbb G_m\ltimes\mathbb G_a$.
The action of $\mathbb G_a$ is faithfull in restriction to each fiber $E_w$, unipotent and fixing the unique subbundle $\OOX\subset E$;
the action of $\mathbb G_m$ is scalar as usual.

For a convenient open covering $(U_i)$ of $X$, the bundle $E$ is defined by a matrix cocycle of the form 
$$M_{ij}=\begin{pmatrix} 1 & b_{ij}\\ 0 & 1 \end{pmatrix}$$
where $(b_{ij})\in \mathrm{H}^1(X,\OOX)$ is a non trivial scalar cocycle. Moreover, from the short exact sequence
$$0\to \mathrm{H}^0(X,\Omega^1_{X})\to \mathrm{H}^1(X,\C)\to \mathrm{H}^1(X,\OOX)\to 0,$$
$(b_{ij})$ may be lifted to $\mathrm{H}^1(X,\C)$, so that $E$ is flat: the local connections $\mathrm{d}_X$ over $U_i$ glue together
to form a global connection  (non-canonical) $\nabla_0$ with unipotent monodromy.
Conversely, if a connection $(E,\nabla)$ has unipotent monodromy, defined by say
$$A_1=\begin{pmatrix}1&a_1\\ 0&1\end{pmatrix},\ B_1=\begin{pmatrix}1&b_1\\ 0&1\end{pmatrix},\ A_2=\begin{pmatrix}1&a_2\\ 0&1\end{pmatrix},\ B_2=\begin{pmatrix}1&b_2\\ 0&1\end{pmatrix}$$
(with respect to the standard basis (\ref{eq:standardfondgroupg2})), then $E$ is either the trivial bundle, or a unipotent bundle; in fact, we are in the former case if, and only if,
$\left(a_1,b_1,a_2,b_2\right)$ is the period data of a holomorphic $1$-form on $X$.

\begin{prop}Let $\nabla_0$ be a unipotent connection on $E$ like above.
Then the general connection on $E$ can be described as 
$$\nabla=\nabla_0+\lambda_1\theta_1+\lambda_2\theta_2+\lambda_3\theta_3+\lambda_4\theta_4$$
with $(\lambda_i)\in \C^4$ so that the $\mathrm G_a$-action is given by
$$\left(\ c\ ,\ \begin{pmatrix}
\lambda_1\\
\lambda_2\\
\lambda_3\\
\lambda_4
\end{pmatrix}\ \right)\longrightarrow\begin{pmatrix}
\lambda_1\\
\lambda_2-c\lambda_1\\
\lambda_3+2c\lambda_2-c^2\lambda_1\\
\lambda_4
\end{pmatrix}$$
Moreover, reducible (resp. unipotent) connections are given by $\lambda_1=0$ (resp. $\lambda_1=\lambda_2=0$).
The moduli space of irreducible connections on $E$ identifies with $\C^*\times\C^2$.
\end{prop}

\begin{proof}A general  trace-free connection on $E$ is defined by a collection 
$$\mathrm{d}+\theta_i\ \ \ \text{where}\ \ \ \theta_i=\begin{pmatrix} \alpha_i & \beta_i \\ \gamma_i & -\alpha_i \end{pmatrix}$$
are matrices of holomorphic $1$-forms on $U_i$ satisfying the compatibility condition 
$$\theta_j=M_{ij}^{-1}\theta_i M_{ij}+M_{ij}^{-1} \mathrm{d}M_{ij}$$
on $U_i\cap U_j$ or, equivalently, 
\begin{equation}\label{eq:compatconnecnilpbundle}
\left\{\begin{matrix}
\alpha_i-\alpha_j&=&b_{ij}\gamma_i \hfill \\
\beta_i-\beta_j&=&-2b_{ij}\alpha_i+b_{ij}^2\gamma_i \\
\gamma_i-\gamma_j&=& 0\hfill
\end{matrix}\right.
\end{equation}
When $\alpha_i=\gamma_i=0$, we precisely obtain all those connections with unipotent monodromy on $E$;
the second equation (\ref{eq:compatconnecnilpbundle}) then tells us that $(\beta_i)$ defines
a global holomorphic $1$-form $\beta\in H^0(X,\Omega^1_{X})$. 

When $\gamma_i=0$, we get all reducible connections on $E$. The first equation (\ref{eq:compatconnecnilpbundle})
tells us that $(\alpha_i)$ defines a global holomorphic $1$-form $\alpha\in H^0(X,\Omega^1_{X})$.
To solve the second equation (\ref{eq:compatconnecnilpbundle}), we need that the image under Serre duality
$$\begin{matrix}
H^1(X,\OOX)\times H^0(X,\Omega^1_{X}) & \to & H^1(X,\Omega^1_{X}) & \stackrel{\sim}{\rightarrow} & \C \\
\left(\ (b_{ij})\ ,\ \alpha\ \right) & \mapsto & (b_{ij}\alpha) &&
\end{matrix}$$
is the zero cocycle. In other words, $\alpha$ must belong to the orthogonal $(b_{ij})^\perp$ (with respect to Serre duality).
In this case, we can solve $(\beta_i)$, and the solution is unique up to addition by a global holomorphic $1$-form $\beta$.

Irreducible connections occur for $\gamma\not=0$ (note that the third equation (\ref{eq:compatconnecnilpbundle})
states that $(\gamma_i)$ is a global $1$-form). 
Then, the first equation (\ref{eq:compatconnecnilpbundle}) imposes that $\gamma\in (b_{ij})^\perp$ (the orthogonal
for Serre duality).
Therefore, the collection $(\alpha_i)$ solving the cocycle $(b_{ij}\gamma)$ is unique up to the addition of a global holomorphic $1$-form
$\alpha\in H^0(X,\Omega^1_{X})$. Finally, to solve the second equation in $(\beta_i)$, we have to insure that the cocycle
$$\left(\ -2b_{ij}\alpha_i+b_{ij}^2\gamma\ \right)\in H^1(X,\Omega^1_{X})$$
is zero, which can be achieved by conveniently using the freedom $\alpha$ when solving the first equation.
Precisely, if we add some global $1$-form $\alpha$ to the collection $(\alpha_i)$, 
then we translate the previous cocycle by $(-2b_{ij}\alpha)$. For a convenient choice of $\alpha$ (or $(\alpha_i)$), 
the cocycle becomes trivial. Note that we still have the freedom to add any $1$-form $\alpha$ belonging
to the orthogonal $(b_{ij})^\perp$. At the end, we can find a solution $(\beta_i)$ which is unique up to addition by a global
holomorphic $1$-form $\beta\in H^0(X,\Omega^1_{X})$.

Given an irreducible connection as above, defined by $(\alpha_i)$, $(\beta_i)$ and $\gamma\not=0$, and given 
a global holomorphic $1$-form $\beta\not\in (b_{ij})^\perp$,
 it follows from above case-by-case discussion that any connection $\nabla$ on $E$ takes the form 
$$d_X+\lambda_1\begin{pmatrix} \alpha_i & \beta_i \\ \gamma & -\alpha_i \end{pmatrix}+\lambda_2\begin{pmatrix} \gamma & -2\alpha_i \\ 0 & -\gamma \end{pmatrix}+\lambda_3\begin{pmatrix} 0 & \gamma \\ 0 & 0 \end{pmatrix}+\lambda_4\begin{pmatrix} 0 & \beta \\ 0 & 0 \end{pmatrix}$$
over charts $U_i$, for convenient scalars $\lambda_i$. Unipotent bundle automorphisms are given in these charts
by a constant matrix $\begin{pmatrix}1 & c \\ 0 & 1 \end{pmatrix}$, with $c\in\C$ not depending on $U_i$, and it is 
straightforward to check that the action on $\lambda_i$ is the one of the statement.
\end{proof}

\subsubsection{Affine bundles}

When $L_0^2\not=\OOX$, then $\P \mathrm{H}^0\left(X,L_0^{-2}\otimes\OOMX\right)$ reduces 
to a single point: there is only one non-trivial extension of $L_0^{-1}$ by $L_0$ up to isomorphism. Following \cite{MaruyamaAut}, the automorphism group of  such a bundle $E$ is $\mathrm{Aut}\left(E\right)=\mathbb{C}^*$. In particular, we have $\mathbb{P}\mathrm{Aut}\left(E\right)=\{1\}$ and \emph{the space of holomorphic connections on $E$ up to automorphisms is an affine $\mathbb{A}^3$ space with homogeneous part $\mathrm{H}^0\left(\mathfrak{sl}(E)\otimes \Omega^1_X\right)\simeq \mathbb{C}^3$}. 

A curious phenomenon occurs for affine bundles $E$: {\it all connections on $E$ are reducible, none of them is totally reducible}. 
Indeed, $L_0$
is the unique subbundle of degree $0$, but is not invariant by the hyperelliptic involution.
Therefore, the vector bundle $E$ itself is not $\iota$-invariant. This implies that the monodromy
of a connection $\nabla$ on $E$ can be neither irreducible, nor totally reducible.
Note that this phenomenon 
 does not occur for higher genus 
(see \cite{HitchinSelfDual}, Prop. (3.3), p.70). Note further that even if affine bundles do not allow hyperelliptic descent, we can see them in smooth charts of the moduli space of flat bundles using Tyurin's approach (see Section \ref{SecTyurinSubBundle}).

\subsection{Unstable and indecomposable: the $6+10$ Gunning bundles}\label{Gunnbdle}

 There are $16$ \emph{theta characteristics}, {\it i.e.}  square-roots of $\OOMX=\OOX( \KX)$. They  split into
 \begin{itemize}
 \item $6$ odd theta characteristics $\OOX\left([w_i]\right)$, $i=0,1,r,s,t,\infty$;
 \item $10$ even theta characteristics $\OOX\left([w_i]+[w_j]-[w_\infty]\right)$, $i\not=j\not=\infty$.
 \end{itemize}
Given a theta characteristic $\vartheta$, there is a unique non-trivial extension $0\to \vartheta\to E_\vartheta\to \vartheta^{-1}\to 0$ up to isomorphism, 
which is called the {\it Gunning bundle} $E_\vartheta$ associated to $\vartheta$. We talk about {\it even or odd} Gunning bundle
depending on the nature of $\vartheta$.
We have $\mathrm {Aut}\left(E_\vartheta\right)\simeq \mathbb G_m\ltimes H^0\left(X,\OOMX\right)$ (see  \cite{MaruyamaAut});
the group $H^0\left(X,\OOMX\right)$ is acting by unipotent bundle automorphisms on fibers $E\vert_w$, fixing the subbundle $\vartheta$.

A connection $\nabla$ on $E$ necessarily satisfies Griffiths transversality with respect to the flag 
$0\subset\vartheta\subset E_\vartheta$ and defines an ''oper'' (see \cite{BeilinsonDrinfeld}).
Following \cite{GunningCoord}, the data of $\nabla$ up to automorphism of $E$ 
is equivalent to the data of a projective 
structure on $X$. Moreover, any two projective structures
differ on $X$ by a quadratic differential: once a projective structure has been chosen,
the moduli space identifies with $\mathrm{H}^0\left(X,\OOX\left(2 \KX\right)\right)$. However, there is no
natural choice of ''origin'', \emph{i.e.} there is no canonical projective structure on $X$ from an algebraic point of view (see   \cite{LorayMarin}).
\emph{The moduli space of (irreducible) connections on $E_\vartheta$ is therefore an $\mathbb{A}^3$-affine space.}

Recall that the Narasimhan-Ramanan classifying map is defined only for semi-stable
bundles and thus not for Gunning bundles. This has the following
geometric reason: We say two rank 2 vector bundles $E$ and $E'$ with trivial determinant over $X$ are \emph{arbitrarily close}
if there are smooth families of vector bundles $(E_t)_{t\in \mathbb{A}^1}$ and $(E'_t)_{t\in \mathbb{A}^1}$ over $X$ such that $E_t=E'_t$ for each $t\neq 0$ and $E_0=E$, $E'_0=E'$. By the Narasimhan-Ramanan-theorem, two arbitrarily close semi-stable vector bundles  are $S$-equivalent. 
It turns out that the Gunning bundle $E_\vartheta$
is arbitrarily close to any semi-stable extension $0\to \vartheta^{-1}\to E'\to \vartheta\to 0$ (see Proposition \ref{prop:NonSepExt}). These are precisely the semi-stable bundles whose
corresponding divisor $D_{E'}\sim 2\Theta$ (see Section \ref{NarRam}) passes through the point 
$\vartheta$ on $\mathrm{Pic}^1\left(X\right)$. They define a $2$-plane in $\mathcal M_{\mathrm{NR}}$ which we will call {\it Gunning plane}  
and denote it by $\Pi_\vartheta$.

The intersection of $\Pi_\vartheta$ with the Kummer surface is easily described as
$$\Pi_\vartheta\cap\mathrm{Kum}\left(X\right)=\{L_0\oplus L_0^{-1}\ |\ L_0\in \vartheta^{-1}\cdot\Theta\}.$$
In fact, the $16$ Gunning planes $\Pi_\vartheta$ are well-known; each of them is tangent to the Kummer surface along a conic
passing through $6$ of the $16$ nodes. The above description gives a natural parametrization 
of the hyperelliptic cover of this marked conic by the curve $X$ itself (via the $\Theta$ divisor).
Precisely, for each $\Pi_\vartheta$, the $6$ corresponding nodes are those parametrized
by the $2$-torsion points $\vartheta^{-1}\otimes \mathcal{O}\left([w_i]\right)$ where $w_i$ runs over
the six Weierstrass points.
Conversely, through each node pass $6$ of the $16$ planes. This so-called $\left(16,6\right)$ configuration
is classical (see \cite{Hudson,GonzalezDorrego}) as well as the interpretation in terms
of the moduli space of vector bundles (see \cite{NR,Bolognesi}). However, 
the interpretation of $\Pi_\vartheta$ in terms of semi-stable bundles arbitrarily close to Gunning bundles seems 
to be new so far.

\begin{prop}\label{prop:NonSepExt}Given two extensions 
$$0\to L\to E_0\to L'\to 0\ \ \ \text{and}\ \ \ 0\to L'\to E'_0\to L\to 0$$
of the same (but permuted) line bundles, there are two deformations
$E_t$ and $E'_t$ of these bundles (parametrized by $\mathbb{A}^1$) such that 
$E_t\simeq E'_t$ for $t\not=0$.
\end{prop}

\begin{proof}The vector bundles $E_0$ and $E_0'$ are respectively defined
by a cocycle of the form
\begin{equation}\label{cocycles}\begin{pmatrix}a_{ij}&b_{ij}\\ 0&d_{ij}\end{pmatrix}\ \ \ \text{and}\ \ \ \begin{pmatrix}a_{ij}&0\\ c_{ij}&d_{ij}\end{pmatrix}\end{equation}
for a convenient open covering $\left(U_i\right)$ of $X$. Here, $\left(a_{ij}\right)$ and $\left(d_{ij}\right)$ are cocycles respectively defining $L$ and $L'$.
{\it We claim} that this can be achieved with only two Zariski open sets $X=U_1\cup U_2$ so that we can neglect the cocycle
condition. Before proving this claim, let us show how to conclude the proof. Consider the deformations $E_t$ and $E_t'$ respectively defined by
$$\begin{pmatrix}a_{ij}&b_{ij}\\ t c_{ij}&d_{ij}\end{pmatrix}\ \ \ \text{and}\ \ \ \begin{pmatrix}a_{ij}&tb_{ij}\\  c_{ij}&d_{ij}\end{pmatrix}.$$
They define the same vector bundle up to isomorphism for $t\not=0$ since these cocycles
are conjugated by the automorphism of $L\oplus L'$ defined in the matrix way by
$\begin{pmatrix}t&0\\ 0&1\end{pmatrix}$. On the other hand, we clearly have $E_t\to E_0$ and $E_t'\to E_0'$
when $t\to 0$. For a general open covering, these matrices fail to satisfy the cocycle condition $A_{ij}A_{jk}A_{ki}=I$; this is why 
we need to work with only two open sets.

Although it might be standard, let us prove the claim. Up to tensoring by a very ample line bundle
$\tilde{ L}=\OOX\left(\tilde{D}\right)$, we may assume that  $L$, $L'$, $E_0$ and $E_0'$ are all generated
by global holomorphic sections. Choose one such section $s_1$ for $L$; it is then easy to construct
another section $s_2$ such that the corresponding (effective) divisors $D_1$ and $D_2$ have disjoint
support. Indeed, given any non-zero section $s_2$, for any common zero with $s_1$ one can find some 
section $s$ non-vanishing at that point: one can then perturbe $s_2:=s_2+\epsilon\cdot s$. 
This means that $L$ may be trivialized on each open set $U_i=X\setminus\mathrm{supp}\left(D_i\right)$, $i=1,2$,
and therefore defined with respect to this covering by a single cocycle $a_{12}$. In a similar way,
we can construct sections $\sigma_1$ and $\sigma_2$ of $E_0$ such that the two sections
$s_i\wedge\sigma_i$ of $\det\left(E_0\right)$ have disjoint zeroes. In other words, possibly by deleting more points 
in the open sets $U_i$, the vector bundle $E_0$ can simultaneously be trivialized on each of these open sets,
and is therefore defined by a cocycle of the above form. To deal simultaneously with $L'$ and $E_0'$,
the easiest way is to consider the zero set of sections $s_i\wedge\sigma_i\wedge s_i'\wedge\sigma_i'$
of $\det\left(E_0\oplus E_0'\right)$. Finally, the same manipulation can be done with sections $\tilde s_i$ of  the ample bundle $\tilde L$: considering the zeros of sections $s_i\wedge\sigma_i\wedge s_i'\wedge\sigma_i'\wedge \tilde s_i$
of $\det\left(E_0\oplus E_0'\oplus \tilde L\right)$ we can assume that the sections 
$ \sigma_i,s_i, s_i'$ and $\tilde s_i$ have no common zeroes for $i=1,2$. Tensoring with $\tilde{L}^{\otimes -1}$ we have constructed bases $\left(s_i, \sigma_i\right)$ (resp. $\left(\sigma_i', s_i'\right)$) of $E|_{U_i}$ (resp. $E'|_{U_i}$) with $i=1,2$ such that the corresponding cocyles of $E$ and $E'$ respectively are of the form (\ref{cocycles}).
\end{proof}

In particular, two semi-stable rank 2 bundles are arbitrarily close if and only if they are $S$-equivalent. 

\subsection{Computation of a system of coordinates}\label{SecComputeNR}
In this section, we construct coordinates on the Narasimhan-Ramanan moduli space allowing us to express explicitly the Kummer surface of strictly semi-stable bundles as well as the involutions of the moduli space given by tensor products with 2-torsion line bundles. 

For all computations, the curve $X$ is the smooth compactification of the affine complex curve defined by 
$$X\ :\ y^2=x(x-1)(x-r)(x-s)(x-t)$$
where $0,1,r,s,t\in\C$ are pair-wise distinct; we denote by $\infty$ the point at infinity.

Let us first calculate a basis of $\mathrm{H}^0(\mathrm{Pic}^1(X),\OX{2\Theta})$ in order to introduce explicit projective 
coordinates on the three-dimensional projective space 
$$\P^3_{\mathrm{NR}}:=\mathbf{P}\mathrm{H}^0(\mathrm{Pic}^1(X),\OX{2\Theta}).$$
Since $\mathrm{Pic}^1(X)$ is birationally equivalent to the symmetric product $X^{(2)}$,
rational functions on $\mathrm{Pic}^1(X)$ can be conveniently expressed as symmetric rational functions on $X\times X$.
$$\begin{xy} \xymatrix{ X \times X \ar@{->>}[r]& X^{(2)}  \ar@{..>}[r]^{\hspace{-.1cm}\phi^{(2)}}&\mathrm{Pic}^2(X) \ar[r]_{\sim}^{-[\infty]}&\mathrm{Pic}^1(X)
 }\end{xy}$$
 \vspace{-.7cm}\\
$\textrm{ } $ \hspace{4.9cm} {\small $\begin{xy} \xymatrix{ \{P,Q\}\ar@{|->}[r] & [P]+[Q]}\end{xy}$}\vspace{.3cm}\\
The pull-back of the divisor $\Theta \subset  \mathrm{Pic}^1(X)$ (resp. $\Theta	 + [\infty] \subset \mathrm{Pic}^2(X)$)  to $X\times X$ is  
$\overline{\Delta}+\infty_1+\infty_2$, where 
\begin{itemize}
\item $\overline{\Delta}$ is the anti-diagonal $\{(P,Q)\in X\times X~|~Q=\iota(P)\}$, 
\item $\infty_1$ is the divisor $\{\infty\}\times X$ and  
\item $\infty_2$ the divisor $X\times\{\infty\}$.
\end{itemize}
The pull-back  to $X\times X$  of $2\Theta$, viewed as a divisor on $\mathrm{Pic}^1(X)$ is then  (see Figure \ref{Pic2X}):  $$2\overline{\Delta}+2\infty_1+2\infty_2.$$

\begin{lem}\label{Lem:GenFunctTheta}
Let $(P_1,P_2)=((x_1,y_1),(x_2,y_2))$ be coordinates of $X \times X$.
Then $$\mathrm{H}^0\left(X\times X, \OOX^{sym}(2\overline{\Delta}+2\infty_1+2\infty_2)\right) = \mathrm{Vect}_{\mathbf{C}}(1,Sum,Prod,Diag)$$
with
{\begin{equation}\label{diag1}  
\begin{array}{rll}1 : (P_1,P_2) & \mapsto & 1\vspace{.2cm}\\ Sum : (P_1,P_2)& \mapsto &x_1+x_2\vspace{.2cm}\\ Prod : (P_1,P_2)&\mapsto& x_1x_2,\\
 Diag :(P_1,P_2)& \mapsto &\left(\frac{y_2-y_1}{x_2-x_1}\right)^2-(x_1+x_2)^3+(1+\sigma_1)(x_1+x_2)^2 \hspace{.1cm}+\\&&+x_1x_2(x_1+x_2)-(\sigma_1+\sigma_2)(x_1+x_2)\end{array}\end{equation}}
where $\sigma_1, \sigma_2$ and $\sigma_3$ are the following constants: $\sigma_1=r+s+t, \sigma_2=rs+st+tr,  \sigma_3=rst$.
\end{lem}

\begin{proof}
We have $\mathrm{h}^0(X\times X, \OOX^{sym}(2\overline{\Delta}+2\infty_1+2\infty_2)) =\mathrm{h}^0\left(\mathrm{Pic}^1(X), \mathcal{O}(2\Theta)\right) =4$ (see \cite{NR} or \cite{Mumford}). 
The function $Diag$ can be rewritten as 
\begin{equation}\label{diag} 
\begin{array}{rrr} Diag = \frac{1}{(x_1-x_2)^2}&\cdot&\left[-2y_1y_2-2(1+\sigma_1)x_1^2x_2^2-(\sigma_2+\sigma_3)(x_1^2+x_2^2)\right.\\&&\left.+(x_1+x_2)\cdot \left(x_1^2x_2^2+(\sigma_1+\sigma_2)x_1x_2+\sigma_3\right)\right]\end{array}\end{equation} 
The expression of $Diag$ in (\ref{diag1}) shows that it has no poles off the anti-diagonal and the infinity (and in particular no poles on the diagonal). From the  expression (\ref{diag}) follows easily that $Diag$ has polar divisor $2\overline{\Delta}+2\infty_1+2\infty_2$. Indeed,
if $u_1$ is the local parameter for $X_1$ near $\infty_1$ defined by $x_1=\frac{1}{u_1^2}$, then the principal part of the generating functions
is given by 
$$1,\ \ \ Sum=\frac{1}{u_1^2}+x_2,\ \ \ Prod=\frac{x_2}{u_1^2}\ \ \ \text{and}\ \ \ Diag\sim\frac{x_2^2}{u_1^2}-\frac{y_2}{u_1}+\cdots$$
As a section of $\mathrm{H}^0(\mathrm{Pic}^1(X),\OX{2\Theta})$, the function $1$ vanishes twice along $\Theta$
while the other ones do not vanish identically on $\Theta$.
\end{proof}

\begin{figure}[H]
\centerline{\resizebox{85mm}{!}{\input{generateursNR1.pstex_t}}\hspace{1cm} \resizebox{52mm}{!}{\input{generateursNR2.pstex_t}}}  
\centerline{$ $ \hspace{2cm}\resizebox{69mm}{!}{\input{generateursNR3.pstex_t}}}  
\vspace{.5cm}   
\caption{$X^2$ as a rational cover of $\mathrm{Jac}(X)$.}\label{Pic2X}
\end{figure}

In the sequel, denote by $(v_0:v_1:v_2:v_3)$ the projective coordinate on $\P^3_{\mathrm{NR}}$ representing the function
$$v_0+v_1\cdot Sum+v_2\cdot Prod+v_3\cdot Diag.$$

In order to compute the strictly semi-stable locus, namely the Kummer surface embedded in $\mathcal{M}_{\mathrm{NR}}$, it is enough to consider
the image in $\P^3_{\mathrm{NR}}$ of decomposable semi-stable bundles. Let $L=\mathcal{O}_X([\underline{P}_1]+[\underline{P}_2]-[\infty])\in \mathrm{Pic}^1(X)$ be a line bundle such that $L^2\neq \OX{ \KX}$ and denote by $\widetilde{L}$ the associated degree 0 bundle $\widetilde{L}=\mathcal{O}_X([\underline{P}_1]+[\underline{P}_2]-2[\infty])$.
Let us now calculate the explicit coordinates of the corresponding Narasimhan-Ramanan divisor $\widetilde{L}\cdot\Theta +\widetilde{L}^{-1}\cdot \Theta$ on $\mathrm{Pic}^1(X)$, which is linearly equivalent to the divisor $2\Theta$ (see Section \ref{SecDecFlatX}). The first component $\widetilde{L}\cdot\Theta$ is parametrized by 
$$X\to\mathrm{Pic}^1(X)\ ;\ Q\mapsto [\underline{P}_1]+[\underline{P}_2]+[Q]-2[\infty].$$ 
Setting $[\underline{P}_1]+[\underline{P}_2]+[Q]-2[\infty]\sim[P_1]+[P_2]-[\infty]$, we get that 
$[\underline{P}_1]+[\underline{P}_2]+[Q]$ belongs to the linear system $[P_1]+[P_2]+[\infty]$.
This latter one is generated by the two functions $1$ and $f(P):=\frac{y+y_1}{x-x_1}-\frac{y+y_2}{x-x_2}$
on the curve. Therefore, $[P_1]+[P_2]-[\infty]\in \widetilde{L}\cdot\Theta$ (the support of) if, and only if,
$f(\underline{P}_1)=f(\underline{P}_2)$; this gives the following equation for $P_1=(x_1,y_1)$ and $P_2=(x_2,y_2)$:
$$\frac{\underline{y}_1+y_1}{\underline{x}_1-x_1}-\frac{\underline{y}_1+y_2}{\underline{x}_1-x_2}\ =\ \frac{\underline{y}_2+y_1}{\underline{x}_2-x_1}-\frac{\underline{y}_2+y_2}{\underline{x}_2-x_2}.$$
 The equation for the other component $\widetilde{L}^{-1}\cdot\Theta$ is deduced by changing signs
$\underline{y}_i\to-\underline{y}_i$ for $i=1,2$. Taking into account the two equations, we get an equation for $\widetilde{L}\cdot\Theta +\widetilde{L}^{-1}\cdot \Theta$:
\begin{center}$\left(\frac{\underline{y}_1+y_1}{\underline{x}_1-x_1}-\frac{\underline{y}_1+y_2}{\underline{x}_1-x_2}-\frac{\underline{y}_2+y_1}{\underline{x}_2-x_1}+\frac{\underline{y}_2+y_2}{\underline{x}_2-x_2}\right)
\left(\frac{-\underline{y}_1+y_1}{\underline{x}_1-x_1}-\frac{-\underline{y}_1+y_2}{\underline{x}_1-x_2}-\frac{-\underline{y}_2+y_1}{\underline{x}_2-x_1}+\frac{-\underline{y}_2+y_2}{\underline{x}_2-x_2}\right)=0$\end{center}
 which,
after reduction, writes 
\begin{equation}\label{Dualequation}-Diag(\underline{P}_1,\underline{P}_2)\cdot 1\ +\ Prod(\underline{P}_1,\underline{P}_2)\cdot Sum\ -\ 
Sum(\underline{P}_1,\underline{P}_2)\cdot Prod\ +\ 1 \cdot Diag\ =\ 0\end{equation}

\begin{rem}The symmetric form of this equation is due to the fact that for any vector bundle $E \in \mathcal{M}_{\mathrm{NR}}$ and any line bundle $L\in \mathrm{Pic}^1(X)$ such that $\mathrm{h}^0(X,E\otimes L)>0$, the divisor $D_E$ associated to $E$ and the divisor  $\widetilde{L}\cdot\Theta +\widetilde{L}^{-1}\cdot \Theta$ associated to $\widetilde{L}\oplus \widetilde{L}^{-1}$ intersect precisely in $L$ and $\iota(L)$ on $\mathrm{Pic}^1(X)$.  
\end{rem}

Hence, according to equation (\ref{Dualequation}), the Kummer embedding 
$$\begin{array}{ccc}
 \mathrm{Jac}(X)&\to &\mathrm{Kum}(X)\subset\P^3_{\mathrm{NR}}\\
\OX{[\underline{P}_1]+[\underline{P}_2]-2[\infty]}&\mapsto&(v_0:v_1:v_2:v_3)
\end{array}$$
is explicitely given by
\begin{equation}\label{DecompInNR}
(v_0:v_1:v_2:v_3)=\left(-Diag(\underline{P}_1,\underline{P}_2)\ :\ Prod(\underline{P}_1,\underline{P}_2)\ :\ -Sum(\underline{P}_1:\underline{P}_2)\ :\ 1\right)
\end{equation}
One can now eliminate parameters $\underline{P}_1$ and $\underline{P}_2$ from (\ref{DecompInNR}) as follows: 
express $\underline{y}_1\underline{y}_2$ in terms of functions $\underline{x}_1+\underline{x}_2$ and $\underline{x}_1\underline{x}_2$
and variable $v_0/v_3$, so that the square can be replaced by 
$$\left(\underline{y}_1\underline{y}_2\right)^2=\prod_{w=0,1,r,s,t}\left(w^2-(\underline{x}_1+\underline{x}_2)w+\underline{x}_1\underline{x}_2\right);$$
then replace $\underline{x}_1\underline{x}_2$ and $\underline{x}_1+\underline{x}_2$ by $v_1/v_3$ and $-v_2/v_3$ respectively.
We get 
$$\begin{array}{rrcc}\mathrm{Kum}\left(X\right)~:\\
0=&(v_0v_2-v_1^2)^2&\cdot &1\vspace{.2cm}\\
&-2\left[[(\sigma_1+\sigma_2)v_1+(\sigma_2+\sigma_3)v_2](v_0v_2-v_1^2)\right.\\&+\left.2(v_0+\sigma_1 v_1)(v_0
+v_1)v_1+2(\sigma_2v_1+\sigma_3 v_2)(v_1+v_2)v_1\right]&\cdot&v_3\vspace{.2cm}\\
&-2\sigma_3(v_0v_2-v_1^2)+\left[\left[(\sigma_1+\sigma_2)^2v_1+(\sigma_2+\sigma_3)^2v_2\right](v_1+v_2)\right.\\&\left.-(\sigma_1+\sigma_3)^2v_1v_2+4[(\sigma_2+\sigma_3)v_0-\sigma_3v_2]v_1\right]&\cdot&\vspace{.2cm}v_3^2\\
&-2\sigma_3\left[(\sigma_1+\sigma_2)v_1-(\sigma_2+\sigma_3)v_2\right]&\cdot&\vspace{.2cm}v_3^3\\
&+ \sigma_3^2&\cdot&v_3^4.\end{array}$$
Here, we see that $v_3=0$ is a (Gunning-) plane tangent to $\mathrm{Kum}\left(X\right)$ along a conic.
In the formula above as in the following, we denote:

\begin{center}
\begin{tabular}{|ccc|}
\hline 
\rowcolor{yellow} $\sigma_1=r+s+t,\quad $ & $\sigma_2=rs+st+tr,\quad$ &  $\sigma_3=rst$.\\
\hline
\end{tabular}
\end{center}


Following formula (\ref{DecompInNR}), we can compute the locus of the trivial bundle $E_0$ \\
and its $15$ twists $E_\tau:=E_0\otimes\OX{\tau}$,  where $\tau=[w_i]-[w_j]$ with $i\not=j$.

\begin{table}[H]
\begin{center}
\begin{tabular}{|c|c|}
\hline 
$E_\tau$ &$(v_0:v_1:v_2:v_3)$\\
\hline \hline
$E_0$ &$(1:0:0:0)$\\
\hline
$E_{[w_0]-[w_\infty]}$ & $(0:0:1:0)$\\
\hline
$E_{[w_1]-[w_\infty]}$ & $(1:-1:1:0)$\\
\hline
$E_{[w_r]-[w_\infty]}$ & $(r^2:-r:1:0)$\\
\hline
$E_{[w_s]-[w_\infty]}$ & $(s^2:-s:1:0)$\\
\hline
$E_{[w_t]-[w_\infty]}$ & $(t^2:-t:1:0)$\\
\hline
\end{tabular}
\begin{tabular}{|c|c|}
\hline 
$E_\tau$ &$(v_0:v_1:v_2:v_3)$\\
\hline \hline
$E_{[w_0]-[w_1]}$ & $(rs+st+rt:0:-1:1)$\\
\hline
$E_{[w_0]-[w_r]}$ & $(r(st+s+t):0:-r:1)$\\
\hline
$E_{[w_0]-[w_s]}$ & $(s(rt+r+t):0:-s:1)$\\
\hline 
$E_{[w_0]-[w_t]}$ & $(t(rs+r+s):0:-t:1)$\\
\hline
$E_{[w_1]-[w_r]}$ & $((1+r)st:r:-1-r:1)$\\
\hline
$E_{[w_1]-[w_s]}$ & $((1+s)rt:s:-1-s:1)$\\
\hline
$E_{[w_1]-[w_t]}$ & $((1+t)rs:t:-1-t:1)$\\
\hline
$E_{[w_r]-[w_s]}$ & $((r+s)t:rs:-r-s:1)$\\
\hline
$E_{[w_r]-[w_t]}$ & $((r+t)s:rt:-r-t:1)$\\
\hline
$E_{[w_s]-[w_t]}$ & $((s+t)r:st:-s-t:1)$\\
\hline
\end{tabular}
\end{center}
\end{table}

The Gunning planes $\Pi_\vartheta$ are the planes passing through $6$ of these $16$ singular points. 
Precisely, the odd Gunning plane
with $\vartheta=[w_i]$ is passing through all $E_\tau$ with $\tau=[w_i]-[w_j]$ (including the trivial bundle $E_0$ for $i=j$);
for an even Gunning plane with $\vartheta=[w_i]+[w_j]-[w_k]\sim[w_l]+[w_m]-[w_n]$, where $\{i,j,k,l,m,n\}=\{0,1,r,s,t,\infty\}$,
we get
$$
\left.\begin{matrix} E_{[w_i]-[w_j]},E_{[w_j]-[w_k]},E_{[w_i]-[w_k]}\\
E_{[w_l]-[w_m]},E_{[w_m]-[w_n]},E_{[w_l]-[w_n]}
\end{matrix}\right\}\ \in\ 
\Pi_{[w_i]+[w_j]-[w_k]}=\Pi_{[w_l]+[w_m]-[w_n]}.$$
In particular, we can derive explicit equations, for instance:

\begin{table}[htdp]
\begin{center}
\begin{tabular}{|c|c|}
\hline 
$\Pi_{[w_0]}$ &$v_1=0$\\
\hline 
$\Pi_{[w_1]}$&$v_1+v_2+v_3=0$\\
\hline
$\Pi_{[w_\infty]}$&$v_3=0$\\
\hline
$\Pi_{[w_0]+[w_1]-[w_\infty]}$&$v_0+v_1=(rs+st+rt)v_3$\\
\hline
\end{tabular}
\end{center}
\end{table}

We can also compute the $16$-order linear group given by twisting the general bundle $E$ by a $2$-torsion 
line bundle $\OX{\tau}$, $\tau=[w_i]-[w_j]$, by looking at the induced permutation on Kummer's singular points.
For instance, we get

\begin{xy}\xymatrix{(v_0:v_1:v_2:v_3)\ar[rr]^{\hspace{-3cm}\otimes E_{[w_0]-[w_\infty]}}&&((\sigma_2+\sigma_3)v_1+\sigma_3v_2:\sigma_3v_3:v_0-(\sigma_2+\sigma_3)v_3:v_1)
}\end{xy}
\begin{xy}
\xymatrix{
(v_0:v_1:v_2:v_3)\ar[rr]^{\hspace{-3cm}\otimes E_{[w_1]-[w_\infty]}}&&(v_0:v_1:v_2:v_3)\cdot {\begin{pmatrix}1&\sigma_1+\sigma_3&\sigma_2&0\\
-1&-1&0&\sigma_2\\
1&0&-1&-(\sigma_1+\sigma_3)\\
0&1&1&1
\end{pmatrix}}^T}\end{xy}

One can find in \cite{Hudson,GonzalezDorrego} classical equations for Kummer surfaces
which are nicer than the above one, but they require coordinate changes that are non-rational in $(r,s,t)$. For instance, we can choose
$E_0$, $E_{[w_0]-[w_1]}$, $E_{[w_1]-[w_\infty]}$ and $E_{[w_0]-[w_\infty]}$ as a projective frame
so that the Gunning bundles $\Pi_{[w_0]}$, $\Pi_{[w_1]}$, $\Pi_{[w_\infty]}$ and $\Pi_{[w_0]+[w_1]-[w_\infty]}$
become coordinate hyperplanes. The Kummer equation therefore becomes quadratic in each coordinate.
However, to reach the nice form given in \S 54 (page 83) of \cite{Hudson}, we must choose 
square roots $\alpha^2=rst$ and $\beta^2=(r-1)(s-1)(t-1)$. Then, setting 
$$(u_0:u_1:u_2:u_3)=((v_0+v_1-\sigma_2 v_3):\beta v_1:\alpha(v_1+v_2+v_3):\alpha\beta v_3),$$
we get the new equation 
$$\begin{array}{rl}\mathrm{Kum}\left(X\right)~:\\
0=&\left(u_0^2u_3^2+u_1^2u_2^2\right)+\beta^2\left(u_0^2u_2^2+u_1^2u_3^2\right)+\alpha^2\left(u_0^2u_1^2+u_2^2u_3^2\right)\vspace{.2cm}\\
&-2\beta\left(u_0u_2-u_1u_3\right)\left(u_0u_3+u_1u_2\right)-2\alpha\left(u_0u_3-u_1u_2\right)\left(u_0u_1-u_2u_3\right)\vspace{.2cm}\\
&-2\alpha\beta\left(u_0u_1+u_2u_3\right)\left(u_0u_2+u_1u_3\right)-2\left(\sigma_1+\sigma_2-2\sigma_3-2\right)u_0u_1u_2u_3.\end{array}$$

In these coordinates, the translations computed above simply become:

\begin{xy}\xymatrix{&&&(u_0:u_1:u_2:u_3)\ar[rr]^{\hspace{-0cm}\otimes E_{[w_0]-[w_\infty]}}&&(u_2:u_3:u_0:u_1)
}\end{xy}
\begin{xy}
\xymatrix{&&&
(u_0:u_1:u_2:u_3)\ar[rr]^{\hspace{-0.2cm}\otimes E_{[w_1]-[w_\infty]}}&&(u_1:-u_0:-u_3:u_2).}\end{xy}

Another classical presentation of the Kummer surface consists in normalizing the action of the finite translation group to be generated by double-transpositions of variables and double-changes of signs. Then the equation of the Kummer surface takes the very nice form (see \S 53 page 80-81 of \cite{Hudson})
\begin{equation}\label{EqKummerHudson}
\begin{array}{rcr}(t_0^4+t_1^4+t_2^4+t_3^4)+2D(t_0t_1t_2t_3)&&\\
+A(t_0^2t_3^2+t_1^2t_2^2)+B(t_1^2t_3^2+t_0^2t_2^2)+C(t_2^2t_3^2+t_0^2t_1^2)&=&0
\end{array}\end{equation}
with coefficients $A,B,C,D$  satisfying the following relation
$$4-A^2-B^2-C^2+ABC+D^2=0.$$

Note that any coordinate change commuting with the (already normalized) actions of $E_{[w_0]-[w_\infty]}$ and $E_{[w_1]-[w_\infty]}$ in the coordinates $(u_0:u_1:u_2:u_3)$ takes the form
$${\begin{pmatrix}t_0\\t_1\\t_2\\t_3\end{pmatrix}}={\begin{pmatrix}a&b&c&d\\-b&a&d&-c\\ c&d&a&b\\d&-c&-b&a\end{pmatrix}}\cdot{\begin{pmatrix}u_0\\u_1\\u_2\\u_3\end{pmatrix}}. $$

If, moreover, we want to normalize the action of all the translation group to the one given in the table below  for example,\begin{table}[H]
\begin{center}
\begin{tabular}{|c|c|}
\hline 
$\tau$ &$(t_0:t_1:t_2:t_3)\otimes E_\tau$\\
\hline \hline
$0$ &$(t_0:t_1:t_2:t_3)$\\
\hline
${[w_0]-[w_\infty]}$ & $(t_2:t_3:t_0:t_1)$\\
\hline
${[w_1]-[w_\infty]}$ & $(t_1:-t_0:-t_3:t_2)$\\
\hline
${[w_r]-[w_\infty]}$ & $(t_0:-t_1:-t_2:t_3)$\\
\hline
${[w_s]-[w_\infty]}$ & $(t_1:t_0:-t_3:-t_2)$\\
\hline
${[w_t]-[w_\infty]}$ & $(t_2:t_3:-t_0:-t_1)$\\
\hline
\end{tabular}
\end{center}
\end{table}
then the variables $a,b,c,d$ have to satisfy (up to a common factor):
$$\begin{array}{rcl}
a&=&rst(r-s)\beta+t\gamma\delta\vspace{.2cm}-rt(r-1)\delta-st\beta\gamma\vspace{.2cm}\\
b&=&-st(s-1)\gamma+rt\beta\delta\vspace{.2cm}\\
c&=&t(r-s)\alpha\beta-t(r-1)\alpha\delta\vspace{.2cm}\\
d&=&-t(r-1)(s-1)(r-s)\alpha+t(s-1)\alpha\gamma\end{array}$$
where $\alpha,\beta,\gamma,\delta$ satisfy
$$\alpha^2=rst,\ \beta^2=(r-1)(s-1)(t-1),\ \gamma^2=r(r-1)(r-s-(r-t)\ \text{and}\ \delta^2=s(s-1)(s-r)(s-t).$$
The coefficients of the resulting Kummer equation (\ref{EqKummerHudson}) are
\begin{equation}\label{EqKummerHudson2}
\begin{array}{rclcrcl}A&=&-2\frac{s(t-1)+(t-s)}{t(s-1)}&& B&=&-2\frac{r+(r-t)}{t}\vspace{.3cm}\\C&=&2\frac{(r-1)+(r-s)}{s-1}&& D&=&-4\frac{r(s-t)+(r-s)}{t(s-1)}.\end{array}\end{equation}

The five $t$-polynomials occuring in the Kummer equation (\ref{EqKummerHudson}) are fundamental invariants for the action 
of the translation group and define a natural map $\P^3_{\mathrm{NR}}\to\P^4$ whose image is a quartic hypersurface
(see \cite{Dolgachev2}, Proposition 10.2.7). 

\begin{cor} The quartic in $\P^4$ defined by the natural map $\P^3_{\mathrm{NR}}\to\P^4$ is a coarse moduli space of $S$-equivalence classes of semi-stable $\P^1$-bundles over $X$.
\end{cor}
\begin{rem} Recall that a $\P^1$-bundle $S$ over $X$ is called semi-stable if 
$\# (s,s) \geq 0$ for every section $s : X\to S$. If $E$ is a rank 2 vector bundle over $X$ such that $\P E=S$, then the (semi-)stability of $S$ is equivalent to the semi-stability of $E$ \cite{Beauville}.
\end{rem}

\begin{proof} Let $T$ be a smooth parameter space and $\mathcal{S}\to X\times T$ a family of  $\P^1$-bundles over $X$. Denote by $\pi_T$ the projection $X\times T\to T$. The $\P^1$-bundle $\mathcal{S}$ lifts to a rank 2 bundle $\mathcal{E}\to X\times T$ such that $\det(\mathcal{E})=\pi_T^*\mathcal{O}_X$ and $\P\mathcal{E}=\mathcal{S}$. This vector bundle is unique up to tensor product with $\pi_T^*(L)$ where $L$ is a 2-torsion line bundle on $X$. According to Theorem \ref{NR}, the classification map $T \to \mathcal{M}_{\mathrm{NR}}$ then is a morphism as is its composition with the natural map  $\P^3_{\mathrm{NR}}\to\P^4$. The resulting morphism $T \to \P^4$ no longer depends on the choice of $\mathcal{E}$. \end{proof}

\section{Anticanonical subbundles}\label{SecTyurin}
 
We will now enrich our point of view of hyperelliptic decent by its relations to the classical approaches of Tyurin \cite{Tyurin} (see also \cite{Jacques2}) and Bertram \cite{Bertram}, as well as more recent works
of Bolognesi \cite{Bolognesi,Bolognesi2}. By our main construction (see Section \ref{SecMainConstruction}), we see $\BUN\left(X/\iota\right)$
as the moduli space of hyperelliptic parabolic bundles $\left(E,\p\right)$ together with the forgetful map 
$$\BUN\left(X/\iota\right)\to\BUN\left(X\right);\left(E,\p\right)\mapsto E.$$ The Bertram-Bolognesi point of view (see Section \ref{SecBertram}) arises from the  moduli space of hyperelliptic flags $\left(E,L\right)$ with $E\supset L\simeq\mathcal{O}\left(- \KX\right)$: Bertram considered in \cite{Bertram} the projective space 
of non-trivial extensions $$0\longrightarrow \mathcal{O}\left(- \KX\right)\longrightarrow E\longrightarrow\mathcal{O}\left( \KX\right)\longrightarrow 0$$ on which
the hyperelliptic involution acts naturally. Bertram's moduli space is the invariant hyperplane, \emph{i.e.} the set of hyperelliptic 
extensions.

Tyurin however considers rank $2$ vector bundles with trivial determinant over $X$ that can be obtained from $\OOX(-\mathrm{K}_X)\oplus \OOX(-\mathrm{K}_X)$ by positive elementary transformations on a parabolic structure carried by a divisor in $|2\KX|$. Again we obtain a moduli space, which is a rational two-cover of an open set of $\BUN (X)$. It will turn out later (see Section \ref{SecSummary}), that the moduli spaces for each of these points of view are all birationally equivalent. 

Let $E$ be a flat vector bundle with trivial determinant bundle on $X$. Given an irreducible connection $\nabla$ on $E$, Corollary \ref{lift}
provides a lift $h:E\to\iota^*E$ of the hyperelliptic involution $\iota:X\to X$ whose action on the Weierstrass
fibers is non-trivial, with two distinct eigenvalues $\pm1$. Consider the set of line subbundles
$\mathcal{O}\left(- \KX\right)\hookrightarrow E$ and how $h$ acts on it. In Section \ref{SecTyurinSubBundle},
we will prove that a generic $E\in \BUN\left(X\right)$ carries a $1$-parameter family of such subbundles, only two
of them being $h$-invariant: 
\begin{itemize}
\item $L^+\subset E$ on which $h$ acts as $\mathrm{id}_{L^+}$,
\item $L^-\subset E$ on which $h$ acts as $-\mathrm{id}_{L^-}$. 
\end{itemize}
In the generic case, the two parabolic structures $  \p$ and $  \p'$ discussed in Sections \ref{diralg} and \ref{SecSymGal}
are therefore respectively defined by the fibres over the Weierstrass points of the line subbundles $L^+$ and $L^-$ of $E$. We also investigate the non generic case. The results are summarized in Table \ref{TableTyuBundles}.

\begin{table}[htdp]
\hspace{-0cm}\begin{tabular}{|l|l|c|}
\hline
bundle type &  $\begin{array}{c}\textrm{degenerate invariant } \\\textrm{Tyurin subbundles} \end{array}$& $\begin{array}{c}\textrm{ parabolic }\\\textrm{structures } \p^\pm\\  \textrm{(up to autom.) } \\\textrm{ determined by }L^\pm \end{array}$  \\
\hline
 stable off Gunning planes &  $\emptyset$ & 2 out of 2\\
generic on $\Pi_{[w_i]}$&  ${L}^+ = \mathcal{O}_X(-[w_i]) $& 1 out of 2\\
stable on $\Pi_{[w_i]}\cap \Pi_{[w_j]}$&   ${L}^+ = \mathcal{O}_X(-[w_i]), \,L^-=  \mathcal{O}_X(-[w_j]) $  & 0 out of 2
\\
\hline
generic decomposable &  $\emptyset$ & 1 out of 1  \\
$L_0\oplus L_0$ with $L^2=\mathcal{O}_X$ &   ${L}^+ =L_0, \,L^-= L_0 $ & 1 out of 1
\\ 
\hline
generic unipotent & $L^+=\mathcal{O}_X$ & 2 out of 2  \\
special unipotent & $L^+=\mathcal{O}_X,\, L^-=\mathcal{O}(-[w])$ & 1 out of 2 \\
twists of unipotent & $ L^+=\mathcal{O}_X([w_i]-[w_j])$ & 1 out of 2\\
\hline
affine &$\emptyset$ & 0 out of 0\\
\hline
even Gunning bundle &  $L^+=L^-=\mathcal{O}_X(\vartheta)$  & 2 out of 2\\
odd Gunning bundle & $L^+=\mathcal{O}_X(\vartheta)$& 2 out of 2
\\\hline
\end{tabular}
\caption{Invariant Tyurin subbundles for the different types of bundles. By definition non-degenerate subbundles are isomorphic to $\mathcal{O}_X(-\KX).$ }
\label{TableTyuBundles}
\end{table}

\subsection{Tyurin subbundles}\label{SecTyurinSubBundle}

Let $\left(E,\nabla\right)$ be an irreducible  trace-free connection over $X$, and let $h:E\to\iota^*E$ be the lift of 
the hyperelliptic involution $\iota:X\to X$ given by Corollary \ref{lift}. Recall that $h$
acts non-trivially with two distinct eigenvalues on each Weierstrass fiber $E\vert_{w}$.
The involution $\iota$ acts linearly on $\mathcal{O}\left(- \KX\right)$ and therefore $h$ acts 
on $\mathrm{H}^0\left(\Hom\left( \mathcal{O}\left(- \KX\right), E\right)\right)$. Since it is involutive, this action induces
a splitting  $$\mathrm{H}^0\left(\Hom\left( \mathcal{O}\left(- \KX\right), E\right)\right)=H^+\oplus H^-$$ into eigenspaces 
(relative to $\pm1$ eigenvalues).
We call {\it Tyurin subbundle} of $E$ the line subbundles $L$ obtained by saturation of  the inclusion of locally trivial sheaves
$\mathcal{O}\left(- \KX\right) \hookrightarrow E$ 
defined by any non zero element $\varphi\in\mathrm{H}^0\left(\Hom\left( \mathcal{O}\left(- \KX\right), E\right)\right)$. In the following, we prefer to consider $\varphi$ as a holomorphic map from the total space of $\mathcal{O}\left(- \KX\right)$ to the total space of $E$.  From this point of view, if $x_1, \ldots, x_n$ are the points of $X$ such that $\varphi|_{\mathcal{O}\left(- \KX\right)_{x_i}}$ is identically zero, then the corresponding Tyurin subbundle $L$ satisfies $L\simeq \OOX([x_1]+\ldots [x_n]-\KX)$. In other words, if $\varphi$ is injective (as a map between total spaces of vector bundles), then $L\simeq\OX{- \KX}$. We say that a Tyurin subbundle $L\subset E$ is \emph{degenerate} if $L\not\simeq\OX{- \KX}$.

\begin{prop}\label{PropMainTyurinSubbundle}
Let $E$ and $h$ be as above. The vector space
$\mathrm{H}^0\left(\Hom_{\OOX}\left( \mathcal{O}\left(- \KX\right), E\right)\right)$ is $2$-dimensional except in the following cases
\begin{itemize}
\item[$\bullet$] $E$ is either unipotent, or an odd Gunning bundle, and then the dimension is $ 3$,
\item[$\bullet$] $E$ is the trivial bundle, and then  the dimension is 4.
\end{itemize}
If $E$ is not an even Gunning bundle, the images of these morphisms span the vector bundle $E$ at a generic point.
The two eigenspaces $H^+$ and $H^-$ then have positive dimension; they correspond to morphisms into two  distinct $h$-invariant subbundles,
$L^+$ and $L^-$. There are no other $h$-invariant Tyurin subbundles.
\end{prop}

\begin{rem} As we shall see in Section \ref{casspecialGunning}, in the case of even Gunning bundles, 
the eigenspaces $H^+$ and $H^-$ still have positive dimension, but the associated $h$-invariant subbundles $L^+$ and $L^-$ are equal.
\end{rem}

\begin{proof}First we have $\Hom\left( \mathcal{O}\left(- \KX\right), E\right)\simeq  E\otimes\OX{\KX}$ and by the Riemann-Roch formula $h^0\left(E\otimes\OX{ \KX}\right)-h^0\left(E\right)=2$. Here, we use Serre duality and the fact that $E$ is selfdual (because $\mathrm{rank}\left(E\right)=2$ and $\det\left(E\right)=\OOX$). We promptly deduce that $h^0\left(E\otimes\OX{\KX}\right)\ge 2$ and $>2$ if and only if $E$
has non-zero sections or, equivalently, if it contains a subbundle $L$ of the form $L=\OOX$, $\OOX\left([p]\right)$ or 
$\deg\left(L\right)>1$.  Because of flatness (see Section \ref{SecFlatOnX}), the only possibilities are actually $L=\OOX$ or $\OOX\left([w]\right)$ for some Weierstrass point $w\in X$.

When the image of a $2$-dimensional subspace of $\mathrm{H}^0\left(\Hom\left( \OX{- \KX}, E\right)\right)$ is degenerate,
\emph{i.e.} contained in a proper subbundle $L\subset E$, then $h^0\left(L\otimes\OX{ \KX}\right)=2$ which implies 
$L=\OOX$ or $L=\vartheta$, a theta characteristic. Yet in the cases when $L$ is trivial or an odd theta characteristic, we have  
$\mathrm{h}^0\left(\Hom\left( \OX{- \KX}, E\right)\right)>2 =\mathrm{h}^0\left(\Hom\left( \OX{- \KX}, L\right)\right)$
and thus not all morphisms take values into $L$: we get enough freedom to span $E$ at a generic point.

Now, given two morphisms $\varphi_i : \OX{- \KX}\to E$ for $i=1,2$, 
taking values into two different subbundles $L_i\subset E$, $L_1\not=L_2$,
we get a morphism
$\varphi_1 \oplus \varphi_2 : \OX{- \KX}\oplus  \OX{- \KX} \to E$
whose image spans the vector bundle $E$ at all fibers but those corresponding to
the (effective) zero divisor of 
$\varphi_1 \wedge \varphi_2 : \OX{-2 \KX}\to \OOX$.
Such a divisor takes the form $[P_1] + [\iota\left(P_1\right)]+[P_2]+[\iota\left(P_2\right)]$ for some $P_1,P_2 \in X$. We thus get an isomorphism between the 2-dimensional vector space
$\mathrm{Vect}_\C\left(\varphi_1,\varphi_2\right)\subset \mathrm{H}^0\left(\Hom\left( \OX{- \KX}, E\right)\right)$
and the fiber of $E$ over each point of $X\setminus\{P_1, \iota\left(P_1\right), P_2, \iota\left(P_2\right)\}$. 
In particular, over a Weierstrass point $w\not=P_1,P_2$, we have  
$E\vert_w\simeq\mathrm{Vect}_\C\left(\varphi_1,\varphi_2\right)$ and since the action $h$ on $\mathrm{H}^0\left(\Hom\left( \OX{- \KX}, E\right)\right)$
is non-trivial, neither $H^+$ nor $H^-$ is reduced to $\{0\}$.
Moreover, $\varphi_1$ and $\varphi_2$ cannot belong to a common eigenspace of the action of $h$ on  $\mathrm{H}^0\left(\Hom\left( \OX{- \KX}, E\right)\right)$. 
In other words, any two morphisms belonging to the same eigenspace $H^{\pm}$ take image in the same subbundle,
say $L^{\pm}$. 

Let now $L$ be a Tyurin subbundle distinct from $L^+$ and $L^-$: $L$ is generated by 
$\varphi=\varphi_1+\varphi_2$ for some $\varphi_1\in H^+$ and $\varphi_2\in H^-$.
Again, there is a Weierstrass point $w$ where $\varphi_1 \wedge \varphi_2$ does not vanish:
the action of $h$ is homothetic on the $\varphi_i$ with opposite eigenvalues and cannot fix 
the direction $\C\cdot\varphi\left(w\right)$. Thus $L$ is not $h$-invariant.
\end{proof}

Note that if the line subbundles $L^{\pm}$ are non-degenerate, their fibres over the Weierstrass points define the parabolic structures $\p^\pm$.
As we shall see, any flat vector bundle $E$ has degenerate Tyurin subbundles; some of them can be $h$-invariant,
even in the stable case.

In the following paragraphs, we will study the Tyurin subbundles for each type of bundle, following the list of Section \ref{SecFlatOnX}.
The reader might want to skip the non stable cases at first, and refer to them later, when needed.

\subsubsection{Stable bundles}\label{sec:TyurinStable}

When $E$ is stable, any holomorphic connection on $E$ is irreducible.
Since the only bundle automorphisms of $E$ are homothecies,  
the same bundle isomorphism $h:E\to \iota^*E$ works for all connections 
and it therefore only depends on the bundle (up to a sign).
The two $h$-invariant Tyurin bundles $L^+$ and $L^-$ 
depend (up to permutation) only on $E$.

Consider two elements $\varphi^+,\varphi^-\in\mathrm{H}^0\left(\Hom\left( \mathcal{O}\left(- \KX\right), E\right)\right)$
generating $L^+$ and $L^-$ (at a generic point) and consider the divisor
$\mathrm{div}\left(\varphi^+\wedge\varphi^-\right)= [P]+[\iota\left(P\right)]+[Q]+[\iota\left(Q\right)]$.
This divisor $D_E^T\in\vert 2 \KX\vert$ is an invariant of the bundle, we call it the {\it Tyurin divisor}. 
Let $D_E\in\vert 2\Theta\vert$ be the divisor on $\mathrm{Pic}^1\left(X\right)$
defined by Narasimhan-Ramanan (see Section \ref{NarRam}).

\begin{prop}\label{Prop:TyurinNRDivisor}
Let $E$ be stable. Then the divisor $D_E^T$ is the intersection
between the divisor $D_E$ and the natural embedding $X\to\Theta; P\mapsto [P]$ on $\mathrm{Pic}^1\left(X\right)$:
$$D_E^T = D_E\cdot\Theta.$$
For each point $P$ of the support of $D_E^T$, there is exactly one subbundle
$L_P\equiv\OOX\left(-[P]\right)$ of $E$. These are precisely the degenerate Tyurin subbundles.
Such a degenerate Tyurin subbundle $L_P$ is $h$-invariant if, and only if, $P=w$ is a Weierstrass point.
This happens precisely when $E$ lies on the odd Gunning plane $\Pi_{[w]}$.
\end{prop}

\begin{proof}First note that $D_E^T=\mathrm{div}\left(\varphi_1\wedge\varphi_2\right) $ for any basis $\left(\varphi_1, \varphi_2\right)$ of the vector space $\mathrm{H}^0\left(\Hom\left( \mathcal{O}\left(- \KX\right), E\right)\right)$. A point $P\in X$ belongs to 
the support of $D_E^T$ if and only if $\iota\left(P\right)$ does. This is equivalent to the fact that $\varphi^+$
and $\varphi^-$ are colinear at $\iota\left(P\right)$. Equivalently, there is a morphism $\varphi_P \in \mathrm{H}^0\left(\Hom\left( \mathcal{O}\left(- \KX\right), E\right)\right)$ which vanishes at $\iota\left(P\right)$ (and can be completed to a basis with $\varphi^+$
or $\varphi^-$). By stability of the vector bundle $E$, the morphism $\varphi_P$ cannot vanish elsewhere. Denote by $L_P$ the line subbundle corresponding to $\varphi_P$. Finally, we have $P \in D_E^T$ if and only if there is a line subbundle $L_P$ of $E$ such that $L_P \simeq \mathcal{O}\left([\iota\left(P]\right)- \KX\right)= \mathcal{O}\left(-[P]\right)$. On the other hand, $P$ belongs to the support of $D_E.\Theta$ if and only if there is a line subbundle $L_P\simeq \mathcal{O}\left(-[P]\right)$ of $E$. 
Since these divisors are generically reduced, we can conclude by continuity that $D_E^T = D_E.\Theta.$
 
 Now suppose $E$ has two line subbundles $L_P$. A linear combination of the two corresponding homomorphisms in $\mathrm{H}^0\left(\Hom\left( \mathcal{O}\left(- \KX\right), E\right)\right)$ then would have a double zero at $P$, which is impossible by stability of $E$.
 So for each point $P$ in the support of $D_E^T$, we get a unique subbundle $L_P\simeq\OOX\left(-[P]\right)$
and there are no other degenerate Tyurin subbundles. 

Finally, note that the finite set of (at most $4$) degenerate Tyurin subbundles must be $h$-invariant.
Thus such a bundle $L_P$ is invariant if, and only if, $P$ is $\iota$-invariant.
\end{proof}

\begin{cor}\label{CorTyurinStablePar}
When $E$ is stable and outside of odd Gunning planes $\Pi_{[w_i]}$, 
there are exactly two $h$-invariant subbundles $L^+,L^-\simeq\mathcal O\left(- \KX\right)$ in $E$
that are invariant under the hyperelliptic involution $h$. The two parabolic 
structures $\p$ and $\p'$ defined in Sections \ref{diralg} and \ref{SecSymGal}
then are precisely the fibres  over the Weierstrass points of two line subbundles $L^{\pm}$.
\end{cor}

Another important consequence of the proposition above  is the Tyurin parametrization of the moduli space of stable bundles (see section \ref{SecTyurinPar})
which relies on the following

\begin{cor}\label{CorStableTyurin}When $E$ is stable and the Tyurin divisor $D_E^T=[P]+[\iota\left(P\right)]+[Q]+[\iota\left(Q\right)]$ is reduced
($4$ distinct points), then the natural map
$$\varphi^+\oplus\varphi^-:\mathcal O\left(- \KX\right)\oplus\mathcal O\left(- \KX\right)\to E$$
is a positive elementary transformation for the parabolic structure defined over 
$D_E^T=[P]+[\iota\left(P\right)]+[Q]+[\iota\left(Q\right)]$ by the fibres of the line subbundles $L_P$, $L_{\iota\left(P\right)}$, $L_Q$ and $L_{\iota\left(Q\right)}$ over the corresponding points.
\end{cor}

\begin{rem}{When $E$ belongs to an odd Gunning plane $\Pi_{[w]}$, then one of the two $h$-invariant Tyurin subbundles
is degenerate, say $L^-=\OOX\left(-[w]\right)$, and fails to determine the parabolic structure $\p^-$
over the Weierstrass point $w$. When $E\in \Pi_{[w_i]}\cap \Pi_{[w_j]}$, then the two $h$-invariant 
Tyurin subbundles are degenerate and neither $\p^+$, nor $\p^-$ are determined
by these bundles.}\end{rem}

\subsubsection{Generic decomposable bundles}\label{sec:TyurinDecomposable}

Let $E=L_0\oplus L_0^{-1}$, where $L_0=\mathcal O\left([P]+[Q]- \KX\right)$ is not $2$-torsion: $L_0^{2}\not=\OOX$.
There is (up to scalar multiple) a unique morphism $\varphi:\mathcal O\left(- \KX\right)\to L_0$ 
(resp. $\varphi':\mathcal O\left(- \KX\right)\to L_0^{-1}$) vanishing at $[P]+[Q]$ (resp. $[\iota\left(P\right)]+[\iota\left(Q\right)]$).
They generate all Tyurin subbundles and they are the only degenerate ones. 
Clearly, neither $L_0$ nor $L_0^{-1}$ is invariant.
The projective part $\mathbb G_m$ of the automorphism group $\mathrm{Aut}\left(E\right)$ 
fixes both $L_0$ and $L_0^{-1}$ and acts transitively on the remaining part of the family.
Any involution $h$ interchanges $L_0$ and $L_0^{-1}$ while it fixes two generic members
$L^+$ and $L^-$ of the family. The parabolic structures are defined by the fibres of these two bundles  over the Weierstrass points.
Another choice of lift $h'=g\circ h\circ g^{-1}$, $g\in\mathrm{Aut}\left(E\right)$, just translates the two
subbundles $L^{\pm}$ by $g$.  
Finally, {\it up to automorphism, there is a unique invariant Tyurin bundle, and thus a unique 
parabolic structure}.

\subsubsection{The trival bundle and its $15$ twists}

When $E$ is the trivial bundle, the space of morphisms $\mathrm{H}^0\left(\Hom\left( \mathcal{O}\left(- \KX\right), E\right)\right)$ is $4$-dimensional 
and generated by $2$-dimensional subspaces $\mathrm{H}^0\left(\Hom\left( \mathcal{O}\left(- \KX\right), \OOX\right)\right)$
for two distinct embeddings $\OOX\hookrightarrow E$. We get $1$-parameter family of degenerate Tyurin sub bundles formed
by all embeddings $\OOX\hookrightarrow E$. Given any irreducible connection, 
the corresponding lift $h$ fixes only two (degenerate) subbundles
(see Section \ref{SecTrivialFlatX}). The two parabolic structures are defined by the line bundles associated to  
these two embeddings $\OOX\hookrightarrow E$. Therefore, 
{\it up to automorphism, there is exactly one parabolic structure on the trivial vector bundle}.

When $E=L_0\oplus L_0$ with $L_0=\mathcal O\left([w_i]+[w_j]- \KX\right)$ and $i\not=j$,
then the vector space $\mathrm{H}^0\left(\Hom\left( \mathcal{O}\left(- \KX\right), E\right)\right)$ is $2$-dimensional and all
Tyurin subbundles are degenerate in this case: they form the $1$-parameter family of subbundles $L\hookrightarrow E$ with $L\simeq L_0$.
Still $\mathrm{Aut}\left(E\right)=\mathrm{GL}_2\left(\C\right)$ acts transitively on them. 
Recall that the two parabolic structures $\p^\pm$ on $E$ can be deduced from the case of trivial bundles 
just by permuting the role of the two parabolics over $w_i$ and $w_j$ with respect to $\p^+$ and $\p^-$ (see Section \ref{SecSymGal}).
Each parabolic structure is thus distributed on two embeddings $L\hookrightarrow E$ with $L\simeq L_0$.
Since $\mathrm{Aut}\left(E\right)$ acts $2$-transitively on the family of such line subbundles, {\it there is a unique parabolic
structure up to automorphisms}.

\subsubsection{Unipotent bundles and their $15$ twists}

Let $0\to \OOX\to E\to \OOX\to 0$ be a non-trivial extension.
Here the space of morphisms $\mathrm{H}^0\left(\Hom\left(\mathcal O\left(- \KX\right),E\right)\right)$ has dimension $3$
and the subbundle $\OOX\subset E$ is responsible for this extra dimension: the space
$\mathrm{H}^0\left(\Hom\left(\mathcal O\left(- \KX\right),\OOX\right)\right)$ has dimension $2$.
There are many lifts $h$ of the hyperelliptic involution $\iota$ 
since there are non-trivial automorphisms on $E$: any other lift is, up to a sign,
given by $g\circ h\circ g^{-1}$ for some $g\in\mathrm{Aut}\left(E\right)$. But once $h$
is fixed, we can apply Proposition \ref{PropMainTyurinSubbundle} and get
that there are exactly two $h$-invariant Tyurin subbundles $L^{\pm}$, 
one of them is the unique embedding $\OOX\hookrightarrow E$. Possibly
 replacing $h$ by $-h$, we may assume $L^+=\OOX$.

Let $\varphi^+$ be a non-zero element of $\mathrm{H}^0\left(\Hom\left(\mathcal O\left(- \KX\right),E\right)\right)$ taking values in $L^+$ and
vanishing at say $[P]+[\iota\left(P\right)]$.
Let $\varphi^-$  be a non-zero element of $\mathrm{H}^0\left(\Hom\left(\mathcal O\left(- \KX\right),E\right)\right)$ taking values in $L^-$. Consider the divisor defined by zeroes of $\varphi^+\wedge\varphi^-$:
as an element of the linear system $\vert 2 \KX\vert$, it takes the form $[P]+[\iota\left(P\right)]+[Q]+[\iota\left(Q\right)]$
including the vanishing divisor of $\varphi^+$. Since $\varphi^-$ is unique up to a constant, the divisor $[Q]+[\iota\left(Q\right)]$ is an invariant of the bundle,
while $[P]+[\iota\left(P\right)]$ can be chosen arbitrarily by switching to another $\varphi^+$.

\begin{prop}The divisor $[Q]+[\iota\left(Q\right)]$ characterizes the extension $E$:
we thus get a natural identification between the space $\P\left(\mathrm{H}^0\left(\Hom\left(\mathcal O\left( \KX\right)\right)\right)^\vee\right)$ 
parametrizing extensions and  $\P\left(\mathrm{H}^0\left(\Hom\left(\mathcal O\left( \KX\right)\right)\right)\right)$ parametrizing those divisors $[Q]+[\iota\left(Q\right)]$.

The bundle $L^-$ is degenerate if, and only if, $[Q]+[\iota\left(Q\right)]=2[w_i]$ where $w_i$ is a Weierstrass point. In this case, $L^-=\OOX\left(-[w_i]\right)$ (and $\varphi^-$ vanishes at $w_i$).
\end{prop}

\begin{proof}The morphism $\varphi^-$ defines a natural morphism
$$\mathrm{id}\vert_{L^+}\oplus\varphi^-:\OOX\oplus\mathcal O\left(- \KX\right)\to E$$
whose determinant map vanishes at $[Q]+[\iota\left(Q\right)]$. When $Q\not=\iota\left(Q\right)$,
this is a positive elementary transformation on the vector bundle $\OOX\oplus\mathcal O\left(- \KX\right)$ for a parabolic structure defined 
over $[Q]+[\iota\left(Q\right)]$. None of the two parabolics can be contained in the total space of the destabilizing line subbundle $\OOX$, otherwise $E$ would be unstable. Moreover, the two parabolics cannot both be contained in a same line subbundle isomorphic to $\mathcal O\left(- \KX\right)$, 
otherwise $E$ would be decomposable. Up to automorphism
of the bundle $\OOX\oplus\mathcal O\left(- \KX\right)$, there is a unique
parabolic structure over $[Q]+[\iota\left(Q\right)]$ satisfying these conditions. Hence $E$ is well determined by the divisor $[Q]+[\iota\left(Q\right)]$.
This provides a natural identification as stated, outside of the $6$ special bundles
for which $Q=\iota\left(Q\right)=w_i$; it extends by continuity at those points.

Since $E$ is semi-stable and indecomposable, we have $\deg\left(L^-\right)<0$.
In the degenerate case, the only possibility is that $\varphi^-$ has a single zero, at say $Q$,
and $L^-=\OOX\left(-[\iota\left(Q\right)]\right)$. But $L^-$ being $h$-invariant, $Q=\iota\left(Q\right)$ has to be
 a Weierstrass point, $w_i$ say. Conversely, if $Q=w_i$, we can chose $P\not=w_i$ making sure that $\varphi^+$ doesn't vanish at $w_i$. The two sections $\varphi^+$ and $\varphi^-$ are however colinear at $w_i$ and there is a linear combination $\varphi=\varphi^-+\lambda \varphi^+$ vanishing at $w_i$. The corresponding Tyurin subbundle $L$ of $E$ then is isomorphic to $\OOX(-[w_i])$ and thus invariant under the hyperelliptic involution. Since $L^\pm$ are the only invariant Tyurin subbundles, we have $L^-=L\simeq \OOX(-[w_i])$. \end{proof}

The two hyperelliptic parabolic structures associated to $h$ are defined by these two bundles,
except for the $6$ special extensions $E$ for which $L^-$ is degenerate. Consider  now another lift $h'$ of the hyperelliptic involution, given by $h'=g\circ h\circ g^{-1}$ for some automorphism $g\in \mathrm{Aut}\left(E\right)$. The $h'$-invariant Tyurin subbundles then are $L^+$ and $g\left(L^-\right)$ since $\mathrm{Aut}\left(E\right)$ fixes the subbundle $L^+\simeq\OOX$. Yet the $\mathbb G_a$-part of $\mathrm{Aut}\left(E\right)$ acts transitively on the set on non-degenerate Tyurin subbundles.
Therefore, \emph{there are exactly two hyperelliptic parabolic structures on $E$ up to automorphism}.

\begin{rem} In the geometric picture, the $1$-parameter family of extensions $\left(E_t\right)_{t\in \P^1}$ of the trivial line bundle can be seen as the tangent cone to the Kummer surface after blowing up the singular point corresponding to the trivial bundle. The strict transform of the Gunning plane $\Pi_{[w_i]}$ then intersects this $\P^1$ in a unique point which is the bundle satisfying $L^-\simeq\mathcal O\left(-[w_i]\right)$ as above.
\end{rem}

Let $L_0=\mathcal O\left([w_i]+[w_j]- \KX\right)$ be a non-trivial $2$-torsion point of $\mathrm{Pic}^0\left(X\right)$, $i\not=j$,
and consider a non-trivial extension $0\to L_0\to E\to L_0\to 0$. This time, the vector space $\mathrm{H}^0\left(\Hom\left(\mathcal O\left(- \KX\right),E\right)\right)$
has dimension $2$ and generates a $1$-parameter family of Tyurin subbundles. One of them is $L_0$, 
the only one having degree $0$. It is degenerate and  must be invariant, say $L^+$. The group
$\mathrm{Aut}\left(E\right)$ acts transitively on the remaining part of the family and, like for unipotent bundles,
$h$ fixes one of them, say $L^-$. 
The intersection $L^+\cap L^-$ has to be $[w_i]+[w_j]$ and $L^-$ is therefore non-degenerate and defines the parabolic structure $\p^-$. 

\subsubsection{Affine bundles}\label{casspecialaff}
We should also consider the case of affine bundles. We have already seen that these bundles are not invariant under the hyperelliptic involution. Hence they do not arise from elements of $\BUN \left( X/\iota\right)$ and our parabolic structures $\p^\pm$ are not defined. Yet Tyurin's construction naturally includes this type of bundle. Indeed, even if the notion of invariant line subbundles does not make sense here, we can of course consider the space of Tyurin subbundles of  an affine bundle.  Let $L_0=\OOX ([P]+[Q]-\mathrm{K}_X)=\OOX (\mathrm{K}_X-[\iota(P)]-[\iota(Q)])$ be a degree 0 line bundle such that $L_0^{\otimes 2}\neq \OOX$ and let $E$ be the unique non-trivial extension
$$0\longrightarrow L_0  \longrightarrow E \longrightarrow L_0^{-1} \longrightarrow 0.$$ Then 
$\mathrm{h}^0\left(\Hom\left(\mathcal{O}_X\left(- \KX\right), E\right)\right)=2$. Moreover, we have
$\mathrm{h}^0\left(\Hom\left(\mathcal{O}_X\left(- \KX\right), L_0\right)\right)=1$. In other words, $E$ possesses a $1$-parameter family of Tyurin subbundles. Precisely three of them are degenerated: 
$L_0$, a unique line subbundle $L_{P}\simeq \OOX (-P)$ of $E$ and a unique line subbundle $L_{Q}\simeq \OOX (-Q)$ of $E$. They define a parabolic structure on $E$ over the Tyurin divisor $D_E^T= [P]+[Q]+[\iota\left(P\right)]+[\iota\left(Q\right)]$ (the parabolics over $\iota(P)$ and $\iota(Q)$ are both given by $L_0$) and the four negative elementary transformations on $E$ defined by these parabolics yield
$\OOX\left(- \KX\right) \oplus \OOX\left(- \KX\right)$.

\subsubsection{The $6+10$ Gunning bundles}\label{casspecialGunning}

Let $\vartheta\in\mathrm{Pic}^1\left(X\right)$ be a theta characteristic and $E_\vartheta$ be the associated
Gunning bundle. The subbundle $\vartheta\subset E_\vartheta$ is the unique one having degree $>-1$;
it is a degenerate  $h$-invariant Tyurin subbundle.

When $\vartheta$ is an even theta characteristic $\vartheta=[w_i]+[w_j]+[w_k]- \KX$, we have
$$\mathrm{h}^0\left(\Hom\left(\mathcal{O}_X\left(- \KX\right),\vartheta\right)\right)=\mathrm{h}^0\left(\Hom\left(\mathcal{O}_X\left(- \KX\right), E_\vartheta\right)\right)=2$$
and all morphisms $\varphi:\mathcal{O}_X\left(- \KX\right)\to E_\vartheta$ factor through the subbundle $\vartheta\subset E_\vartheta$:
there is a unique Tyurin bundle in this case. Through the identification  
$$\Hom\left(\mathcal{O}\left(- \KX\right),\vartheta\right)\simeq \OOX([w_i]+[w_j]+[w_k]),$$
the space global sections of the sheaf of morphisms is generated by
$$1,\frac{\left(x-x_l\right)\left(x-x_m\right)\left(x-x_n\right)}{y}\ \in\mathrm{H}^0\left(X,\mathcal{O}\left( [w_i]+[w_j]+[w_k]\right)\right),$$ where $\{i,j,k,l,m,n\}=\{0,1,r,s,t,\infty\}$ and $w_i=\left(x_i,0\right)\in X$. The hyperelliptic involution acts as $\mathrm{id}$ on the first one and $-\mathrm{id}$ on the second one. 
There are two types of hyperelliptic parabolic structures on $E_\vartheta$:
\begin{itemize}
\item[$\bullet$]parabolics corresponding to $w_i,w_j$ and $w_k$ lying in $\vartheta\hookrightarrow E_\vartheta$, the others outside;
\item[$\bullet$]parabolics corresponding to $w_l,w_m$ and $w_n$ lying in $\vartheta\hookrightarrow E_\vartheta$, the others outside.
\end{itemize}
This implies that \emph{up to automorphism, there are exactly two  parabolic structures on a Gunning bundle $E_\vartheta$ with even theta characteristic.}
 
Let us now consider the case where $\vartheta$ is  an odd theta characteristic $\vartheta=\OOX([w])$. The $h$-invariant Tyurin subbundles $L^+$ and $L^-$ are distinct and one of them is the maximal subbundle of $E_\vartheta$, say $L^+=\vartheta$, which is the only degenerate Tuyrin subbundle of $E_\vartheta$. For any $P\in X$ we can choose a holomorphic section $\varphi^+$ of the line subbundle $L^+\otimes \OOX\left( \KX \right)$ of $E_\vartheta\otimes \OOX\left( \KX \right)$ such that $\mathrm{div}_0(\varphi^+) = [w]+[P]+[\iota (P)]$, whereas any holomorphic section $\varphi^-$ of $L^-\otimes \OOX\left( \KX \right)$ is nowhere vanishing. Moreover, the fibres of the corresponding line subbundles of $E_\vartheta$ are colinear only over the point $w$. In other words, the Tyurin divisor of $E_\vartheta$ is well-defined only after the choice of a section of the destabilizing line subbundle. The birational map $\varphi^+\oplus \varphi^- :\mathcal O\left(- \KX\right)\oplus \mathcal O\left(- \KX\right) \to E_\vartheta$ then decomposes as four successive positive elementary transformations with (Tyurin)-parabolics given by the fibres of $L^-$ and its strict transform over $[w], [w], [P]$ and $[\iota (P)]$. The hyperelliptic parabolic $\p_i^-$ on the other hand is defined by $L^+|_{w}=L^-|_{w}$ and $\p_i^+$ is elsewhere.
Since $\mathrm{Aut}\left(E_\vartheta\right)$ fixes $L^+$ and acts transitively on the set of line subbundles of the form $\mathcal{O}\left(- \KX\right)$, \emph{there are, up to automorphism,  exactly two parabolic structures on a Gunning bundle $E_\vartheta$ with odd theta characteristic.}

\subsection{Extensions of the canonical bundle}\label{SecBertram}

Here, we recall some results obtained by Bertram in \cite{Bertram}, completed
in the genus $2$ case by Bolognesi in \cite{Bolognesi,Bolognesi2} (see also \cite{Kumar}).

The space of non trivial extensions $0\to\mathcal O\left(- \KX\right)\to E\to \mathcal O\left( \KX\right)\to 0$
is $\P  \mathrm{H}^1\left(-2 \KX\right)$ which identifies, by Serre duality, to $\P  \mathrm{H}^0\left(3 \KX\right)^\vee$.
This space naturally parametrizes the moduli space of those pairs $\left(E,L\right)$ where $L\subset E$
is a non-degenerate Tyurin bundle. The hyperelliptic involution $\iota$ acts naturally
on $ \mathrm{H}^0\left(3 \KX\right)$ and thus on its dual: the invariant subspace is an hyperplane 
$\P_B^3\subset \P  \mathrm{H}^0\left(3 \KX\right)^\vee\simeq\P^4$ that naturally parametrizes those pairs
$\left(E,L\right)$ that are invariant under the involution. As we have seen in Section 
\ref{SecTyurinSubBundle}, most  stable bundles $E$ admit exactly two invariant
and non-degenerate Tyurin subbundles and most decomposable bundles $E$
admit only one. This suggests that $\P_B^3$ is a birational model for the $2$-fold
cover of $\P_{\mathrm{NR}}^3$ ramified over the Kummer surface.

A cubic differential $\omega\in  \mathrm{H}^0\left(3 \KX\right)$ writes $\omega=\left(a_0+a_1x+a_2x^2+a_3x^3+a_4y\right)\left(\frac{\mathrm{d}x}{y}\right)^{\otimes 3}$ uniquely so that the coefficients $a_i$ provide a full set of coordinates. Let $\left(b_0:b_1:b_2:b_3:b_4\right)$ be dual 
homogeneous coordinates for $\P^4_B:=\P  \mathrm{H}^0\left(3 \KX\right)^\vee$.
We have the following description (see introductions of \cite{Bertram,Kumar} and \S 5 of \cite{Bolognesi}) 

The locus of unstable bundles  is given by the natural embedding of the curve $X$:
$$X\hookrightarrow\P^4_B;\,\left(x,y\right)\mapsto \left(1:x:x^2:x^3:y\right).$$
The locus of strictly semi-stable bundles is given by the quartic hypersurface $\Wed\subset \P^4_B$ 
spanned by the $2$-secant lines of $X$. 
The natural action of the hyperelliptic involution $\iota:X\to X$ on cubic differentials induces an involution
on $\P^4_B$ that fixes the hyperplane $\P^3_B=\{b_4=0\}$ and the point $\left(0:0:0:0:1\right)$.

The Narasimhan-Ramanan moduli map
$$\P^4_B\dashrightarrow\P_{\mathrm{NR}}^3$$
is given by the full linear system of quadrics that contain $X$;
it restricts to $\P^3_B$ as the full linear system of quadrics (of $\P^3_B$) that contain the six points $X\cap\P^3_B$.
After blowing-up the locus $X$ of unstable bundles, we get a morphism
$$\widetilde\P^4_B\to\P_{\mathrm{NR}}^3$$
namely a conic bundle; its restriction to the strict transform $\widetilde\P^3_B$ of $\P^3_B$
is generically $2:1$, ramifying over the Kummer surface $\Kum\subset\P^3_{\mathrm{NR}}$. 
The quartic hypersurface $\Wed$ restricts to $\P^3_B$ as the (dual) Weddle surface;
it is sent onto the Kummer surface. 

There is a Poincar\'e vector bundle $\mathcal E\to X\times\P^4_B$ realizing the classifying map above.
Hence by restriction, there is a Poincar\'e bundle $\mathcal E\to X\times \P^3_B$ on the double cover  $\P^3_B$ of $\P^3_{\mathrm{NR}}$. The projectivized Poincar\'e bundle $\mathbf{P}\left(\mathcal{E}\right)\to X\times \P^3_B$ defines a conic bundle $\mathcal{C} \to X\times \P^3_{\mathrm{NR}}$ over the quotient $\P^3_{\mathrm{NR}}$. For each vector bundle $E \in \P^3_{\mathrm{NR}}$, the fibre  $\mathcal{C}_E$ of the conic bundle represents the family of Tyurin-subbundles of $E$. Yet the conic bundle $\mathcal{C}$ is not a projectivized vector bundle over $\P^3_{\mathrm{NR}}$, not even up to birational equivalency, because a Poincar\'e bundle over a Zariski-open set of  $\P^3_{\mathrm{NR}}$ does not exist \cite{NR2}. 


\subsection{Tyurin parametrization}\label{SecTyurinPar}

Let $E$ be a flat rank two vector bundle with trivial determinant bundle over $X$. It follows from Corollary \ref{CorStableTyurin} that,
when $E$ is stable and off the odd Gunning planes, then $E$ can be deduced from $\OX{- \KX}\oplus\OX{- \KX}$
by applying $4$ positive elementary transformations, namely over the Tyurin divisor $D_E^T$. In fact, if we allow non reduced divisors,
then this remains true for all flat bundles except even Gunning bundles. Indeed, it follows from Proposition \ref{PropMainTyurinSubbundle}
that we have a non degenerate map 
$$\varphi^+\oplus\varphi^-:\OX{- \KX}\oplus\OX{- \KX}\to E$$
by selecting $\varphi^+$ and $\varphi^-$ generating $H^+$ and $H^-$ respectively; non degenerate means that 
the image spans the generic fiber. Comparing the degree of both vector bundles, we promptly deduce that this map decomposes into 
$4$ successive positive elementary transformations, possibly over non distinct points (this happens when the divisor $D_E^T\in|2 \KX |$
is non reduced). 

Conversely, let us consider a divisor, say reduced for simplicity: 
$$D = [\underline{P}_1]+[\iota\left(\underline{P}_1\right)]+[\underline{P}_2]+[\iota\left(\underline{P}_2\right)] \in |2 \KX |,$$
and consider also a parabolic structure $\q$ over $D$ on the trivial bundle $E_0\to X$: 
given $e_1$ and $e_2$ two independent sections of $E_0$, 
the parabolic structure is defined by 
$$\left(\lambda_{\underline{P}_1}, \lambda_{\iota\left(\underline{P}_1\right)}, \lambda_{\underline{P}_2}, \lambda_{\iota\left(\underline{P}_2\right)}\right) \in \left(\mathbb{P}^1\right)^4$$
where $e_1+\lambda_{\underline{P}_i}e_2$ generates the parabolic direction over $\underline{P}_i$, 
and similarly for $\iota(\underline{P}_i)$.
From this data, one can associate a vector bundle with trivial determinant $E$ by
\begin{equation}\label{TyurinConstr} \mathcal{O}\left(-\KX\right)\otimes \mathrm{elm}^+_D\left(E_0, \q\right)\to E.\end{equation} Table \ref{TyruinPossibilities} lists all types of vector bundles $E$ that can be obtained in that way.

\bgroup
\def\arraystretch{1.5}%
 \begin{table}[p]
 \centering
        \rotatebox{90}{
                \begin{minipage}{\textheight}
 \begin{tabular}{| l | l | c | c | c |}
\hline
\multicolumn{2}{|c|}{bundle type}
& {Tyurin divisor} &  reduced & parabolic structure
\\
\hline
 stable & off $\Pi_{[w_i]}$&  $D_E^T$& yes & generic\\
 \cline{2-5}
& on $\Pi_{[w_i]}$, off $\Pi_{[w_j]}$&  $2[w_i]+[P]+[\iota(P)] $& no &$\begin{array}{c}(\lambda_{w_i},\lambda_P,\lambda_{\iota(P)})=(0,1,\infty)\\\textrm{(but } P \textrm{ is free on }X\setminus W \textrm{)}\end{array}$\\
 \cline{2-5}
& on $\Pi_{[w_i]}\cap \Pi_{[w_j]}$&  $2[w_i]+2[w_j] $& no &$\lambda_{w_i}\neq \lambda_{w_j}$\\
\hline
unipotent & generic & $[P]+[\iota(P)]+[Q]+[\iota(Q)]$&yes &$\lambda_P=\lambda_{\iota(P)}$ (but $Q$ is free)\\\cline{2-5}
&special & $[P]+[\iota(P)]+2[w]$& no & $\lambda_P=\lambda_{\iota(P)}$\\\cline{2-5}
& twisted by $ \OOX([w_i]-[w_j])$ & $2[w_i]+2[w_j]$ & no & $\lambda_{w_i}=\lambda_{w_j}$\\\hline
\hspace{-.2cm}$\begin{array}{l}\textrm{affine } \\ {L_0\to E\to L_0^{-1}} \end{array}$ 
& \hspace{-.2cm}$\begin{array}{l}L_0^{\otimes 2}\neq \OOX \\ L_0=\OOX([P]+[Q]-\mathrm{K}_X)\end{array}$ & $[P]+[\iota(P)]+[Q]+[\iota (Q)]$ & yes & 
$\begin{array}{l}(\lambda_P,\lambda_{\iota(P)},\lambda_Q,\lambda_{\iota(Q)})\\=(0,1,0,\infty)\end{array}$
\\\hline
\hspace{-.2cm}$\begin{array}{l}\textrm{semi-stable } \\\textrm{decomposable} \end{array}$ & \hspace{-.2cm}$\begin{array}{l}\textrm{generic: } L_0^{\otimes 2}\neq \OOX \\ L_0=\OOX([P]+[Q]-\mathrm{K}_X)\end{array}$
&$[P]+[\iota(P)]+[Q]+[\iota(Q)]$&yes &$\lambda_P=\lambda_Q\neq \lambda_{\iota (P)}=\lambda_{\iota (Q)}$
\\\cline{2-5}
$L_0\oplus L_0^{-1}$& trivial: $L_0=\OOX$&$[P]+[\iota(P)]+[Q]+[\iota(Q)]$&yes &$\lambda_P=\lambda_{\iota (P)}\neq \lambda_{Q}=\lambda_{\iota (Q)}$\\\cline{2-5}
&twist: $L_0=\OOX([w_i]-[w_j])$&$2[w_i]+2[w_j]$&no&$\lambda_{w_i}\neq \lambda_{w_j}$
\\\hline
unstable & $L=\OOX([P])$,  $P\not \in W$ & $[P]+[\iota(P)]+[Q]+[\iota(Q)]$&yes &$\lambda_{\iota (P)}\neq \lambda_{P}= \lambda_{Q}=\lambda_{\iota (Q)}$\\\cline{2-5}
decomposable& $L=\OOX([w])$&   $2[w]+[Q]+[\iota (Q)]$ & no & $\lim_{P\to w}$ of the previous one \\\cline{2-5}
$L\oplus L^{-1}$& $L=\OOX(\mathrm{K}_X)$& $[P]+[\iota(P)]+[Q]+[\iota (Q)]$ & yes & $\lambda_P=\lambda_{\iota (P)}= \lambda_{Q}=\lambda_{\iota (Q)}$
\\\hline
odd Gunning bundle & $E_{w}$ & $\begin{array}{c}2[w]+[P]+[\iota (P)] \\(P \textrm{ arbitrary}) \end{array}$ & no &$\lambda_w=\lambda_P=\lambda_{\iota (P)}$ \\\hline
\end{tabular}
\end{minipage}}
\caption{List of vector bundles that can be obtained from Tyurin's construction.}\label{TyruinPossibilities}
\end{table}
\egroup

\begin{rem}This list is mostly a summary of the case by case study in Section \ref{SecTyurinSubBundle}.  Reasoning on the possible preimages of the destabilizing subbundle,  it is straightforward to check that the above mentioned decomposable bundles are the only possible ones.
Even Gunning bundles cannot be obtained: otherwise  two distinct trivial subbundles of the trivial bundle would generate two distinct Tyurin subbundes on an even Gunning bundle. 
\end{rem}

As a consequence, the moduli space $\mathcal{M}_{\mathrm{NR}}$ is birational to the moduli space of parabolic
structures over $D$ on $E_0$, when $D$ runs over the linear system $|2 \KX |$.
Let us be more precise. Consider the parameter space 
$$(\underline{P}_1,\underline{P}_2,\lambda)\in X\times X\times \P^1$$
and associate to each such data the parabolic structure defined on the vector bundle $\OX{- \KX}\oplus\OX{- \KX}$
by
$$\left(\lambda_{\underline{P}_1}, \lambda_{\iota\left(\underline{P}_1\right)}, \lambda_{\underline{P}_2}, \lambda_{\iota\left(\underline{P}_2\right)}\right):=\left(\lambda,-\lambda,\frac{1}{\lambda}, -\frac{1}{\lambda}\right).$$
Equivalently, one can view the parabolic structure as the collection of points
$$\left(\underline{P}_1,\lambda\right),\ \ \ \left(\iota(\underline{P}_1),-\lambda\right),\ \ \ \left(\underline{P}_2,\frac{1}{\lambda}\right)\ \ \ \text{and}\ \ \ \left(\iota(\underline{P}_2),-\frac{1}{\lambda}\right)$$
on the total space $X\times\P^1$ of the projectivized $\P^1$-bundle $\P\left(\OX{- \KX}\oplus\OX{- \KX}\right)$.
The natural rational map $X\times X\times \P^1\dashrightarrow\P^3_{\mathrm{NR}}$ is not birational however, since for a given bundle over $X$ there are several possibilities
to choose $\underline{P}_1$, $\underline{P}_2$ and $\lambda$. One can first independently permute $\underline{P}_1\leftrightarrow\iota(\underline{P}_1)$, $\underline{P}_2\leftrightarrow\iota(\underline{P}_2)$ and 
$\underline{P}_1\leftrightarrow \underline{P}_2$: this generates a order $8$ group of permutations. Moreover, once $\underline{P}_1$ and $\underline{P}_2$ have been chosen to parametrize 
the linear system $|2 \KX |$,
there is still a freedom in the choice of $\lambda$: our choice of normalization, characterized by 
$$\lambda_{\underline{P}_1}+\lambda_{\iota(\underline{P}_1)}=\lambda_{\underline{P}_2}+\lambda_{\iota(\underline{P}_2)}=0\ \ \ \text{and}\ \ \ \lambda_{\underline{P}_1}\cdot\lambda_{\underline{P}_2}=1,$$
is invariant under the Klein $4$ group $<z\mapsto-z,z\mapsto\frac{1}{z}>$ acting on the projective variable
$e_1+ze_2$. The transformation group taking into account all this freedom is 
generated by the following $4$ transformations
$$
\begin{matrix}
(\underline{X}_1\times \underline{X}_2\times \P^1_\lambda)\times(X\times\P^1_z)&\longrightarrow \quad (\underline{X}_1\times \underline{X}_2\times \P^1_\lambda)\times(X\times\P^1_z)\vspace{.2cm}\\
\left((\underline{P}_1,\underline{P}_2,\lambda),((x,y),z)\right)&\left\{\begin{matrix}
\stackrel{\sigma_{12}}{\longmapsto}&\left((\underline{P}_2,\underline{P}_1,\frac{1}{\lambda}),((x,y),z)\right)\vspace{.1cm}\\
\stackrel{\sigma_{\iota}}{\longmapsto}&\left((\iota(\underline{P}_1),\iota(\underline{P}_2),-\lambda),((x,y),z)\right)\vspace{.1cm}\\
\stackrel{\sigma_{iz}}{\longmapsto}&\left((\underline{P}_1,\iota(\underline{P}_2),i\lambda),((x,y),iz)\right)\vspace{.1cm}\\
\stackrel{\sigma_{1/z}}{\longmapsto}&\left((\underline{P}_1,\underline{P}_2,\frac{1}{\lambda}),((x,y),\frac{1}{z})\right)
\end{matrix}\right.
\end{matrix}
$$
(here, $i=\sqrt{-1}$). In fact, our choice of normalization for 
$\left(\lambda_{\underline{P}_1}, \lambda_{\iota\left(\underline{P}_1\right)}, \lambda_{\underline{P}_2}, \lambda_{\iota\left(\underline{P}_2\right)}\right)$ may not the most naive one,
which would have consisted to fix $3$ of them to $0$, $1$ and $\infty$; but our choice has the advantage that the transformation
$$
\begin{matrix}
(\underline{X}_1\times \underline{X}_2\times \P^1_\lambda)\times(X\times\P^1_z)&\longrightarrow&(\underline{X}_1\times \underline{X}_2\times \P^1_\lambda)\times(X\times\P^1_z)\vspace{.2cm}\\
\left((\underline{P}_1,\underline{P}_2,\lambda),((x,y),z)\right)&
\longmapsto&\left((\underline{P}_1,\underline{P}_2,\lambda),((x,-y),-z)\right)
\end{matrix}
$$
preserves the parabolic structure, and corresponds to the projectivized hyperelliptic involution $h:E\to\iota^*E$.
In particular, the subbundles $z=0$ and $z=\infty$ generated respectively by $e_1$ and $e_2$ precisely 
correspond to the two $\iota$-invariant Tyurin subbundles of $E$.

The $32$-order group $\langle\sigma_{12},\sigma_{\iota},\sigma_{iz},\sigma_{1/z}\rangle$ acts faithfully on the parameter space
$\underline{X}_1\times \underline{X}_2\times \P^1_\lambda$. Setting $\underline{P}_1=(\underline{x}_1,\underline{y}_1)$ and $\underline{P}_2=(\underline{x}_2,\underline{y}_2)$, the field of rational invariant functions is generated
by 
$$\underline{\s}:=\underline{x}_1+\underline{x}_2,\ \ \ \underline{\p}:=\underline{x}_1\underline{x}_2\ \ \ \text{and}\ \ \ \boldsymbol{\lambda}:=\left(\lambda^2+\frac{1}{\lambda^2}\right)\underline{y}_1\underline{y}_2$$
so that a quotient map (up to birational equivalence) is given by
\begin{equation}\label{TyurinQuotient}\begin{matrix}
\underline{X}_1\times \underline{X}_2\times \P^1_\lambda&\stackrel{(32:1)}{\dashrightarrow}& \P^2_D\times\P^1_{\boldsymbol{\lambda}}\\
\left((\underline{x}_1,\underline{y}_1),(\underline{x}_2,\underline{y}_2),\lambda\right)&\mapsto&\left((1:-\underline{x}_1-\underline{x}_2:\underline{x}_1\underline{x}_2),\left(\lambda^2+\frac{1}{\lambda^2}\right)\underline{y}_1\underline{y}_2\right)
\end{matrix}\end{equation}
Here, $\P^2_D=|2 \KX |$ is just the linear system parametrizing those divisors $D^T_E$. This quotient is our sharp Tyurin
configuration space, and we get a natural birational map
$$\P^2_D\times\P^1_{\boldsymbol{\lambda}}\dashrightarrow\P^3_{\mathrm{NR}}$$
which can be explicitely described as follows.

\begin{prop}\label{prop:TyurinToNR}The natural classifying map $\P^2_D\times\P^1_{\boldsymbol{\lambda}}\dashrightarrow\P^3_{\mathrm{NR}}$ writes 
$$\begin{array}{rl}(\underline{\s},\underline{\p},\boldsymbol{\lambda})\mapsto &(v_0:v_1:v_2:v_3)\vspace{.2cm}\\&=\left(\frac{\boldsymbol{\lambda}-\underline{\s}\underline{\p}^2+2(1+\sigma_1)\underline{\p}^2-(\sigma_1+\sigma_2)\underline{\s}\underline{\p}+(\sigma_2+\sigma_3)(\underline{\s}^2-2\underline{\p})-\sigma_3\underline{\s}}{\underline{\s}^2-4\underline{\p}}:\underline{\p}:-\underline{\s}:1\right).\end{array}$$
\end{prop}

Before proving it, let us make some observations. First, the fibration $\P^2_D\times\P^1_{\boldsymbol{\lambda}}\to\P^2_D$
is send onto the pencil of lines of $\P^3_{\mathrm{NR}}$ passing through the trivial bundle $E_0:(1:0:0:0)$.
In fact, the surface $\{\boldsymbol{\lambda}=\infty\}$ in Tyurin parameter space, 
corresponding to $\lambda=0$ or $\infty$, is the locus of the trivial bundle.
Also, the surface defined by $\lambda=\{1,-1,i,-i\}$ corresponds to generic decomposable flat bundles
and is sent onto the Kummer surface; we note that it is also defined by $\boldsymbol{\lambda}^2=4(y_1y_2)^2$
which, after expansion, writes
$$\begin{array}{rcccl}\boldsymbol{\lambda}^2&=&\underline{\p}(\underline{\p}-\underline{\s}+1)&\cdot &\left(\underline{\p}^3-\sigma_1\underline{\p}^2\underline{\s}+\sigma_2\underline{\p}\underline{\s}^2-\sigma_3\underline{\s}^3 +(\sigma_1^2-2\sigma_2)\underline{\p}^2+(3\sigma_3-\sigma_1\sigma_2)\underline{\p}\underline{\s}\right.\vspace{.2cm}\\
&&&&\left.+\sigma_1\sigma_3\underline{\s}^2
+(\sigma_2^2-2\sigma_1\sigma_3)\underline{\p}-\sigma_2\sigma_3\underline{\s}+\sigma_3^2\right)\end{array}$$
which allow us to retrieve the equation of $\mathrm{Kum}(X)\subset\P^3_{\mathrm{NR}}$.

\begin{proof}


Assume we are given $(\underline{P}_1,\underline{P}_2,\lambda)$ and the associated parabolic structure on $(E_0,\q)\to(X,D_E^T)$;
then, we want
to compute the Narasimhan-Ramanan divisor $D_E\subset\mathrm{Pic}^1(X)$ for the corresponding 
vector bundle $E$ obtained after $4$ elementary transformations. 
Given a degree $3$ line bundle $L_0$,
we can look at holomorphic sections $s_0:X\to E_0\otimes L_0$; it is straightforward to check that a section $s_1e_1+s_2e_2$ taking value
in Tyurin parabolic directions over $D_E^T$ will produce, after elementary transformations, a holomorphic section of $E\otimes L_0( \KX-D_E^T)$, showing that  $L_0( \KX-D_E^T)=L_0(- \KX)\in D_E$. Since sections of $L_0=\OX{[P_1]+[P_2]+[\infty]}$ are generated by $\langle 1,\frac{y+y_1}{x-x_1}-\frac{y+y_2}{x-x_2}\rangle$,
up to automorphisms of $E_0$, we can assume $s_1=1$ and $s_2=f:=\frac{y+y_1}{x-x_1}-\frac{y+y_2}{x-x_2}$.
Therefore, computing the cross-ratio, we get
$$\gamma:=\frac{\lambda_{\underline{P}_2}-\lambda_{\underline{P}_1}}{\lambda_{\iota(\underline{P}_1)}-\lambda_{\underline{P}_1}} : 
\frac{\lambda_{\underline{P}_2}-\lambda_{\iota(\underline{P}_2)}}{\lambda_{\iota(\underline{P}_1)}-\lambda_{\iota(\underline{P}_2)}}
=\frac{f(\underline{P}_2)-f(\underline{P}_1)}{f(\iota(\underline{P}_1)-f(\underline{P}_1)} : 
\frac{f(\underline{P}_2)-f\iota(\underline{P}_2))}{f(\iota(\underline{P}_1))-f(\iota(\underline{P}_2))}$$
which,  after reduction, gives 
$$\begin{array}{rcr}\frac{4\underline{y}_1\underline{y}_2\gamma}{(\underline{x}_1-\underline{x}_2)^2}&=&\left(-Diag(\underline{P}_1,\underline{P}_2)\cdot 1\ +\ Prod(\underline{P}_1,\underline{P}_2)\cdot Sum\right. \\&&\left.-\ 
Sum(\underline{P}_1,\underline{P}_2)\cdot Prod\ +\ Diag\right)\end{array}$$
with notations of Section \ref{SecComputeNR}. On the other hand, from Tyurin parameters, we get
$$\gamma=-\frac{(1-\lambda^2)^2}{4\lambda^2}$$
hence the result.
\end{proof}

The total space $(X_1\times X_2\times \P^1_\lambda)\times(X\times\P^1_z)$
is equipped with the $4$ rational sections
$$\left(\underline{P}_1,\lambda\right),\left(\iota(\underline{P}_1),-\lambda\right),\left(\underline{P}_2,\frac{1}{\lambda}\right),\left(\iota(\underline{P}_2),-\frac{1}{\lambda}\right)\ :\ \left(X_1\times X_2\times \P^1_\lambda\right)\ \to\ \left(X\times\P^1_z\right)$$
which are globally invariant under the action of $\langle\sigma_{12},\sigma_{\iota},\sigma_{iz},\sigma_{1/z}\rangle$. The quotient
provides a projective Poincar\'e bundle, namely a (non trivial) $\P^1$-bundle over 
$\left(\P^2_D\times\P^1_{\boldsymbol{\lambda}}\right)\times X$ (actually, over an open set of the parameters)
equipped with a universal parabolic structure. After positive elementary transformation, we get a universal $\P^1$-bundle over an 
open subset of $\P^3_{\mathrm{NR}}$. However, we cannot lift the construction to a vector bundle because the action of 
$<z\mapsto-z,z\mapsto\frac{1}{z}>$ (induced by $\langle\sigma_{iz}^2,\sigma_{1/z}\rangle$) does not lift to a linear
$\GL$-action (indeed, $\left(\begin{smallmatrix}-i &0\\0&i\end{smallmatrix}\right)$ and $\left(\begin{smallmatrix}0&1\\-1&0\end{smallmatrix}\right)$ do not commute). This is the reason why there is no Poincar\'e bundle for $\P^3_{\mathrm{NR}}$, but only a projective version of it.
The ambiguity is killed-out if we do not take $\sigma_{1/z}$ into account, meaning that we choose one of the two
$h$-invariants Tyurin subbundles: we then obtain Bolognesi's  Poincar\'e bundle mentioned in Section \ref{SecBertram}, which here is explicitely given as follows. Consider the vector bundle
$$\widetilde{\mathcal{E}} = p^*( \mathcal{O}_X( \KX))\otimes \mathrm{elm}^+_{\delta_1, \delta_2, \delta_3, \delta_4} ( (X_1\times X_2 \times \mathbb{P}^1) \times(X \times \mathbb{C}^2) )$$ over $ (X_1\times X_2 \times \mathbb{P}^1) \times X$,
where $\delta_1 : p = p_1, \ \delta_2 : p = \iota(p_1),\ \delta_3 : p = p_2, \ \delta_4 : p = \iota(p_2)$
if $ p$ denotes the projection from $X_1\times X_2 \times \mathbb{P}^1\times X$ to $X$ and $p_i$ the projection to $X_i$; and the parabolic structure over these divisors is given respectively by 
$$
\left(\underline{P}_1,\underline{P}_2,\lambda,\underline{P}_1, \left(\begin{smallmatrix} \lambda\\1\end{smallmatrix}\right)\right),\ \left(\underline{P}_1,\underline{P}_2,\lambda,\iota(\underline{P}_1), \left(\begin{smallmatrix} -\lambda\\1\end{smallmatrix}\right)\right),\ \left(\underline{P}_1,\underline{P}_2,\lambda,\underline{P}_2, \left(\begin{smallmatrix} 1\\\lambda \end{smallmatrix}\right)\right), \ 
\left(\underline{P}_1,\underline{P}_2,\lambda,\iota(\underline{P}_2), \left(\begin{smallmatrix} 1\\-\lambda\end{smallmatrix}\right)\right).
$$
This vector bundle is clearly invariant for the action 
$$
\left(\begin{matrix}\underline{P}_1,\underline{P}_2,\lambda,P,Z\end{matrix}\right) \left\{\begin{matrix}
\stackrel{\sigma_{12}}{\longmapsto}&\left(\underline{P}_2,\underline{P}_1,\frac{1}{\lambda},P,Z\right)\vspace{.2cm}\\
\stackrel{\sigma_{\iota}}{\longmapsto}&\left(\iota(\underline{P}_1),\iota(\underline{P}_2),-\lambda,P,Z\right)\vspace{.2cm}\\
\stackrel{\sigma_{iz}}{\longmapsto}&\left((\underline{P}_1,\iota(\underline{P}_2),i\lambda,P,\left(\begin{smallmatrix}\sqrt{i} &0\\0&\frac{1}{\sqrt{i}}\end{smallmatrix}\right)Z\right),\end{matrix}\right.
$$
\emph{i.e.} $\widetilde{\mathcal{E}} \simeq \sigma^*\widetilde{\mathcal{E}} $ for each $\sigma \in \langle\sigma_{12}, \sigma_\iota, \sigma_{iz}\rangle$. The quotient (in the sense of \cite{BiswasOrbifold})
thus defines a universal vector bundle $\mathcal{E}\to X\times B$ with trivial determinant bundle parametrized by the 2-cover 
$B=(X_1\times X_2 \times \mathbb{P}^1)/_{\langle\sigma_{12}, \sigma_\iota, \sigma_{iz}\rangle}=\P^2_D\times\P^1_\lambda$ of an open set of
$\mathcal{M}_{\mathrm{NR}}$.

\section{Flat parabolic vector bundles over the quotient $X/\iota$}\label{SecFlatPar}

The aim of this section is to completely describe the space $\BUN (X/\iota)$ of flat parabolic vector bundles over the quotient $X/\iota$. It can be covered by $3$-dimensional projective charts patched together by birational transition maps. Our main focus will lay on the Bertram chart $\P^3_{B}$ (see Section \ref{SecBertram}). We will see that this chart has a particularly rich geometry (see Figure 3). 
Precisely, there is a natural embedding $X/\iota\hookrightarrow\P^3_{B}$
as a twisted cubic and $\phi\vert_{\P^3_{B}}: ~\P^3_{B}\subset \BUN(X/\iota)\stackrel{2:1}{\longrightarrow} \BUN(X)$ is defined by the linear system of quadrics passing through the $6$ conic points
of $X/\iota$. The Galois involution 
$\Upsilon:\BUN(X/\iota)\stackrel{\sim}{\longrightarrow}\BUN(X/\iota)$
of $\phi$ 
is defined by elementary transformations: $\Upsilon=\OP{-3}\otimes\mathrm{elm}_{\underline{W}}^+$.
After restriction to the chart $\P^3_{B}$, it
is known as Geiser involution (see Dolgachev \cite{Dolgachev});
its decomposition as sequence of blow-up and contraction directly follows from the study of wall-crossing phenomena
when weights varry inside $\frac{1}{6}<\mu<\frac{5}{6}$.
In this picture, unipotent bundles come from the parabolic bundles parametrized by the cubic $X/\iota$,
and twisted unipotent bundles come from the $15$ lines passing through $2$ among $6$ points. The Gunning planes 
with even theta-characteristic come from  the $20$ planes passing through $3$ among $6$ points,
while odd Gunning planes come form the $6$ conic points of $X/\iota$, that are indeterminacy points for $\phi\vert_{\P^3_{B}}$.
Finally, the Kummer surface lifts as the dual Weddle surface (another quartic birational model of $\mathrm{Kum}(X)$). These results are summarized in the Figure 3.

 \begin{figure}[h]\centering
\hspace{-0cm} \resizebox{160mm}{!}{\input{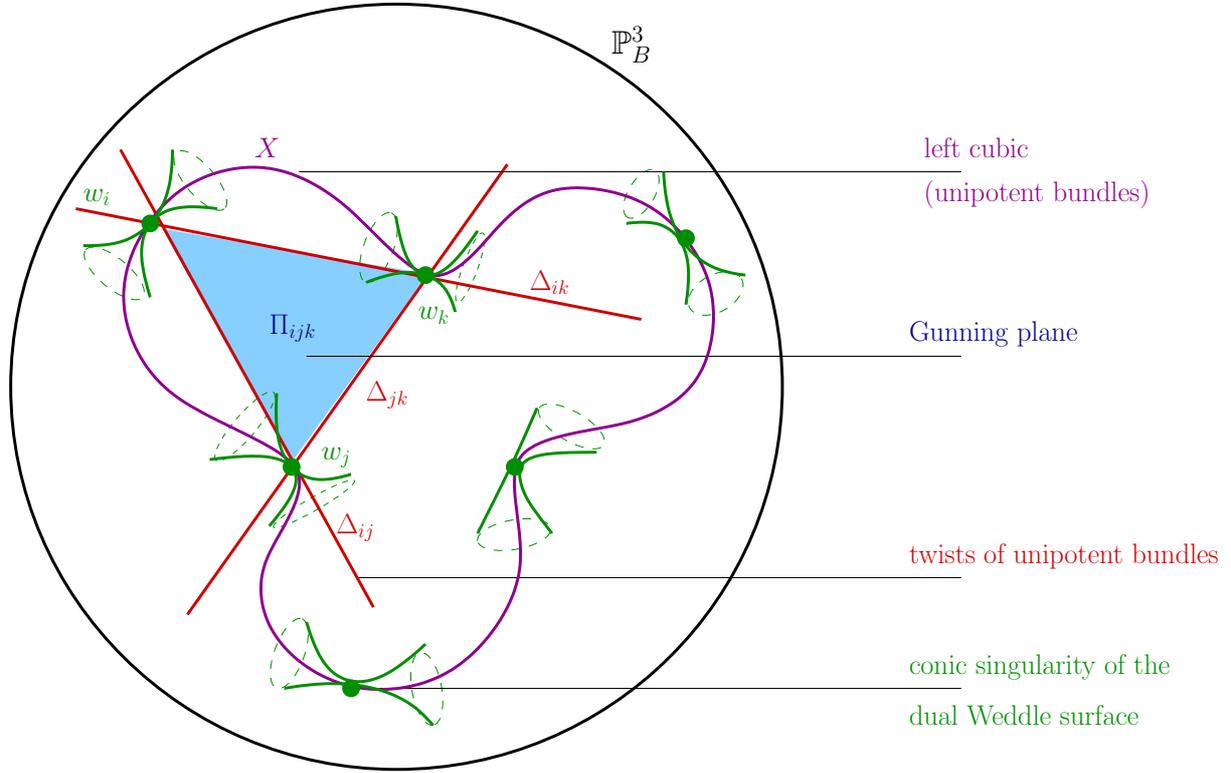}}\caption{Special bundles in the chart $\P^3_{B}$.}      
\end{figure}

\subsection{Flatness criterion}
Consider the data $\left(\underline{E},\underline{\nabla},\underline{\p}\right)$ where 
\begin{itemize} 
\item[$\bullet$] $\underline{E}$ is a rank $2$ vector bundle
over $\P^1$, 
\item[$\bullet$] $\underline{\nabla}:\underline{E}\to \underline{E}\otimes\OMP{\underline{W}}$ is a rank $2$ logarithmic
connection on $\underline{E}$ with polar divisor $\underline{W}=[0]+[1]+[r]+[s]+[t]+[\infty]$ and residual eigenvalues $0$ and $\frac{1}{2}$ over each pole,
\item[$\bullet$] $\underline{\p}=\left(\underline{p_0},\underline{p_1},\underline{p_r},\underline{p_s},\underline{p_t},\underline{p_\infty}\right)$ the parabolic
structure defined by the $\frac{1}{2}$-eigendirections over $x=0,1,r,s,t,\infty$. 
\end{itemize}
Via the Riemann-Hilbert correspondance, an equivalent data is the monodromy representation
$\pi_1\left(\P^1\setminus\{0,1,r,s,t,\infty\}\right)\to\GL$
with local monodromy $\sim\left(\begin{smallmatrix}1&0\\ 0&-1\end{smallmatrix}\right)$ at the punctures.
We denote by $\CON(X/\iota)$ the coarse moduli space of such parabolic connections $\left(\underline{E},\underline{\nabla},\underline{\p}\right)$.
Note that the parabolic structure $\underline{\p}$ is actually determined by the connection $\left(\underline{E},\underline{\nabla}\right)$
so that we do not need to specify it. However, it plays a crucial role in the bundle map.

We denote by $\BUN(X/\iota)$ the coarse moduli space of the parabolic bundles $\left(\underline{E},\underline{\p}\right)$
subjacent to some irreducible parabolic connection $\left(\underline{E},\underline{\nabla},\underline{\p}\right)$.
We note that, from Fuchs relations, we get that 
$$\deg(E)=-3\ \ \ \text{for any}\ \ \ \left(\underline{E},\underline{\p}\right)\in \BUN(X/\iota).$$
Following \cite{BiswasWeil,ArinkinLysenko}, we have the complete characterization of flat parabolic bundles:

\begin{prop}\label{prop:ParFlatCriterium}
Given a parabolic bundle $\left(\underline{E},\underline{\p}\right)$ like above, there exists 
a connection $\underline{\nabla}$ compatible with the parabolic structure like above
if and only if $\deg\left(\underline{E}\right)=-3$ and
\begin{itemize}
\item[$\bullet$] either $\left(\underline{E},\underline{\p}\right)$ is indecomposable,
\item[$\bullet$] or $\underline{E}=\OP{-1}\oplus\OP{-2}$ 
with $2$ parabolics defined by the the fibres over the Weierstrasspoints of the line subbundle $\OP{-1}$, the $4$ other ones by $\OP{-2}$,
\item[$\bullet$] or $\underline{E}=\OOP\oplus\OP{-3}$ 
with all parabolics defined by $\OP{-3}$.
\end{itemize} 
Moreover, in each case, one can choose $\underline{\nabla}$ irreducible.
\end{prop}

\begin{proof} We refer to the proof of Proposition 3 in \cite{ArinkinLysenko} to show that
indecomposable parabolic bundles are flat: this part of their proof does not use genericity 
of eigenvalues. In the decomposable case, $\underline{E}=L_1\oplus L_2$ and parabolics are distributed 
along $L_1$ and $L_2$ giving a decomposition $\underline{W}=D_1+D_2$. If it exists, a connection $\underline{\nabla}$
writes in  matrix form 
$$\underline{\nabla}=\begin{pmatrix}\nabla_1&\theta_{1,2}\\ \theta_{2,1}&\nabla_2\end{pmatrix}$$
where
\begin{itemize} 
\item[$\bullet$] $\nabla_i:L_i\to L_i\otimes\Omega^1_{\P^1}\left(D_i\right)$ is a logarithmic connection with 
eigenvalues $\frac{1}{2}$ for $i=1,2$; 
\item[$\bullet$] $\theta_{i,j}:L_j\to L_i\otimes\Omega^1_{\P^1}\left(D_i\right)$ is a morphism for $i\not=j$.
\end{itemize}
Fuchs relation for $\underline{E}$ gives $\deg\left(\underline{E}\right)=-3$, and for $\nabla_i$, gives 
$$-2\deg\left(L_i\right)=\text{number of parabolics lying on }L_i.$$
It follows that the only flat decomposable parabolic bundles are those listed in the statement.
Now we note that  connections $\nabla_i$ exist and are uniquely determined by above conditions. Setting $\theta_{i,j}=0$, 
we get a (totally reducible) parabolic connection $\nabla$ on $\left(E,\p\right)$.
In all cases, $\theta_{i,j}$ are morphisms $\OOP\left(n\right)\to\OOP\left(n+1\right)$
for some $n$ and live in a $2$-dimensional vector space. We claim that 
\begin{itemize} 
\item[$\bullet$] $\underline{\nabla}$ is reducible if, and only if, one of the $\theta_{i,j}=0$,
\item[$\bullet$] $\underline{\nabla}$ is totally reducible if, and only if, all $\theta_{i,j}=0$.
\end{itemize}
Indeed, if a line bundle $L\hookrightarrow E$ is $\underline{\nabla}$-invariant,
then Fuchs relation for $\underline{\nabla}\vert_L$ gives the following possible cases:
\begin{itemize} 
\item[$\bullet$] $\deg\left(L\right)=-3$ and $L$ contains all parabolics;
\item[$\bullet$] $\deg\left(L\right)=-2$ and  $L$ contains $4$ parabolics;
\item[$\bullet$] $\deg\left(L\right)=-1$ and  $L$ contains $2$ parabolics;
\item[$\bullet$] $\deg\left(L\right)=0$ and  $L$ contains no parabolics.
\end{itemize}
This forces $L$ to be one of direct summands of the decomposable cases above. For instance, 
when $\deg\left(L\right)=-3$, either $L\hookrightarrow\OOP\oplus\OP{-3}$
and must coincide with the second direct summand (since both must contain all parabolics), or 
$L\hookrightarrow\OP{-1}\oplus\OP{-2}$
but then $L$ intersects the first direct summand at only one point and thus cannot share the $2$ parabolics on $\OP{-1}$.
\end{proof}

It follows from Proposition \ref{prop:ParFlatCriterium} above that the only flat decomposable parabolic bundles are
\begin{itemize}
\item[$\bullet$] $\underline{E}=\OP{-1}\oplus\OP{-2}$ 
with $2$ parabolics defined by $\OP{-1}$, the $4$ other ones by $\OP{-2}$, and
\item[$\bullet$] $\underline{E}=\OOP\oplus\OP{-3}$ 
with all parabolics defined by $\OP{-3}$.
\end{itemize} 
For each such bundle $\left(\underline{E},\underline{\p}\right)$, the space of connections is $\C_{\theta_{1,2}}^2\times\C_{\theta_{2,1}}^2$ (the $\theta_{i,j}$ are those defined in the proof of Proposition \ref{prop:ParFlatCriterium})
where $\{0\}\times\C^2$ and $\C^2\times\{0\}$ stand for reducible connections
and $\{0\}\times\{0\}$ for the unique totally reducible one. The automorphism group of $\left(\underline{E},\underline{\p}\right)$
is $\C^*$ acting as follows:
$$\C^*\times\C^4\to \C^4\ ;\ 
\left(\lambda,\left(a_0,a_1,b_0,b_1\right)\right)\mapsto\left(\lambda a_0,\lambda a_1,\lambda^{-1}b_0,\lambda^{-1}b_1\right).$$
The GIT quotient is the affine threefold $xy=zw$ where $x=a_0a_1$, $y=b_0b_1$, $z=a_0b_1$ 
and $w=a_1b_0$; the singular point $x=y=z=w=0$ stands for reducible connections.

\begin{figure}[H]\label{fig1}
 \centering
 \resizebox{145mm}{!}{\input{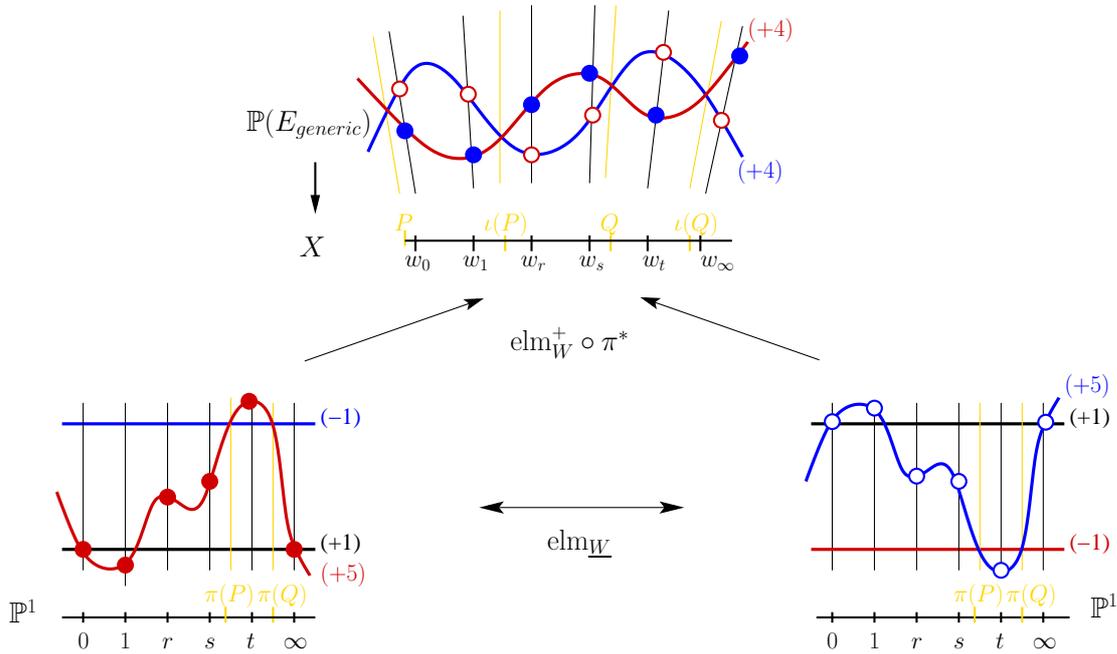}}\caption{A generic stable bundle on $X$.}      
\end{figure}

\subsection{Dictionary: how special bundles on $X$ occur as special bundles on $X/\iota$}\label{SecSpecialParBundle}

Let us recall the construction of the map $\phi:\BUN(X/\iota)\to \BUN(X)$ (see Sections \ref{diralg} and \ref{SecSymGal}). Given a flat parabolic bundle 
$(\underline{E},\underline{\p})$ in $\BUN(X/\iota)$, we lift it up to the curve $X$ as $\pi^*(\underline{E},\underline{\p})=(\tilde{E},\tilde{\p})$,
then apply elementary transformations $(E,\p):=\elm^+_W(\tilde{E},\tilde{\p})$ over the Weierstrass points
and get a determinant-free vector bundle $E$ over $X$, an element of $\BUN(X)$. Conversely, given a generic 
bundle $E$ on $X$, say stable and off the Gunning planes, then it has exactly two $\iota$-invariant anti-canonical 
subbundles $\OX{- \KX}\hookrightarrow E$ (see Corollary \ref{CorTyurinStablePar});
consider the parabolic structure $\p$ defined by the fibres over the Weierstrass points of one of them $L\subset E$. Then after applying elementary transformations 
over the Weiertrass points $(\tilde{E},\tilde{\p}):=\elm^-_W(E,\p)$, 
we get the lift of a unique parabolic bundle $(\underline{E},\underline{\p})$ on $X/\iota$; precisely, 
$\tilde E=\OX{- \KX}\oplus\OX{-2 \KX}$ and $\underline{E}=\OP{-1}\oplus\OP{-2}$.
The two anti-canonical subbundles $L,L'\subset E$, being $\iota$-invariant,
descend as two subbundles of $(\underline{E},\underline{\p})$; one easily checks that they are the destabilizing bundle 
$\underline L=\OP{-1}\subset\underline{E}\simeq\OP{-1}\times \OP{-2}$ and the unique 
$\underline{L}'\simeq\OP{-4}\subset\underline{E}$ containing all parabolics $\underline{p}$.

  \begin{figure}[H]\label{fig1a}
 \centering
 \resizebox{145mm}{!}{\input{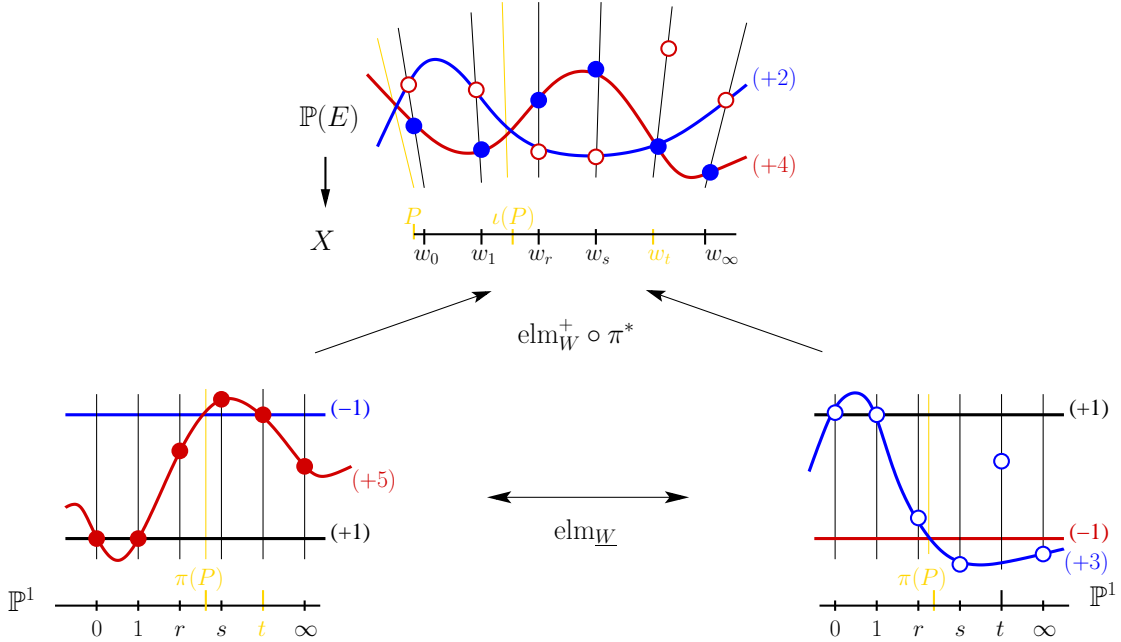}}\caption{A stable bundle belonging to the odd Gunning plane $\Pi_{[w_t]}$.}      
\end{figure}

In Figure 4, we can see the projectivized total space of the parabolic bundle associated to $E$ (a ruled surface), and its two preimages $\underline{E}$ and $\underline{E}'$ in $\BUN(X/\iota)$. The anti-canonical subbundles $L$ and $L'$ of $E$, and the corresponding subbundles
of $\underline{E}$ and $\underline{E}'$, are the blue and red curves (sections) on the ruled surfaces.
We can see the self-intersection of the curves in each case.
Parabolics are just points in Weierstrass fibers; those corresponding to $\p$ and $\underline{\p}$ (defined by the blue curve $L$ up-side)
are the red ones and those corresponding to $\p'$ and $\underline{\p}'$ (defined by the red curve $L'$ up-side)
are the blue ones. The intersection of the two curves determines (in each ruled surface) the Tyurin divisor $D_E^T$. 
The Galois involution of $\phi:\BUN(X/\iota)\to \BUN(X)$ permutes the roles of $L$ and $L'$; down-side, the elementary transformation permutes the role of the two curves.
The special case drawn in figure \ref{fig1a} where one of the two $+4$-curves is reducible, we obtain a stable bundle on an odd Gunning plane. Another special case, drawn in figure \ref{fig1b} arises when  $\mathbb{P} E$ possesses an invariant (but not anti-canonical) $+2$-curve (drawn here in green) containing three parabolics of each type (red and blue). This configuration corresponds to a stable bundle on an even Gunning plane.

 \begin{figure}\label{fig1b}
 \centering
 \resizebox{145mm}{!}{\input{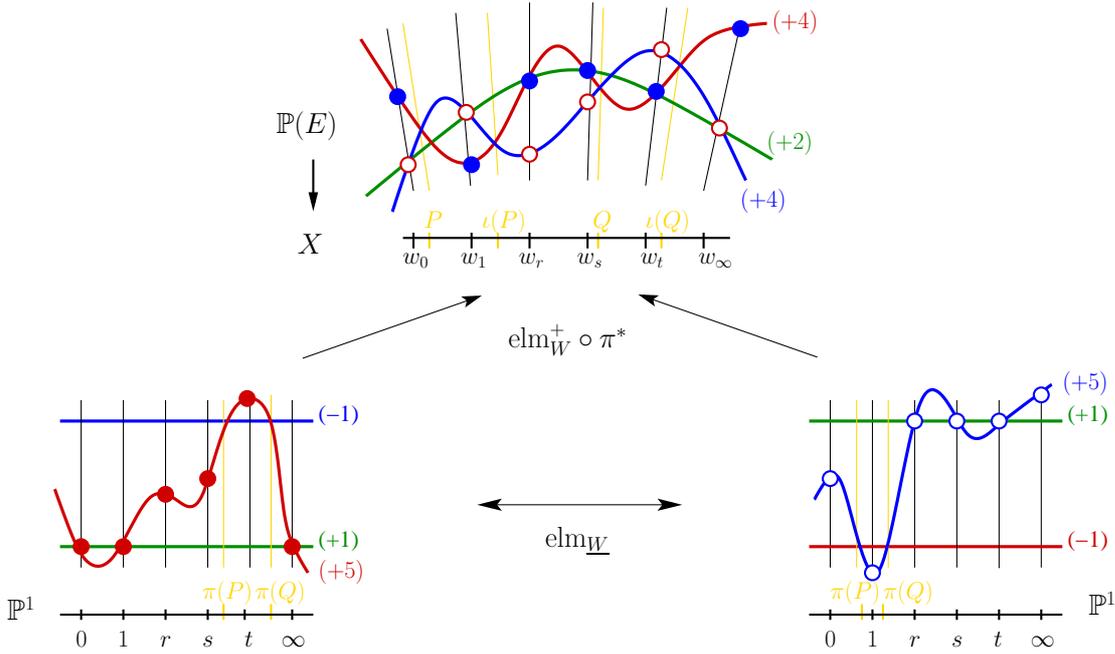}}\caption{A stable bundle belonging to the even Gunning plane $\Pi_{[w_r]+[w_s]-[w_t]}$.}      
\end{figure}

We now list the parabolic bundles of $\BUN(X/\iota)$ giving rise to special bundles of $\BUN(X)$ and illustrate 
on pictures the corresponding configurations of curves and points on the ruled surfaces.

\subsubsection{Generic decomposable bundles}

Let $E=L_0\oplus L_0^{-1}$, where $L_0=\mathcal O\left([P]+[Q]- \KX\right)$ is not $2$-torsion: $L_0^{2}\not=\OOX$.
Assume also, for simplicity, that neither $P$, nor $Q$ is a Weierstrass point.
Recall (see Section \ref{sec:TyurinDecomposable}) that, up to automorphism, there is a unique parabolic structure $\p$
which is defined by the line subbundle associated to any embedding $\OX{- \KX}\hookrightarrow E$. On the projective bundle $\P E$, there are 
two sections $\sigma_0,\sigma_\infty:X\to\P E$ coming from the two direct summands $L_0$ and $L_0^{-1}$ respectively,
both having $0$ self-intersection. They are  permuted by the involution $\iota:X\to X$. On the other hand, the anticanonical embedding
defines a section $\sigma :X\to\P E$ intersecting $\sigma_0$ at $[P]+[Q]$ and $\sigma_\infty$ at $[\iota(P)]+[\iota(Q)]$.
One can view $\P E$ as the fiber-wise compactification of $\OX{[P]+[Q]-[\iota(P)]-[\iota(Q)]}$ with $\sigma_0$ as the zero section
and $\sigma_\infty$ as the compactifying section; then $\sigma$ is a rational section with divisor $[P]+[Q]-[\iota(P)]-[\iota(Q)]$.

For the corresponding parabolic bundle $(\underline{E},\underline{\p})$, the anticanonical embedding descends as the destabilizing
subbundle $\OP{-1}\hookrightarrow \underline{E}=\OP{-1}\oplus\OP{-2}$. On the other hand, $\sigma_0$ and $\sigma_\infty$, 
being permuted by the involution $\iota$, descend as a $2$-section $\Gamma\subset\P(\underline E)$, thus intersecting
a generic member of the ruling twice. Moreover, $\Gamma$ intersects twice the section $\sigma_{-1}:\P^1\to \P(\underline E)$
defined by the destabilizing bundle $\OP{-1}\hookrightarrow \underline{E}$, namely at $\pi(P)$ and $\pi(Q)$ 
(where $\pi:X\to \P^1=X/\iota$ is the hyperelliptic projection). The restriction of 
the ruling projection $\P(\underline E)\to\P^1$ to the curve $\Gamma$:
$$\Gamma\to\P^1\ (=X/\iota)$$
is a $2:1$-cover branching precisely over the branching divisor $\underline W$ of $\pi:X\to \P^1$
(orbifold points of $X/\iota$). The parabolic structure $\underline{\p}$ is precisely located at the double points of $\Gamma\subset \P(\underline E)$ over $\underline W$.

Conversely, a parabolic structure $\underline{\p}$ on $\underline E=\OP{-1}\oplus\OP{-2}$ gives rise to a decomposable 
bundle $E$ if, and only if, there is a smooth curve $\Gamma\subset\P(\underline E)$ belonging to the linear system
defined by $\vert 2[\sigma_{-1}]+4[f] \vert$ (with $f$ any fiber of the ruling and $\sigma_{-1}$ the negative section as before)
such that $\Gamma$ passes through all $6$ parabolic points $\underline{\p}$ and is moreover vertical at these points
(i.e. tangent to the ruling). 

\begin{figure}[H]
\hspace{-0cm} \resizebox{130mm}{!}{\input{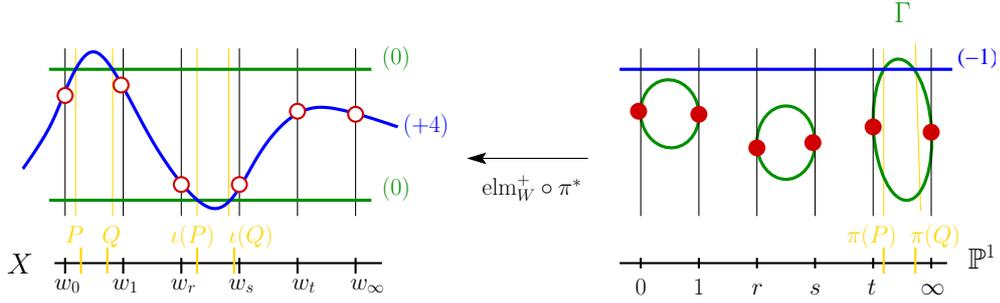}}\caption{A generic decomposable bundle on $X$.}      
\end{figure}

\subsubsection{The trivial bundle and its $15$ twists}\label{sec:TrivialBundleDescent}

Up to automorphism, the trivial bundle $E_0=\OOX\oplus\OOX$ has a unique parabolic structure $\p$, which is  defined by any line subbundle $\OOX\hookrightarrow E_0$.
Descending to $\P^1$, we get the decomposable bundle $\underline{E}_0=\OOP\oplus\OP{-3}$
with parabolic structure  $\underline{\p}$ defined by any line subbundle $\OP{-3}\hookrightarrow \underline{E}_0$.
Note that $\left(\underline{E}_0, \underline{\p}\right) $ is a fixed point of the Galois involution $\OOP(-3)\otimes \mathrm{elm}^+_{\underline W}$.
Similarly, $E_\tau=\tau\otimes E_0$ with $\tau=\OOX\left([w_i]-[w_j]\right)$ a $2$-torsion line bundle, comes from the decomposable parabolic bundle
$\underline{E}=\OP{-1}\oplus\OP{-2}$ having parabolics $\underline{p_i}$ and $\underline{p_j}$
lying in the first direct summand, the other ones in the second.

These $16$ parabolic bundles are exactly the flat decomposable bundles listed in Proposition \ref{prop:ParFlatCriterium}.

 \begin{figure}[H]
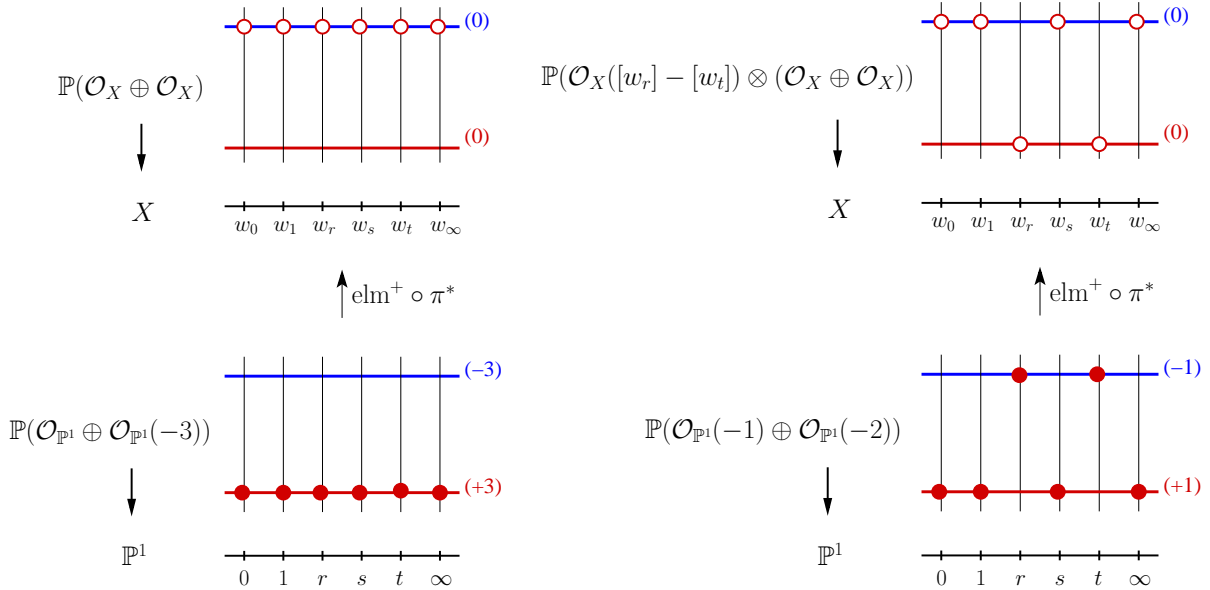
\centering
\hspace{-1.5cm} \resizebox{65mm}{!}{\input{trivialbundle.pstex_t}}\hspace{.5cm}\resizebox{87mm}{!}{\input{trivialbundletwist.pstex_t}}
\caption{The trivial bundle over $X$ and one of its twists.}      
\end{figure}

\subsubsection{The unipotent family and its $15$ twists}

A generic non trivial extension $0\to\OOX\to E\to \OOX\to 0$ has two hyperelliptic parabolic structures:
\begin{itemize}
\item[$\bullet$] $\p$ defined by some embedding $\OOX\left(- \KX\right)\hookrightarrow E$
(unique up to bundle automorphism); 
\item[$\bullet$] $\p'$ defined by the destabilizing bundle $\OOX\hookrightarrow E$. 
\end{itemize}
They respectively descend to elements of
\begin{itemize}
\item[$\bullet$] $\underline\Delta=\left\{(\underline{E},\underline{\p})\ ;\ \underline{E}=\OP{-1}\oplus\OP{-2}\ \text{and}\ 
\underline{\p}\subset\OP{-3}\subset\underline{E}\right\}$;
\item[$\bullet$] $\underline\Delta'=\left\{(\underline{E}',\underline{\p'})\ ;\ \underline{E}'=\OOP\oplus\OP{-3}\ \text{and}\ 
\underline{\p}'\subset\OP{-4}\subset\underline{E}'\right\}$.
\end{itemize}

The study of non-trivial extensions $0\to \tau\to E\to \tau\to 0$ where 
 $\tau=\OOX\left([w_i]-[w_j]\right)$ is a $2$-torsion line bundle, 
 can be deduced from the study of the corresponding unipotent bundles $\tau\otimes E$ 
 by  applying $\mathcal{O}_{\mathbb{P}^1}\left(-1\right)\otimes \mathrm{elm}_{[w_i]+[w_j]}^+$ on $\underline{E}$ or, equivalently, 
 by interchanging on $\tau\otimes E$ the parabolic directions $p_i$ and  $p_j$ with $p_i'$ and $p_j' $ respectively.
 We get a $1$-parameter family
$\Delta_{i,j}$ naturally parametrized by $X/\iota$.
There are two hyperelliptic parabolic structures for such a bundle $E$:
\begin{itemize}
\item[$\bullet$] $\p$ with parabolics $p_i$ and $p_j$ on $\OOX\hookrightarrow E$
and the others outside; 
\item[$\bullet$] $\p'$ with parabolics $p_i$ and $p_j$ outside $\OOX\hookrightarrow E$
and the others on it.
\end{itemize}
They respectively descend as elements of 
\begin{itemize}
\item[$\bullet$] $\underline\Delta_{i,j}=\left\{(\underline{E},\underline{\p})\ ;\ \underline{E}=\OP{-1}\oplus\OP{-2}\ \text{and}\ 
\underline{p_k}\subset\OP{-2},\ \forall k\not=i,j\right\}$;
\item[$\bullet$] $\underline\Delta_{i,j}'=\left\{(\underline{E}',\underline{\p}')\ ;\ \underline{E}'=\OP{-1}\oplus\OP{-2}\ \text{and}\ 
\underline{p_i}',\underline{p_j}'\subset\OP{-1}\right\}$.
\end{itemize}

Again, $\mathcal{O}\left(-3\right) \otimes \mathrm{elm}_{\underline W}^+$  point-wise permutes $\underline\Delta_{i,j}$ and $\underline\Delta_{i,j}'$.

\begin{figure}[H]\centerline{
\resizebox{150mm}{!}{\input{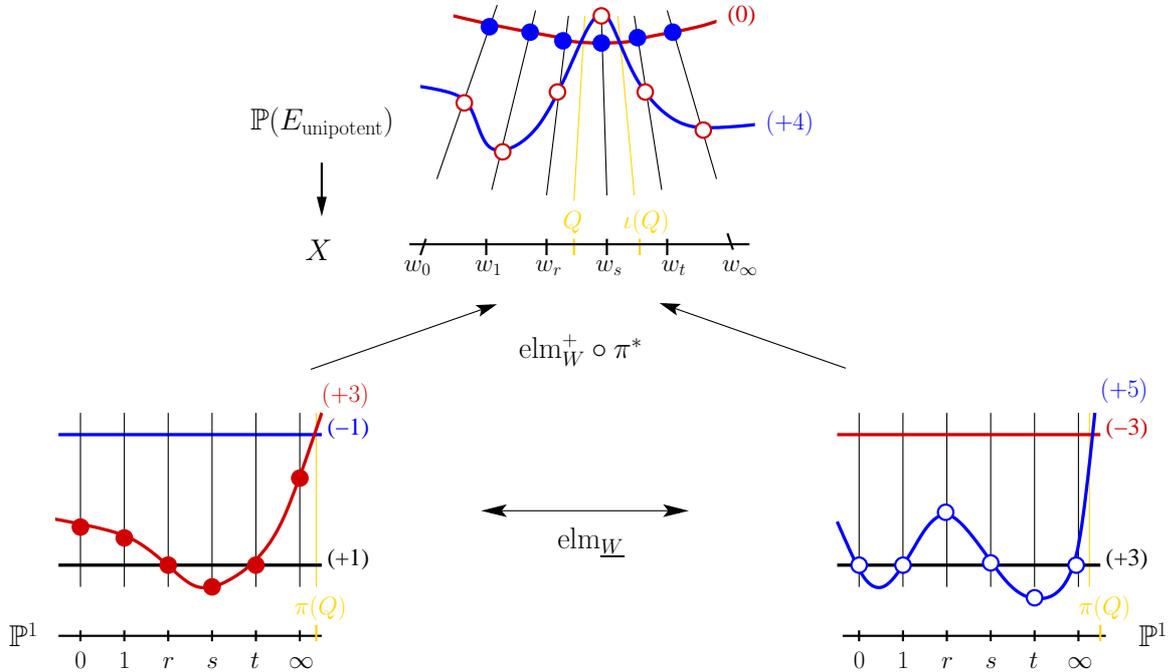}}}
\caption{A unipotent bundle over $X$.}      
\end{figure}

\vspace{1cm}
 \begin{figure}[H]\centerline{
\resizebox{150mm}{!}{\input{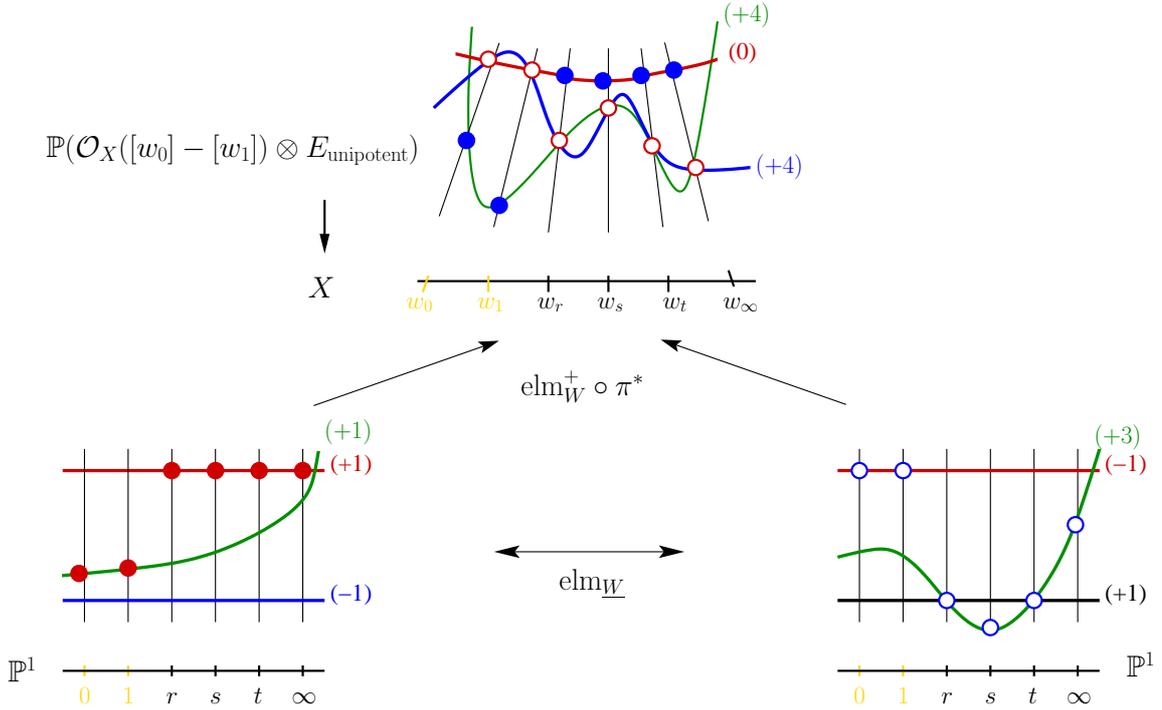}}}
\caption{Twist of a unipotent bundle over $X$.}      
\end{figure}

Denote by $\Delta$ the $1$-parameter family of unipotent bundles in $\mathfrak{Bun}\left(X\right)$ 
and by $\underline\Delta$ and $\underline\Delta'$
its respective preimages on $\BUN(X/\iota)$. Both of these families are naturally parametrized
by our base $X/\iota$: the extension class of $E\in\Delta$ is characterized by the intersection locus
of the two special subbundles $\OX{- \KX},\OOX\hookrightarrow E$, an element of $\vert\OX{ \KX}\vert\simeq\vert\OP{1}\vert$. Conversely, the intersection locus of two subbundles
$\OP{-1},\OP{-2}\hookrightarrow \underline{E}$, respectively  $\OOP,\OP{-4}\hookrightarrow \underline{E}'$, defines an element of $\vert\OP{1}\vert$.
This unambiguously defines  isomorphisms $\Delta\simeq\underline\Delta\simeq\underline\Delta'$, 
the latter one being induced by $\mathcal{O}\left(-3\right) \otimes \mathrm{elm}_{\underline W}^+$. Remind (see \cite{LoraySaito}) that, despite the point-wise identification 
just mentioned, any point of $\underline\Delta$ is arbitrary close to any point of $\underline\Delta'$ in the sense that they can be simultaneously
approximated
by some deformation of stable parabolic bundles. This will give rise to a flop phenomenon when we will compare 
certain semi-stable projective charts. The same phenomenon occurs for twisted unipotent bundles. 

\subsubsection{Affine bundles}
Affine bundles cannot occur from elements in $\BUN \left( X/\iota\right)$ since they are not invariant under the hyperelliptic involution.

\subsubsection{The $6+10$ Gunning bundles and Gunning planes}\label{sec:GunningLiftConf}

We now list how arise the unique non trivial extensions $0\to\OX{\vartheta}\to E\to \OX{-\vartheta}\to 0$
where $\vartheta$ runs over the $16$ theta characteristics $\vartheta^2= \KX$.

{\bf Six odd theta characteristics.}
For odd theta characteristics $\vartheta=[w_i]$ (lying on the divisor $\Theta$)
the two hyperelliptic parabolic structures are:
\begin{itemize}
\item[$\bullet$] $\p$ with parabolic $p_i$  in $\OX{\vartheta}\hookrightarrow E$
and the others outside; 
\item[$\bullet$] $\p'$ with all parabolics in $\OX{\vartheta}\hookrightarrow E$
except  $p_i$.
\end{itemize}
They respectively descend as 
\begin{itemize}
\item[$\bullet$] $Q_i:\ \underline{E}=\OP{-1}\oplus\OP{-2}\ \text{and}\ 
\underline{p_k}\subset\OP{-2},\ \forall k\not=i$;
\item[$\bullet$] $Q_i':\ \underline{E}'=\OOP\oplus\OP{-3}\ \text{and}\ 
\underline{p_i}'\subset\OOP$.
\end{itemize}

\noindent Analogously, the Gunning plane $\Pi_\vartheta$
descends as 
\begin{itemize}
\item[$\bullet$] $\underline\Pi_i=\left\{(\underline{E},\underline{\p})\ ;\ \underline{E}=\OP{-1}\oplus\OP{-2}\ \text{and}\ 
\underline{p_i}\subset\OP{-1}\right\}$;
\item[$\bullet$] $\underline\Pi_i'=\left\{(\underline{E}',\underline{\p}')\ ;\ \underline{E}'=\OP{-1}\oplus\OP{-2}\ \text{and}\ 
\underline{p_k}'\subset\OP{-3},\ \forall k\not=i\right\}$.
\end{itemize}

{\bf Ten even theta characteristics.}
Somehow different is the case of even theta characteristics $\vartheta=[w_i]+[w_j]-[w_k]$.
Denote by $\underline W=\{i,j,k\}\cup\{l,m,n\}$. The two hyperelliptic parabolic structures are:
\begin{itemize}
\item[$\bullet$] $\p$ with parabolics $p_i$, $p_j$ and $p_k$ 
in $\OX{\vartheta}\hookrightarrow E$ and the others outside; 
\item[$\bullet$] $\p'$ with parabolics $p_l$, $p_m$ and $p_n$
in $\OX{\vartheta}\hookrightarrow E$ and the others outside. 
\end{itemize}
They respectively descend as elements of
\begin{itemize}
\item[$\bullet$] $Q_{i,j,k}:\ \underline{E}=\OP{-1}\oplus\OP{-2}\ \text{and}\  
\underline{p_l},\underline{p_m},\underline{p_n}\subset\OP{-1}$;
\item[$\bullet$] $Q_{l,m,n}:\ \underline{E}'=\OP{-1}\oplus\OP{-2}\ \text{and}\  
\underline{p_i}',\underline{p_j}',\underline{p_k}'\subset\OP{-1}$.
\end{itemize}

\noindent The corresponding Gunning planes descend to
\begin{itemize}
\item[$\bullet$] $\underline\Pi_{i,j,k}=\left\{(\underline{E},\underline{\p})\ ;\ \underline{E}=\OP{-1}\oplus\OP{-2}\ \text{and}\ 
\underline{p_i},\underline{p_j},\underline{p_k}\subset\OP{-2}\subset\underline{E}\right\}$;
\item[$\bullet$] $\underline\Pi_{l,m,n}=\left\{(\underline{E}',\underline{\p}')\ ;\ \underline{E}'=\OP{-1}\oplus\OP{-2}\ \text{and}\ 
\underline{p_l}',\underline{p_m}',\underline{p_n}'\subset\OP{-2}\subset\underline{E}'\right\}$.
\end{itemize}

 \begin{figure}[H]\centerline{\resizebox{150mm}{!}{\input{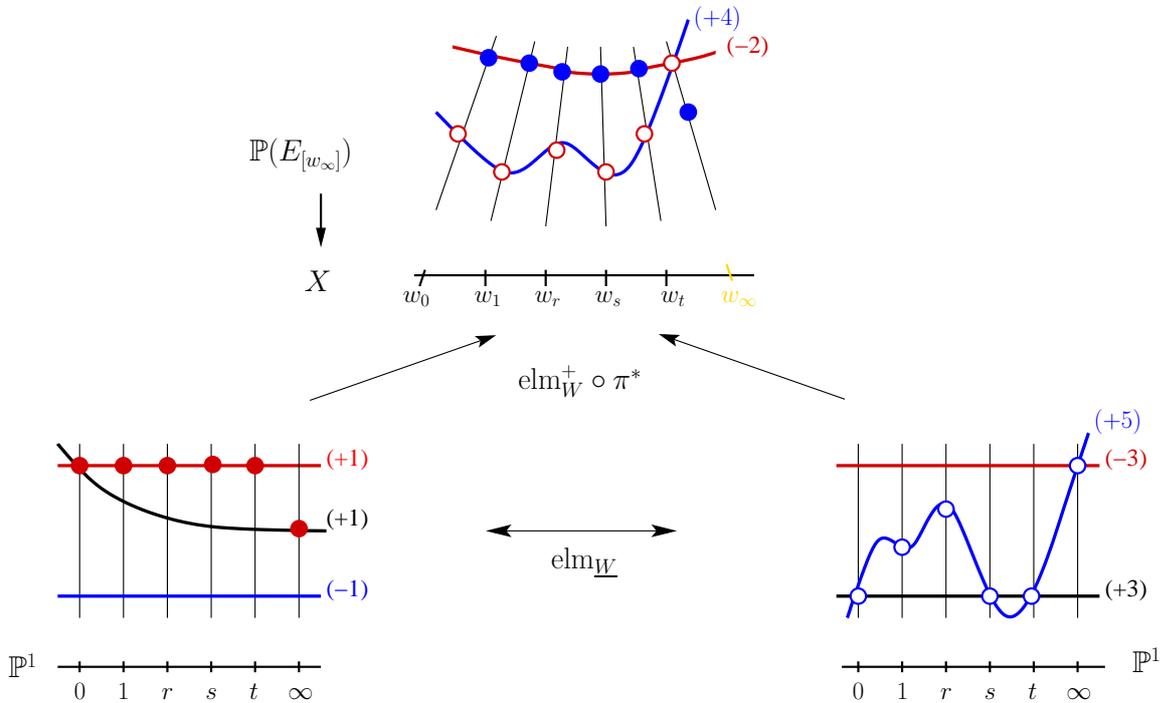}}}
\caption{An odd Gunning bundle over $X$.}      
\end{figure}

 \begin{figure}[H]\centerline{\resizebox{150mm}{!}{\input{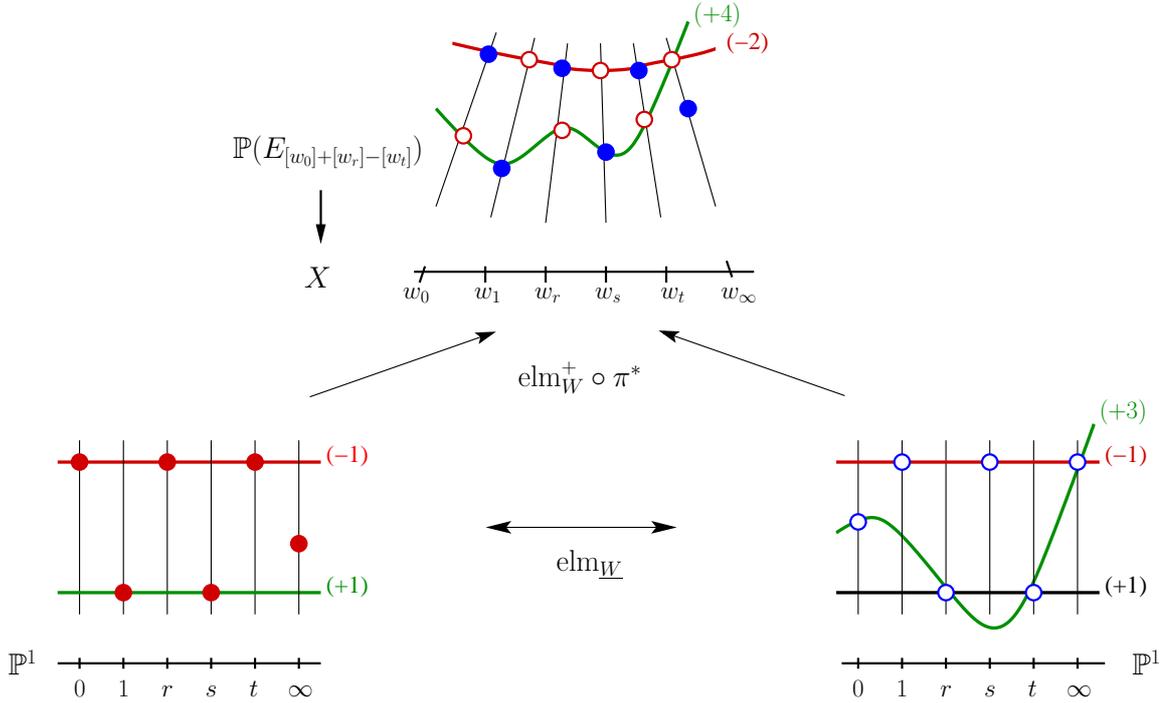}}}
\caption{An even Gunning bundle over $X$.}      
\end{figure}

\subsection{Semi-stable bundles and projective charts}\label{SecProjChartsParaBundles}

The coarse moduli space $\BUN^{ind}(X/\iota)$ of rank $2$ indecomposable parabolic bundles 
$\left(\underline{E},\underline{\p}\right)$ 
over $\P^1=X/\iota$ is studied in \cite{ArinkinLysenko,LoraySaito}. From the previous section,  $\BUN(X/\iota)\setminus\BUN^{ind}(X/\iota)$ only consists of $16$ bundles, that correspond to the trivial bundle and its $15$ twists on $X$
(see Section \ref{sec:TrivialBundleDescent}).
It turns out that a parabolic bundle $\left(\underline{E},\underline{\p}\right)$ is indecomposable if, and only if, 
it is stable for a good choice of 
weights $\Mu=\left(\mu_0,\mu_1,\mu_r,\mu_s,\mu_t,\mu_\infty\right)\in[0,1]^6$ (see \cite{LoraySaito}). 
One can thus cover the moduli space $\BUN^{ind}(X/\iota)$ by 
projective charts $\Bun^{ss}_{\Mu}(X/\iota)$ for a finite collection of weights, giving $\BUN^{ind}(X/\iota)$
a structure of non separated scheme. In this context, two parabolic bundles $\left(\underline{E},\underline{\p}\right)$ and $\left(\underline{E}',\underline{\p}'\right)$ are said to be \emph{arbitrarily close} if there are families  $\left(\underline{E}_t,\underline{\p}_t\right)_{t\in \mathbb{A}^1}$ and $\left(\underline{E}'_t,\underline{\p}_t'\right)_{t\in \mathbb{A}^1}$ such that $\left(\underline{E}_t,\underline{\p}_t\right)\simeq \left(\underline{E}'_t,\underline{\p}_t'\right)$ for each $t\neq 0$ but $\left(\underline{E}_0,\underline{\p}_0\right)\simeq \left(\underline{E},\underline{\p}\right)$ and $\left(\underline{E}_0',\underline{\p}_0'\right)\simeq \left(\underline{E}',\underline{\p}\right)$. If two parabolic bundles over $\mathbb{P}^1$ are arbitrarily close then of course the corresponding vector bundles over $X$ are arbitrarily close in the sense of Section \ref{Gunnbdle}.
By the way, $\BUN^{ind}(X/\iota)$ can be covered by charts isomorphic to $\left(\P^1\right)^3$ (see \cite{ArinkinLysenko}) or $\P^3$ (see \cite{LoraySaito}). We provide a finite set of charts covering $\BUN(X/\iota)$. We use two types of charts whose main representatives we now present: 

\subsubsection{The chart $\P^1_R\times\P^1_S\times\P^1_T$}\label{defrst}

The first one (see \cite{ArinkinLysenko} and \cite{LoraySaito} section 3.4) is given by weights of the form 
$$\mu_0=\mu_1=\mu_\infty=\frac{1}{2}\ \ \ \text{and}\ \ \ \mu_r=\mu_s=\mu_t=0$$
and is isomorphic to $\P^1_R\times\P^1_S\times\P^1_T$. Precisely, $\Mu$-stable bundles $\left(\underline{E},\underline{\p}\right)$
are given by $\underline{E}=\OP{-1}\oplus\OP{-2}$ with $\underline{p}_0,\underline{p}_1,\underline{p}_\infty$ outside of
$\OP{-1}\subset\underline{E}$ and not all three of them contained in the same $\OP{-2}\hookrightarrow\underline{E}$. Within the $2$-parameter family of line subbundles isomorphic to $\OP{-2}$ we can choose one  containing at least $\underline{p_0}$ and $\underline{p_\infty}$ say, and then 
choose meromorphic sections $e_1$ and $e_2$ of $\OP{-1}$ and $\OP{-2}$ (whose divisor is supported at $x=\infty$)
such that the parabolic structure is normalized to
$$\underline{p}_i=\lambda_ie_1+e_2\ \ \ \text{with}\ \ \ (\lambda_0,\lambda_1,\lambda_\infty)=(0,1,0)\ \ \ 
\text{and}\ \ \ (\lambda_r,\lambda_s,\lambda_t)=(R,S,T)\in\P^1_R\times\P^1_S\times\P^1_T.$$
To compare to the point of view of \cite{ArinkinLysenko}, note that
$$\OP{1}\otimes\elm^+_{\infty}\left(\underline{E},\underline{\p}\right)=(\underline{E}_0',\underline{\p}')$$
is the trivial bundle $\underline{E}_0'=\OOP\otimes\OOP$ equipped with a parabolic structure having $\underline{p_0}'$, $\underline{p_1}'$ 
and $\underline{p_\infty}'$ pairwise disctinct (with respect to the trivialization of the bundle).
From this chart, we can compute the two-fold cover $\phi:\BUN(X/\iota)\to\BUN(X)$. 

\begin{prop}\label{PropFormulesRSTtoNR}
The classifying map $\P^1_R\times\P^1_S\times\P^1_T\dashrightarrow\P^3_{\mathrm{NR}}$ is explicitely 
given by $(R,S,T)\mapsto(v_0:v_1:v_2:v_3)$ where
 $$\begin{array}{rcl}
 v_0&=& s^2t^2(r^2-1)(s-t)R-r^2t^2(s^2-1)(r-t)S+s^2r^2(t^2-1)(r-s)T+\\ &&+t^2(t-1)(r^2-s^2)RS-s^2(s-1)(r^2-t^2)RT+r^2(r-1)(s^2-t^2)ST\vspace{.3cm}\\
 v_1& =&rst\left[( (r-1)(s-t)R-(s-1)(r-t)S+(t-1)(r-s)T+\right.\\ &&\left.+ (t-1)(r-s)RS-(s-1)(r-t)RT+(r-1)(s-t)ST \right]\vspace{.3cm}\\
 v_2&=&-st(r^2-1)(s-t)R+rt(s^2-1)(r-t)S-rs(t^2-1)(r-s)T-\\&&-t(t-1)(r^2-s^2)RS+s(s-1)(r^2-t^2)RT-r(r-1)(s^2-t^2)ST\vspace{.3cm}\\
 v_3&=& st(r-1)(s-t)R-rt(s-1)(r-t)S+sr(t-1)(r-s)T+\\
 &&+ t(t-1)(r-s)RS-s(s-1)(r-t)RT+r(r-1)(s-t)ST
 \end{array}$$
 This map is generically $(2:1)$ with indeterminacy points
 $$(R,S,T)=(0,0,0),\ \ \ (1,1,1),\ \ \ (\infty,\infty,\infty)\ \ \ \text{and}\ \ \ (r,s,t).$$
 The Galois involution $(R,S,T) \mapsto(\widetilde{R},\widetilde{S},\widetilde{T})$ of this covering map is given by
 $$\begin{array}{l}
 \widetilde{R}= \lambda(R,S,T)\cdot  \frac{(s-t)+(t-1)S-(s-1)T}{-t(s-1)S+s(t-1)T+(s-t)ST}\\
\widetilde{S}= \lambda(R,S,T)\cdot \frac{(r-t)+(t-1)R-(r-1)T}{-t(r-1)R+r(t-1)T+(r-t)RT}\vspace{.3cm}\\
\widetilde{T}=\lambda(R,S,T)\cdot \frac{(r-s)+(s-1)R-(r-1)S}{-s(r-1)R+r(s-1)S+(r-s)RS}\vspace{.3cm}
\end{array}$$
$$\text{where}\ \ \ \begin{array}{l}\lambda(R,S,T)=\frac{t(r-s)RS-s(r-t)RT+r(s-t)ST}{(s-t)R-(r-t)S+(r-s)T}\end{array}.$$ 
The ramification locus is over the Kummer surface; its lift on $\P^1_R\times\P^1_S\times\P^1_T$ is given by the equation
$$((s-t)R+(t-r)S+(r-s)T)RST$$
$$+t((r-1)S-(s-1)R)RS+r((s-1)T-(t-1)S)ST+s((t-1)R-(r-1)T)RT$$
$$-t(r-s)RS-r(s-t)ST-s(t-r)RT=0.$$
\end{prop}

\begin{proof} For computations, we work with the parabolic bundle 
$$(\underline{E}_0',\underline{\p}'):=\OP{1}\otimes\elm^+_{\infty}\left(\underline{E},\underline{\p}\right)$$
where $\underline{E}_0'=\OOP\otimes\OOP$ is the trivial bundle, generated by sections $e_1'$ and $e_2'$, and $\underline{\p}'$
is the parabolic structure defined by
$$\underline{p}_i'=\lambda_ie_1'+e_2'\ \ \ \text{with}\ \ \ (\lambda_0,\lambda_1,\lambda_\infty)=(0,1,\infty)\ \ \ 
\text{and}\ \ \ (\lambda_r,\lambda_s,\lambda_t)=(R,S,T)\in\P^1_R\times\P^1_S\times\P^1_T.$$
Let now $E$ be the vector bundle over $X$ obtained by 
$$E:=\elm^+_W\ \left(\pi^*\ (\underline{E},{\p})\right)
=\elm^+_W\ \left(\pi^*\ \left( \OP{-1}\otimes\elm^-_{[\infty]}\ (\underline{E}_0',{\p'})\right)\right);$$
this can be rewritten as
$$E:=\OX{-3[w_{\infty}]}\otimes\ \elm^+_W\ \left(\pi^*\ (\underline{E}_0',{\p'})\right)
=\OX{-3[w_{\infty}]}\otimes\ \elm^+_W\ (E_0,\pi^*{\p'})$$
where $E_0$ is the trivial vector bundle on $X$.

In order to calculate the classifying map, we need to make the Narasimhan-Ramanan divisor  $D_E$ explicit in our coordinates. 
We may assume that $E$ is generic (\emph{i.e.} stable), so that $D_E$ precisely describes the $1$-parameter family of degree
$-1$ line bundles $L\subset E$. After applying $\OX{-3[\infty]}\otimes\ \elm^+_W$, we get the family of degree $-4$ subbundles
$L'\subset E_0$ (the trivial bundle over $X$) containing all $6$ parabolics $\p'$. Precisely, if $L=\OX{[w_{\infty}]-[P_1]-[P_2]}$,
then $L'=\OX{-3[w_{\infty}]}\otimes L=\OX{-2[w_{\infty}]-[P_1]-[P_2]}$. 
In other words, the Narasimhan-Ramanan divisor $D_E\subset\mathrm{Pic}^1(X)$ is directly given by the 1-parameter family of points $\{P_1,P_2\}$ such that there is a line subbundle $L=\mathcal{O}_X(-[P_1]-[P_2]-2[\infty]) \hookrightarrow E_0$ coinciding with the parabolic structure over $W$. 
Let $\sigma=(\sigma_1,\sigma_2) : X \to \mathbf{C}^2$ be a meromorphic section of $L$ with divisor $-[P_1]-[P_2]-2[\infty]$ with  $P_i=(x_i,y_i) \in X, ~i=1,2$: 
$${\sigma_1 \choose \sigma_2} = 
{{\alpha+\beta x+ \gamma (\frac{y-y_1}{x-x_1}-\frac{y-y_2}{x-x_2})}\choose {\delta+\varepsilon x+ \varphi (\frac{y-y_1}{x-x_1}-\frac{y-y_2}{x-x_2})}}.$$ 
After normalizing $\alpha=1$, 
 there is a unique choice of $\beta, \gamma, \delta, \varepsilon, \varphi \in \mathbb{C}$ such that $$
\sigma(0,0) \parallel {0 \choose 1}, ~\sigma(1,0)\parallel {1 \choose 1}, ~\sigma(r,0) \parallel {R \choose 1}, ~\sigma(s,0) \parallel {S \choose 1}, ~\sigma(\infty, \infty) \parallel {1 \choose 0}.$$ The condition $\sigma(t,0) \parallel {T \choose 1}$ depends now only on the choice of $\{P_1,P_2\}$ and writes (after convenient reduction)
$$v_0\cdot 1+v_1\cdot Sum(P_1,P_2)+v_2\cdot Prod(P_1,P_2)+v_3\cdot Diag(P_1,P_2)=0$$
with $v_i$ as given in the proposition.

One can easily deduce that a generic point  $(v_0:v_1:v_2:v_3)\in \mathcal{M}_{\mathrm{NR}}$ has precisely two preimages in $\P^1_R\times\P^1_S\times\P^1_T$ : 
$$\begin{array}{rcl}
R&=&\frac{r(t-1)(v_0+rv_1-r(s+t+st)v_3)T}{t(r-1)(v_0+tv_1-t(r+s+rs)v_3)-(r-t)(v_0+v_1-\sigma_2v_3)T}\vspace{.3cm}\\
S&=&\frac{s(t-1)(v_0+sv_1-s(r+t+rt)v_3)T}{t(s-1)(v_0+tv_1-t(r+s+rs)v_3)-(s-t)(v_0+v_1-\sigma_2v_3)T} ,
\end{array}$$
where $T$ is any solution of $aT^2+btT+ct^2=0$ with $$\begin{array}{rcl}
a&=&(v_1+v_2t+v_3t^2)(v_0+v_1-\sigma_2v_3)\\
b&=&-(1+t)(v_0v_2+v_1^2+tv_1v_3)-2(v_0v_1+tv_0v_3+tv_1v_2)\\&&+\sigma_2(tv_1+v_2+tv_3)v_3+(r+s+rs)(v_1+t^2v_2+t^2v_3)v_3\\
c&=&(v_1+v_2+v_3)(v_0+tv_1-t(r+s+rs)v_3). \end{array}$$ The discriminant of this polynomial leads again to our equation of the Kummer surface in the coordinates $(v_0:v_1:v_2:v_3)$ given in Section \ref{SecComputeNR}. We can easily calculate the Galois involution of the classifying map 
$\P^1_R\times\P^1_S\times\P^1_T\dashrightarrow\P^3_{\mathrm{NR}}$. Its fixed points provide the equation in coordinates $(R,S,T)$ of the lift of the Kummer surface. 
\end{proof}

\subsubsection{The chart $\P^3_{\boldsymbol{b}}$}
The other chart (namely the main chart $\P^3_{\boldsymbol{b}}$ of \cite{LoraySaito}) is defined by democratic weights
$$\frac{1}{6}< {\mu_0}={\mu_1}={\mu_r}={\mu_s}={\mu_t}={\mu_\infty}<\frac{1}{4}$$
and corresponds to the moduli space of the indecomposable parabolic structures on $\underline{E}:=\OP{-1}\oplus\OP{-2}$ having no parabolic in the total space of $\OP{-1}$. 
Parabolic bundles belonging to this chart are exactly those  given by extensions
$$0\to (\OP{-1},\emptyset)\to \left(\underline{E},\underline{\p}\right)\to (\OP{-2},\underline{W})\to 0$$
\emph{i.e.} elements of $\P  \mathrm{H}^1\left(\P^1,\Hom(\OP{-2}\otimes\OP{\underline{W}},\OP{-1})\right)$, which by Serre duality,
identifies to $\P  \mathrm{H}^0\left(\P^1,\OP{-1}\otimes\OMP{\underline{W}}\right)^\vee$. After lifting them on $X\to\P^1$, applying elementary transformations and forgetting the parabolic structure,
we precisely get those extensions 
$$0\to\mathcal O\left(- \KX\right)\to E\to \mathcal O\left( \KX\right)\to 0$$
i.e. by those points of $\P^4_B=\P  \mathrm{H}^0\left(X,\OX{3 \KX}\right)^\vee$, that are $\iota$-invariant. 
Thus, the projective chart $\P^3_{\boldsymbol{b}}$ of \cite{LoraySaito} naturally identifies with 
$\P^3_B$ introduced by Bertram (see Section \ref{SecBertram}). From this point of view, we have natural
projective coordinates $\boldsymbol{b}=(b_0:b_1:b_2:b_3)$, dual to the coordinates of $\iota$-invariant
cubic forms $\left(a_0+a_1x+a_2x^2+a_3x^3\right)\left(\frac{\mathrm{d}x}{y}\right)^{\otimes 3}.$
After computation, we get

\begin{prop}\label{PropBertramToRST}
The natural birational map $\P^3_B\dashrightarrow\P^1_R\times\P^1_S\times\P^1_T$
is given by 
$$(b_0:b_1:b_2:b_3)\mapsto\left\{\begin{matrix}
R&=&r\frac{b_3-(s+t+1)b_2+(st+s+t)b_1-stb_0}{b_3-\sigma_1b_2+\sigma_2b_1-\sigma_3b_0}\vspace{.2cm}\\
S&=&s\frac{b_3-(r+t+1)b_2+(rt+r+t)b_1-rtb_0}{b_3-\sigma_1b_2+\sigma_2b_1-\sigma_3b_0}\vspace{.2cm}\\
T&=&t\frac{b_3-(r+s+1)b_2+(rs+r+s)b_1-rsb_0}{b_3-\sigma_1b_2+\sigma_2b_1-\sigma_3b_0}
\end{matrix}\right.$$
The inverse map is given by 
$(R,S,T)\mapsto (b_0:b_1:b_2:b_3)$ with$$\left\{\begin{array}{rcl}
b_0&=&\frac{R-r}{r(r-1)(r-s)(r-t)}+\frac{S-s}{s(s-1)(s-r)(s-t)}+\frac{T-t}{t(t-1)(t-r)(t-s)}\vspace{.2cm}\\
b_1&=&\frac{R-r}{(r-1)(r-s)(r-t)}+\frac{S-s}{(s-1)(s-r)(s-t)}+\frac{T-t}{(t-1)(t-r)(t-s)}\vspace{.2cm}\\
b_2&=&\frac{r(R-r)}{(r-1)(r-s)(r-t)}+\frac{s(S-s)}{(s-1)(s-r)(s-t)}+\frac{t(T-t)}{(t-1)(t-r)(t-s)}\vspace{.2cm}\\
b_3&=&\frac{r^2R}{(r-1)(r-s)(r-t)}+\frac{s^2S}{(s-1)(s-r)(s-t)}+\frac{t^2T}{(t-1)(t-r)(t-s)}-\frac{1}{(r-1)(s-1)(t-1)}\end{array}\right.$$

\end{prop}

This will be proved in Section \ref{SecApparentRST}, using Higgs fields. 
The geometry of this birational map is explained in section \ref{GeomBRST} and summarized in Figure \ref{figPBRST}.

Combination of the Propositions \ref{PropBertramToRST} and \ref{PropFormulesRSTtoNR} yields

\begin{cor}\label{CorCoordBNR}
The natural map $\P^3_B\dashrightarrow\P^3_{NR}$ is given by
$$(b_0:b_1:b_2:b_3)\mapsto\left\{\begin{matrix}
v_0&=&b_2b_3-(1+\sigma_1)b_2^2+(\sigma_1+\sigma_2)b_1b_2-(\sigma_2+\sigma_3)b_0b_2+\sigma_3b_0b_1\\
v_1&=&b_2^2-b_1b_3\\
v_2&=&b_0b_3-b_1b_2\\
v_3&=&b_1^2-b_0b_2 
\end{matrix}\right.$$
Moreover, the (dual) Weddle surface, \emph{i.e.} the lift to $\P^3_B$ of the Kummer equation, writes
$$(-b_0b_2b_3^2+b_1^2b_3^2+b_1b_2^2b_3-b_2^4)+(1+\sigma_1)(b_0b_2^2b_3-2b_1^2b_2b_3+b_1b_2^3)
+(\sigma_1+\sigma_2)(-b_0b_2^3+b_1^3b_3)$$
$$+(\sigma_2+\sigma_3)(-b_0b_1^2b_3+2b_0b_1b_2^2-b_1^3b_2)
+\sigma_3(b_0^2b_1b_3-b_0^2b_2^2-b_0b_1^2b_2+b_1^4)=0.$$
\end{cor}

This Corollary has to be compared to Section \ref{SecBertram}. Indeed, the components 
of the map $\P^3_B\dashrightarrow\P^3_{\mathrm{NR}}$ exactly correspond to the restriction to $\P^3_B$
of the natural quadratic forms on $\P^4_B$ vanishing along the embedding 
$$X\hookrightarrow\P^4_B\ ;\ (x,y)\mapsto(b_0:b_1:b_2:b_3:b_4)=(1:x:x^2:x^3:y).$$
Indeed, the first one is the restriction of
$$b_4^2-(b_2b_3-(1+\sigma_1)b_2^2+(\sigma_1+\sigma_2)b_1b_2-(\sigma_2+\sigma_3)b_0b_2+\sigma_3b_0b_1)$$
which vanishes along $X\hookrightarrow\P^4_B$ from
$$y^2=x(x-1)(x-r)(x-s)(x-t)=x^5-(1+\sigma_1)x^4+(\sigma_1+\sigma_2)x^3-(\sigma_2+\sigma_3)x^2+\sigma_3x.$$
The other $3$ quadratic forms just come from the following equalities on $X$
$$b_0b_2=b_1^2=x^2,\ \ \ b_0b_3=b_1b_2=x^3\ \ \ \text{and}\ \ \ b_1b_3=b_2^2=x^4.$$

It is quite surprising that the most natural basis both appearing from Bertram
point of view, and Narasimhan-Ramanan point of view, are so compatible. They provide the same system of coordinate
on $\P^3_{\mathrm{NR}}$ which is however not considered in the classical theory of Kummer surfaces 
(see \cite{Hudson,GonzalezDorrego}).

In Section \ref{sec:MovingWeights} we provide four other charts, also with democratic weights, by varying $\mu$ in $[0,1]$. 
This has the advantage, with respect to an arbitrary choice of a covering of  $\BUN(X/\iota)$ by charts  $\BUN_{\Mu}^{ss}(X/\iota)$, to make the geometry of (birational) transition maps between charts quite clear.

 \begin{figure}
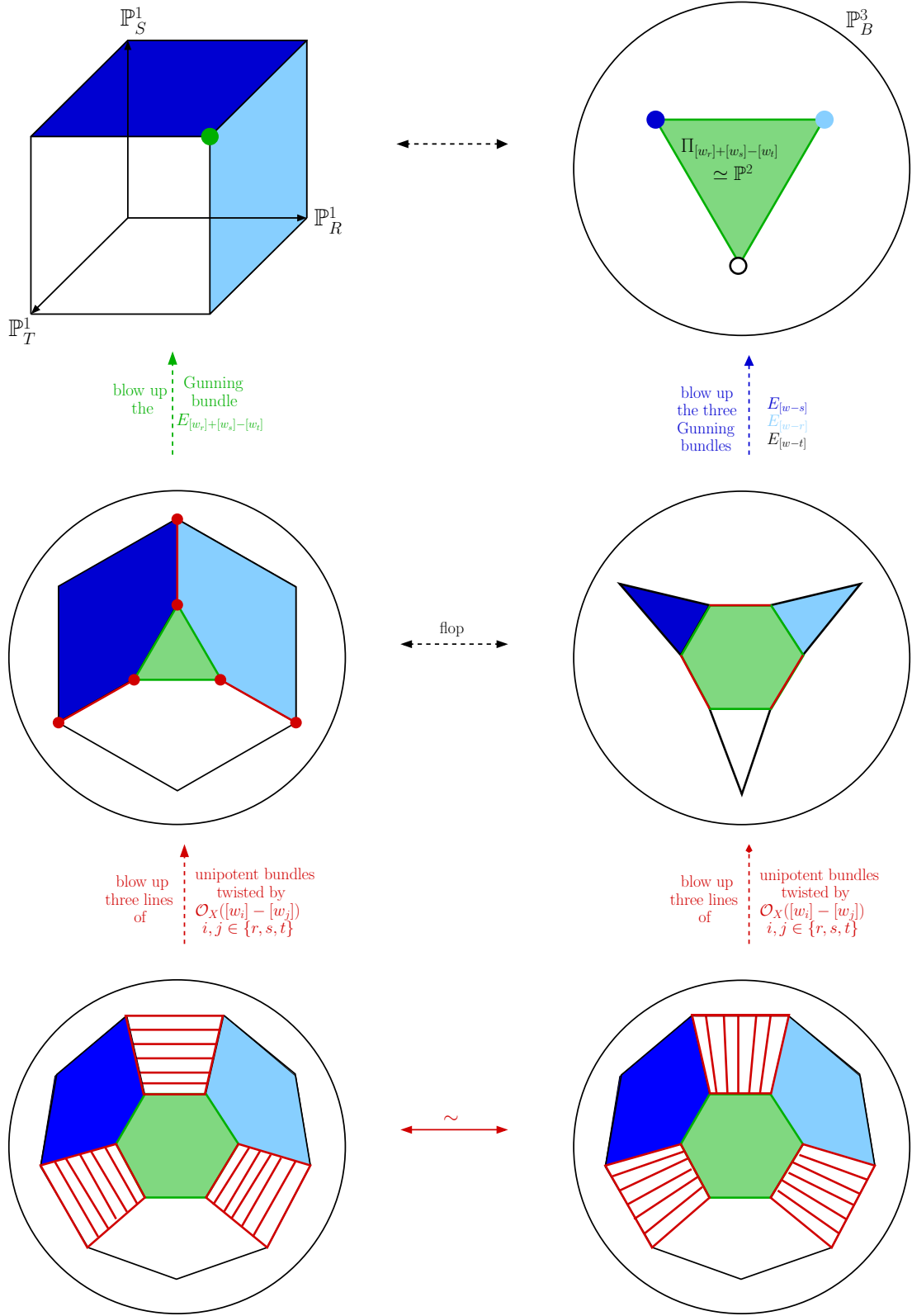
\label{PBRST}\centerline{\resizebox{80mm}{!}{\input{PBRST1b.pstex_t}}\vspace{.5cm}\hspace{1cm}\resizebox{53.9mm}{!}{\input{PBRST6b.pstex_t}}}
 \centerline{\resizebox{80mm}{!}{\input{PBRST2b.pstex_t}}\vspace{.5cm}\hspace{1cm}\resizebox{53.9mm}{!}{\input{PBRST5b.pstex_t}}}
  \centerline{\resizebox{80mm}{!}{\input{PBRST3b.pstex_t}}\vspace{.5cm}\hspace{1cm}\resizebox{53.9mm}{!}{\input{PBRST4b.pstex_t}}}
\caption{Geometry of the natural birational map $\P^3_B\dashrightarrow\P^1_R\times\P^1_S\times\P^1_T$.}\label{figPBRST}      
\end{figure}

\subsubsection{Special bundles in the chart $\P^3_{\b}$}
Here is the list of those special parabolic bundles of Section \ref{SecSpecialParBundle}
that are semi-stable for $\frac{1}{6}< {\mu_0}={\mu_1}={\mu_r}={\mu_s}={\mu_t}={\mu_\infty}<\frac{1}{4}$
and how they occur as special points in the chart $\P^3_{\b}$. 

\begin{prop} The only special bundles occuring (as semi-stable parabolic bundles) in $\Bun^{ss}_{\Mu}(X/\iota)=\P^3_{\b}$
are generic bundles of the following families
\begin{itemize}
\item {\bf Unipotent bundles} $\underline\Delta$: this $1$-parameter family corresponds to the twisted cubic parametrized by
$$X/\iota \to \P^3_{\b}\ ;\ x\mapsto (1:x:x^2:x^3).$$
\item {\bf Odd Gunning bundles} $Q_i$: they are the $6$ special points of the previous embedding $X/\iota \to \P^3_{\b}$,
namely $Q_i$ is the image of the Weierstrass point $w_i$.
\item {\bf twisted unipotent bundles} $\underline\Delta_{i,j}$: lines of $\P^3_{\b}$ passing through $Q_i$ and $Q_j$.
\item {\bf even Gunning planes} $\underline\Pi_{i,j,k}$: planes of $\P^3_{\b}$ passing through $Q_i$, $Q_j$ and $Q_k$.
\item {\bf odd Gunning planes} $\underline\Pi_i'$: the quadric surface of $\P^3_{\b}$ with a conic singular point at $Q_i$
that contains the $5$ lines $\Delta_{i,j}$ and the cubic $\Delta$.
\end{itemize}
\end{prop}

\begin{proof} It is easy to check which special parabolic bundles are semi-stable or not. For instance, 
the trivial bundle $E_0$ descends as the vector bundle $\underline{E}_0=\OOP\oplus\OP{-3}$
equipped with the decomposable parabolic structure $\underline{\p}$ defined by the fibres of the line subbundle $\OP{-3}\hookrightarrow \underline{E}_0$
(see Section \ref{sec:TrivialBundleDescent}); then $\OOP$ is destabilizing. 

Once this has been done,
for each family occuring in  $\P^3_{\b}$, we already know from Section \ref{SecSpecialParBundle} where they are sent on $\P^3_{\mathrm{NR}}$,
we known the corresponding explicit equations from Section \ref{SecComputeNR}
and we can deduce equations on $\P^3_{\b}$ by using explicit formula from Corollary \ref{CorCoordBNR}.
\end{proof}

\begin{rem} The proposition above is stated only for generic bundles of each type. Indeed, only an open set of the family $\underline{\delta}$ of unipotent bundles occurs in $\Bun^{ss}_{\Mu}(X/\iota)=\P^3_{\b}$,
namely the complement of Weierstrass points (which are replaced by Gunning bundles $Q_i$). One can easily check that this is the only obstruction for the proposition to hold for all bundles of the respective families. 
\end{rem}

The preimage of the  Kummer surface $\mathrm{Kum}\left(X\right)$ in the chart $\P^3_{\boldsymbol{b}}$ is nothing but
the dual Weddle surface $\mathrm{Wed}\left(X\right)$, another birational model of $\mathrm{Kum}\left(X\right)$:
it is also a quartic surface, but with only $6$ nodes (see \cite{Hudson,GonzalezDorrego}). Precisely, the $16$ singular points
of $\mathrm{Kum}\left(X\right)$ are blown-up and replaced by the lines $\underline\Delta_{i,j}$; the $6$ Gunning planes $\Pi_i$
are contracted onto the points $Q_i$, giving rise to new conic points. In particular, all $16$ quasi-unipotent families
$\underline{\Delta}$ and $\underline{\Delta}_{i,j}$ are contained in $\mathrm{Wed}$.

Actually, the map $\phi:\P^3_{\b}\dashrightarrow\P^3_{\mathrm{NR}}$ is defined by the linear sytem of quadrics passing through the $6$ points $Q_i$;
indeed, for a general plane $\Pi\in\P^3_{\mathrm{NR}}$, $\phi^*\Pi$ must intersect each contracted $\Pi_i$. 
We thus recover the quadric system in \cite{Dolgachev}, \S 4.6.
Those $\Pi$ tangent to $\mathrm{Kum}\left(X\right)$ have a  singular lift $\underline\Pi$; when $\Pi$ runs over the tangent planes of 
$\mathrm{Kum}\left(X\right)$, the singular point of $\underline\Pi$ runs over the Weddle surface.

\begin{rem}The complement of the (dual) Weddle surface covers the open set of stable bundles
in $\P^3_{\mathrm{NR}}$
$$\P^3_{\b}\setminus\mathrm{Wed}\left(X\right)\ \stackrel{\phi}{\twoheadrightarrow}\P^3_{\mathrm{NR}}\ \setminus\mathrm{Kum}\left(X\right).$$
However, this is not a covering since over odd Gunning planes, only $\Pi_i'$ occurs in $\P^3_{\b}$.
\end{rem}

\subsection{Moving weights and wall-crossing phenomena}\label{sec:MovingWeights}

For a generic weight $\Mu$, semi-stable bundles are automatically stable; in this case, 
the moduli space $\Bun^{ss}_{\Mu}(X/\iota)$ is projective, smooth
and a geometric quotient. The special weights $\Mu$, for which some bundles are strictly semi-stable, form a finite collection of affine
planes in the weight-space $[0,1]^6\ni\Mu$ called {\it walls}. They cut-out $[0,1]^6$ into finitely many {\it chambers}: the connected
components of the complement of walls. Along walls, the moduli space is no more a geometric quotient, but a categorical quotient,
identifying some semi-stable bundles together to get a (Hausdorff) projective manifold, which might be singular in this case;
outside of the strictly semi-stable locus, $\Bun^{ss}_{\Mu}(X/\iota)$ is still smooth and a geometric quotient.
The moduli space $\Bun^{ss}_{\Mu}(X/\iota)$ is constant in a given chamber; if not empty, it has the right dimension $3$
and contains as an open set the geometric quotient of those bundles $(\underline{E},\underline{\p})$ with $\underline{E}=\OP{-1}\oplus\OP{-2}$
and parabolics $\underline{\p}$ in general position:
\begin{itemize}
\item no parabolic in $\OP{-1}\hookrightarrow E$,
\item no $3$ parabolics in the same $\OP{-2}\hookrightarrow E$,
\item no $5$ parabolics in the same $\OP{-3}\hookrightarrow E$.
\end{itemize}
Between (non empty!) moduli spaces in any two chambers, we get a natural birational map
$$\mathrm{can}:\Bun^{ss}_{\Mu}(X/\iota)\stackrel{\sim}{\dashrightarrow}\Bun^{ss}_{\Mu'}(X/\iota)$$
arising from the identification of the generic bundles occuring in both of them.
The indeterminacy locus comes from those special parabolic bundles that are stable for $\Mu$ but not for $\Mu'$
and \emph{vice-versa}; this configuration occurs each time we cross a wall.
The moduli space $\BUN^{ind}(X/\iota)$ of indecomposable bundles can be covered by a finite collection 
of such moduli spaces, by choosing one $\Mu$ in each non empty chamber; therefore, $\BUN^{ind}(X/\iota)$
can be constructed by patching together these moduli spaces by means of canonical maps along the open set
of common bundles. This gives $\BUN^{ind}(X/\iota)$ a structure of smooth non separated scheme.
However, in our case, we have also decomposable flat bundles that are not taken into account in this picture. 
For instance the preimage $(\underline{E}_0,\underline{\p}^0):=\phi^{-1}(E_0)$ of the trivial bundle on $X$ (see Section \ref{sec:TrivialBundleDescent}), 
being decomposable, can only arise as a singular point in semi-stable projective charts $\Bun^{ss}_{\Mu}(X/\iota)$.
Indeed, if the bundle $\underline{E}_0=\OOP\oplus\OP{-3}$ equipped with the decomposable parabolic
structure $\underline{\p}^0$ defined by the fibres of $\OP{-3}\hookrightarrow \underline{E}_0$ is semi-stable for some choice of weights $\Mu$,
then all other parabolic structures $\underline{\p}$ on $\underline{E}_0$ with no parabolics in the total space of $\OOP\subset\underline{E}_0$
are also semi-stable and arbitrarily close to $\underline{\p}^0$; they are represented by the same point in the Hausdorff quotient
$\Bun^{ss}_{\Mu}(X/\iota)$. One can check that this point is necessarily singular.

\subsubsection{Wall-crossing between our two main charts}\label{GeomBRST}
If we want to understand the geometry of the birational map
 $\P^3_B \dashrightarrow \P^1_R\times\P^1_S\times\P^1_T$ explicitly given in proposition \ref{PropBertramToRST} we have to consider a path in $\BUN^{ind}(X/\iota)$ linking the corresponding chambers and the wall-crossing phenomena along this path. Since $\P^3_B$ corresponds to the weight $\Mu=\left(\frac{1}{5},\frac{1}{5},\frac{1}{5},\frac{1}{5},\frac{1}{5},\frac{1}{5}\right)$ and $\P^1_R\times\P^1_S\times\P^1_T$ corresponds to the weight  $\Mu=\left(\frac{1}{2},\frac{1}{2},0,0,0,\frac{1}{2}\right)$, a possibility to do so consists in considering the walls between chambers of the form $\Mu=\left(\mu,\mu, \lambda, \lambda, \lambda,\mu \right)$ with $\lambda, \mu \in [0,1]$. A projective parabolic bundle belongs to such a wall if it possesses a section with self-intersection number $k\in 2\mathbb{Z}+1$ containing $m$ parabolics over $\{0,1, \infty\}$ and $\ell$ parabolics over $\{r,s,t\}$ such that $$0=k+(3-2m)\mu+(3-2\ell)\lambda$$ for some $\lambda, \mu\in [0,1]$.   Table \ref{configswalls} lists all possible configurations. They are visualized in Figure \ref{walls}

 \begin{table}\label{configswalls}
 \begin{center}
 \begin{tabular}{| l | l | c |} \hline
 Possible configuration & $(k,m,\ell)$& Walls\\
 \hline
 $\begin{array}{rcl} \lambda  &=& -\mu  +\frac{5}{3}\end{array}$ & $(-5,0,0), \quad (5,3,3)$& $ $ \\
 \hline
 $\begin{array}{rcl} \lambda  &=& -\mu  +1\end{array}$ & $(-3,0,0), \quad (-1,1,1),\quad (1,2,2),\quad (3,3,3)$& $ $ \\
 \hline
 {$\begin{array}{rcl} \lambda  &=& -\mu  +\frac{1}{3}\end{array}$ }& $(-1,0,0),\quad (1,3,3)$& $ $ \\
 \hline
{\color{red}  $\begin{array}{rcl} \lambda  &=& -3\mu  +1\end{array}$} & ${\color{red}(-1,0,1)},\quad (1,3,2)$&{\color{red} \textcircled{3}} \\
 \hline
   $\begin{array}{rcl} \lambda  &=& -3\mu  +3\end{array}$ & $(-3,0,1),\quad (3,3,2)$& $ $ \\
 \hline
   $\begin{array}{rcl} \lambda  &=& -\frac{1}{3}\mu  +1\end{array}$ & $(-3,1,0),\quad (3,2,3)$& $ $ \\
 \hline
   $\begin{array}{rcl} \lambda  &=& -\frac{1}{3}\mu  +\frac{1}{3}\end{array}$ & $(-1,1,0),\quad (1,2,3)$& $ $ \\
 \hline
 {\color{red}  $\begin{array}{rcl} \lambda  &=& 3\mu  -1\end{array}$ }& ${\color{red}(-1,0,2)},\quad (1,3,1)$ & {\color{red} \textcircled{2}} \\
 \hline
   $\begin{array}{rcl} \lambda  &=& \frac{1}{3}\mu  +\frac{1}{3}\end{array}$ & $(-1,2,0),\quad (1,1,3)$ & $ $\\
 \hline
   $\begin{array}{rcl} \lambda  &=& \mu  +\frac{1}{3}\end{array}$ & $(-1,3,0), \quad (1,0,3)$& $ $ \\
 \hline
 {\color{red} $\begin{array}{rcl} \lambda  &=& \mu  -\frac{1}{3}\end{array}$} & ${\color{red}(-1,0,3)}, \quad (1,3,0) $& {\color{red} \textcircled{1}}\\
 \hline
\end{tabular}
 \end{center}
 \caption{Possible wall-configurations for weights of the form $\Mu=\left(\mu,\mu, \lambda, \lambda, \lambda,\mu \right)$.}
 \end{table}

 \begin{figure}\label{walls}\centerline{\resizebox{115mm}{!}{\input{walls.pstex_t}}}
\caption{Chambers of moduli spaces for the weights $\Mu=\left(\mu,\mu, \lambda, \lambda, \lambda,\mu \right)$.}      
\end{figure}

 Following the pink line in Figure \ref{walls} means studying wall-crassing phenomena for moduli spaces $\Bun^{ss}_{\Mu}(X/\iota)$ with democratic weights $\Mu=\left(\mu,\mu, \mu, \mu, \mu,\mu \right)$. As we see, walls occur for $\mu \in \left\{ \frac{1}{6}, \frac{1}{4}, \frac{1}{2}, \frac{3}{4}, \frac{5}{6}\right\}$. They will considered in Section \ref{SectDemWeights}.  
 
 First, we want to consider the crossing of the walls { \color{red} \textcircled{1},\textcircled{2}} and  { \color{red}\textcircled{3}} in order to describe the birational map   $\P^3_B \dashrightarrow \P^1_R\times\P^1_S\times\P^1_T$.
 The configuration $(k,m,\ell)=(1,3,0)$ is not stable in $\P^1_R\times\P^1_S\times\P^1_T$, but $(k,m,\ell)=(-1,0,3)$ is. It corresponds to the even Gunning bundle $E_\vartheta$ with $\vartheta = [w_r]+[w_s]-[w_t]$. The point in the moduli space $\P^1_R\times\P^1_S\times\P^1_T$ corresponding to this bundle is blown up when crossing the wall \textcircled{1} and replaced by the corresponding Gunning plane:  $(k,m,\ell)=(1,3,0)$. Passing on to wall \textcircled{2}, the three lines $(k,m,\ell)=(-1,0,2)$ in the moduli space corresponding to the unipotent bundles tensored by $\OOX\left([w_r]-[w_s]\right), \OOX\left([w_s]-[w_t]\right)$ and $\OOX\left([w_s]-[w_t]\right)$ respectively are no longer stable. Here a flop phenomenon occurs: these three lines are blown up and the resulting planes are contacted to three lines $(k,m,\ell)=(1,3,1)$ corresponding to the families of the same types of unipotent bundles. Passing on to wall \textcircled{3}, the three planes  $(k,m,\ell)=(-1,0,1)$ corresponding to the odd Gunning planes with characteristic $\vartheta \in \left\{[w_r],[w_s],[w_t]\right\}$ are contracted and replaced by three points corresponding to the configurations  $(k,m,\ell)=(1,3,2)$: the Gunning bundles with characteristic $\vartheta$. 
 
\subsubsection{Democratic weights}\label{SectDemWeights} Let us consider in this section the family of moduli spaces $\Bun^{ss}_{\Mu}(X/\iota)$
with weights $\Mu=\left(\mu,\mu,\mu,\mu,\mu,\mu\right)$, for $\mu\in[0,1]$. One can easily check which family of special
bundle is semi-stable, depending on the choice of $\mu$; this is summarized in Table \ref{TableDemWeights}.
 {\renewcommand{\arraystretch}{1}
\begin{table}[htdp]\begin{center}
\begin{tabular}{| c | c c c c c c c c c c c c c |}
\hline
$\mu$&$0$&&$\frac{1}{6}$&&$\frac{1}{4}$&&$\frac{1}{2}$&&$\frac{3}{4}$&&$\frac{5}{6}$&&$1$\\
\hline
unipotent bundles&&&\vline &&$\Delta$&&\vline&&$\Delta'$
&&\vline&&\\
\cline{2-14}
(and twists)&&&\vline &&  $\Delta_{ij}$    &&\vline&&$\Delta_{ij}'$&&\vline&&\\
\hline
odd Gunning&&&\vline &$Q_i$&\vline&&&$\Pi_i$&&&\vline&&\\
\cline{2-14}
bundles and planes&&&\vline &&&$\Pi_i'$&&&\vline&$Q_i'$&\vline&&\\
\hline
even Gunning planes&&&\vline &&&&$\Pi_{ijk}$&&&&\vline&&\\
\hline
\end{tabular}
\end{center}
\caption{Moving weights.}\label{TableDemWeights}
\label{figMovMu}
\end{table}}

\subsubsection*{For $\mu \in [0, \frac{1}{6}[$}
The moduli space $\Bun^{ss}_{\Mu}(X/\iota)$ is empty since $\OP{-1}$
is destabilizing the generic parabolic bundle (even if it carries no parabolic).

\subsubsection*{For $\mu=\frac{1}{6}$} The moduli space $\Bun^{ss}_{\Mu}(X/\iota)$ reduces to a single point. 
Indeed, it also contains the (non flat) 
decomposable bundle $\underline{E}=\OP{-1}\oplus\OP{-2}$ with all parabolics $\underline{\p}$ lying in the total space of $\OP{-2}$.
But the generic parabolic bundle is arbitrarily close to this decomposable bundle so that they have to be 
identified in the Hausdorff quotient $\Bun^{ss}_{\Mu}(X/\iota)$.

\subsubsection*{For $\mu \in ]\frac{1}{6},\frac{1}{4}[$} Here, we recover our chart $\P^3_{\b}:= \Bun^{ss}_{]\frac{1}{6},\frac{1}{4}[}(X/\iota)$ 
with special families $\Delta$, $\Delta_{ij}$, $Q_i$, $\Pi_i'$ and $\Pi_{ijk}$. The natural map $\phi:\Bun^{ss}_{]\frac{1}{6},\frac{1}{4}[}(X/\iota)\dashrightarrow\P^3_{\mathrm{NR}}$ has indeterminacy points at all $6$ points $Q_i$.

\subsubsection*{For $\mu=\frac{1}{4}$} Now, odd Gunning planes $\Pi_i$ become semi-stable, but arbitrarily close the the corresponding point $Q_i$,
so that they are identified in the quotient $\Bun^{ss}_{\Mu}(X/\iota)$. Therefore, the moduli space is still the same $\P^3_{\b}$
but no more a geometric quotient.

\subsubsection*{For $\mu \in ]\frac{1}{4},\frac{1}{2}[$} Odd Gunning bundles $Q_i$ are no more semi-stable and are replaced by the corresponding
Gunning planes $\Pi_i$. The natural map 
$$\mathrm{can}:\Bun^{ss}_{]\frac{1}{4},\frac{1}{2}[}(X/\iota)\to \Bun^{ss}_{]\frac{1}{6},\frac{1}{4}[}(X/\iota)$$
is the blow-up of $\P^3_{\b}$ at all $6$ points
$Q_i$, and the exceptional divisors represent the corresponding planes $\Pi_i$. The natural map $\phi:\Bun^{ss}_{]\frac{1}{4},\frac{1}{2}[}(X/\iota)\to\P^3_{\mathrm{NR}}$ is a morphism.

\subsubsection*{For $\mu=\frac{1}{2}$} the trivial bundle and its $15$ twists become semi-stable (and just for this special value of $\mu$).
In particular, unipotent families are identified with these bundles in the moduli space, which has  the effect to contract 
the strict transforms of lines $\Delta_{ij}$ and the rational curve $\Delta$ to $16$ singular points of $\Bun^{ss}_{\Mu}(X/\iota)$.
This moduli space is exactly the double cover of $\P^3_{\mathrm{NR}}$ ramified along $\Kum(X)$, therefore singular with conic points
over each singular point of $\Kum(X)$. The natural map 
$$\mathrm{can}:\Bun^{ss}_{]\frac{1}{4},\frac{1}{2}[}(X/\iota)\to \Bun^{ss}_{\frac{1}{2}}(X/\iota)$$
is a minimal resolution. 

\subsubsection*{For $\mu \in ]\frac{1}{2},\frac{3}{4}[$} The families $\Delta$ and $\Delta_{ij}$ are no more semi-stable, 
and are replaced by the families $\Delta'$ and $\Delta_{ij}'$.
But mind that the canonical map
$$\mathrm{can}:\Bun^{ss}_{]\frac{1}{4},\frac{1}{2}[}(X/\iota)\dashrightarrow \Bun^{ss}_{]\frac{1}{2},\frac{3}{4}[}(X/\iota)$$
is not biregular: there is a flop phenomenon around each of the $16$ above rational curves. Precisely, after blowing-up
the $16$ curves, we exactly get the resolution $\widehat{\Bun^{ss}_{\frac{1}{2}}(X/\iota)}$ of the previous moduli space
by blowing-up the $16$ conic points. Then, exceptional divisors are $\simeq\P^1\times\P^1$ and we can contract them
back to rational curves by using the other ruling; this is the way the map $\mathrm{can}$ is constructed here.
In particular, we get a second minimal resolution of $\Bun^{ss}_{\frac{1}{2}}(X/\iota)$.

\subsubsection*{For $\mu\in [\frac{3}{4},\frac{5}{6}[$} Here, we finally contract the strict transforms of $\Pi_i'$ to the points $Q_i$.

\subsection{Galois and Geiser involutions}\label{sec:GaloisGeiser}

The Galois involution of the ramified cover $\phi:\BUN(X/\iota)\stackrel{2:1}{\longrightarrow}\BUN(X)$
$$\Upsilon:=\OP{-3}\otimes\mathrm{elm}_{\underline{W}}^+:\BUN(X/\iota)\stackrel{\sim}{\longrightarrow}\BUN(X/\iota)$$
 induces isomorphisms between moduli spaces 
$$\Upsilon:\Bun^{ss}_{\Mu}(X/\iota)\stackrel{\sim}{\longrightarrow}\Bun^{ss}_{\Mu'}(X/\iota)$$
where $\Mu'$ is defined by $\mu_i'=\frac{1}{2}-\mu_i$ for all $i$. In particular, it underlines the symmetry of our special family of moduli spaces
around $\mu=\frac{1}{2}$ (see Section \ref{sec:MovingWeights}):
the Galois involution induces a biregular involution of $\Bun^{ss}_{\frac{1}{2}}(X/\iota)$, as well as isomorphisms
$$\Bun^{ss}_{]\frac{1}{4},\frac{1}{2}[}(X/\iota)\stackrel{\sim}{\longleftrightarrow}\Bun^{ss}_{]\frac{1}{2},\frac{3}{4}[}(X/\iota)\ \ \ \text{and}\ \ \ 
\Bun^{ss}_{]\frac{1}{6},\frac{1}{4}]}(X/\iota)\stackrel{\sim}{\longleftrightarrow}\Bun^{ss}_{[\frac{3}{4},\frac{5}{6}[}(X/\iota).$$
Considering now the composition
$$\Bun^{ss}_{]\frac{1}{6},\frac{1}{4}]}(X/\iota)\stackrel{\mathrm{can}}{\dashrightarrow}\Bun^{ss}_{[\frac{3}{4},\frac{5}{6}[}(X/\iota)\stackrel{\Upsilon}{\longrightarrow} \Bun^{ss}_{]\frac{1}{6},\frac{1}{2}[}(X/\iota),$$ 
we get the (birational) Galois involution of the map $\phi:\P^3_{\b}\dashrightarrow\P^3_{\mathrm{NR}}$ described in Corollary \ref{CorCoordBNR}.
This is known as the Geiser involution (see \cite{Dolgachev}, \S 4.6); it is a degree $7$ birational map.
The combination of all wall-crossing phenomena described in Section \ref{sec:MovingWeights},
when $\mu$ is varying from $\frac{1}{6}$ to $\frac{5}{6}$, provides a complete decomposition of this map (see Table \ref{TableGeiser}):
\begin{itemize}
\item first blow-up $6$ points (the $Q_i$ along the embedding $X/\iota\stackrel{\sim}{\longrightarrow}\underline{\Delta}\subset\P^3_{\b}$),
\item flop $16$ rational curves (the strict transforms of the twisted cubic $\underline{\Delta}$ and all lines $\underline{\Delta}_{ij}$),
\item contract $6$ planes (namely strict transforms of $\underline{\Pi}_i'$ onto $Q_i$),
\item then compose by the unique isomorphism sending $Q_i'\to Q_i$. 
\end{itemize}

\begin{table}[h]
{\large $$ \xymatrix{
 & \widehat{\Bun^{ss}_{\frac{1}{2}}(X/\iota)} \ar[dl]_{\text{$\underline{\Delta},\underline{\Delta}_{ij}$ blow-up}}^{\text{($16$ curves)}} \ar@{.>}[dd] \ar[dr]^{\text{$\underline{\Delta}',\underline{\Delta}_{ij}'$ blow-up}} & \\
 \Bun^{ss}_{]\frac{1}{4},\frac{1}{2}[}(X/\iota) \ar[dd]_{\text{$Q_i$ blow-up}}^{\text{($6$ points)}}  \ar@{.>}[dr]   &   &   \Bun^{ss}_{]\frac{1}{2},\frac{3}{4}[}(X/\iota) \ar@{.>}[dl]  \ar[dd]^{\text{$Q_i'$ blow-up}}    \\
 & \Bun^{ss}_{\frac{1}{2}}(X/\iota)  \\
\Bun^{ss}_{]\frac{1}{6},\frac{1}{4}]}(X/\iota)  \ar@/_2pc/@{<->}[rr]_{\Upsilon}^{\sim}
  &  &   \Bun^{ss}_{[\frac{3}{4},\frac{5}{6}[}(X/\iota)
  }$$}
  \caption{Geometry of the Geiser involution.}\label{TableGeiser}
  \end{table}

\begin{rem}Even Gunning bundles $Q_{ijk}$ are semi-stable if, and only if, $\mu=1$. 
This is why they do not appear in our family of moduli spaces. However, for some other choices
of weights $\Mu$, they appear as stable points, and therefore smooth points of some projective charts.
%
\end{rem}

\subsection{Summary: the moduli stack $\BUN(X)$ } \label{SecSummary}
Recall from the introduction that we have defined $\BUN(X/\iota)$ as the moduli space of parabolic rank 2 vector bundles with determinant $\OOX(-3)$ over $\mathbb{P}^1$ that can be endowed with a logarithmic connection (with poles over the Weierstrass points and prescribed residues). Denote by $\BUN(X)$ the set of rank 2 vector bundles with trivial determinant over $X$ that can be endowed with an  tracefree holomorphic connection and $\BUN^*(X)$ as the complement of the affine bundles in $\BUN(X)$. The map $\phi : \mathfrak{Bun}\left(X/\iota\right)\to\mathfrak{Bun}^*\left(X\right)$ defined by the "hyperelliptic lift" $\mathrm{elm}^+_{W}\circ \pi^*$ (see Section \ref{SecMainConstruction}) is surjective. Recall further that $\BUN^{ind}(X/\iota)$ denotes the set of indecomposable parabolic bundles and that its image under $\phi$ is precisely the complement in $\BUN^*(X)$ of the trivial bundle and its 15 twists (see Proposition \ref{prop:ParFlatCriterium} and Section \ref{sec:TrivialBundleDescent}).  

As mentioned before, \cite{LoraySaito} provides  methods to choose a finite set of smooth projective charts covering the moduli space $\BUN^{ind}(X/\iota)$ of indecomposable parabolic bundles. In the present paper however, we chose to present charts with particular geometrical meaning and natural explicit coordinates. Similarly to the construction in \ref{SecTyurinPar}, we can express explicitly the universal bundle in (affine parts of) each of these charts, giving $\BUN(X)$ the structure of a moduli stack.
Table \ref{Formulae} references the explicit maps between the charts given in this paper. 

\begin{table}[htdp]
{\large $$ \begin{xy}
\xymatrix{
{\left(X^2 \times\P^1_{{\lambda}}\right)_{ \diagup_{\langle \sigma_\iota , \sigma_{12} , \sigma_{iz}\rangle }}}\ar@{<.>}[rrr]^{\textrm{Prop. \ref{PropTyurinToBertram}}}_{1:1}   \ar@{.>}[dd]^{\textrm{Eq. (\ref{TyurinQuotient})}}_{2:1} 
&&&
\P^3_{\b}=\P^3_{B}\ar@{<.>}[rrr]^{\textrm{Prop. \ref{PropBertramToRST}}}_{1:1}\ar@{.>}[dd]^{\textrm{Cor. \ref{CorCoordBNR}}}_{2:1}&&&\mathbb{P}^1_R\times\mathbb{P}^1_S\times\mathbb{P}^1_T \ar@/^/@{.>}[ddlll]^{\textrm{Prop. \ref{PropFormulesRSTtoNR}}}_{2:1}\\\\
\P^2_D\times\P^1_{\boldsymbol{\lambda}}\ar@{<.>}[rrr]^{\textrm{Prop. \ref{prop:TyurinToNR}}}_{1:1}&&&\P^3_{\mathrm{NR}}
}
\end{xy}$$}
\caption{Collection of explicit formulae.}
\label{Formulae}
\end{table}

  As we can easily convince ourselves with the help of our dictionary in Section \ref{SecSpecialParBundle}, Table \ref{compare} lists which elements of $\BUN (X)$ occur in the image under $\phi$ of the respective charts. Here we use a checkmark sign ($\checkmark$) if every bundle of a given type can be found in the image of this chart and no checkmark sign if no bundle of the given type can be found in the image of this chart.

Consider the two democratic charts $\P^3_{\b}=\Bun^{ss}_{\frac{1}{5}}(X/\iota)$ and $\Bun^{ss}_{\frac{1}{3}}(X/\iota)$. Their birational relation has been thoroughly described in Sections \ref{sec:MovingWeights} and \ref{sec:GaloisGeiser}. According to Table \ref{compare} , the union of the images under $\phi$ of these two charts 
 covers the whole space $\BUN^*(X)$  minus the ten even Gunning bundles and the decomposable bundles. Moreover, the Galois-involution $\Upsilon:=\OP{-3}\otimes\mathrm{elm}_{\underline{W}}^+$ (see Section \ref{sec:GaloisGeiser})  sends $\P^3_{\b}=\Bun^{ss}_{\frac{1}{5}}(X/\iota)$ to $\Upsilon(\P^3_{\b})=\Bun^{ss}_{\frac{4}{5}}(X/\iota)$ and $\Bun^{ss}_{\frac{1}{3}}(X/\iota)$ to $\Upsilon(\P^3_{\b})=\Bun^{ss}_{\frac{2}{3}}(X/\iota)$. Hence the two charts $\P^3_{\b}, \Bun^{ss}_{\frac{1}{3}}(X/\iota)$ and their images under $\Upsilon$ are sufficient to cover $\BUN^{ind}(X/\iota)$ minus the (pre-images of) even Gunning bundles.

Similar to the construction of our chart $\mathbb{P}^1_R\times\mathbb{P}^1_S\times\mathbb{P}^1_T$, we can construct charts $\P^1_i\times\P^1_j\times\P^1_k$ for any choice of three distinct elements $i,j,k\in \{0,1,r,s,t,\infty\}$ by setting $\mu_i=\mu_j=\mu_k=0$ and all the other weights equal to $1 \over 2$. The geometry and explicit formulae of the transition maps between these charts are obvious. Moreover, the explicit formulae in Proposition \ref{PropFormulesRSTtoNR} can easily be generalized to any of these charts. 
According to Table \ref{compare}, ten of these charts (one for each partition $\{i,j,k\}\cup\{k,l,m\}=\{0,1,r,s,t,\infty\}$) are sufficient if we want their image to cover the whole space $\BUN^*(X)$ minus the decomposable bundles. We can of course cover the whole space $\BUN^*(X)$ by adding the singular chart  $\Bun^{ss}_{\frac{1}{2}}(X/\iota)$ to the 10 charts of type $\P^1_i\times\P^1_j\times\P^1_k$. Transition maps to this chart can be obtained from Proposition \ref{PropBertramToRST}. 
Note however that even the union of all 20  charts of type $\mathbb{P}^1_i\times\mathbb{P}^1_j\times\mathbb{P}^1_k$ is however not sufficient to cover $\BUN(X/\iota)$ minus the decomposable bundles (we don't have the images under the Galois involution of the Gunning bundles for example).

 \begin{table}
  \centering
        \rotatebox{90}{
                \begin{minipage}{\textheight}
{\begin{tabular}{|l|l|c|c|c|c|c|}
\hline
\multicolumn{2}{|c|}{bundle type}
& $\P^3_{\b}$ &  $\Bun^{ss}_{\frac{1}{3}}(X_{/\iota})$  &  $\Bun^{ss}_{\frac{1}{2}}(X_{/\iota})$ & $\P^1_R\times\P^1_S\times\P^1_T$&$\bigcup_{j=1}^4 X^{(2)}\diagup_{\{\iota,\iota\}}\times \mathbb{P}^1_{{\tilde{\lambda}_j}}$
\\
\hline
{stable}&off the odd Gunning planes &  $\checkmark$& $\checkmark$&$\checkmark$ & $\checkmark$& $\checkmark$
\\
\cline{2-7}
&on the odd Gunning planes &  $\checkmark$& $\checkmark$&$\checkmark$ & $\checkmark$& $ $
\\
\hline
unipotent & generic & $\checkmark$&$\checkmark$&$\checkmark$ &$\checkmark$&$\checkmark $\\\cline{2-7}
&\hspace{-.2cm}$\begin{array}{l} \textrm{special, corresponding } \\ \textrm{to } {[w_r],[w_s]} \textrm{ or } {[w_t]}\end{array}$& & $\checkmark$ &$\checkmark$& $\checkmark$&$\checkmark$\\\cline{2-7}
&\hspace{-.2cm}$\begin{array}{l} \textrm{special, corresponding } \\ \textrm{to } {[w_0],[w_1]} \textrm{ or } {[w_\infty]}\end{array}$& &$\checkmark$& $\checkmark$ &$ $&$\checkmark $\\\cline{2-7}
&{twisted }  &$\checkmark$ & $\checkmark$ &$\checkmark$&$\checkmark$& $\checkmark$\\\hline
 \multicolumn{2}{|l|}{affine}&&&&& $\checkmark $
 \\\hline 
\hspace{-.2cm}$\begin{array}{l}\textrm{semi-stable} \\\textrm{decomposable} \end{array}$ & \hspace{-.2cm}$\begin{array}{l} L_0=\OOX([P]-[Q]) \textrm{ with } \\ P,Q \in X\setminus W \textrm{ and } P\neq Q\end{array}$
& $\checkmark$& $\checkmark$ &$\checkmark$& $\checkmark$&$ $
\\\cline{2-7}
\hspace{-.2cm}$\begin{array}{l} L_0\oplus L_0^{-1}\\ \textrm{ } \end{array}$ &\hspace{-.2cm}$\begin{array}{l} L_0=\OOX([P]-[w_i]) \textrm{ with } \\ P \in X\setminus W \textrm{ and } i \in \{0,1,\infty\} \end{array}$&  &$\checkmark$& $\checkmark$ && $ $\\\cline{2-7}
&\hspace{-.2cm}$\begin{array}{l} L_0=\OOX([P]-[w_i]) \textrm{ with } \\  P \in X\setminus W \textrm{ and } i \in \{r,s,t\} \end{array}$& &$\checkmark$& $\checkmark$ & $\checkmark$&$ $
\\\cline{2-7}
& $L_0^{\otimes 2}=\OOX$& &&$\checkmark$& $ $  &$ $ 
\\\hline
 \hspace{-.2cm}$\begin{array}{l}\textrm{Gunning}\\\textrm{bundle}\end{array}$  & \hspace{-.2cm}$\begin{array}{rl}\textrm{even, } \vartheta&=[w_r]+[w_s]-[w_t]\\ &\simeq [w_0]+[w_1]-[w_\infty] \end{array}$&&&&$\checkmark$&$ $\\\cline{2-7}
&even, other $\vartheta$ &&&&& $ $\\\cline{2-7}
& odd, $\vartheta\in \{[w_1],[w_0],[w_\infty]\}$  &$\checkmark$ &&&$\checkmark$&$\checkmark$\\\cline{2-7}
& odd, other $\vartheta$  &$\checkmark$ &&&$ $&$\checkmark$
\\\hline
\end{tabular}}
\end{minipage}}
\caption{Types of bundles that can be found in the image under $\phi$ of the charts  $\P^3_{\b}$, $\Bun^{ss}_{\frac{1}{2}}(X/\iota)$, $\P^1_R\times\P^1_S\times\P^1_T$  and $X^{(2)}\diagup_{\{\iota,\iota\}}\times \mathbb{P}^1_{{\tilde{\lambda}_j}}$
 respectively.}\label{compare}
\end{table}
If we wish to cover $\BUN(X)$ entirely (and not only $\BUN^*(X)$), we have to add non-hyperelliptic charts, for example \emph{via} the construction in Section \ref{SecTyurinPar}. Indeed, there we have a parabolic structure over some divisor in $|2\KX]$  defined by four points in $\P^1$. We can normalize three of them to $0,1$ and $\infty$ respectively, the fourth then is given by some $\tilde{\lambda}_j \in \mathbb{P}^1$. We deduce four natural charts $$ X^{(2)}\diagup_{\{\iota,\iota\}}\times \mathbb{P}^1_{{\tilde{\lambda}_j}}\stackrel{2:1}{\longrightarrow} \mathfrak{Bun}(X) \quad \textrm{ where } \quad j\in \{1,2,3,4\}.$$
Here $\{\iota,\iota\}$ denotes the diagonal action $\{P,Q\} \mapsto \{\iota(P),\iota (Q)\}$ on the symmetric product $X^{(2)}$, which leaves the cross-ratio of the corresponding parabolics invariant.

\section{The moduli stack $\mathfrak{Higgs}(X)$ and the Hitchin fibration}\label{SecHiggsCon}

A \emph{Higgs bundle} on a Riemann surface $X$ is a vector bundle $E\to X$ endowed with a \emph{Higgs field, i.e.} an $\OO_X$-linear morphism
$$\Theta : E \to E \otimes \Omega^1_X(D),$$
where $D$ is an effective divisor. If $D$ is reduced, then $\Theta $ is called \emph{logarithmic} and   for any $x\in D$, the residual morphism
$\mathrm{Res}_x(\Theta)\in \mathrm{End}(E_x)$ is well-defined. 
As usual, we will only consider the case where  $E$ is a rank 2 vector bundle with trivial determinant bundle and $\Theta$ is trace-free.
By definition, a holomorphic  ($D=\emptyset$) and trace-free Higgs-field on $E$ is an element of $\mathrm{H}^0(X,\mathfrak{sl} (E)\otimes\Omega^1_X)$, which, by Serre duality, is isomorphic to $\mathrm{H}^1(X,\mathfrak{sl} (E))^\vee$. On the other hand, stable bundles
are \emph{simple}: they possess no non-scalar automorphism.  For such bundles
$E$, the vector space $\mathrm{H}^1(X,\mathfrak{sl} (E))$ is precisely the tangent space in $E$ of our moduli space $\mathfrak{Bun}(X)$ of flat vector bundles over $X$. Therefore, \emph{in restriction to the open set of stable bundles} the moduli space $\mathfrak{Higgs}(X)$ 
of Higgs bundles identifies in a natural way to 
$$\mathfrak{Higgs}(X):=\mathrm{T}^*\mathfrak{Bun}(X).$$ 
Just as naturally, we can define $$\mathfrak{Higgs}(X/\iota):=\mathrm{T}^*\mathfrak{Bun}(X/\iota),$$ but we need to clarify its meaning. 
Let $(\underline E, \underline \p)$ be a parabolic bundle in  $\mathfrak{Bun}(X/\iota)$. Then   $\mathrm{T}^*_{(\underline E, \underline \p)}\mathfrak{Bun}(X/\iota)=\mathrm{H}^0(X,\mathfrak{sl} (\underline E, \underline \p)\otimes\Omega^1_{\P^1})$, where $\mathfrak{sl} (\underline E, \underline \p)$ denotes the space of trace-free endomorphisms of $\underline E$ leaving $\underline \p$ invariant. Now consider the image of the natural embedding
$$\mathrm{H}^0(\P^1,\mathfrak{sl} (\underline E, \underline \p)\otimes\Omega^1_{\P^1}) \hookrightarrow \mathrm{H}^0(\P^1,\mathfrak{sl} (\underline E)\otimes\Omega^1_{\P^1}(\underline W )).$$ \emph{Via} the (meromorphic) gauge transformation $$\OO_{\P^1}(-3)\otimes\mathrm{elm}_{\underline{W}}^+\in\mathrm{H}^0(\P^1,\mathrm{SL} ( \underline E\otimes\OO_{\P^1}(\underline W ),\underline \p))$$ it corresponds precisely to those logarithmic Higgs fields $ \underline \Theta$  in $\mathrm{H}^0(\P^1,\mathfrak{sl} (\underline E)\otimes\Omega^1_{\P^1}(\underline W ))$ that have  \emph{apparent singularities} in $\p$ over $\underline W$:  the residual matrices are congruent to $ \left(\begin{smallmatrix}0 &1\\0&0 \end{smallmatrix}\right)$ and
 $\underline \p$ corresponds to their  eigenvectors. We shall denote this set of apparent logarithmic Higgs fields on $\underline E$  by 
$$\mathrm{H}^0(\P^1,\mathfrak{sl} (\underline E)\otimes\Omega^1_{\P^1}(\underline W ))^{\mathrm{app}_{\underline \p}}\simeq \mathrm{H}^0(\P^1,\mathfrak{sl} (\underline E, \underline \p)\otimes\Omega^1_{\P^1}) .$$

On the other hand, if we  see  $\mathfrak{Bun}(X/\iota)$ as a space of bundles $ E$ over $X$ with a lift $h$ of the  hyperelliptic involution, then the space of $h$-invariant Higgs fields on $ E$ also naturally identifies to the cotangent space $\mathrm{T}^*_{( E, h)} \mathfrak{Bun}(X/\iota)$. Indeed, let $(\underline E, \underline \p)$ be a parabolic bundle in  $\mathfrak{Bun}(X/\iota)$ and consider the corresponding parabolic bundle
 $( E,  \p)=\mathrm{elm}^+_W(\pi^*(\underline E, \underline \p))$ over $X$ together with its unique isomorphism $h: E \simeq \iota^*E$ such that $ \p$ corresponds to the $+1$-eigenspaces of $h$. Let $\underline \Theta$ be a logarithmic Higgs field in 
 $$\mathrm{T}^*_{(\underline E, \underline \p)}\mathfrak{Bun}(X/\iota)\simeq \mathrm{H}^0(\P^1,\mathfrak{sl} (\underline E)\otimes\Omega^1_{\P^1}(\underline W ))^{\mathrm{app}_{\underline \p}}.$$
 The corresponding Higgs bundle $( E,  \Theta)=\mathrm{elm}^+_W(\pi^*(\underline E, \underline \p))$ then is $h$-invariant and holomorphic by construction. 
 
   Similarly to the case of connections, we obtain $$\pi_*( E, \Theta) = \bigoplus_{i=1}^2(\underline{E}_i,\underline \Theta_i),$$
where $(\underline{E}_i,\underline \Theta_i)$ are apparent logarithmic Higgs bundles on $\P^1$ with $D=\underline{W}$.

\subsection{A Poincar\'e family on the $2$-fold cover  $\HIGGS(X/\iota)$}\label{SecPoincareRST}

Since we get a universal vector bundle on an open part of $\BUN(X/\iota)$ for our moduli problem
(for instance over $\P^3_B$, see Section \ref{SecBertram}),
we can expect to find a universal family of Higgs bundles (resp. connections) there,
which we will now construct over an open subset of the projective chart $\P^1_R\times \P^1_S\times \P^1_T$,
namely when $(R,S,T)\in\C^3$ is finite.

For $(i,z_i)=(r,R),(s,S),(t,T)$, define the Higgs field $\Theta_i$ given on a trivial chart $(\P^1\setminus \{\infty\})\times \mathbb{C}^2$ of $\underline E = \OO_{\P^1}(-1)\oplus \OO_{\P^1}(-2)$ by
$$\Theta_i:=\frac{\mathrm{d}x}{x}\begin{pmatrix}0&0\\ 1-z_i&0\end{pmatrix}
+\frac{\mathrm{d}x}{x-1}\begin{pmatrix}z_i&-z_i\\ z_i&-z_i\end{pmatrix}
+\frac{\mathrm{d}x}{x-i}\begin{pmatrix}-z_i&z_i^2\\ -1&z_i\end{pmatrix}$$
These parabolic Higgs fields are independent over $\C$ (they do not share the same poles) and 
any other Higgs field $\Theta$ on $\underline E$ respecting the parabolic structure $\underline \p$ given by $(R,S,T)$ is a linear combination of these $\Theta_i$:
$$\Theta=c_r\Theta_r+c_s\Theta_s+c_t\Theta_t\ \ \ \text{for unique}\ \ \ c_r,c_s,c_t\in\C.$$
These generators are chosen such that the coefficient $(2,1)$ of $\Theta_i$ vanishes at $x=j$ and $k$
where $\{i,j,k\}=\{r,s,t\}$. They are also very natural on our chart 
$\Bun^{ss}_{\Mu}(X/\iota)=\P^1_R\times \P^1_S\times \P^1_T$ with $\mu\in]\frac{1}{6},\frac{1}{4}[$. Indeed, for our choice of chart and generators, we precisely get:

\begin{prop}The differential $1$-form $\mathrm{d}z_i $ on the affine chart $(R,S,T)\in\C^3\subset\P^1_R\times \P^1_S\times \P^1_T$ identifies under Serre duality with 
the Higgs bundle $\Theta_i \in \mathrm{H}^0(\P^1,\mathfrak{sl}_2(\underline E)\otimes\Omega^1_{\P^1}(\underline W ))^{\mathrm{app}_{\underline \p}}$ for $(i,z_i)=(r,R),(s,S),(t,T)$.
\end{prop}

\begin{proof}In an intrinsic way, the tangent space of the moduli space of parabolic bundles at a point $(\underline{E},\underline{\p})$ is given by $\mathrm{H}^1(\P^1,\mathfrak{sl}(\underline{E},\underline{\p}))$ where $\mathfrak{sl}(\underline{E},\underline{\p})$ is the sheaf
of trace-free endomorphisms of $E$ over $\P^1_x$ that preserve the parabolic structure. For instance the vector field
$\frac{\partial}{\partial R} \in \mathrm{T}_{(R,S,T)}\P^1_R\times \P^1_S\times \P^1_T$ can be represented  by the two charts  $U_0=\P^1_x\setminus\{r\}$ and $U_1$ an analytic disc surrounding $x=r$ together with the cocycle 
$$\phi_{0,1}=\begin{pmatrix}0&0\\ 1&0\end{pmatrix}$$
on the punctured disc $U_{0,1}=U_0\cap U_1$. Indeed, if we glue the restrictions $(\underline{E},\underline{\p})\vert_{U_0}$
and $(\underline{E},\underline{\p})\vert_{U_1}$ by the map 
$$\exp(\zeta\phi_{0,1})=\begin{pmatrix}1&0\\ \zeta&1\end{pmatrix}:\left((\underline{E},\underline{\p})\vert_{U_{1}}\right)\vert_{U_{0,1}}\to (\underline{E},\underline{\p})\vert_{U_{0}},$$
we get the new parabolic bundle defined by $\underline{\p}=(0,1,R+\zeta, S,T,\infty)$, \emph{i.e.} the point defined by the time-$\zeta$ map generated by the vector field $\frac{\partial}{\partial R}$. Let us now compute  the perfect pairing
$$\langle\cdot,\cdot\rangle: \mathrm{H}^0(\P^1,\mathfrak{sl}_2(\underline E)\otimes\Omega^1_{\P^1}(\underline W ))^{\mathrm{app}_{\underline \p}} \times \mathrm{H}^1(\P^1,\mathfrak{sl}(\underline{E},\underline{\p}))\to \mathrm{H}^1(\P^1,\OOMP)\simeq\C;$$
defining Serre duality in our coordinates. Given a Higgs field $\Theta \in \mathrm{H}^0(\P^1,\mathfrak{sl}_2(\underline E)\otimes\Omega^1_{\P^1}(\underline W ))^{\mathrm{app}_{\underline \p}}$, the image in $\mathrm{H}^1(\P^1,\OOMP)$ is given by the cocycle
$$\langle\Theta,\phi_{0,1}\rangle=\mathrm{trace}(\Theta\cdot\phi_{0,1})$$
on $U_{0,1}$, that is the $(1,2)$-coefficient of $\Theta$ restricted to $U_{0,1}$ (note that $\Theta$ is holomorphic there). 
We fix an isomorphism $\mathrm{H}^1(\OOMP)\to\C$ as follows. Given a cocycle $(U_{0,1},\omega_{0,1})\in \mathrm{H}^1(\OOMP)$, 
one can easily write $\omega_{0,1}=\alpha_0 - \alpha_1$ for meromorphic $1$-forms $\alpha_i$ on $U_i$.
Then $\omega_{0,1}$ is trivial in $\mathrm{H}^1(\OOMP)$ if, and only if, $\omega_{0,1}=\omega_0 - \omega_1$ for 
holomorphic $1$-forms $\omega_i$ on $U_i$, or, equivalently, if  the principal part defined
by $(\alpha_i)_i$ is that of a global meromorphic $1$-form $(\alpha_i-\omega_i)_i$. Since the obstruction 
is given precisely by the Residue Theorem, we are led to define 
$$\mathrm{Res}:\mathrm{H}^1(\P^1,\OOMP)\to\C$$
as the map which to a principal part $(\alpha_i)_i$ representing the cocycle, associates the sum of residues.
For instance, 
$$\omega_{0,1}:=\langle\Theta_r,\phi_{0,1}\rangle=(1-R)\frac{\mathrm{d}x}{x}+R\frac{\mathrm{d}x}{x-1}-\frac{\mathrm{d}x}{x-r}$$
can be represented by the cocycle
$$\alpha_0:=0\ \ \ \text{and}\ \ \ \alpha_1:=-\omega_{0,1}$$
so that the principal part is just defined by  $\frac{\mathrm{d}x}{x-r}$ at $x=r$ and we get
$$\mathrm{Res}\langle\Theta_r,\phi_{0,1}\rangle=1$$
i.e. $\left\langle\Theta_r,\frac{\partial}{\partial R}\right\rangle=1$. Similarly, we have 
$$\left\langle\Theta_i,\frac{\partial}{\partial z_j}\right\rangle=\left\{\begin{matrix}1\ \text{if}\ i=j\\0\ \text{if}\ i\not=j\end{matrix}\right.$$
\end{proof}

\begin{cor}The Liouville form on $\mathrm{T}^*\Bun^{ss}_{\Mu}(X/\iota)$ defines a holomorphic 
symplectic $2$-form on the moduli space
of Higgs bundles defined in the chart $(R,S,T,c_r,c_s,c_t)\in\C^{6}$ by
$$\omega=\mathrm{d}R\wedge \mathrm{d}c_r+\mathrm{d}S\wedge \mathrm{d}c_s+\mathrm{d}T\wedge \mathrm{d}c_t.$$
\end{cor}

\subsection{The Hitchin fibration}\label{SecHitchin}

On the moduli space of Higgs bundles on $X$, the Hitchin fibration is defined by the map
$$\mathrm{Hitch}\ :\ \HIGGS(X)\to \mathrm{H}^0(X,2 \KX)\ ;\ (E,\Theta)\mapsto\det(\Theta).$$
Viewing $\HIGGS(X)$ as the total space of the cotangent bundle $\mathrm{T}^*\BUN(X)$ (over the open set of stable bundles),
the Liouville form defines a symplectic structure on $\HIGGS(X)$. The above map defines
a completely integrable system on this space: writing a quadratic differential as 
$(h_2x^2+h_1x+h_0)\left(\frac{dx}{y}\right)^{\otimes 2}$, the $3$ components of $\mathrm{Hitch}$
$$h_0,h_1,h_2:\HIGGS(X)\to \C$$
are holomorphic functions commuting to each other for the Poisson structure.
Moreover, fibers of the map $\mathrm{Hitch}$ are (open sets of) $3$-dimensional abelian varieties.
One can also associate to $(E,\Theta)$ the spectral curve $\mathrm{spec}(\Theta)$ which is the 
double-section of the projectivized bundle $\P E\to X$ defined by the eigendirections of $\Theta$.
This curve $\mathrm{spec}(\Theta)$ is thus a two-fold ramified cover of $X$, ramifying at zeroes
of the quadratic form $\mathrm{Hitch}(E,\Theta)$; the spectral curve is thus constant along Hitchin fibers
and its Jacobian is the compactification of the fiber (see for example \cite{Jacques1}).
 
\subsection{Explicit Hitchin Hamiltonians on $\HIGGS(X/\iota)$ }
Viewing a Higgs field as the difference of two connections, we have seen that Higgs bundles are
invariant under involution and descend, likely as connections, as parabolic Higgs fields on $\P^1_x=X/\iota$.
The induced map
$$ \HIGGS(X/\iota)\stackrel{2:1}{\rightarrow}\HIGGS(X)$$
allows us to compute the Hitchin fibration easily.
Note that, applying an elementary transformation to some Higgs bundle $(E,\Theta)$
does not modify $\det(\Theta)$ since an elementary transformation is just a birational
bundle transformation, acting by conjugacy on $\Theta$. Therefore, to get Hitchin Hamiltonians
on the chart $(R,S,T,c_r,c_s,c_t)$, we just have to compute
$$\det(c_r\Theta_r+c_s\Theta_s+c_t\Theta_t)=(h_2x^2+h_1x+h_0)\frac{(dx)^{\otimes2}}{x(x-1)(x-r)(x-s)(x-t)}.$$
A straightforward computation yields the explicit Hitchin Hamiltonians for  $\HIGGS(X/\iota)$  given in Table \ref{HitchinRST}.

\begin{table}
$$\begin{array}{rcl}
h_0&=& \left( c_r(R-1)+c_s(S-1)+c_t(T-1)
 \right)  \left(c_rst (R-r)R+c_srt(S-s)S+c_trs(T-t)T\right)\\\\
h_1&=&+{c_r}\,
 \left(c_r(s+t)(r+1)+c_ss(t+1)+c_tt(s+1) \right) {R}^{2}-{{\it c_r}}^{2} \left( {\it t}+{\it s} \right) {R}^
{3}\\&&+{\it c_s}
\, \left(c_s(r+t)(s+1)+c_rr(t+1)+c_tt(r+1)
 \right) {S}^{2}-{{\it c_s}}^{2} \left( {\it t}+{\it r} \right) {S}^
{3}\\&&+{\it c_t}\, \left( {\it c_t}(r+s)(t+1)+{\it c_r}r(s+1)+{\it c_s}{\it s}(r+1)
 \right) {T}^{2}-{{\it c_t}}^{2} \left( {\it r}+{\it s} \right) {T}^{
3}\\&&-c_rc_s(t(R-1+S-1)+r(S-s)+s(R-r))RS \\&&-c_rc_t(s(R-1+T-1)+r(T-t)+t(R-r))RT\\&&-c_sc_t(r(S-1+T-1)+s(T-t)+t(S-s))ST\\&&- \left( {\it c_t}{\it t}(r+s)+{\it c_r}{\it r}({\it s}+t) +{
\it c_s}{\it s}(r+t) \right) (c_rR+c_sS+c_tT)
\\
\\
h_2&=&\left(c_r(R-1)R+c_s(S-1)S+c_t(T-1)T\right)  \left( {\it c_r}(R-r)+{\it c_s}(S-s)+
{\it c_t}(T-t) \right)
\end{array}
$$
\caption{Explicit Hitchin Hamiltonians for the chart  $\P^1_R\times\P^1_S\times\P^1_T$ of $\BUN (X/\iota)$}\label{HitchinRST}
\end{table}

It is easy to check that these functions indeed Poisson-commute:
for any $f,g\in\{h_0,h_1,h_2\}$, we have
$$\sum_{i=r,s,t}\frac{\partial f}{\partial p_i}\frac{\partial g}{\partial q_i}-\frac{\partial f}{\partial q_i}\frac{\partial g}{\partial p_i}=0$$
in Darboux notation $(p_r,p_s,p_t,q_r,q_s,q_t):=(R,S,T,c_r,c_s,c_t)$.

In Proposition \ref{PropBertramToRST}, we specified the birational map $\P^1_R\times \P^1_S\times \P^1_T \dashrightarrow \P^3_B$, allowing us to express the Bertram coordinates $(b_0:b_1:b_2:b_3)$ as functions of $(R,S,T)$. Setting $$c_r\mathrm{d}R+c_s\mathrm{d}S+c_t\mathrm{d}T=\lambda_1\mathrm{d}\frac{b_1}{b_0}+\lambda_2\mathrm{d}\frac{b_2}{b_0}+\lambda_3\mathrm{d}\frac{b_3}{b_0}$$ allows us to express the coeffcients $c_r,c_s,c_t$ as functions of the Bertram coordinates as well. The Hitchin map in Bertram coordinates then writes
$$(h_2x^2+h_1x+h_0)\frac{(dx)^{\otimes2}}{x(x-1)(x-r)(x-s)(x-t)}$$
where the Hitchin Hamiltonians $h_0,h_1,h_2$ are given explicitly in Table \ref{HitchinB}.

\begin{table}
\bgroup
\def\arraystretch{1.4}%

        \rotatebox{90}{
                \begin{minipage}{20.5cm}
$$
\begin{array}{l}
\begin{array}{rcl}
h_0&=&  \frac{\lambda_1 b_1+\lambda_2 b_2 +\lambda_3 b_3}{b_0^4} \cdot \left\{ 
\begin{array}{rl}
-b_0\sigma_3\cdot&\left[\lambda_1b_0 b_{10}+\lambda_2(b_0b_{21}+b_1 b_{10}) +\lambda_3(b_0 b_{32}+b_1b_{21}+b_2  b_{10})\right]
\\
+b_0\sigma_2 \cdot&\left[\lambda_1 b_1   b_{10}+\lambda_2 (b_2   b_{10}+b_1  b_{21}) +\lambda_3(b_1  b_{32}+b_2  b_{21}+b_3  b_{10})\right]
\\
-\sigma_1\cdot &\left[\lambda_1 b_1^2  b_{10}+\lambda_2 b_2(b_1  b_{10}+b_0  b_{21})+\lambda_3(b_0 b_2   b_{32}+b_0 b_3   b_{21}+b_1 b_3   b_{10})\right]
\\
+1\cdot &\left[\lambda_1 (b_1^2   b_{21}+b_2 (b_1^2-b_0 b_2)) +\lambda_2 b_2 (b_1   b_{21}+b_2   b_{10})+\lambda_3 b_3 (b_0   b_{32}+b_1   b_{21}+b_2   b_{10})\right]\end{array}\right.\end{array}
\vspace{.3cm}\\
\begin{array}{rcl}
h_1&=& \frac{1}{b_0^4} \cdot \left\{ 
\begin{array}{rl}
b_0\sigma_3 \cdot &   \left[\lambda_2^2   b_0 ( b_{21}^2-  b_2  b_{20})+\lambda_3^2 (-  b_0   b_2 (2   b_3-  b_2)-  b_1   b_3 (  b_1-2   b_0))
-\lambda_1 \lambda_2   b_0   b_1  b_{10}-\lambda_1 \lambda_3   b_1^2 b_{10}\right.\\&\left.+\lambda_2 \lambda_3 (2   b_0   b_1   b_2+  b_2   b_0 (  b_0-2   b_2)-  b_0   b_3 (2   b_1-  b_0)-  b_2b_{10}^2)\right]
\\
+b_0 \sigma_2 \cdot &\left[
\lambda_2^2 b_2 (b_{10}^2+b_0 b_{20})+\lambda_3^2 b_3 (b_0 (b_3-2 b_2)+b_1 (2 b_2-b_1))+\lambda_1 \lambda_2 b_1^2 b_{10}\right.\\&+\lambda_1 \lambda_3 (b_1^2 (2 b_2-b_1)-b_0 b_2^2)+\left.\lambda_2 \lambda_3 (b_1^2 b_{32}+2 b_2^2 b_{10}+2 b_0 b_3 b_{21})
\right]
\\
+b_0 \sigma_1\cdot & \left[ \lambda_2^2 b_2 (b_{21}^2-b_2 b_{20})+\lambda_3^2 b_3 (-b_2 (b_2-2 b_1)-b_3 (2 b_1-b_0))+\lambda_1 \lambda_2 (b_0 b_2^2-b_1^2 (2 b_2-b_1))\right.\\&+\lambda_1 \lambda_3 (b_0 b_2 (2 b_3-b_2)-b_1 (b_{21}^2+b_1 (2 b_3-b_1)))\left.+\lambda_2 \lambda_3 (b_3 (b_1^2-2 b_2 (2 b_1-b_0))-b_2^2 (b_2-2 b_1))
\right]
\\
+1\cdot &\left[
\lambda_1^2 (b_1^2 (b_1^2-2 b_0 b_2)+b_0^2 b_2^2)+\lambda_2^2 b_2^2 (  b_{10}^2+b_0   b_{20})\right.+\lambda_3^2 b_3 (-b_0 (b_2^2+2 b_1 b_3)+b_3 (b_1^2+2 b_0 b_2))\\
&+\lambda_1 \lambda_2 b_2 (2 b_1^2-b_0 b_2)   b_{10}+\lambda_1 \lambda_3 (-b_0 b_1 b_2^2+2 b_1^2 b_3   b_{10}-b_0^2 b_3 (b_3-2 b_2))\\
&\left.+\lambda_2 \lambda_3 (b_0 b_2^2 (3 b_3-b_2)+2 b_1 b_2 b_3 (b_1-2 b_0))\right]
\end{array}\right.
\end{array}
\vspace{.3cm}\\
\begin{array}{rcl}
h_2 &=& \frac{\lambda_3}{b_0^3} \cdot \left\{ 
\begin{array}{rl}
\sigma_3 \cdot &\left[ -\lambda_1 b_0 b_1   b_{10}-b_0 \lambda_2 (b_1   b_{21}+b_2   b_{10}) -\lambda_3 b_0(b_1  b_{32}+b_2  b_{21}+b_3  b_{10})\right]
\\
+\sigma_2\cdot &\left[\lambda_1 b_1^2   b_{10}+\lambda_2 b_2 (b_1  b_{10}+b_0  b_{21}) +\lambda_3(b_0b_2  b_{32}+b_0b_3  b_{21}+b_1b_3  b_{10})\right]
\\
-\sigma_1\cdot & \left[\lambda_1(b_1^2  b_{21}+b_2(b_1^2-b_0b_2))+\lambda_2b_2(b_2  b_{10}+b_1  b_{21})+\lambda_3b_3(b_0  b_{32}+b_1  b_{21}+b_2  b_{10})\right]
\\
+1\cdot &\left[\lambda_1((b_1^2-b_0b_2)(2b_3-b_2)+b_1b_2  b_{21})+\right.\lambda_2(b_2^2(b_2-2b_1)+b_3(2b_1b_2-b_0b_3))\\ &
\left.+\lambda_3(b_3^2(2b_1-b_0)+b_2b_3(b_2-2b_1))\right]
\end{array}\right.
\end{array}\end{array}
$$
\end{minipage}}\egroup
\caption{Explicit Hitchin Hamiltonians for the coordinates $(b_0:b_1:b_2:b_3)$ of $\P^3_B$. Here $b_{ij}$ denotes $b_i-b_j$.}\label{HitchinB}
\end{table}

\subsection{Explicit Hitchin Hamiltonians on $\HIGGS(X)$ }

We can now push-down formulae onto $X$ to give the explicit Hitchin Hamiltonians on $\HIGGS(X)\simeq \mathrm{T}^*\BUN(X)$. 
In order to do this, we consider the natural rational map $\phi^*:\mathrm{T}^*\P_{\mathrm{NR}}^3\dashrightarrow \mathrm{T}^*\P^1_R\times\P^1_S\times\P^1_T$
induced by the explicit map $\phi:\P^1_R\times\P^1_S\times\P^1_T\dashrightarrow\P_{\mathrm{NR}}^3$ of Proposition \ref{PropFormulesRSTtoNR}.
Then, for a general section $\mu_0 \mathrm{d}\left(\frac{v_0}{v_3}\right)+\mu_1 \mathrm{d}\left(\frac{v_1}{v_3}\right)+\mu_2 \mathrm{d}\left(\frac{v_2}{v_3}\right)$,
the Hitchin Hamiltonians are given, after straightforward computation, by the explicit formula in Table \ref{HitchinV}. 

\begin{table}
\bgroup
\def\arraystretch{1.4}%

        \rotatebox{90}{
                \begin{minipage}{20.5cm}
$$
\begin{array}{l}
\begin{array}{rcl}
h_0&=&  \frac{1}{v_3^3} \cdot \left\{ 
\begin{array}{rl}
\mu_0^2\cdot&\left[v_0^3-(2\sigma_{23} v_0 + \sigma_3 v_1-(\sigma_{12}\sigma_3+\sigma_{23}^2)v_3)v_0v_3
+\sigma_3(\sigma_{23} v_1+\sigma_3 v_2+(\sigma_3-\sigma_{123}\sigma_2)v_3)v_3^2\right]
\\+v_1\mu_1^2 \cdot&\left[v_0v_1+\sigma_3 v_2v_3 \right]\\
+v_1\mu_2^2 \cdot &\left[ v_0v_3+v_1v_2+(1+\sigma_1)v_1v_3-\sigma_{23}v_3^2\right]\\
+\mu_0\mu_1\cdot&\left[ 2(v_0 - \sigma_{23}v_3)v_0v_1+(v_0v_2-(v_1-\sigma_{12}v_3)v_1-(\sigma_{23}v_2+\sigma_3v_3)v_3)\sigma_3v_3\right]\\
+\mu_0\mu_2 \cdot &\left[ v_0^2v_2+v_0v_1^2 + \sigma_{23}(\sigma_{23}v_2 +\sigma_3v_3)v_3^2\right.\\
&\left.-(\sigma_{12}v_1 +2\sigma_{23}v_2+\sigma_3v_3)v_0v_3-(\sigma_{23}v_1 
+2\sigma_3v_2 + (\sigma_1\sigma_3-\sigma_{123}\sigma_2+2\sigma_3)v_3)v_1v_3 \right]\\
+v_1\mu_1\mu_2 \cdot &\left[  v_0v_2+v_1^2-(\sigma_{12}v_1+\sigma_{23}v_2-\sigma_3v_3)v_3  \right]
\end{array}\right.\end{array}
\vspace{.3cm}\\
\begin{array}{rcl}
h_1&=&  \frac{1}{v_3^3} \cdot \left\{ 
\begin{array}{rl}
\mu_0^2
\cdot&\left[2v_0^2v1 -2\sigma_{23}v_0v_1v_3+\sigma_3v_0v_2v_3 -\sigma_3v_1^2v_3+ \sigma_{12}\sigma_3v_1v_3^2-\sigma_3\sigma_{23}v_2v_3^2-\sigma_3^2v_3^3
\right]\\
+v_1\mu_1^2 \cdot&\left[ v_0v_2+v_1^2-\sigma_{12}v_1v_3-\sigma_{23}v_2v_3+\sigma_3v_3^2 
 \right]
\\
+\mu_2^2 \cdot &\left[v_0v_2v_3-v_1^2v_3+ 2v_1v_2^2+2(1+\sigma_1)v_1v_2v_3 + \sigma_{12}v_1v_3^2 - \sigma_{23}v_2v_3^2-\sigma_3v_3^2
\right]\\
+\mu_0\mu_1\cdot&\left[  v_0^2v_2+(3v_0 -\sigma_{23}v_3)v_1^2   +((\sigma_{123}\sigma_2-(\sigma_1+2)\sigma_3)v_1+\sigma_{23}^2v_2 +\sigma_{23}\sigma_3v_3)v_3^2 
\right.\\&\left.
-( \sigma_{12}v_1 +2\sigma_{23}v_2+\sigma_3v_3)v_0v_3
\right]
\\
+\mu_0\mu_2 \cdot &\left[ v_0(2v_0v_3 +4v_1v_2+4(1+\sigma_1)v_1v_3 +\sigma_{12}v_2v_3 -2\sigma_{23}v_3^2)
+\sigma_{12}(v_1 -\sigma_{12}v_3)v_1v_3 +2\sigma_3v_2^2v_3\right.\\
&\left. -(\sigma_{123}\sigma_2-(\sigma_1+2)\sigma_3)v_2v_3^2
+\sigma_{12}\sigma_3v_3^3 
 \right]
 \\
+\mu_1\mu_2 \cdot &\left[v_0v_2^2+ 3v_1^2v_2+2((1+\sigma_1)v_1-\sigma_{23}v_3)v_1v_3 -(\sigma_{12}v_1 +\sigma_{23}v_2 +\sigma_3v_3)v_2v_3
\right]
\end{array}\right.\end{array}
\vspace{.3cm}\\
\begin{array}{rcl}
h_2&=&  \frac{1}{v_3^3} \cdot \left\{ 
\begin{array}{rl}
v_1\mu_0^2\cdot&\left[
v_0v_1+\sigma_3v_2v_3 \right]
\\
+v_1\mu_1^2 \cdot&\left[ v_0v_3+v_1v_2+(1+\sigma_1)v_1v_3 -\sigma_{23}v_3^2 \right]
\\
+\mu_2^2 \cdot&\left[ v_0v_3^2 - v_1v_2v_3+\sigma_{12}v_2v_3^2 +(1+\sigma_1)v_2^2v_3+v_2^3\right]
\\
+v_1\mu_0\mu_1\cdot&\left[ v_0v_2+v_1^2-\sigma_{12}v_1v_3-\sigma_{23}v_2v_3+\sigma_3v_3^2\right]
\\
+\mu_0\mu_2 \cdot &\left[ -2v_0v_1v_3+v_0v_2^2+v_1^2v_2-\sigma_{12}v_1v_2v_3-\sigma_{23}v_2^2v_3-\sigma_3v_2v_3^2 \right]
\\
+\mu_1\mu_2 \cdot &\left[  v_0v_2v_3-v_1^2v_3+2v_1v_2^2+2(1+\sigma_1)v_1v_2v_3+\sigma_{12}v_1v_3^2-\sigma_{23}v_2v_3^2-\sigma_3v_3^3
  \right]
\end{array}\right.\end{array}
\end{array}
$$
\end{minipage}}\egroup
\caption{Explicit Hitchin Hamiltonians for the coordinates $(v_0:v_1:v_2:v_3)$ of $\mathcal{M}_{\mathrm{NR}}$. Here we denote $\sigma_{ij}=\sigma_i+\sigma_j$, $ij=12,13,23$, and $\sigma_{123}=\sigma_1+\sigma_2+\sigma_3$, where $\sigma_1=r+s+t, \sigma_2=rs+st+tr$ and $\sigma_3=rst$ as usual.
}\label{HitchinV}
\end{table}

\noindent In section \ref{SecComputeNR}, we introduced symmetric coordinates $(t_0:t_1:t_2:t_3)$ of $\mathcal{M}_{\mathrm{NR}}$ given by
$${\begin{pmatrix}t_0\\t_1\\t_2\\t_3\end{pmatrix}}={\begin{pmatrix}a&b&c&d\\-b&a&d&-c\\ c&d&a&b\\d&-c&-b&a\end{pmatrix}}\cdot\begin{pmatrix}1&1&0&-\sqrt{\sigma_3}\\0&\sqrt{\sigma_4}&0&0\\0&\sqrt{\sigma_3}&\sqrt{\sigma_3}&\sqrt{\sigma_3}\\0&0&0&\sqrt{\sigma_3\sigma_4}\end{pmatrix}\cdot{\begin{pmatrix}v_0\\v_1\\v_2\\v_3\end{pmatrix}}, \vspace{.2cm}$$
 $\begin{array}{lccrcl}
\textrm{where }&&&a&=&rst(r-s)\sqrt{\sigma_4}+t\sqrt{\rho_r \rho_s}\vspace{.2cm}-rt(r-1)\sqrt{ \rho_s}-st\sqrt{\sigma_4\rho_r}\vspace{.2cm}\\
&&&b&=&-st(s-1)\sqrt{\rho_r}+rt\sqrt{\sigma_4 \rho_s}\vspace{.2cm}\\
&&&c&=&t(r-s)\sqrt{\sigma_3\sigma_4}-t(r-1)\sqrt{\sigma_3 \rho_s}\vspace{.2cm}\\
&&&d&=&-t(r-1)(s-1)(r-s)\sqrt{\sigma_3}+t(s-1)\sqrt{\sigma_3\rho_r}\end{array}$\vspace{.3cm}\\
 $\begin{array}{lllll}
\textrm{and }&&&\sqrt{\sigma_3}^2=rst,& \sqrt{\sigma_4}^2=(r-1)(s-1)(t-1),\\
&&&\sqrt{\rho_r}^2=r(r-1)(r-s)(r-t),&  \sqrt{\rho_s}^2=s(s-1)(s-r)(s-t)
.\end{array}$\vspace{.3cm}\\
 We obtain rational Hitchin Hamiltonians for the coordinates $(t_0:t_1:t_2:t_3)$ given explicitly in  Table \ref{HitchinT} with respect to a general section  $\eta_0 \mathrm{d}\left(\frac{t_0}{t_3}\right)+\eta_1 \mathrm{d}\left(\frac{t_1}{t_3}\right)+\eta_2 \mathrm{d}\left(\frac{t_2}{t_3}\right)= \mu_0 \mathrm{d}\left(\frac{v_0}{v_3}\right)+\mu_1 \mathrm{d}\left(\frac{v_1}{v_3}\right)+\mu_2 \mathrm{d}\left(\frac{v_2}{v_3}\right).$
\vspace{-.5cm}\\
\bgroup
\def\arraystretch{1.3}
\begin{table}[H]
$$\begin{array}{l}
\begin{array}{rcl}
h_0&=&  \frac{1}{4 t_3^4} \cdot \left\{ 
\begin{array}{rl}
rst\cdot & \left[\eta_0 (t_0^2-t_3^2)+\eta_1 (t_0 t_1+t_2 t_3)+\eta_2 (t_0 t_2+t_1 t_3)\right]^2\\
-st\cdot & \left[\eta_0 (t_0 t_1-t_2 t_3)+\eta_1 (t_1^2+t_3^2)+\eta_2 (t_0 t_3+t_1 t_2)\right]^2\\
+4rs\cdot & \left(\eta_0 t_0+\eta_1 t_1\right)^2t_3^2\\
-rt \cdot & \left[\eta_0 (t_0^2+t_3^2)+\eta_1 (t_0 t_1+t_2 t_3)+\eta_2 (t_0 t_2-t_1 t_3)\right]^2
\end{array}\right.
\end{array}
\vspace{.2cm}\\

\begin{array}{rcl}
h_1&=&  \frac{1}{4 t_3^4} \cdot \left\{ 
\begin{array}{rl}
t\cdot &  \left(t_0^2+t_1^2+t_2^2+t_3^2\right)  \left[(\eta_0^2+\eta_1^2+\eta_2^2) t_3^2+(\eta_0 t_0+\eta_1 t_1+\eta_2 t_2)^2\right]\\
+st\cdot &  \left(t_0^2-t_1^2+t_2^2-t_3^2\right)  \left[(\eta_0^2-\eta_1^2+\eta_2^2) t_3^2-(\eta_0 t_0+\eta_1 t_1+\eta_2 t_2)^2\right]
\\
+4r\cdot &  \left(t_0 t_2-t_1 t_3\right) t_3  \left[\eta_0 \eta_2 t_3+(\eta_0 t_0+\eta_1 t_1+\eta_2 t_2) \eta_1\right]
\\
+4sr\cdot &  \left(t_0 t_2+t_1 t_3\right) t_3  \left[\eta_0 \eta_2 t_3-(\eta_0 t_0+\eta_1 t_1+\eta_2 t_2) \eta_1\right]
\\
+4s\cdot &  \left(t_0 t_3+t_1 t_2\right) t_3  \left[\eta_1 \eta_2 t_3-(\eta_0 t_0+\eta_1 t_1+\eta_2 t_2) \eta_0\right]
\\
+4rt\cdot &  \left(t_0 t_1+t_2 t_3\right) t_3  \left[\eta_0 \eta_1 t_3-(\eta_0 t_0+\eta_1 t_1+\eta_2 t_2) \eta_2\right]
\end{array}\right.
\end{array}

\vspace{.2cm}\\

\begin{array}{rcl}
h_2&=&  \frac{1}{4 t_3^4} \cdot \left\{ 
\begin{array}{rl}
s\cdot &  \left[\eta_0 (t_0 t_2+t_1 t_3)+\eta_1 (t_0 t_3+t_1 t_2)+\eta_2 (t_2^2-t_3^2)\right]^2
\\
-1\cdot & \left[\eta_0 (t_0 t_2-t_1 t_3)+\eta_1 (t_0 t_3+t_1 t_2)+\eta_2 (t_2^2+t_3^2)\right]^2
\\
-t\cdot &  \left[\eta_0 (t_0 t_1+t_3 t_3)-\eta_2 (t_0 t_3-t_1 t_2)+\eta_1 (t_2^2+t_3^2)\right]^2\\
+4 r \cdot &  \left(\eta_1 t_1+\eta_2 t_2\right)^2 t_3^2

\end{array}\right.
\end{array}
\end{array}$$
\caption{Explicit Hitchin Hamiltonians for the coordinates $(t_0:t_1:t_2:t_3)$ of $\mathcal{M}_{\mathrm{NR}}$. 
}\label{HitchinT}
\end{table}

\egroup

\subsection{Comparison to existing formulae}
 In \cite{Emma}, B. van Geemen and E. Previato conjectured a projective version of explicit Hitchin Hamiltonians (up to multiplication by functions from the base), which has been confirmed in \cite{GTNB}. These Hamiltonians are expressed in symmetric coordinates $(t_0:t_1:t_2:t_3)$ of $\mathcal{M}_{\mathrm{NR}}$ and with respect to a genus $2$ curve $X$ given by 
 $$y^2=\prod_{i=1}^6 (x-\lambda_i).$$ The coefficients $A,B,C,D$ of the Kummer surface (\ref{EqKummerHudson}) can be made explicit, allowing us to uniquely identify
$$(\lambda_1,\, \lambda_2,\, \lambda_3,\, \lambda_4,\,\lambda_5,\,\lambda_6)=(  0,\, t ,\,  1,\,  s,\, r ,\,\infty)$$ with respect to our coordinates. Indeed,

\bgroup
\def\arraystretch{1.6}
\begin{table}[H]
\centerline{
 \begin{tabular}{|c|c|c|}
\hline value of  & in the paper \cite{Emma} & in equation (\ref{EqKummerHudson2})\\
\hline 
$A$ & $2\frac{  (2 (\lambda_1 \lambda_2+ \lambda_3 \lambda_4)-(\lambda_1+\lambda_2) (\lambda_3+\lambda_4)}{(\lambda_3-\lambda_4) (\lambda_1-\lambda_2)}$ & $-2\frac{s(t-1)+(t-s)}{t(s-1)}$\\
\hline 
$B$ & $-2\frac{2(\lambda_1 \lambda_2+ \lambda_5 \lambda_6)-(\lambda_1+\lambda_2) (\lambda_5+\lambda_6)}{(\lambda_5-\lambda_6) (\lambda_1-\lambda_2)}$ & $-2\frac{r+(r-t)}{t}$\\
\hline
$C $& $2\frac{2 ( \lambda_3 \lambda_4+ \lambda_5 \lambda_6)-(\lambda_3+\lambda_4) (\lambda_5+\lambda_6)}{(\lambda_5-\lambda_6) (\lambda_3-\lambda_4)}$ & $2\frac{(r-1)+(r-s)}{s-1}$ \\
\hline 
$D $& $-4\frac{(\lambda_1+\lambda_2) (\lambda_5 \lambda_6-\lambda_3 \lambda_4)+(\lambda_3+\lambda_4) (\lambda_1 \lambda_2-\lambda_5 \lambda_6)+(\lambda_5+\lambda_6) (\lambda_3 \lambda_4-\lambda_1 \lambda_2)}{(\lambda_5-\lambda_6) (\lambda_3-\lambda_4) (\lambda_1-\lambda_2)}$ & $-4\frac{r(s-t)+(r-s)}{t(s-1)}$\\
\hline
\end{tabular}
}\end{table}\egroup

Let us denote $$h(x):=h_2x^2+h_1x+h_0,$$ where $h_i$ for $i\in \{0,1,2\}$ are the Hitchin Hamiltonians given with respect to the symmetric coordinates in the affine chart $\left(\frac{t_0}{t_3}:\frac{t_1}{t_3}:\frac{t_2}{t_3}:1\right)$ as in Table \ref{HitchinT}. 
The Hamiltonians $H_1,\ldots H_6$ in \cite{Emma}
can then be expressed in terms of the Hitchin Hamiltonians as
$$\begin{array}{rcc cc rcc}
H_1 &=& \frac{4h(0)}{rst}  &\textrm{ }  &\textrm{ } & H_4  &=& \frac{4h(s)}{s(s-1)(s-r)(s-t)}\vspace{.3cm}\\
H_2 &=& -\frac{4h(t)}{t(t-1)(t-r)(t-s)} &&& H_5 &=& \frac{4h(r)}{r(r-1)(r-s)(r-t)}\vspace{.3cm}\\
H_3 &=& \frac{4h(1)}{(r-1)(s-1)(t-1)}&&& H_6 &=&0
\end{array}$$

\section{The moduli stack $\CON (X)$}

Note that if $\nabla_1$ and $\nabla_2$ are connections on the same vector bundle $E\to X$, then $(E,\nabla_1-\nabla_2)$ is a Higgs bundle. 
Hence  $\mathfrak{Con}(X)$ (resp. $\mathfrak{Con}(X/\iota)$) can be seen as an affine extension of $\mathfrak{Higgs}(X)$ (resp. $\mathfrak{Higgs}(X/\iota)$).
One connection on the parabolic bundle $\left(\mathcal{O}_{\mathbb{P}^1}(-1)\oplus\mathcal{O}_{\mathbb{P}^1}(-2),\underline{\p}\right)$ attached to a parameter $(R,S,T)$ is given in the affine chart $(\mathbb{P}^1\setminus\{\infty\})\times \mathbb{C}^2$ with coordinates $(x,Y)$ by
\begin{equation}\label{TheConnectionRST}
\begin{array}{rcl} \nabla_0&:=&\mathrm{d}+\begin{pmatrix}0&0\\ -1&\frac{1}{2}\end{pmatrix}\frac{\mathrm{d}x}{x}
+\begin{pmatrix}1&- \frac{1}{2}\\1&- \frac{1}{2}\end{pmatrix}\frac{\mathrm{d}x}{x-1}
\vspace{.2cm}\\ &&
+\frac{1}{2}\begin{pmatrix}0&R\\ 0&1\end{pmatrix}\frac{\mathrm{d}x}{x-r}
+\frac{1}{2}\begin{pmatrix}0&S\\ 0&1\end{pmatrix}\frac{\mathrm{d}x}{x-s}
+\frac{1}{2}\begin{pmatrix}0&T\\ 0&1\end{pmatrix}\frac{\mathrm{d}x}{x-t}
\end{array}\end{equation}
and hence in the affine chart $(\mathbb{P}^1\setminus\{0\})\times \mathbb{C}^2$ with coordinates $(\tilde x,\widetilde Y)=\left(\frac{1}{x}, \left(\begin{smallmatrix} x & 0\\ 0& x^2 \end{smallmatrix}\right) Y\right)$ its 
 residual part at $\tilde{x}=0$  is given by $\mathrm{d}+\left(\begin{smallmatrix}0 & 0\\ -1&  \frac{1}{2}\end{smallmatrix}\right)\frac{\mathrm{d}\tilde x}{\tilde x}$.   
Any other connection on this bundle writes uniquely as
$$\nabla=\nabla_0+c_r\Theta_r+c_s\Theta_s+c_t\Theta_t,$$
where the Higgs bundles $\Theta_i$ are defined in Section \ref{SecPoincareRST}.
This provides a universal family of parabolic connections on a large open subset
of the moduli space. Note that the residual part at infinity is given for such a connection by
$$\mathrm{d}+\left(\begin{matrix}0 & 0\\ -1-c_r(R-r)-c_s(S-s)-c_t(T-t)&  \frac{1}{2}\end{matrix}\right)\frac{\mathrm{d}\tilde x}{\tilde x}.$$   

\subsection{An explicit atlas}\label{SecAtlasCon}
We can use the above construction to cover the moduli space $\CON^*(X/\iota)$ by affine charts, in each of which we can explicitly describe the Poincar\'e family. Here $\CON^*(X/\iota)$ denotes the space of all those parabolic connections $(\underline{E}, \underline{\nabla})$ such that $\Phi(\underline{E}, \underline{\nabla})$ is not a twist of the trivial connection on the trivial vector bundle over $X$.

The map $\Phi$ then induces on the set $\CON^*(X)$ of irreducible or abelian but non trivial $\mathfrak{sl}_2\C$- connections on $X$ the structure of a moduli stack. 

Fix exponents $\kappa_i\in\mathbb C$ for $$i\in \{0,1, r,s,t,\infty\}$$ and define $\rho\in\C$ by
$$\kappa_0+\kappa_1+\kappa_{r}+\kappa_{s}+\kappa_{t}+\kappa_\infty+2\rho=1.$$
The universal connection on $\OOP\oplus\OOP(-1)$ with eigenvalues 
\begin{equation}\label{RiemannScheme}
\begin{pmatrix}x=0 & x=1 & x=r & x=s & x=t & x=\infty\\
0 & 0 & 0 & 0 & 0 & \rho \\
\kappa_0 & \kappa_1 & \kappa_{r} & \kappa_{s} & \kappa_{t} & \kappa_\infty+\rho\end{pmatrix}
\end{equation}
can be written as follows:
\begin{equation}\label{UniversalConnectionKappai}
\nabla=\nabla_0+c_r\Theta_r+c_s\Theta_s+c_t\Theta_t
\end{equation}
with
$$\nabla_0=d+\begin{pmatrix}0&0\\ \rho&\kappa_0\end{pmatrix}\frac{dx}{x}
+\begin{pmatrix}-\rho& \rho+\kappa_1\\ -\rho&\rho+\kappa_1\end{pmatrix}\frac{dx}{x-1}
+\sum_{i\in \{r,s,t\}}\begin{pmatrix} 0& z_i\kappa_{i}\\ 0& \kappa_{i}\end{pmatrix}\frac{dx}{x-i}$$
and
$$\Theta_i=\begin{pmatrix}0&0\\ 1-z_i&0\end{pmatrix}\frac{dx}{x}
+\begin{pmatrix}z_i& -z_i\\ z_i&-z_i\end{pmatrix}\frac{dx}{x-1}
+\begin{pmatrix} -z_i & z_i^2\\ -1& z_i\end{pmatrix}\frac{dx}{x-i}$$
Here we have normalized the parabolic data to
\begin{equation}\label{RiemannScheme}\begin{matrix}x=0 & x=1 & x=r & x=s & x=t & x=\infty\\
\begin{pmatrix}0\\ 1\end{pmatrix} &\begin{pmatrix}1\\ 1\end{pmatrix} &\begin{pmatrix}z_r\\ 1\end{pmatrix} &\begin{pmatrix}z_s\\ 1\end{pmatrix} &\begin{pmatrix}z_t\\ 1\end{pmatrix} & \OOP(-1)
\end{matrix}\end{equation}
We note that eigendirections with respect to $0$-eigenvalue are generated by
$$\begin{matrix}x=0 & x=1 & x=r & x=s & x=t\\
\begin{pmatrix}-\frac{\kappa_0}{\rho+\sum_{i\in \{r,s,t\}}c_i(z_i-1)}\\ 1\end{pmatrix} &\begin{pmatrix}1+\frac{\kappa_1}{\rho+\sum_{i\in \{r,s,t\}}c_iz_i}\\ 1\end{pmatrix} &\begin{pmatrix}z_r
-\frac{\kappa_{r}}{c_r}\\ 1\end{pmatrix} &\begin{pmatrix}z_s-\frac{\kappa_{s}}{c_s}\\ 1\end{pmatrix} &\begin{pmatrix}z_t-\frac{\kappa_{t}}{c_t}\\ 1\end{pmatrix} 
\end{matrix}$$

This matrix connection can be thought of, via an elementary transformation at $x=\infty$,
as a parabolic system (a parabolic connection on the trivial bundle) with shifted eigenvalues $(\rho,\rho+\kappa_\infty-1)$
and parabolic now associated to $\rho$ normalized to $e_1={}^t(1,0)$.
Similarly, after twisting by the (unique) rank $1$ connection $(\OP{-1},\zeta)$ having a single pole at infinity,
we get a universal family for those connections on $\OP{-1}\oplus\OP{-2}$
with shifted eigenvalues $(\rho+1,\rho+\kappa_\infty+1)$ at $x=\infty$.

We obtain an atlas of charts of $\CON^*(X/\iota )$, each possessing a universal connection, as follows  (the birational transition maps between charts are straightforward to calculate and will not be carried out explicitly):
\begin{itemize}
\item[$\bullet$] {\bf Canonical chart:}
When $\kappa_i=\frac{1}{2}$ for all $i\in \underline W$  (and thus $\rho=-1$), we obtain our first affine chart $$(z_r,z_s,z_t,c_r,c_s,c_t)\in \C^6=:U_0$$ together with its  universal family  of connections (\ref{UniversalConnectionKappai}).
 \item[$\bullet$]  {\bf  Switch:} Choose $J\subset \{r,s,t\}$. 
 Set $\kappa_j=-\frac{1}{2}$ for all $j\in J$ and  $\kappa_i=\frac{1}{2}$ for all  $i\in\underline{W}\setminus J$. Tensorize the corresponding universal family of connections (\ref{UniversalConnectionKappai}) with the (unique) logarithmic rank $1$ connection on $\eta : \OOP\to   \OOP\otimes \Omega^1_{\P^1}(J+[\infty])$ having eigenvalues $+\frac{1}{2}$ over each element in $J$ and eigenvalue $-\frac{\#J}{2}$ over $\infty$.\\ 
We thereby obtain the universal  family of connections on $\OOP(-1)\oplus \OOP(-2)$ having eingenvalues $0$ and $\frac{1}{2}$ over each point in $\underline W$, where the $\frac{1}{2}-$ eigendirections over $i\in\underline{W}\setminus J$ and the $0$-eigendirections over $J$ are normalized to (\ref{RiemannScheme}).
\item[$\bullet$]  {\bf  Twist:} Set $\kappa_0=\kappa_1=-\frac{1}{2}$ and $\kappa_i=\frac{1}{2}$ for all $i\in\underline{W}\setminus \{0,1\}$. Apply positive elementary transformations in the  parabolic directions corresponding to the $-\frac{1}{2}$ eigendirections of the universal family of connections (\ref{UniversalConnectionKappai}). We obtain a new universal family of connections with eigenvalues $
0$ and $\frac{1}{2}$ over each $i\in\underline W$ such that the parabolic structure corresponding $\frac{1}{2}$-eigendirections over each $i\in\underline{W}$ is normalized to 
$$\begin{matrix}x=0 & x=1 & x= r & x=s & x=t & x=\infty\\
\begin{pmatrix}1\\ 0\end{pmatrix} &\begin{pmatrix}1\\ 0\end{pmatrix} &\begin{pmatrix}z_r\\ 1\end{pmatrix} &\begin{pmatrix}z_s\\ 1\end{pmatrix} &\begin{pmatrix}z_t\\ 1\end{pmatrix} & \OOP(-1)
\end{matrix}.$$
\item[$\bullet$]  {\bf Permutation :} For any $\sigma \in \mathfrak{S}(\{0,1,r,s,t,\infty\}$, we obtain similar constructions for parabolic data normalized to 
$$\begin{matrix}x=\sigma(0) & x=\sigma(1) & x=\sigma(r) & x=\sigma(s) & x=\sigma(t) & x=\infty\\
\begin{pmatrix}0\\ 1\end{pmatrix} &\begin{pmatrix}1\\ 1\end{pmatrix} &\begin{pmatrix}z_r\\ 1\end{pmatrix} &\begin{pmatrix}z_s\\ 1\end{pmatrix} &\begin{pmatrix}z_t\\ 1\end{pmatrix} & \OOP(-1)
\end{matrix}$$
\item[$\bullet$]  {\bf Galois involution:} Choose any of the above charts $U$ together with its universal family of connections. First apply positive elementary transformations in all $\frac{1}{2}$-eigendirections of the universal family of connections. We obtain logarithmic connections with eigenvalues $0$ and $-\frac{1}{2}$ over each Weierstrass point. Then tensorize this connection by the unique logarithmic (rank $1$) connection on $\OOP (-3)$ having eigenvalues $\frac{1}{2}$ over each Weierstrass point. We obtain a new chart $U'$ together with a universal connection such that $\Phi(U)=\Phi(U')$.
\end{itemize}

By construction, the above are indeed all affine charts of $\CON (X/\iota )$. Moreover, the transition maps between charts are all birational. 
\begin{prop} The moduli space $\CON (X/\iota )$ is entirely covered by the above charts, except for the preimages  under $\Phi$ of the trivial connection on the trivial bundle $E_0$ and its twists. \end{prop}

\begin{proof}
Firstly, note that all possible  parabolic vector bundles $(\underline E, \underline p)$ underlying a connection $\underline \nabla$ in $\Con (X/\iota )$, where $\underline p$ is given by the $\frac{1}{2}$-eigendirections, occur in the above charts.
 \begin{itemize}
\item[$\bullet$] If $\underline E = \OOP(-1)\oplus \OOP(-2)$ and $\underline \p$ is undecomposable, then at least three of the parabolics are not included int the total space of the destabilizing subbundle $\OOP(-1)\subset \underline E$ and are not included in the same $\OOP(-2)\hookrightarrow \underline E$. Up to permutation, we can assume that this is the case for $\underline {\p}_0,\underline{\p}_1$ and $\underline{\p}_\infty$. Any such configuration appears in the Canonical chart or a Switch. More precisely, we need to switch each parabolic contained in the destabilizing subbundle. 
\item[$\bullet$] If $\underline E = \OOP(-1)\oplus \OOP(-2)$ and $\underline \p$ is decomposable, then we have  two parabolics, which we can assume to be $\underline {\p}_{0},\underline{\p}_{1}$ by permutation, defined by the total space of  $\OOP(-1)\subset \underline E$ and the four others by  some $\OOP(-2)\hookrightarrow \underline E$. This configuration arises in the Twist chart. 
\item[$\bullet$] If $\underline E = \OOP\oplus \OOP(-3)$ and $\underline \p$ is undecomposable, then at most one parabolic in included in the total space of  the destabilizing subbundle and the Galois involution leads to  an undecomposable  parabolic configuration on $\OOP(-1)\oplus \OOP(-2)$, which we have already treated.
\item[$\bullet$] If $\underline E = \OOP\oplus \OOP(-3)$ and $\underline \p$ is decomposable, then all parabolics are defined by the total space of a same $\OOP(-3)\hookrightarrow \underline E$. This configuration arises from the Twist chart, namely when the $0$-eigendirections over $0$ and $1$, as well as the $\frac{1}{2}$-parabolics corresponding over $\underline W \setminus \{0,1\}$ of some logarithmic connection on $\OOP(-1)\oplus \OOP(-2)$ are included in the total space of a same subbundle $\OOP(-2)\hookrightarrow \OOP(-1)\oplus \OOP(-2)$. 
 \end{itemize}
Secondly, we know from Section \ref{SecFlatOnX} that  for every vector bundle $E$ in $\BUN^* (X)$, except for the trivial bundle and its twists, the space of $\iota$-invariant $\mathfrak{sl}_2\C$- connections on $E$ is an affine $\C^3$-space. Yet by construction, the universal connection we established for our charts  provide a $\C^3$-space of two-by-two non isomorphic connections on each  parabolic bundle $(\underline E, \underline \p)$.
The moduli space of irreducible or abelian connections on $E_0\otimes L$ with  $L^{\otimes 2}=\mathcal{O}_{X_t}$ in $\Con (X)$, if we exclude the trivial connections, is only birationally isomorphic to $\C^3$ (see Section \ref{SecDecFlatX}). The fact that we do indeed cover all of the mentioned connections follows from a more detailed but straightforward analysis. 
  \end{proof}

\subsection{The apparent map on $\CON(X/\iota)$}\label{SecApparentRST}
Folllowing \cite{LoraySaito}, we will now recall the construction of the so-called \emph{apparent map}, allowing us to prove Proposition \ref{PropBertramToRST}. For a parabolic connection $(\underline{E},\underline{\p},\nabla)$
defined on the main vector bundle $\underline{E}=\OP{-1}\oplus\OP{-2}$, we can associate
a morphism 
$$\nabla\ \ \ \mapsto\ \ \ \varphi_\nabla\in\mathrm{H}^0\left(\Hom(\OP{-1},\OP{-2}\otimes\OMP{\underline{W}})\right)\simeq 
\mathrm{H}^0\left(\P^1,\OP{-1}\otimes\OMP{\underline{W}}\right)$$
by composition of 
$$\OP{-1}\hookrightarrow \underline{E}\stackrel{\nabla}\longrightarrow \underline{E}\otimes\OMP{\underline{W}}\to\OP{-2}\otimes\OMP{\underline{W}}$$
where the last arrow is just the projection on the second direct summand. 

\begin{rem}Geometrically, the zeroes of the apparent map (which is an element of $\mathrm{H}^0\left(\P^1,\OP{3}\right)$) are the coordinates of the (three) tangencies between the destabilizing section $\sigma_{-1}$ of $\P ( \underline{E} )$ and the foliation  on $\P ( \underline{E} )$ defined by flat sections of $\P (\nabla)$. On the other hand, these are precisely the positions of the apparent singular points appearing when we
derive the associate $2^{\text{nd}}$ order fuchsian equation from the ``cyclic vector'' $\OP{-1}\hookrightarrow E$.
\end{rem}

We can extend the definition of the apparent map to so-called \emph{$\lambda$-connections}
$$\nabla=\lambda\cdot\nabla_0+c_r\Theta_r+c_s\Theta_s+c_t\Theta_t,\ \ \ (\lambda,c_r,c_s,c_t)\in\C^4,$$
 including Higgs fields (for $\lambda=0$). There is a natural $\mathbb{G}_m$-action by multiplication on the moduli space 
of $\lambda$-connections so that a generic element $\nabla$, with $\lambda\not=0$, is equivalent to a unique connection (in the usual sense),
namely $\frac{1}{\lambda}\nabla$. After projectivization, we thus obtain a natural compactification of the moduli space of
connections on $\underline{E}$ (an affine $3$-space) by the moduli space of projective Higgs fields (i.e. up to $\mathbb{G}_m$-action).
In our coordinates, an element $(\lambda:c_r:c_s:c_t)\in\P^3$ denotes either a connection (when $\lambda\not=0$) or a projective
class of a Higgs field. It is proved in \cite{LoraySaito}, Theorem 4.3, that the map $\nabla\mapsto\P\varphi_{\nabla}$,
which is invariant under $\mathbb{G}_m$-action, defines an isomorphism from the moduli space of $\lambda$-connections
up to $\mathbb{G}_m$-action onto $\P \mathrm{H}^0\left(\P^1,\OP{-1}\otimes\OMP{\underline{W}}\right)$. Moreover, 
we deduce a map 
$$\Bun^{ss}_{\Mu}(X/\iota)\to\P \mathrm{H}^0\left(\P^1,\OP{-1}\otimes\OMP{\underline{W}}\right)^\vee$$
which to a parabolic bundle $(\underline{E},\underline{\p})$ associates the image under $\P \varphi$ 
of the hyperplane locus of Higgs bundles $\lambda=0$. For $\frac{1}{6}<\mu<\frac{1}{4}$, this map 
is also an isomorphism. 

On the other hand, looking at $\Bun^{ss}_{\Mu}(X/\iota)$ as extensions (see Section \ref{SecProjChartsParaBundles}), 
we also get a natural isomorphism
$$\Bun^{ss}_{\Mu}(X/\iota)\stackrel{\sim}{\to}\P \mathrm{H}^0\left(\P^1,\OP{-1}\otimes\OMP{\underline{W}}\right)^\vee$$
It follows from \cite{LoraySaito}, proof of Theorem 4.3, that these two maps coincide.

\begin{proof}[Proof of Proposition \ref{PropBertramToRST}]For $(R,S,T)\in\C^3$ finite, 
the corresponding parabolic bundle also belongs to $\Bun^{ss}_{\Mu}(X/\iota)$
and we can use the apparent map to compute the corresponding point $(b_0:b_1:b_2:b_3)\in\P^3_{\boldsymbol{b}}$.
Precisely, the apparent map $\varphi_{\Theta_r}$ is given by the $(2,1)$-coefficient of $\Theta_r$
$$\P\varphi_{\Theta_r}=\frac{R-1}{R-r}(x-r)(x-s)(x-t)
\in \P \mathrm{H}^0\left(\P^1,\OP{-1}\otimes\OMP{\underline{W}}\right)\simeq \vert\OP{3}\vert.$$
This provides a first equation 
$$(R-r)b_0-(\sigma_1R-r(1+s+t))b_1+(\sigma_2R-r(s+t+st))b_2-\sigma_3(R-1)b_3=0;$$
similar equations for $\Theta_s$ and $\Theta_t$ give the result.
\end{proof}

\subsection{A Lagrangian section of $\CON(X)\to\BUN(X)$}\label{lagr}

The rational section 
$$\nabla_0:\BUN(X/\iota)\dashrightarrow\CON(X/\iota)$$ constructed
in Section \ref{SecPoincareRST} over the chart $\P^1_R\times\P^1_S\times\P^1_T$
is not invariant by the Galois involution of $\Phi:\CON(X/\iota)\stackrel{2:1}{\longrightarrow}\CON(X)$, \emph{i.e.} it defines a $2$-section, but not a  rational section  $\BUN(X)\dashrightarrow\CON(X)$. One can easily deduce a rational section
by taking the barycenter (recall that $\CON(X)\to\BUN(X)$ is an affine bundle) but it is not 
  the simplest one. Here, we start back from the Tyurin parametrization of bundles
to construct such an explicit section.

Like in Section \ref{SecTyurin}, consider a generic data $(\underline{P}_1,\underline{P}_2,\lambda)\in X\times X\times \P^1$
and associate the parabolic structure $\widetilde{\p}$ on $\widetilde{E}:=\OX{- \KX}\oplus\OX{- \KX}$ defined over 
$$D:= [\underline{P}_1]+[\iota\left(\underline{P}_1\right)]+[\underline{P}_2]+[\iota\left(\underline{P}_2\right)] \in |2 \KX |,$$
 by 
$$\left(\lambda_{\underline{P}_1}, \lambda_{\iota\left(\underline{P}_1\right)}, \lambda_{\underline{P}_2}, \lambda_{\iota\left(\underline{P}_2\right)}\right):=\left(\lambda,-\lambda,\frac{1}{\lambda}, -\frac{1}{\lambda}\right)$$
(where $\lambda_Q$ means the direction generated by $\lambda_Q e_1+e_2$ over $Q$, for fixed independent sections $e_1,e_2$ over $X\setminus\{\infty\}$). After $4$ elementary transformations, we get a bundle $E$ with trivial determinant.
A holomorphic connection $\nabla$ on $E:=\mathrm{elm}_D^+(\widetilde{E},\widetilde{\p})$ can be pulled-back to $\OX{- \KX}\oplus\OX{- \KX}$
and we get a parabolic logarithmic connection $\widetilde{\nabla}$ on this bundle with (apparent) singular points over $D$. In the basis $\langle e_1,e_2\rangle$, we can write
$$\widetilde{\nabla}\ :\ \mathrm{d}+\begin{pmatrix}\alpha&\beta\\ \gamma&\delta \end{pmatrix}$$
where the trace is given by 
$$\alpha+\delta=\frac{\mathrm{d}x}{x-x_1}+\frac{\mathrm{d}x}{x-x_2}$$
and the projective part takes the form (here $z$ is the projective variable defined by $ze_1+e_2$)
$$\P\widetilde{\nabla}\ :\ \mathrm{d}z-\gamma z^2+(\alpha-\delta) z+\beta\ \ \ 
\text{with}\ \ \ 
\left\{\begin{matrix}
-\gamma&=&\frac{A(x)}{(x-x_1)(x-x_2)}\frac{\mathrm{d}x}{y}\vspace{.2cm}\\
\alpha-\delta&=&\frac{by}{(x-x_1)(x-x_2)}\frac{\mathrm{d}x}{y}\vspace{.2cm}\\
\beta&=&\frac{C(x)}{(x-x_1)(x-x_2)}\frac{\mathrm{d}x}{y}
\end{matrix}\right.$$
where $A,C$ are degre $3$ polynomials in $x$ and $b\in\C$. This is due to the fact that the connection 
has only simple poles over $D$ and that it is invariant under the (normalized) lift of the hyperelliptic involution $h:(x,y,z)\mapsto(x,-y,-z)$.
We note that $e_1$ and $e_2$ generate the two $\iota$-invariant Tyurin subbundles.
Moreover, these coefficients $\{(A,b,C)\}$ have to satisfy several additional conditions, namely the compatibility with the parabolic data,
that eigenvalues are $0$ and $1$ (parabolic directed by $1$) and the singularity is apparent, in the sense that it disappears after an elementary transformation in the parabolic. This gives $6$ affine equations in the $9$-dimensional space of coefficients $\{(A,b,C)\}$:
$$\begin{array}{rl}\text{parabolic data:}&\left\{\begin{matrix}
\lambda A(x_1)+by_1+\frac{1}{\lambda}C(x_1)&=&0\\
\frac{1}{\lambda}A(x_2)+by_2+\lambda C(x_2)&=&0
\end{matrix}\right.\vspace{.2cm}\\
\text{eigenvalues:}& \left\{\begin{matrix}
2\lambda A(x_1)+by_1&=&y_1(x_2-x_1)\\
\frac{2}{\lambda}A(x_2)+by_2&=&y_2(x_1-x_2)
\end{matrix}\right.\vspace{.2cm}\\\text{apparent:}& \left\{\begin{matrix}
2y_1(\lambda A'(x_1)+\frac{1}{\lambda}C'(x_1))+bF'(x_1)&=&0\\
2y_2(\frac{1}{\lambda}A'(x_2)+\lambda C'(x_2))+bF'(x_2)&=&0
\end{matrix}\right.\end{array}$$
where $F(x)=x(x-1)(x-r)(x-s)(x-t)$.
Viewing a Higgs field $\widetilde{\Theta}=\left(\begin{smallmatrix}\alpha&\beta\\ \gamma&\delta \end{smallmatrix}\right)$ as the difference of two connections,
we get $\alpha+\delta=0$ and, for the projective part, the corresponding linearized equations (with $0$ right-hand-side).
Starting with a connection $\widetilde{\nabla}$ on $\widetilde{E}=\OX{- \KX}\oplus\OX{- \KX}$ as above, \emph{via}
 $4$ elementary transformations, we get a holomorphic $\mathfrak{sl}_2$-connection 
 $$(E,\nabla)= \mathrm{elm}^+_D(\widetilde{E}, \widetilde{\nabla},\widetilde{\p})$$
 on $X$ whose parabolic data $\p$ is supported by the strict transform of the line bundle $L_1:=\C\langle e_1\rangle$ (we suppose $\lambda \not \in \{0,\infty\}$ by genericity). Pushing it down, we get a parabolic connection on $X/\iota=\P^1_x$ for which $L_1$ becomes the destabilizing subbundle $\OP{-1}$. 
 
Selecting the $\iota$-invariant Tyurin subbundle $L_1$, we have a natural generically finite map 
$$X\times X\times\P^1_\lambda\stackrel{16:1}{\dashrightarrow}\P^3_B$$
with Galois group generated by $\langle\sigma_{12},\sigma_{\iota},\sigma_{iz}\rangle$ (see Section \ref{SecTyurin}).
Then, the Galois involution of $\P^3_B\stackrel{2:1}{\dashrightarrow}\P^3_{\mathrm{NR}}$ is induced by $\sigma_{1/z}$
which is permuting $e_1$ and $e_2$. We can thus compute the apparent map 
of a connection $\widetilde{\nabla}$ (or 
a Higgs field $\widetilde{\Theta}$) with respect to $e_1$  and get that $\varphi_{\widetilde{\nabla}}=A(x)$. 

\begin{rem}
The three zeroes of $A(x)$ define six points on $X$, which are the coordinates of the tangencies between $e_1$ and the foliation $\P (\widetilde{\nabla})$ on $\P (\widetilde{E})$.
\end{rem}

Like in the proof of Proposition \ref{PropBertramToRST} (see Section \ref{SecApparentRST})
we can use the apparent map for Higgs fields to compute the corresponding Bertram coordinates of $\P^3_B$.
A straightforward computation yields:

\begin{prop}\label{PropTyurinToBertram}
The natural map $X\times X\times\P^1_\lambda\to\P^3_B$
is given by
$$\left\{\begin{matrix}
b_0&=&\lambda y_2-\frac{1}{\lambda}y_1\\
b_1&=&\lambda x_1y_2-\frac{1}{\lambda}x_2y_1\\
b_2&=&\lambda x_1^2y_2-\frac{1}{\lambda}x_2^2y_1\\
b_3&=&\lambda x_1^3y_2-\frac{1}{\lambda}x_2^3y_1
\end{matrix}\right.$$
\end{prop}

It follows from \cite{LoraySaito} that a connection on a parabolic bundle belonging to the chart
$\P^3_B$ is determined by its apparent map. It is particularly easy to see this fact in above equations:
after prescribing $\varphi_{\widetilde{\nabla}}=A(x)\in\P^3_A$ (up to homothecy), \emph{i.e}. after prescribing the roots of $A(x)$,
we get a unique solution $(A,b,C)$ except when $A(x)$ lies in the hyperplane of Higgs bundles
defined by Proposition \ref{PropTyurinToBertram} above. In the latter case,  there is a solution $(A,b,C)$
as a Higgs field which is unique up to an homothecy. Note that the group $\langle\sigma_{12},\sigma_{\iota},\sigma_{iz}\rangle$
acts on connections (and Higgs fields) and the induced action on the coefficient $A(x)$ is by homothecy. It follows that the corresponding point $A(x)\in\P^3_A$ is invariant. The fourth involution $\sigma_{1/z}$
however permutes $A(x)$ and $C(x)$ (and changes the sign).
In order to construct a rational section $\nabla_0:\BUN(X)\dashrightarrow\CON(X)$,
we can consider connections for which $A(x)$ and $C(x)$ are homothetic to each other, \emph{i.e.} define the same point in $\P^3_A$.
A straightforward computation shows that there are exactly two possibilities:
$$\nabla^+\ :\ b=\frac{\lambda^2+1}{\lambda^2-1}(x_1-x_2)\ \ \ \text{and}\ \ \ A(x)=C(x)=$$
$$\left.\frac{1}{2(x_1-x_2)^2}\frac{\lambda}{\lambda^2-1}
\right((y_1-y_2)(4x^3-6(x_1+x_2)x^2+12x_1x_2x)-6x_1x_2(x_2y_1-x_1y_2)$$
$$\left.+2(x_2^3y_1-x_1^3y_2)-(x_1-x_2)(x-x_1)(x-x_2)\left(y_1\frac{F'(x_1)}{F(x_1)}(x-x_2)+y_2\frac{F'(x_2)}{F(x_2)}(x-x_1)\right)\right)$$
and
$$\nabla^-\ :\ b=\frac{\lambda^2-1}{\lambda^2+1}(x_1-x_2)\ \ \ \text{and}\ \ \ A(x)=-C(x)=$$
$$\left.\frac{1}{2(x_1-x_2)^2}\frac{\lambda}{\lambda^2+1}
\right((y_1+y_2)(4x^3-6(x_1+x_2)x^2+12x_1x_2x)-6x_1x_2(x_2y_1+x_1y_2)$$
$$\left.+2(x_2^3y_1+x_1^3y_2)-(x_1-x_2)(x-x_1)(x-x_2)\left(y_1\frac{F'(x_1)}{F(x_1)}(x-x_2)-y_2\frac{F'(x_2)}{F(x_2)}(x-x_1)\right)\right)$$
This provides two ``universal connections'' over the parameter space $X\times X\times\P^1_\lambda$ which are each invariant
under $\sigma_{1/z}$ and $\langle\sigma_{12},\sigma_{\iota},\sigma_{iz}^2\rangle$, but permuted by $\sigma_{iz}$.
Taking the barycenter of these two connections for each parameter $(\underline{P}_1,\underline{P}_2,\lambda)$ yields a fully invariant section
$$\nabla_0:=\frac{\nabla^++\nabla^-}{2}$$
whose coefficients are given by
{\Large
\begin{equation}\label{TheSection}
\begin{array}{ccrl}
\frac{b_0y}{(x-x_1)(x-x_2)}&:=&\frac{\lambda^4+1}{\lambda^4-1} &\left(\frac{1}{x-x_1}-\frac{1}{x-x_2}\right)\vspace{.2cm}\\
\frac{A_0(x)}{(x-x_1)(x-x_2)}&:=&\frac{\lambda}{\lambda^4-1} &\left\{\left(\frac{y_2}{x-x_2}-\frac{\lambda^2y_1}{x-x_1}\right)+\frac{(\lambda^2y_1-y_2)(2x-x_1-x_2)}{(x_1-x_2)^2}\right.\vspace{.1cm}\\&&&\left.-\frac{\lambda^2y_1\frac{F'(x_1)}{F(x_1)}(x-x_2)+y_2\frac{F'(x_2)}{F(x_2)}(x-x_1)}{2(x_1-x_2)}\right\}\vspace{.2cm}\\\frac{C_0(x)}{(x-x_1)(x-x_2)}&:=&\frac{\lambda}{\lambda^4-1}&\left\{\left(\frac{\lambda^2y_2}{x-x_2}-\frac{y_1}{x-x_1}\right)+\frac{(y_1-\lambda^2y_2)(2x-x_1-x_2)}{(x_1-x_2)^2}\right.\vspace{.1cm}\\
&&&\left.-\frac{y_1\frac{F'(x_1)}{F(x_1)}(x-x_2)+\lambda^2y_2\frac{F'(x_2)}{F(x_2)}(x-x_1)}{2(x_1-x_2)}\right\}.\end{array}
\end{equation}
}

\begin{prop}The induced rational section 
$$\nabla_0:\BUN(X)\dashrightarrow\CON(X)$$
is Lagrangian, and moreover regular over the open set of stable bundles.
\end{prop}

\begin{proof} This connection is well-defined provided that $\lambda^4\not=1$ and $x_2\not=x_1$. We get a universal connection 
for all stable bundles. Indeed, we first check that all stable bundles off odd Gunning planes
are covered by the open subset where the connection $\nabla_0$ is well-defined:
$$X\times X\times\P^1_{\lambda}\setminus\left(\{\lambda^4=1\}\cup\{x_1=x_2\}\right)\ \twoheadrightarrow\ 
\P^3_{\mathrm{NR}}\setminus\left(\Kum(X)\cup\Pi_{w_0}\cup\cdots\cup\Pi_{w_\infty}\right).$$
We thus get a rational section $\nabla_0:\BUN(X)\dashrightarrow\CON(X)$ which is holomorphic
over stable bundles, off odd Gunning planes. We can check that it actually extends holomorphically along odd Gunning planes. 
It is sufficient to extend it outside intersections of odd Gunning planes since those form a codimension $2$ subset.
The Gunning plane $\Pi_{w_0}$ comes from the indeterminacy locus $\{w_0\}\times X\times \{0\}$ of the map
$X\times X\times\P^1_{\lambda}\to\P^3_{\mathrm{NR}}$. Precisely, a generic element of  $\Pi_{w_0}$
is obtained as follows. We first renormalize $z=w/\lambda$ so that parabolic directions become
$$(\lambda_{\underline{P}_1},\lambda_{\iota(\underline{P}_1)},\lambda_{\underline{P}_2},\lambda_{\iota(\underline{P}_2)})
=(\lambda^2,-\lambda^2,1,-1)$$
and then make the first two parabolic tending to $0$ while $y_1\to 0$ with some fixed slope $\frac{\lambda^2}{y_1}=c$.
The limiting connection has now a double pole at $w_0$, which disappears after two elementary transformations.

Finally, that this section is Lagrangian directly follows from straightforward verification. Precisely,
following \cite{LoraySaito}, in coordinates $(\boldsymbol a,\boldsymbol b)$ defined 
by coefficients $\boldsymbol a=(a_0:a_1:a_2:a_3)$ of $A_0(x)$ defined in (\ref{TheSection}) and Bertram coefficients
$\boldsymbol b=(b_0:b_1:b_2:b_3)$ defined in Proposition \ref{PropTyurinToBertram},
the symplectic form is defined by
$$\omega=d\eta\ \ \ \text{with}\ \ \ \eta=\frac{a_0db_0+a_1db_1+a_2db_2+a_3db_3}{a_0b_0+a_1b_1+a_2b_2+a_3b_3}.$$
If we pull-back the $1$-form $\eta$ by the map 
$$(\lambda,x_1,y_1,x_2,y_2)\mapsto(\boldsymbol a,\boldsymbol b)$$
then we get the zero $1$-form.
\end{proof}

\begin{rem}Over the open set of stable bundles, the natural map $\Con\to\Bun$ is a locally trivial Lagrangian fibration,
which is also an affine $\mathbb A^3$-fiber bundle. Over an affine open set, any affine bundle reduces to a vector bundle,
namely its linear part $\mathrm{T}^*\BUN^{s}$. The existence of Lagrangian section (regular over the open set) shows that the reduction
is actually symplectic with respect to the Liouville symplectic structure on $\mathrm{T}^*\BUN^{s}$.
\end{rem}

\begin{rem} There are precisely two Higgs fields invariant under $\sigma_{1/z}$: 
$$(x-x_1)(\lambda^2z^2-1)\frac{dx}{y}\ \ \ \text{and}\ \ \ (x-x_2)(z^2-\lambda^2)\frac{dx}{y}.$$
They are also permuted by $\sigma_{iz}$ and invariant under  $\langle\sigma_{12},\sigma_{\iota},\sigma_{iz}^2\rangle$. We obtain a basis of the space of Higgs bundles by adding for example
$\nabla^+-\nabla^-$.
\end{rem}



\section{Application to isomonodromic deformations}\label{sec:Garnier}

Our construction of the stack $\CON^* (X)$ allows us to vary the parameter $(r,s,t)$ defining the base curve $$X : y^2=x(x-1)(x-r)(x-s)(x-t)$$ in 
$$T:=\{(r,s,t)\in \mathbb{C}^3 ~|~r,s,t \neq 0,1 , ~r\neq s,~r\neq t, ~s\neq t\}$$
in order to obtain a family $\mathcal{M} \to T$ such that $\mathcal{M}|_{(r,s,t)} =\CON^* (X_{(r,s,t)})$. Roughly speaking, $\mathcal{M}$ is the moduli space of triples $(X,E,\nabla)$, where $X$ is a curve of genus $2$, $E$ is a rank $2$ vector bundle with trivial determinant bundle and $\nabla$ a holomorphic trace free connection on $E$ with either abelian (but non trivial) or irreducible monodromy. 
Locally in the variable $(r,s,t)\in T$, isomonodromic deformations are fibres of the monodromy map, defining an isomonodromic foliation $\mathcal{F}_{\mathrm{iso}}$ on $\mathcal{M}$. Note that an analytic family of connections over genus 2 curves  with contractible parameter space  is an isomonodromic deformation if and only if the connection is integrable and our isomonodromic foliation is thus defined by the integrability condition.
Our aim in this section is to express explicitly this isomonodromic deformation, via the corresponding moduli space 
$\underline{\mathcal{M}} \to T$ such that $\underline{\mathcal{M}}|_{(r,s,t)} =\CON^* (X_{(r,s,t)}/\iota)$. The integrability condition there is equivalent to a Garnier system. We then prove that the isomonodromic foliation is transversal to the locus of unstable bundles in $\mathcal{M}$ and give a geometric interpretation of this result. 

\subsection{Darboux coordinates}
We will use the notations of Section \ref{SecAtlasCon}.
The classical Darboux coordinates with respect to the symplectic form 
$\omega=\mathrm{d}z_r\wedge \mathrm{d}c_r+\mathrm{d}z_s\wedge \mathrm{d}c_s+\mathrm{d}z_t\wedge \mathrm{d}c_t$
on the Canonical chart $U_0$ are defined as follows. The vector  $e_1={}^t(1,0)$
becomes an eigenvector of the matrix connection for $3$ points $x=q_1,q_2,q_3$ (counted with multiplicity),
namely at the zeroes of the $(2,1)$-coefficient of the matrix connection:
\begin{equation}\label{Equaqk}
-\rho  + \sum_{i\in\{r,s,t\}}c_i\frac{(z_i-i)x-i(z_i-1)}{x-i}=\left(-\rho+ \sum_{i=1}^3c_i(z_i-t_i)\right)\frac{\prod_{k=1}^3(x-q_k)}{\prod_{i\in\{r,s,t\}}(x-i)}
\end{equation}
At each of the three solutions $x=q_k$ of (\ref{Equaqk}), the eigenvector $e_1={}^t(1,0)$ is associated to the eigenvalue
\begin{equation}\label{Equapk}
p_k:=-\frac{\rho}{q_k-1}+\sum_{i\in\{r,s,t\}}c_iz_i\left(\frac{1}{q_k-1}-\frac{1}{q_k-i}\right).
\end{equation}
The equations (\ref{Equaqk}) and (\ref{Equapk}) allow us to express our initial variables 
$(z_r,z_s,z_t,c_r,c_s,c_t)$ as rational functions of new variables $(q_1,q_2,q_3,p_1,p_2,p_3)$ as follows.

Set $\Delta=(q_1-q_2)(q_2-q_3)(q_3-q_1)$ and 
$$\Lambda=\rho+ \sum_{\{k,l,m\}=\{1,2,3\}} \frac{p_k(q_k-r)(q_k-s)(q_k-t)}{(q_k-q_l)(q_k-q_m)}.$$
For $i\in\{r,s,t\}$, denote 
$$\Lambda_i:=\Lambda\vert_{i=1}$$
the rational function obtained by setting $i=1$ in the expression of $\Lambda$.
Then we have, for $\{i,j,k\}=\{r,s,t\}$
\begin{equation}\label{Equacizi}
c_i=-\frac{(q_1-i)(q_2-i)(q_3-i)}{i(i-1)(i-j)(i-k)}\Lambda\ \ \ \text{and}\ \ \ 
z_i=i\frac{\Lambda_i}{\Lambda}.
\end{equation}
The rational map
\begin{equation}\label{Canonical6Cover}
\Pi:\C^6_{q,p}\dashrightarrow \C^6_{z,c}=U_0
\end{equation}
has degree $6$: the (birational) Galois group of this map is the permutation group on indices $k=1,2,3$
for pairs $(q_k,p_k)$.
 In these new coordinates, the symplectic form writes
$$\omega=\sum_{i\in\{r,s,t\}}dz_i\wedge dc_i=\sum_{k=1}^3dq_k\wedge dp_k.$$

We deduce a new atlas of ${\mathcal{M}}$ with charts given locally in $T$ by
\begin{equation}\label{DarbouxPhaseSpace}
\begin{matrix}
T\times \C^6_{q,p}&\stackrel{\Pi}{\dashrightarrow}&T\times \C^6_{z,c}\subset \underline{\mathcal{M}}&\stackrel{\Phi}{\dashrightarrow}&\mathcal{M}\\
((r,s,t),q,p)&\stackrel{6:1}{\mapsto}&((r,s,t),z,c)
\end{matrix}
\end{equation} 
of $\underline{\mathcal{M}}$ and ${\mathcal{M}}$ respectively, each endowed again with a universal family of connections 
 $$(X_{r,s,t},E_{r,s,t,z_1,z_2,z_3},\nabla_{r,s,t,z_1,z_2,z_3,c_1,c_2,c_3}).$$

\subsection{Hamiltonian system}

For $i\in \{r,s,t\}$ define $H_i$ by
$$i(i-1)\prod_{j\in\{r,s,t\}\setminus \{i\}}(j-i)\cdot H_i:=$$
$$\sum_{ l=1}^3\frac{\prod_{k\not= l}(q_k-i)}{\prod_{k\not= l}(q_k-q_ l)}F(q_ l)\left(p_ l^2-G(q_ l)p_ l+\frac{p_ l}{q_ l-i}\right)
\ \ \ +\ \ \ \rho(\rho+\kappa_\infty)\prod_{ l=1}^3(q_ l-i), $$
where $F(x)=x(x-1)(x-r)(x-s)(x-t)$ and $G(x)=\frac{\kappa_0}{x}+\frac{\kappa_1}{x-1}+\frac{\kappa_{r}}{x-r}+\frac{\kappa_{s}}{x-s}+\frac{\kappa_{t}}{x-t}$.
Then, assuming $\kappa_i\not \in \mathbb Z$ for any $i\in\{0,1,r,s,t,\infty\}$, a local analytic map 
$$\chi:(r,s,t)\mapsto(q_1,q_2,q_3,p_1,p_2,p_3)$$
induces an isomonodromic deformation of the connection (\ref{UniversalConnectionKappai}) if, and only if, 
\begin{equation}\label{GarnierHamiltonian}
\frac{\partial q_k}{\partial  i}=\frac{\partial H_i}{\partial  p_k}\ \ \ \text{and}\ \ \ \frac{\partial p_k}{\partial  i}=-\frac{\partial H_i}{\partial  q_k}\ \ \ \forall i\in \{r,s,t\}, ~k\in \{1,2,3\}.
\end{equation}
In other words, the kernel of the $2$-form
$$\Omega=\sum_{k=1}^3\mathrm{d}q_k\wedge \mathrm{d}p_k + \sum_{i\in \{r,s,t\}}\mathrm{d}H_i\wedge \mathrm{d}i,$$
defines a $3$-dimensional (singular holomorphic) foliation $\underline{\mathcal{F}}_{\mathrm{iso}}$ on $\underline{\mathcal{M}}$ which is transversal to $\{(r,s,t)=\mathrm{const}.\}$ and $\chi$ parametrizes
a leaf of this foliation. We will call it the isomonodromy foliation in the sequel. The tangent space to the foliation is defined by the $3$ vector fields $V_r, V_s, V_t$ given by
\begin{equation}\label{IsomonVectFields}
V_i:=\frac{\partial}{\partial i}+\sum_{k=1}^3\left(\frac{\partial H_i}{\partial p_k}\right)\frac{\partial}{\partial q_k}-\sum_{k=1}^3\left(\frac{\partial H_i}{\partial q_k}\right)\frac{\partial}{\partial p_k}.
\end{equation}
Note that the polar locus of these vector fields is given by $(q_1-q_2)(q_2-q_3)(q_1-q_3)=0$, namely the critical
locus of the map (\ref{Canonical6Cover}).
The induced (singular holomorphic) foliation on ${\mathcal{M}}$
will be called the isomonodromy foliation $\mathcal{F}_{\mathrm{iso}}$ in the sequel.

\subsection{Transversality to the locus of Gunning bundles}\label{SecGunning}

\begin{thm}\label{Thm:TransvGunningBundle}
For any even theta-characteristic $\vartheta$, the locus $\{(X,E_\vartheta,\nabla) \in\mathcal{M}\}$ of connections on the Gunning bundle $E_\vartheta$, 
is transversal to the isomonodromy foliation $\mathcal{F}_{\mathrm{iso}}$.
\end{thm}

\begin{proof}
Up to permutation, we can assume $$\vartheta=\mathcal{O}([w_{t_1}]+[w_{t_2}]-[w_{t_3}])=\mathcal{O}([w_0]+[w_1]-[w_\infty]).$$
According to Section \ref{Gunnbdle}, the two pre-images under $\Phi$ of $E_\vartheta$ are $(\underline E, \underline \p)$ with $\underline E=\OOP(-2)\oplus \OOP(-1)$, where
\begin{itemize}
\item[$\bullet$] The parabolics $\p_{i}$ are given by the total space of the fibres of $\OOP(-1)\subset \underline E$ for all $i\in \{r,s,t\}$. Moreover, there is a line subbundle $\OOP(-2)\hookrightarrow \underline E$ such that $\p_0$ and $\p_\infty$ are given by the the corresponding fibres, and $\p_1$ lies on neither of these two line subbundes. 
\item[$\bullet$] The parabolics $\p_{i}$ are given by the total space of the fibres of $\OOP(-1)\subset \underline E$ for all $i\in \{0,1,\infty\}$. Moreover, there is a line subbundle $\OOP(-2)\hookrightarrow \underline E$ such that $\p_{r}$ and $\p_{t}$ are given by the the corresponding fibres, and $\p_{s}$ lies on neither of these two line subbundes. 
\end{itemize}

 \begin{figure}[!h] \centering
 \resizebox{140mm}{!}{\input{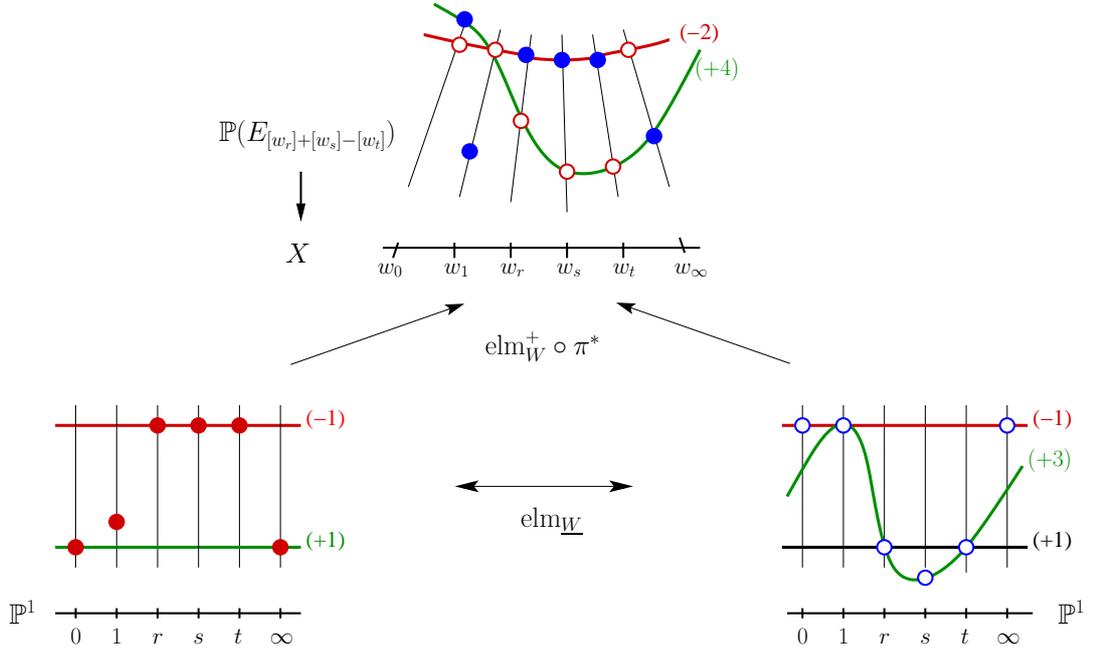}}      
 \caption{The two  parabolic bundles corresponding to $E_\vartheta$ with  $\vartheta=[w_{r}]+[w_{s}]-[w_{t}]$ under the lifting map $\Phi$.}\label{GunningPicture}
\end{figure} 

For any holomorphic connection on $E_\vartheta$, the  parabolic structure on $\underline{E}= \OOP(-1)\oplus \OOP(-2)$ can be normalized to one of the two above. Up to the permutation $\sigma = \left(\begin{array}{cccccc}0&1&r&s&t&\infty\\r&s&0&1&\infty&t\end{array}\right)$, it is enough to consider the first configuration. Note that after a M\"obius-transformation in the basis, the configuration after permutation corresponds to the configuration before permutation, but for different values of $(r,s,t)\in T$.  
We are led to the  parabolic structure $$\begin{matrix}x=0 & x=1 & x=r & x= s & x= t & x=\infty\\
\begin{pmatrix}0\\ 1\end{pmatrix} &\begin{pmatrix}1\\ 1\end{pmatrix} &\begin{pmatrix}1\\ 0\end{pmatrix} &\begin{pmatrix}1\\ 0\end{pmatrix} &\begin{pmatrix}1\\ 0\end{pmatrix} & \OOP(-1)
\end{matrix}$$
on   $\underline{E}= \OOP(-1)\oplus \OOP(-2)$, 
in other words $(z_r,z_s,z_t)=(\infty,\infty,\infty)$. This  parabolic structure is non visible in the canonical chart $U_0$, the Twisted chart and Switched charts for $J\neq \{r,s,t\}$. Moreover, we can discard 
Permuted charts by the argument given above and Galois involution charts because we do not need them to cover $\mathcal{M}.$ 
Therefore, the only chart containing the Gunning bundle $E_\vartheta$ we need to consider is the Switched chart for $J=\{r,s,t\}$, \emph{ i.e.} we consider the Darboux chart with respect 
to the parabolic structure attached to the $0$-eigenvalue over $x=r,s,t$. Consider now spectral data (\ref{RiemannScheme})
with
$$\kappa_0=\kappa_1=\kappa_\infty=\frac{1}{2}\ \ \ \text{and}\ \ \ 
\kappa_{r}=\kappa_{s}=\kappa_{t}=-\frac{1}{2}\ \ \ (\Rightarrow \rho=\frac{1}{2}).$$
Connections on our Gunning bundle $E_\vartheta$ then correspond to those parabolic connections with $z_r,z_s,z_t\in\C$ (finite) but having $e_1$ as eigenvector 
over $x=r,s,t$. This means that $c_r=c_s=c_t=0$ and we are just looking at the family $\Sigma$ 
defined by $\nabla_0$ when the parabolic data $z=(z_r,z_s,z_t)\in\C^3_z$ is arbitrary. 
We want to prove that the isomonodromic foliation is transversal to $$\Sigma:=\{((r,s,t),(z,c))\in T\times \C^6 ~|~c_i=0\}.$$
Locally in $T$,  
the linear subspace $\Sigma\in\C^9_{r,s,t,z,c}$ is the image of
the linear subspace $\Sigma^{\textrm{Darb}}\in\C^9_{r,s,t,q,p}$ defined by 
$$\Sigma^{\textrm{Darb}}=\{q_1=r,q_2=s,q_3=t\}.$$
Precisely, the map $\Pi:\C^9_{r,s,t,q,p}\dashrightarrow\C^9_{r,s,t,z,c}$ induces an affine transformation 
$$\Sigma^{\textrm{Darb}}\to\Sigma\ ;\ (r,s,t,p)\mapsto(r,s,t,z)$$ with $$z_r=r(2(r-1)p_1+1) , \ \ \ z_s=s(2(s-1)p_2+1) , \ \ \  z_t=t(2(t-1)p_3+1).$$

\begin{lem} There is a neighborhood of $\Sigma^{\textrm{Darb}}$ such that $\Pi$ restricted to this neighborhood is a local diffeomorphism. Moreover, 
 $\Sigma^{\textrm{Darb}}$ is sent surjectively onto $\Sigma$.
\end{lem}
\begin{proof}
We can check by direct computation from (\ref{Equaqk}) and (\ref{Equapk}) that
\begin{itemize}
\item the $c_i$'s have poles only at $q_k=q_l$, thus far from $\Sigma^{\textrm{Darb}}$,
\item the $z_i$'s have poles also not intersecting $\Sigma^{\textrm{Darb}}$,
\item the $q_k$'s can be inversely defined near $\Sigma$,
\item the $p_k$'s are then regular near $\Sigma$. 
\end{itemize}
For this last fact, we have to take special care with the indefinite terms $\frac{c_rz_r}{q_1-r}, $ $\frac{c_sz_s}{q_2-s}, $ $\frac{c_tz_t}{q_3-t}, $
occuring in formula  (\ref{Equapk}). For instance, for $p_1$, we can use the (unitary) equation for the $q_i$'s
 given by formula (\ref{Equaqk}) and then substitute 
$$\frac{c_rz_r}{q_1-r}=-\frac{c_rz_r(q_2-r)(q_3-r)}{Q(r)}$$
with $Q(x)=(x-q_1)(x-q_2)(x-q_3)$,
leading to $$p_1=-\frac{\rho}{q_1-1}+\sum_{i\in \{r,s,t\}}\frac{c_iz_i}{q_1-1} +\frac{z_r(q_2-r)(q_3-r)(\rho-\sum_{i\in \{r,s,t\}}c_i(z_i-i))}{r(r-1)(r-s)(r-t)}-\frac{c_sz_s}{q_1-s} -\frac{c_tz_t}{q_1-t}.$$ The right-hand-side is now well-defined and analytic in a neighborhood of $\Sigma$.
\end{proof}

We want to show that the isomonodromy foliation $\underline{\mathcal{F}}_{\mathrm{iso}}$ is transversal to $\Sigma$. By the previous lemma it is enough to 
prove the transversality of $\Sigma^{\textrm{Darb}}$ with the vector fields $V_i$ defined in (\ref{IsomonVectFields}).
Modulo the vector fields $$\frac{\partial}{\partial r}+\frac{\partial}{\partial q_1} ,
 \ \ \ \frac{\partial}{\partial s}+\frac{\partial}{\partial q_2} , \ \ \ \frac{\partial}{\partial t}+\frac{\partial}{\partial q_3}$$ and $$\frac{\partial}{\partial p_1} , \ \ \ \frac{\partial}{\partial p_2} , \ \ \ \frac{\partial}{\partial p_3}, $$
that are tangent to $\Sigma^{\textrm{Darb}}$, the vector fields $V_i$ are equivalent to
$$\tilde V_r=-\frac{\partial}{\partial q_1}+\sum_{k=1}^3\left(\frac{\partial H_r}{\partial p_k}\right)\frac{\partial}{\partial q_k}$$
$$\tilde V_s=-\frac{\partial}{\partial q_2}+\sum_{k=1}^3\left(\frac{\partial H_s}{\partial p_k}\right)\frac{\partial}{\partial q_k}$$
$$\tilde V_t=-\frac{\partial}{\partial q_3}+\sum_{k=1}^3\left(\frac{\partial H_t}{\partial p_k}\right)\frac{\partial}{\partial q_k}$$
But the corresponding matrix writes (with $I$ denoting the identity $3$-by-$3$ matrix)
$$(\tilde V_r,\tilde V_s,\tilde V_t)=\left( \frac{\partial H_i}{\partial p_k}\right)_{i\in \{r,s,t\},k\in \{1,2,3\}} -I$$
and we obtain $$(\tilde V_1,\tilde V_2,\tilde V_3)|_{(q_1,q_2,q_3)=(r,s,t)}=\frac{1}{2}\cdot I,$$
which is clearly invertible. This proves transversality of the isomonodromic foliation $\underline{\mathcal{F}}_{\mathrm{iso}}$ with the locus $\Sigma$ of our even Gunning bundle in $\underline{\mathcal{M}}$. The transversality of ${\mathcal{F}}_{\mathrm{iso}}$ with the locus $\Phi(\Sigma)$ of our even Gunning bundle in ${\mathcal{M}}$ then follows from the fact that the Gunning bundle is not in the ramification locus of the action of the Galois involution in the considered Switched chart : the two-fold cover $(r,s,t,z,c)\stackrel{2:1}{\dashrightarrow}\Phi(r,s,t,z,c)$
is a local diffeomorphism in a neighborhood of $\Sigma$.
\end{proof}

\subsection{Projective structures and Hejhal's theorem}
\subsubsection{Projective structures}
A projective structure on the Riemann surface $X$ is the data of an atlas of charts $f_i:U_i\to\mathbb P^1$
(holomorphic diffeomorphisms) such that transition charts $\varphi_{ij}:=f_j\circ f_i^{-1}$ are Moebius transformations
in restriction to their set of definition: $\varphi_{ij}\in\PGL(\C)$. Two projective atlases define the same projective 
structure if their union (concatenation) also forms a projective atlas. This notion goes back to the works of 
Schwarz on the hypergeometric equation where the projective charts are locally defined as quotients of independant solutions
of a given $2^{\text{nd}}$-order differential equation $u''+f(x)u'+g(x)u=0$; equivalently, after normalization
 $u''+\frac{\phi(x)}{2}u=0$, $\phi=g-\frac{f'}{2}-\frac{f^2}{4}$, local projective charts are solutions of the differential
equation $\{f,x\}=\phi$ where $\{f,x\}=\left(\frac{f''}{f'}\right)'-\frac{1}{2}\left(\frac{f''}{f'}\right)^2$ is the Schwarzian derivative 
with respect to $x$ (see \cite[chapter VIII]{Uniformisation} for this point of view). 
In the hypergeometric case, the projective structure has singular points at poles of the differential equation.
In our case, we can define a (non singular) projective structure on the curve 
$$X_{(r,s,t)}\ :\ \{y^2=F(x)\},\ \ \ F(x)=x(x-1)(x-r)(x-s)(x-t)$$
by a unique differential equation of the form
$$u''+\left(\frac{1}{2}\frac{F'}{F}\right)u'+\left(\frac{x^3+b_2x^2+b_1x+b_0}{2F}\right)$$
(where $u'$ and $F'$ mean partial derivative with respect to $x$). When we let the complex structure $(r,s,t)$
of the curve vary in $T\subset \C^3$, where 
$$T=\{(r,s,t)\in \mathbb{C}^3 ~|~r,s,t \neq 0,1 , ~r\neq s,~r\neq t, ~s\neq t\},$$ the space of projective structures identifies with 
$$T\times \C^3_{b},$$ where
$\C^3_{b}= \{(b_0,b_1,b_3)\in \C^3\}$. 
Following \cite{GunningCoord}, the data of a global non sigular projective structure on $X$ is also equivalent 
to the data of a $\SL$-connection on a Gunning bundle $(E_\vartheta,\nabla)$. In fact, up to the choice of $E_\vartheta$,
we have a one-to-one correspondance between connections $\nabla$ on $E_\vartheta$ and projective structures on $X$.

The monodromy of a projective structure is by definition the monodromy of the connection $\nabla$,
of the $2^{\text{nd}}$-order differential equation, or of any local projective chart (that can be analytically
continuated along any loop). After lifting to the Teichm\"uller space, namely the universal cover $\tilde T\to T$, 
the monodromy map 
$$\mathrm{Mon}\ :\ \tilde T\times \C_b^3 \to \Hom(\pi_1(X,w),\SL)/_{\PGL}.$$
is well-defined and analytic.
\subsubsection{Hejhal's theorem}
A problem which goes back to the work of Poincar\'e on Fuchsian functions was to decide which kind of 
representation $\Hom(\pi_1(X,w),\SL)/_{\PGL}$ arise as the monodromy of a projective structure, 
i.e. as monodromy of $(E_\vartheta,\nabla)$,
maybe deforming the complex structure of $X$. Counting dimensions, we get $3$ parameters $(r,s,t)$
for the curve and then
$3$ other parameters $b=(b_0,b_1,b_2)$ for the projective structure on the curve. Since the dimension of representations
space is $6$, one expect to realize most of them as monodromy. This was indeed proved in \cite{GKM}: 
a representation can be realized if, and only if, it is not conjugated to a unitary representation, 
and it has Zariski dense image in $\SL(\C)$. Some time earlier, D. A. Hejhal proved in \cite{Hejhal} 
a local version:

\begin{thm}[Hejhal]The monodromy map 
$$\mathrm{Mon}\ :\ \tilde T\times \C_b^3 \to \Hom(\pi_1(X,w),\SL)/_{\PGL}$$
is a local diffeomorphism.
\end{thm}

Going back to the isomonodromy point of view, consider the Gunning bundle $E_\vartheta$ with $\vartheta =\OX{[w_0]+[w_1]-[w_\infty]}$.
The locus of projective structures is given by the subspace $\Sigma$ of those triples 
$(X,E,\nabla)$ with $E=E_\theta$ in the total moduli stack $\mathcal{M}$. 
The leaves of the isomonodromy foliation are locally defined as the fibres of the monodromy map $\mathrm{RH}$. That the monodromy map $\mathrm{RH}\vert_\Sigma$ restricted to the locus $\Sigma$ of projectives structures 
is a local diffeomorphism is therefore equivalent to saying that the isomonodromic foliation is transversal to $\Sigma$.
With \ref{Thm:TransvGunningBundle}, we have therefore provided a new proof of Hejhal's theorem.

\begin{rem}The topological transversality of $\Sigma$ with the isomonodromy leaves, or equivalently
the openess of the monodromy map, also follows from the main result in \cite{TheseViktoria}. Indeed, the projective structure induced on $X$ by taking the cyclic vector $\OOP$
has no apparent singular point (since all $q_i=t_i$) and cannot be deformed isomonodromically 
(see \cite[section 1.2]{TheseViktoria}).
\end{rem}

\end{document}